\documentclass[a4paper,11pt]{article}
\usepackage{amsmath, amsfonts, amscd, amssymb, amsthm, enumerate, xypic, float, makeidx}

\def\Fbar{\overline{\mathbb{F}}}

\def\id{\mathrm{id}}
\def\ad{\mathrm{ad}}

\newcommand{\Mat}{\operatorname{M}}

\newcommand{\GL}{\operatorname{GL}}

\newcommand{\disc}{\operatorname{disc}}
\newcommand{\Ker}{\operatorname{Ker}}

\newcommand{\Aut}{\operatorname{Aut}}
\newcommand{\Root}{\operatorname{Root}}
\newcommand{\GA}{\operatorname{GA}}
\newcommand{\BAut}{\operatorname{BAut}}

\newcommand{\End}{\operatorname{End}}
\newcommand{\Irr}{\operatorname{Irr}}
\newcommand{\Inn}{\operatorname{Inn}}

\newcommand{\Vect}{\operatorname{span}}

\newcommand{\im}{\operatorname{Im}}

\newcommand{\car}{\operatorname{char}}
\newcommand{\tr}{\operatorname{tr}}

\newcommand{\SB}{\operatorname{SB}}
\newcommand{\SSB}{\operatorname{SSB}}
\newcommand{\rk}{\operatorname{rk}}

\newcommand{\Hom}{\operatorname{Hom}}

\renewcommand{\setminus}{\smallsetminus}

\def\F{\mathbb{F}}
\def\K{\mathbb{K}}
\def\R{\mathbb{R}}

\def\Q{\mathbb{Q}}
\def\N{\mathbb{N}}
\def\Z{\mathbb{Z}}

\renewcommand{\L}{\mathbb{L}}

\def\calA{\mathcal{A}}
\def\calB{\mathcal{B}}
\def\calC{\mathcal{C}}
\def\calD{\mathcal{D}}
\def\calE{\mathcal{E}}

\def\calH{\mathcal{H}}

\def\calM{\mathcal{M}}
\def\calN{\mathcal{N}}

\def\calQ{\mathcal{Q}}

\def\calU{\mathcal{U}}

\def\calW{\mathcal{W}}

\def\lcro{\mathopen{[\![}}
\def\rcro{\mathclose{]\!]}}

\theoremstyle{definition}
\newtheorem{Def}{Definition}[section]
\newtheorem{Not}[Def]{Notation}

\theoremstyle{plain}
\newtheorem{theo}{Theorem}[section]
\newtheorem{prop}[theo]{Proposition}
\newtheorem{cor}[theo]{Corollary}
\newtheorem{lemma}[theo]{Lemma}

\theoremstyle{remark}
\newtheorem{Rems}{Remarks}[section]
\newtheorem{Rem}[Rems]{Remark}

\title{The free Hamilton algebra}
\author{Cl\'ement de Seguins Pazzis\footnote{Universit\'e de Versailles Saint-Quentin-en-Yvelines, Laboratoire de Math\'ematiques
de Versailles, 45 avenue des \'Etats-Unis, 78035 Versailles cedex, France}
\footnote{e-mail address: clement.de-seguins-pazzis@ac-versailles.fr}}

\makeindex

\begin{document}

\thispagestyle{plain}

\maketitle

\begin{abstract}
Let $p$ and $q$ be monic polynomials with degree $2$ in one variable and coefficients in a field $\F$.
The free Hamilton algebra of the pair $(p,q)$ is the free noncommutative algebra in two generators $a$ and $b$ subject only to the relations $p(a)=0=q(b)$.
Free Hamilton algebras are models of free products of two $2$-dimensional algebras over $\F$.
They can be viewed as the most elementary nontrivial noncommutative algebras over fields.

It has been recently observed that the free Hamilton algebra has surprising connections with quaternion algebras.
Here, we exploit these connections to investigate its zero divisors, group of units, maximal ideals, finite-dimensional subalgebras, and its automorphism group.
\end{abstract}



\begin{center}
\emph{Dedicated to Tom Laffey.}
\end{center}

\tableofcontents

\newpage
\centerline{\textbf{List of notation}}

\bigskip

\def\8[#1;#2]{\hbox to \textwidth{#1 \hfill p. #2}\vskip2truemm}

\8[$p$, $q$;\pageref{page:intro}]
\8[$\calW_{p,q}$;\pageref{page:intro}]
\8[$a$, $b$;\pageref{page:intro}]
\8[$\Irr(\F)$;\pageref{page:intro}]
\8[$\mathrm{P} G$ ($G$ a subgroup of $\calW_{p,q}$;\pageref{page:projectivegroup}]
\8[$\BAut(\calW_{p,q})$ (subgroup of basic automorphisms);\pageref{page:basicautomorphisms}]
\8[$\Aut_C(\calW_{p,q})$ (subgroup of $C$-automorphisms);\pageref{page:Cautomorphisms}]
\8[$\BAut_+(\calW_{p,q})$ (subgroup of positive basic automorphisms);\pageref{page:positivebasic}]
\8[$(-)^\star$ (adjunction mapping);\pageref{page:adjunction}]
\8[$\omega$;\pageref{page:omega}]
\8[$\tr$ (trace mapping);\pageref{page:trace}]
\8[$\langle -,-\rangle$ (inner product mapping);\pageref{page:innerproduct}]
\8[$N$ (norm mapping);\pageref{page:norm}]
\8[$\Lambda_{p,q}$ (fundamental polynomial);\pageref{page:fundamentalpolynomial}]
\8[$\mathfrak{J}$ (fundamental ideal);\pageref{page:fundamentalideal}]
\8[$C$ (center of $\mathcal{W}_{p,q}$);\pageref{page:center}]
\8[$\overline{\calW_{p,q}}$ (extension of $\calW_{p,q}$);\pageref{page:extension}]
\8[$\deg(r)$ (degree of a central element);\pageref{page:degree}]
\8[$d_\calA$ (distance to a basis subalgebra);\pageref{page:distance}]
\8[$\delta$ (absolute distance of a quadratic element);\pageref{page:distance}]
\8[$\calW_{p,q,[r]}$ (specialization of $\calW_{p,q}$ mod $(r)$);\pageref{page:specialization}]
\8[$\tr_r$ (specialized trace);\pageref{page:specializedtraceandnorm}]
\8[$\langle -,-\rangle_r$ (specialized inner product);\pageref{page:specializedtraceandnorm}]
\8[$N_r$ (specialized norm);\pageref{page:specializedtraceandnorm}]
\8[$R_r$;\pageref{page:radicalmodr}]
\8[$\mathfrak{R}_r$;\pageref{page:radicalmodr}]
\8[$\mathfrak{J}_r$;\pageref{page:superfundamental}]
\8[$\mathcal{N}(\calA)$;\pageref{page:nilpotentcone}]
\8[$\mathcal{H}_{n+2}$;\pageref{page:Hn+2}]
\8[$\mathcal{U}_{n+1}$;\pageref{page:Hn+2}]
\8[$\mathcal{U}(\beta,V)$;\pageref{page:degeneratesubalg}]
\8[$\alpha^\sharp$;\pageref{page:alphadiese}]
\8[$\mathcal{H}(\alpha,V)$;\pageref{page:degeneratesubalgidem}]
\8[$\Phi_\star$ (pseudo-adjunction);\pageref{page:pseudoadjunction}]
\8[$\Irr_{\Lambda_{p,q}}$;\pageref{page:irrdivisors}]
\8[$n_r(\Phi)$;\pageref{page:rexponents}]
\8[$\varepsilon_r(\Phi)$;\pageref{page:rsignature}]
\8[$\varepsilon$ (full signature);\pageref{page:fullsignature}]
\8[$\mathrm{SB}(\alpha)$;\pageref{page:semibasicsubgroupelement}]
\8[$\mathrm{SB}(\calA)$;\pageref{page:semibasicsubgroupalgebra}]
\8[$\mathrm{SSB}(\calA)$;\pageref{page:specialsemibasicsubgroup}]

\section{Introduction}

\subsection{Introduction for lay readers: A words play}

Our story starts with a fairly simple game of words. We take two letters $a,b$,
and we construct words in $a$ and $b$ of arbitrary length, requiring that the two letters $a$ and $b$ are never adjacent in a word, so $abababa$ and $bab$ are allowed, but not $baa$.
We also allow the empty word $()$.
Then we can take pondered chains (also known as linear combinations) of those words with real coefficients, like $2.abab-\sqrt{2}.ba+\pi.b+5.()$.
So we can add chains and multiply them with coefficients.
And next we consider a bilinear multiplication $*$ of these linear combinations, defined on the words as follows:
whenever we have two words, we try to concatenate them, but if this concatenation makes a pair $aa$ or $bb$ appear, we drop one of the occurrences of these letters,
so $abab$ and $ba$ multiply as $ababa$, not as $ababba$. Thus
$$(2.abab+3.aba)* (ab+b)=2.ababab+8.abab.$$
Then we can ask very simple questions: can we easily detect the zero divisors, i.e., the nonzero linear combinations of words $x$ for which there exists
another nonzero linear combination of words $y$ such that $x*y=0.()$ or $y*x=0.()$? Can we also easily detect the units, i.e., the linear combinations of words $x$ for which there exists another linear combination of words $y$ such that $x*y=()=y*x$? And our fellow algebraists will ask deeper questions on this algebra: what are the maximal ideals? the algebraic elements? the automorphisms?

This was just a simple game, and we can consider a different set of rules. What if, instead of simplifying the product $aa$ as $a$, we decide to
delete any product of words that would make this subword formally appear, and likewise with $b$?
So, in the above we would have
$$(2.abab+3.aba)* (ab+b)=2.ababab+3.abab.$$
But we could also have a mix of rules for the two letters $a$ and $b$, requiring to replace $aa$ with $a$ and to delete all concatenations that make $bb$ appear.
This way, we would have
$$(2.abab+3.aba)* (ab+b)=2.ababab+6.abab.$$
Another interesting variation of the set of rules would have us, every time we meet $aa$ or $bb$ in product computations, delete
the two occurrences, and then repeat the operation until no two identical letters appear side by side. In that case
$$(2.abab+3.aba)* (ab+b)=2.ababab+2.aba+3.a+3.abab.$$
And finally some of us might enjoy playing under the so-called ``semi-Hamilton rule book", requiring that anytime $aa$ or $bb$ appears formally in a product, one both deletes it \emph{and} multiplies the coefficient of the resulting word by $-1$. Thus, e.g.,
$$(2.abab+3.aba)* (ab+b)=2.ababab-2.aba+3.a+3.abab.$$
The associative algebras we obtain in this way are, in some sense, the simplest examples of noncommutative algebras.
But, as incredible as it may sound, although the representation theory of the first one we have mentioned has been known for about a half-century, little was known until now on the \emph{deep} internal structure of those algebras. Wouldn't it be time for a change?

\subsection{The free Hamilton algebra}

We will now generalize the previous problem.
\label{page:intro}
Throughout, we fix an arbitrary field $\F$, possibly of characteristic $2$. We also fix an indeterminate $t$ and two monic polynomials $p(t)$ and $q(t)$ with degree $2$ in $\F[t]$
(we will simply say that they are \textbf{quadratic polynomials}).
We recall that the trace of $p(t)$ is the opposite of its coefficient in $t$, and we denote it by $\tr(p)$ or $\tr p$. The constant coefficient of $p$ is
denoted by $N(p)$, so that $p(t)=t^2-(\tr p) t+N(p)$.
We also denote by $\Irr(\F)$ the set of all irreducible monic polynomials in $\F[t]$. Whenever possible, we drop the parentheses to designate polynomials, so
$p(t)$ is simply written $p$, and so on.

Our central object of study here is the associative unital $\F$-algebra
$$\calW_{p,q}:=\F\langle \mathbf{a},\mathbf{b}\rangle/(p(\mathbf{a}),q(\mathbf{b})),$$
defined as the quotient algebra of the free associative algebra $\F\langle \mathbf{a},\mathbf{b}\rangle$ in two noncommuting generators $\mathbf{a},\mathbf{b}$
by the two-sided ideal generated by the elements $p(\mathbf{a})$ and $q(\mathbf{b})$. 
\index{Free Hamilton Algebra}
We have decided to call it the \textbf{free Hamilton algebra}\footnote{To this day, no name had been attached to it.} of the pair $(p,q)$, and the homage to the discoverer of quaternions is almost self-evident: if $p$ and $q$ have trace zero then the free Hamilton algebra would be the definition of a (generalized) quaternion algebra given by a distracted student who has forgotten the skew-commutation rule!

Throughout, $a$ and $b$ will denote the respective cosets of $\mathbf{a},\mathbf{b}$ in $\calW_{p,q}$, which we call the \textbf{basic generators},
\index{Basic generators}
and their generated subalgebras $\F[a]$ and $\F[b]$, which will turn out to have dimension $2$ and hence be isomorphic respectively to $\F[t]/(p)$ and $\F[t]/(q)$,
are called the \textbf{basic subalgebras} of $\calW_{p,q}$,
\index{Basic subalgebra}
 and we say that each one is \textbf{opposite} to the other one.
\index{Opposite (basic subalgebra)}
A vector of either $\F[a]$ or $\F[b]$ is called a \textbf{basic vector}, while an element of the subalgebra $\F$ is called \textbf{scalar}.
\index{Basic vector}
\index{Scalar element}
\label{page:equivalence}
For two vectors $x$ and $y$ of $\calW_{p,q}$, we write $x \sim y$ to mean that there exists $\lambda \in \F^\times$
such that $x=\lambda y$.

The algebra $\calW_{p,q}$ will shortly be seen to be isomorphic to the \emph{free product} $\F[a] * \F[b]$ of the associative $\F$-algebras $\F[a]$ and $\F[b]$.
And conversely, given $2$-dimensional algebras $\calA$ and $\calB$ over $\F$,
there exist monic quadratic polynomials $p_1$ and $q_1$ such that $\calA\simeq \F[t]/(p_1)$ and $\calB \simeq \F[t]/(q_1)$,
and hence $\calA * \calB$ appears to be isomorphic to $\calW_{p_1,q_1}$.
It follows that studying free Hamilton algebras is entirely equivalent to studying free products of two $2$-dimensional $\F$-algebras.

Our initial motivation for considering the free Hamilton algebra was its connection to issues in the representation theory of algebras. For example, if $p=q=t^2-t$, of which the reader of our introduction will recognize the first set of rules for the word game, classifying the linear representations of $\calW_{p,q}$ amounts to classifying pairs of idempotent operators of a finite-dimensional vector space up to conjugation by an automorphism, i.e., pairs of idempotent matrices over $\F$ up to simultaneous conjugation by an invertible matrix.
This is actually a special case of the Four Subspace Problem \cite{Brenner,GelfandPonomarev,Nazarova}, in which the quadruple $(V_1,V_2,V_3,V_4)$ of subspaces of a finite-dimensional vector space $V$ is required to satisfy $V=V_1\oplus V_2=V_3 \oplus V_4$
(the correspondence is as follows: any such $4$-tuple $(V_1,V_2,V_3,V_4)$ corresponds to the pair $(P,Q)$
of idempotents in $\End(V)$ such that $V_1=\im P$, $V_2=\Ker P$, $V_3=\im Q$ and $V_4=\Ker Q$).
More generally, when both $p$ and $q$ split with simple roots, the
linear representations of $\calW_{p,q}$ are naturally deduced from the ones of $\calW_{t^2-t,t^2-t}$: popular examples include
the case $p=q=t^2-1$ when $\car(\F) \neq 2$, which amounts to determining the linear representations of the infinite dihedral group
\cite{BermanBuzaki,Djokovic}. 
In fact, the classification of the linear representations of $\calW_{p,q}$ is entirely closed in the case where both $p$ and $q$ split:
the case where both $p$ and $q$ have a double root in $\F$ has been solved independently by Ringel \cite{Ringel}
and Bondarenko \cite{Bondarenko}, and the case where one splits with simple roots and the other one splits with a double root has been studied more recently 
by Bondarenko \cite{Bondarenko2}. The case where one of $p$ and $q$ is irreducible remains largely open, however, 
see nevertheless \cite{Gindi} for the very special case where $\F=\R$ and $p=q=t^2+1$.

The case where $p$ and $q$ are split with simple roots also stands out as a very special case in the general study of
free product of algebras: it is proved in \cite{Buchananetal} that, when $\F$ is algebraically closed,
a free product of (at least two) semi-simple $\F$-algebras is of tame representation type only if it is isomorphic to $\F^2 * \F^2$,
i.e., to $\calW_{t^2-t,t^2-t}$.

We also mention the recent series of articles \cite{dSPregular,dSPsum,dSPprod}, in which a full characterization, in terms of invariant factors, has been given of the endomorphisms $u$ of a finite-dimensional vector space that admit a decomposition into a sum $u=a+b$ where the summands $a$ and $b$ are endomorphisms that satisfy $p(a)=q(b)=0$, and a similar result was obtained for decompositions into products.

As we were working on these problems, the importance of the free Hamilton algebra and its internal structure gradually emerged, and to our bewilderment we discovered that we could find little systematic study of it in the literature.
There are mainly two sets of prior works. On the one hand, there is the systematic work undertaken by P.M. Cohn
\cite{CohnFIR,CohnFreeAssociativeII,CohnFreeAssociativeIII,CohnFreeProductSkewFields}
 in the 1960's and the 1970's on free products of algebras over a field (and even a skew field). Some of Cohn's results apply to the free Hamilton Algebra, but they are essentially concentrated in the special case where both $p$ and $q$ are irreducible.
On the other hand, there has more recently been specialized work on the free Hamilton algebra,
mostly limited to the special case $p=q=t^2-t$ (i.e., the two idempotents case, see \cite{Laffey,Weiss1,Weiss2}), and the only recent reference that considers the general case \cite{SZW} contains some important basic results but only scratches the structure of the free Hamilton algebra.
In particular, all these prior studies have entirely missed the connection between the free Hamilton algebra and quaternion algebras. Some elements of these connections have been laid out in the recent \cite{dSPregular}, with critical applications to the above representation problems, and here we will explore it much more systematically to obtain many new results on the free Hamilton algebra.

At this point, the reader might still want extra motivation for studying the free Hamilton algebra. We simply hope that the sheer beauty of the results will settle this issue, and to make our case even stronger, let us simply state the most remarkable results that will be proved in this work.

To start with, we recall that a zero divisor in a ring $R$ is a \emph{nonzero} $x \in R$ for which there exists a nonzero element $y \in R$ such that $xy=0_R$ or $yx=0_R$.
Our first result is actually a special case of a more general result of Cohn \cite{CohnFreeAssociativeII}, who proved that a free product of division rings over $\F$ has no zero divisor. We will give as many as two new proofs of this special case:

\begin{theo}[Zero Divisors Theorem]\label{theo:introzerodiv}
The algebra $\calW_{p,q}$ has zero divisors if and only if one of $p$ and $q$ splits.
\end{theo}

Hence $\calW_{p,q}$ has zero divisors if and only if at least one of the basic subalgebras $\F[a]$ and $\F[b]$ has a zero divisor, which is the trivial case.
In addition to giving new proofs of this result, we will give a very simple algorithm that detects whether a given element $x$ of
$\calW_{p,q}$ is a zero divisor, and provides a corresponding nonzero left-annihilator $y$ of $x$ if so.

The corresponding result for units (i.e., invertible elements of $\calW_{p,q}$) is more spectacular,
although part of it is also a special case of a general result of Cohn \cite{CohnFreeAssociativeII}.
The units that are basic are the elements of $\F[a]^\times \cup \F[b]^\times$, and we naturally call them the \textbf{basic units},
\index{Basic units}
and from these units we can of course create new units by multiplying the basic units.
The \textbf{monomial units} are the products of basic units, and we denote by $\calM_{p,q}$ the
\index{Monomial units}
subgroup of such units. It is not difficult to prove that
$\calM_{p,q}$ is naturally isomorphic to the amalgamated product
of the subgroups $\F[a]^\times$ and $\F[b]^\times$ over $\F^\times$, which means the following:
say that a formal product $x=\prod_{k=1}^n x_k$ of basic units is \textbf{reduced} when,
\index{Reduced decomposition}
for all $k \in \lcro 1,n-1\rcro$, $x_k$ and $x_{k+1}$ belong to distinct basic subalgebras
(which requires that no factor is scalar if $n>1$).
It is then fairly elementary to prove that in a reduced expression of a given monomial unit into a product of basic units,
each factor is uniquely determined up to multiplication with a nonzero scalar (see the start of Section \ref{section:units1} for details).

We can restate this result as follows:
given a subgroup $G$ of $\calW_{p,q}^\times$ that includes $\F^\times$, we define its \textbf{projective group}
\index{Projective group}
as $\mathrm{P}G:=G/\F^\times$.
\label{page:projectivegroup}
 The elements of $\F^\times$ are precisely the central units in $\calW_{p,q}$
(i.e., the units that commute with all the elements of $\calW_{p,q}$),
so the denomination of projective group definitely makes sense here.
Then $\mathrm{P}\calM_{p,q}$ is naturally isomorphic to the free product of the groups $\mathrm{P}\F[a]^\times$ and $\mathrm{P}\F[b]^\times$.

And then we may of course ask if all units are monomial. Here is the answer:

\begin{theo}[Weak Units Theorem]\label{theo:introunits}
Every unit of $\calW_{p,q}$ is monomial if and only if both $p$ and $q$ are irreducible.
\end{theo}

Cohn \cite{CohnFreeAssociativeII} proved more generally that, in the free product
of $\F$-algebras without zero divisor, every unit is monomial, thereby directly providing the converse implication in the Weak Units Theorem.

Our critical contribution here, apart from proving the existence of non-monomial units when at least one of the basic subalgebras is not a field, is to provide the missing generators for the group of units, as well as a clear understanding of the structure of $\calW_{p,q}^\times$.
We will briefly sketch the results here. For each \emph{basic} zero divisor $\alpha$,
we consider the set $\SB(\alpha)$ of all elements of the form $1+z$ where $z \in \calW_{p,q}$
is such that $\alpha^\star z=z \alpha=0$, where $\alpha^\star$ is the conjugate\footnote{I.e.\ the image of $\alpha$ under the only non-identity involution of $\calA$ over $\F$ if there exists one, or $\alpha$ if there is no involution besides the identity.}
of $\alpha$ in the corresponding basic subalgebra $\calA$.
It can be proved that $z_1z_2=0$ for all such elements $z_1$ and $z_2$, to the effect that all the elements of $\SB(\alpha)$ are units, called the \textbf{semi-basic} units attached to $\alpha$, and $\SB(\alpha)$ is a subgroup of $\calW_{p,q}^\times$ that is isomorphic to the additive group of all $z \in \calW_{p,q}$ that satisfy $\alpha^\star z=z \alpha=0$. It will even be seen that the latter is isomorphic to the additive group $(\F[t],+)$.

Now, let $\calA$ be a basic subalgebra of $\calW_{p,q}$.
For every zero divisor $\alpha$ in $\calA$, the subgroup $\SB(\alpha)$ is clearly normalized by $\calA^\times$.
Two zero divisors in $\calA$ that are scalar multiples of one another give rise to the same group $\SB(\alpha)$,
and it follows that to $\calA$ corresponds exactly one such subgroup if $\calA \simeq \F[t]/(t^2)$ (i.e., $\calA$ degenerates),
and exactly two such subgroups if $\calA \simeq \F \times \F$ (i.e., $\calA$ splits), which are the only two possibilities
when $\calA$ is not a field. Then a special subgroup $\SSB(\calA)$ is defined as the subgroup of $\calW_{p,q}$
generated by the semi-basic units attached to any zero divisor in $\calA$:
of course $\SSB(\calA)=\SB(\alpha)$ when $\calA$ degenerates and $\alpha$ is one if its zero divisors, and
$\SSB(\calA)=\{1\}$ if $\calA$ is a field. And finally one defines $\SB(\calA)$ as the subgroup generated by $\calA^\times$ and $\SSB(\calA)$.
Hence $\SSB(\calA)$ is a normal subgroup of $\SB(\calA)$.

We can now state the new theorems which, combined,
yield a completely clear picture of the group of units $\calW_{p,q}^\times$:

\begin{theo}[Strong Units Theorem]
The group $\calW_{p,q}^\times$ is naturally isomorphic to the amalgamated product of
$\SB(\F[a])$ and $\SB(\F[b])$ over $\F^\times$.
\end{theo}

\begin{theo}
Let $\calA$ be a basic subalgebra of $\calW_{p,q}$.
\begin{enumerate}[(i)]
\item If $\calA$ is a field then $\SB(\calA)=\calA^\times$.
\item If $\calA$ degenerates then $\F^\times$ is a direct factor of $\SSB(\calA)$ in $\SB(\calA)$.
\item If $\calA$ splits and $\alpha$ denotes a nontrivial idempotent in it, then
$\SSB(\calA)$ is an (internal) free product of the subgroups $\SB(\alpha)$ and $\SB(1-\alpha)$,
and $\calA^\times$ is a semi-direct factor of $\SSB(\calA)$ in $\SB(\calA)$.
\end{enumerate}
\end{theo}

The special case $p=q=t^2-1$ with $|\F| \in \{2,3\}$ was already discovered by Mirowicz \cite{Mirowicz},
with a viewpoint that is completely different from ours.

\vskip 3mm
Our next set of results deals with the algebraic elements of $\calW_{p,q}$ over $\F$ and more generally with the finite-dimensional subalgebras.
Remember that an element $x$ of an $\F$-algebra is called \textbf{quadratic} whenever
$x^2 \in \F+\F x$.
\index{Quadratic element}

\begin{theo}
Every element of $\calW_{p,q}$ is either quadratic or transcendental over~$\F$.
\end{theo}

This result was already known to Cohn for free products of skew-fields (theorem 3.5 in \cite{CohnFreeProductSkewFields}).
We also improve on the known results by obtaining the following one:

\begin{theo}
Every quadratic element of $\calW_{p,q}$ is conjugated to a basic vector provided that its
minimal polynomial does not split with a double root.
\end{theo}

We even go further and consider more general finite-dimensional subalgebras.
We start with the case where both $p$ and $q$ are irreducible.

\begin{theo}
If $p$ and $q$ are irreducible, then up to conjugation by a unit the only nontrivial finite-dimensional subalgebras of $\calW_{p,q}$
are the basic subalgebras.
\end{theo}

Our study is not limited to this special case, and the next theorem gives the full picture, allowing any of $p$ and $q$ to split.

\begin{theo}
Up to isomorphism, every $2$-dimensional subalgebra of $\calW_{p,q}$ is isomorphic to one of the basic subalgebras unless
one of $p$ and $q$ splits and none splits with a double root, in which case there are also degenerate $2$-dimensional subalgebras.

If one of $p$ and $q$ splits, then for each integer $n \geq 3$ there exist $n$-dimensional subalgebras of $\calW_{p,q}$,
with exactly two isomorphism types if at least one of $p$ and $q$ splits with simple roots, and exactly one isomorphism type otherwise.
\end{theo}

We will even go as far as to classify the finite-dimensional subalgebras of $\calW_{p,q}$ up to conjugation, but
the results cannot be stated at this point.

As a consequence of the results on subalgebras, we will obtain an unsurprising, yet nontrivial, result on the existence of isomorphisms between free Hamilton algebras:

\begin{theo}[Rigidity Theorem]\label{theo:isomorphismbetween}
Let $p_1,q_1,p_2,q_2$ be quadratic polynomials in $\F[t]$.
The $\F$-algebras $\calW_{p_1,q_1}$ and $\calW_{p_2,q_2}$ are isomorphic if and only if
their respective basic subalgebras $\F[a_1],\F[b_1],\F[a_2],\F[b_2]$ satisfy either one of the following conditions:
\begin{enumerate}[(i)]
\item $\F[a_1] \simeq \F[a_2]$ and $\F[b_1] \simeq \F[b_2]$;
\item $\F[a_1] \simeq \F[b_2]$ and $\F[b_1] \simeq \F[a_2]$.
\end{enumerate}
\end{theo}

Our last main result deals with the automorphism group of the $\F$-algebra $\calW_{p,q}$,
denoted by $\Aut(\calW_{p,q})$.
Say that an automorphism of $\calW_{p,q}$ is \textbf{basic} if it maps every basic vector to a basic vector, which is equivalent to
\index{Basic automorphism}
\label{page:basicautomorphisms}
having it preserve or exchange the basic subalgebras $\F[a]$ and $\F[b]$. The basic automorphisms form a subgroup
of $\Aut(\calW_{p,q})$ (though not a normal one at all!) which we naturally denote by $\BAut(\calW_{p,q})$.
When neither $p$ nor $q$ has a double root in $\F$, the group of basic automorphisms is finite and of very low order:
the maximal order is $8$, and when it is reached $\BAut(\calW_{p,q})$ is isomorphic to the dihedral group $D_4$ (it is the case for instance when $p$ and $q$ split with simple roots). Computing $\BAut(\calW_{p,q})$ is completely elementary but requires a case-by-case discussion, whether one of $p$ and $q$ has a double root, the basic algebras are isomorphic, and so on.

Of course, as in any other algebra there are also the inner automorphisms $x \mapsto \gamma x \gamma^{-1}$ with $\gamma \in \calW_{p,q}^\times$,
which form a normal subgroup of $\Aut(\calW_{p,q})$, denoted by $\Inn(\calW_{p,q})$ and isomorphic to the projective group
$\mathrm{P}\calW_{p,q}^\times$ of units. Now, here is our main result:

\begin{theo}[Automorphisms Theorem]\label{theo:automorphismstheointro}
Every automorphism of the $\F$-algebra $\calW_{p,q}$ splits in a unique way as the composite of a basic automorphism followed by an inner automorphism.
In other words, $\BAut(\calW_{p,q})$ is a semi-direct factor of $\Inn(\calW_{p,q})$ in $\Aut(\calW_{p,q})$.
\end{theo}

In contrast with most of the previous results, in which the special case where $p$ and $q$ are irreducible
can frequently be seen as a special case of more general results of Cohn on free products of skew fields,
the Automorphisms Theorem does not seem to be related to more general results (although Cohn could have, with limited effort, derived
it from his methods, with again the limitation of requiring $p$ and $q$ to be irreducible).
There is an interesting contrast between the Automorphisms Theorem and the classical problem of the automorphisms of free algebras,
an in particular the known isomorphism \cite{Czerniakiewicz,Makar-Limanov}
between the automorphisms of the free algebra $\F\langle x,y\rangle$
in two generators $x$ and $y$ and the automorphisms of the polynomial ring $\F[x,y]$,
the latter of which are described by the celebrated Jung - van der Kulk theorem \cite{Jung,vanderKulk}.
In  $\F\langle x,y\rangle$, the only units are the elements of $\F^\times$
and hence the only inner automorphism is the identity. Yet the endomorphism that fixes $x$ and maps $y$ to $y+x$ is clearly an automorphism
(with inverse the endomorphism that fixes $x$ and maps $y$ to $y-x$),
but it is clear that it is not basic in the sense of free products.
It should be noted finally that to our knowledge Cohn's techniques for studying
free products of skew fields did not produce any substantial result on the automorphisms of such free products.

In light of the above, we have some hints as to why these results have waited for so long to appear.
On the one hand, Cohn was preoccupied with very general results, and probably overlooked the ``double dimension $2$"
case in the free product of algebras as a mere curiosity. On the other hand,
most authors who have tackled the free Hamilton algebra for itself so far have only done so in cases where $p$ and $q$ split \cite{Weiss1,Weiss2}, and it is arguably the situation where the greatest complexity is found. In contrast, among the split cases, the ``simple split" case, where both $p$ and $q$ split with simple roots, is the one where the linear representations of $\calW_{p,q}$ are the easiest to describe. This contrast could be explained by an important difference in the structure
of the automorphism group: the simple split case is the one where the outer automorphism group is finite (with cardinality $8$),
while in the other cases it is infinite if $\F$ is infinite; but as seen earlier it is in the simple split case that
the inner automorphism group is most complex (i.e., that the group of units is most complex).

In studying $\calW_{p,q}$ for itself, and not through its linear representations, we were initially motivated by the
issue of understanding the involutions of $\calW_{p,q}$ (meaning, the $\F$-linear involutions of the algebra that revert the products)
up to the action of the automorphism group by conjugation. We initially thought that would be easy, and discovered that the contrary was true.
Now that we have gone through all this journey, the initial question has limited importance to us, but we will nevertheless solve
it in the end of this document, leaving out however the case of fields with characteristic $2$.

It is customary to end such an expository section with a list of open problems, but we are afraid that the theorems we will prove here are so definitive that very little remains to be said on $\calW_{p,q}$ itself. The main open problem to us deal with potential analogues
of the Automorphism Theorem to free products of finite-dimensional algebras over $\F$
(as we have already stated, there is no correct analogue to it for the free product of two copies of $\F[t]$).
Also, the results on the group of units beg corresponding ones for the problem where the coefficient field $\F$ is replaced with the ring $\Z$ of integers.
The meaningful open questions that remain on $\calW_{p,q}$ all deal with the linear representations of $\calW_{p,q}$, which constitute a whole different matter although
some of the tools that are necessary to study the linear representations are developed here.

\subsection{Strategy and structure of the document}

We have already pointed out that some of our results are special cases of Cohn's results when $p$ and $q$
are irreducible. The reader will be relieved however to learn that we will not require any knowledge of
Cohn's methods and results. Rather, our main idea is to completely break free from Cohn's approach and use a viewpoint that is completely
peculiar to $\calW_{p,q}$. This viewpoint involves structures that are reminiscent of
Clifford algebras, and in particular of generalized quaternions. And it also takes great advantage of a property that is entirely
specific to the situation under consideration, which is the nontriviality of the center of $\calW_{p,q}$.

The remainder of the present work is laid out as follows.
Section \ref{section:structure} introduces our main tools to study the Hamilton algebra, starting from the adjunction, the $\omega$ element, and then building the trace, norm and the associated inner product.
We immediately give interesting applications of these constructions, with an answer to the problem of easily detecting units, zero divisors and quadratic elements,
and an application to the fact that the adjunction commutes with every automorphism.
Then we resume this foundational section with the key analysis of the determinant of the inner product, in which another critical object appears:
the fundamental polynomial associated with the pair $(p,q)$, which is another polynomial of degree $2$ that can be thought as a sort of
``midpoint" in Galois theory. With these tools in place, we review the connection with quaternion algebras over fields, both by extending scalars
(which yields quaternion algebras over a pure transcendental extension of degree $1$ of the base field $\F$) or by specializing
(which yields quaternion algebras over algebraic extensions of $\F$). These concepts are then immediately used to
obtain the center of $\calW_{p,q}$ in a breeze, as well as to obtain a very short proof of a result of Laffey on matrix algebras generated by two idempotents
(Section \ref{section:Laffey}).

The remaining sections are organized thematically.
In Section \ref{section:zerodivisors}, we characterize the existence of zero divisors in $\calW_{p,q}$,
giving a new proof of Cohn's result, and as an application we give a definitive answer to the problem of
embedding $\calW_{p,q}$ in matrix algebras over $\F[t]$ (which can be formulated in various ways).

Section \ref{section:units1} is the first one that is devoted to the structure of the group of units, but it actually has a broader ambition,
as it solves a wealth of other issues when both $p$ and $q$ are irreducible. It is the main section where we do not use the connection with quaternion algebras: here, the methods are more elementary than in the rest of the article, but they still rely on the adjunction, the trace and the inner product. The end results we obtain in this section are mostly special cases of results obtained by Cohn for more general free products of fields, but this section is important because it lays down the essential elements that will ultimately help us decipher the group of units in all cases.

The next section (Section \ref{section:ideals}) deals with the maximal ideals of $\calW_{p,q}$. The key result is that every nonzero ideal has nontrivial intersection with the center of $\calW_{p,q}$, a result which was already known \cite{SZW} but which we easily reprove. By using specializations, we derive  most maximal ideals from this observation, with the exception of those that include the fundamental ideal, and which occupy the remainder of the discussion. There, we lay out essential results that are used later in the study of automorphisms. As an application, we also obtain a new proof of the Zero Divisors Theorem that is almost computation-free.

Section \ref{section:finitedimalg} deals with the algebraic elements in $\calW_{p,q}$ and the finite-dimensional subalgebras.
The main tool there is to consider things locally, by moding out a maximal ideal that includes the fundamental ideal.
There, the internal structure of the finite-dimensional subalgebras of $\calW_{p,q}$ is entirely deciphered,
while the question of their orbits under the respective actions of the inner automorphism group and of the full automorphism
group is postponed to later sections.

Sections \ref{section:automorphismsI} and \ref{section:automorphismsII}, which deal with the structure of the automorphism group of $\calW_{p,q}$ and feature a proof of the Automorphisms Theorem, are the first pinnacle of our study. These sections combine results from the two previous sections to slowly dissect the structure of this automorphism group, by considering ever smaller subgroups and constructing relevant invariants (called signatures). These sections, and more prominently the second one, involve the global and the local viewpoint of quaternion algebras, as well as delicate considerations on ideals that include the fundamental ideal. Section \ref{section:automorphismsI} deals with the gap between the subgroup of automorphisms that fix the elements of the center of $\calW_{p,q}$ and the full automorphism group. Section \ref{section:automorphismsII}, which is the most intricate one, deciphers the group of all automorphisms that fix the central elements, and more specifically the gap between this group and the group of inner automorphisms.

Section \ref{section:units2} returns to the study of the group of units when one of $p$ and $q$ splits, and is the second pinnacle of our study.
There, the method of Section \ref{section:units1} is refined to obtain the full decomposition of the projective group of units into the internal free products of two subgroups, depending on the respective types of $p$ and $q$. The technique allows us to give a new proof of the existence part in the Automorphisms Theorem, one that is completely different from the proof given in Section \ref{section:automorphismsII}.

The last section is devoted to miscellaneous issues that involve many of the preceding results.
There, we give a complete classification of the orbits of quadratic elements and of finite-dimensional subalgebras,
mainly under conjugation and in some cases under the action of the full automorphism group.
We also prove that the center of the automorphism group of $\calW_{p,q}$ is trivial, thereby justifying the importance of the fundamental
adjunction of $\calW_{p,q}$ (as the only anti-automorphism that commutes with all the automorphisms).
Finally, we determine the conjugacy classes of the elements of order $2$ in the automorphism group of $\calW_{p,q}$,
thereby answering the question that initiated the present study.

\section{Basics on the free Hamilton algebra}\label{section:structure}

\subsection{The free Hamiton algebra as a free product}

\index{Free product (of algebras)}
We start by noting that the basic subalgebras $\F[a]$ and $\F[b]$ really have dimension $2$.
Indeed, by the universal property there is a unique homomorphism $\Phi : \calW_{p,q} \rightarrow \F[t,s]/(p(t),q(s))$ of $\F$-algebras
that takes $a$ to $t$ and $b$ to $s$, and it is known that $\F[t,s]/(p(t),q(s))$ is naturally isomorphic to $\F[t]/(p) \otimes_\F \F[t]/(q)$,
so the respective cosets of $t$ and $s$ in  $\F[t,s]/(p(t),q(s))$ are not scalar multiples of the unity, hence neither are $a$ and $b$.

Next, we observe that $\calW_{p,q}$, equipped with the canonical inclusions of $\F[a]$ and $\F[b]$ into it,
is a coproduct of these two algebras in the category of $\F$-algebras, i.e., it is the free product of $\F[a]$ and $\F[b]$.
To see this, we take an $\F$-algebra $\calA$ and homomorphisms $f : \F[a] \rightarrow \calA$ and $g : \F[b] \rightarrow \calA$,
and we simply note that $p(f(a))=f(p(a))=0$ and $q(g(b))=g(q(b))=0$ to see that there is a unique homomorphism $\Phi : \calW_{p,q} \rightarrow \calA$
that takes $a$ to $f(a)$ and $b$ to $g(b)$. From there, it is clear that $\Phi$ coincides with $f$ on $\F[a]$ and with $g$ on $\F[b]$;
the uniqueness of $\Phi$ with respect to that property is also clear.

Note that, as a consequence, we obtain that whenever we have $2$-dimensional algebras $\calA$ and $\calB$ and we take distinct elements $x \in \calA\setminus \F$
and $y \in \calB \setminus \F$ with respective minimal polynomials $p$ and $q$, then $\calA * \calB \simeq \calW_{p,q}$.
As a consequence, the structure of $\calW_{p,q}$ up to isomorphism depends only on the respective structures of its basic subalgebras.

\subsection{The fundamental involution}

\label{page:typeof2dimalg}
Let $\calA$ be a $2$-dimensional algebra over $\F$. Remember that there are three cases:
\begin{itemize}
\item Either $\calA$ is a field;
\item Or $\calA \simeq \F \times \F$, in which case one says that $\calA$ \textbf{splits};
\index{Split (quadratic algebra)}
\item Or $\calA \simeq \F[\varepsilon]/(\varepsilon^2)$, in which case one says that $\calA$ \textbf{degenerates.}
\index{Degenerate (quadratic algebra)}
\end{itemize}
Choosing $x \in \calA \setminus \F$, and denoting by $r$ its minimal polynomial, one sees that $\calA \simeq \F[t]/(r)$,
and then $\calA$ is a field if and only if $r$ is irreducible, it splits if and only if $r$ splits with simple roots, and
it degenerates if and only if $r$ splits with a double root.

In any case there are at most two involutions of $\calA$, and exactly one if $\calA$ is an inseparable field extension
of $\F$, or $\car(\F)=2$ and $\calA$ degenerates.
The \textbf{adjunction} of $\calA$, denoted by $x \mapsto x^{\star_\calA}$, is then the non-identity involution if there is one,
otherwise it is the identity\footnote{
Note that the unifying way of seeing the adjunction is through the Clifford viewpoint: if one sees $\calA$ as isomorphic to the Clifford algebra of some
(potentially degenerate) $1$-dimensional quadratic form over $\F$, then the adjunction is the involution of $\calA$ that corresponds to the Clifford involution
through such an isomorphism. Another option would be to define the adjunction directly as $x \mapsto \tr_{\calA/\F}(x).1_\calA-x$.}.
The adjunction takes every element $x \in \calA \setminus \F$, with minimal polynomial $r=t^2-(\tr r)t+N(r)$, to
the element $(\tr r).1_\calA-x$, called the \textbf{quadratic adjoint} of $x$, which is another root of $r$ in $\calA$
(and actually the only possible extra root unless $\calA$ degenerates).
\index{Quadratic adjoint}

By the universal property of free algebras, there is a unique
\emph{anti}homomorphism of $\F$-algebras
$$x \in \calW_{p,q} \longmapsto x^\star \in \calW_{p,q},$$
that satisfies
$$\forall x \in \F[a], \; x^\star=x^{\star_{\F[a]}} \quad \text{and} \quad \forall y \in \F[b], \; y^\star=y^{\star_{\F[b]}}.$$
Of course, by antihomomorphism, we mean that
$$\forall (\lambda,x,y)\in \F \times \calW_{p,q}^2, \; (\lambda\,x+y)^\star=\lambda x^\star+y^\star, \quad 1^\star=1 \quad \text{and} \quad (xy)^\star=y^\star x^\star.$$
Clearly $x \mapsto (x^\star)^\star$ is an endomorphism of the $\F$-algebra $\calW_{p,q}$ that leaves every basic vector invariant.
Since the basic vectors generate the algebra $\calW_{p,q}$, we deduce that 
$$\forall x \in \calW_{p,q}, \; (x^\star)^\star=x.$$
Hence $(-)^\star$ is an involution of the algebra $\calW_{p,q}$.
This antihomomorphism is called the \textbf{fundamental involution}, or \textbf{adjunction}, of $\calW_{p,q}$,
\label{page:adjunction}
\index{Adjunction}
\index{Fundamental adjunction}
and $x^\star$ is called the \textbf{adjoint} of $x$.
\index{Adjoint (of an element)}

The adjunction has a fundamental property that we will prove later: it commutes with every automorphism of $\calW_{p,q}$ (Proposition \ref{prop:commuteadjunction}).
Note in particular that
$$a^\star=\tr(p)-a \quad \text{and} \quad b^\star=\tr(q)-b.$$

\subsection{The $\omega$ element}\label{section:omega}

\label{page:omega}
The next observation is critical. We introduce the element
$$\omega :=ab^\star+b a^\star=a^\star b+b^\star a,$$
where the equality is easily obtained by expanding the adjoints.
Then we find that $\omega$ commutes with $a$ and $b$. To see this most concisely, write
$$a \omega=a(a^\star b+b^\star a)=N(p)\, b+a b^\star a=(b a^\star+ab^\star)\, a=\omega a,$$
and work likewise with $b$ instead of $a$. As a consequence:

\begin{lemma}
The subalgebra $\F[\omega]$ is included in the center of $\calW_{p,q}$.
\end{lemma}

We will prove later that $\F[\omega]$ is \emph{exactly} the center of $\calW_{p,q}$ (see Theorem \ref{theo:center} in Section \ref{section:traceetc}), but this can wait.
Let us simply mention that it is one of the main results of \cite{SZW} (theorem 4 and its proof there).

Another important remark is that $\omega$ is invariant under the adjunction, i.e., $\omega^\star=\omega$,
to the effect that the adjunction is an endomorphism of the $\F[\omega]$-module $\calW_{p,q}$.

It should be noted that while the adjunction is fundamental to $\calW_{p,q}$, the element $\omega$ will generally be perturbated in
applying an $\F$-automorphism of $\calW_{p,q}$. For example, in taking the $\F$-automorphism $\Phi$ that fixes $a$ and exchanges
$b$ and $b^\star$, we obtain $\Phi(\omega)=ab+b^\star a^\star=(\tr p)(\tr q)-\omega$.

Our next observation is that we can now write
$$ba=-ba^\star+(\tr p)\, b=-\omega+a b^\star+(\tr p)\,b=-ab+(\tr q)\, a+(\tr p)\, b-\omega.$$
This, along with the identities $a^2=(\tr p)a-N(p)$ and $b^2=(\tr q)b-N(q)$, allows one to simplify all monomials in $a$ and $b$ and write them as linear combinations of $1$, $a$, $b$ and $ab$ with coefficients in $\F[\omega]$. In other words, $(1,a,b,ab)$ generates the $\F[\omega]$-module $\calW_{p,q}$.
Better still, by analyzing the formal left-multiplications by $x$ when $x$ ranges in $\{1,a,b,ab\}$
(e.g., try to write $b(ab)$ as a formal linear combination of $1,a,b,ab$ with coefficients in $\F[\omega]$), we naturally come up with
the two matrices
$$A:=\begin{bmatrix}
0 & -N(p) & 0 & 0 \\
1 & \tr(p) & 0 & 0 \\
0 & 0 & 0 & -N(p) \\
0 & 0 & 1 & \tr(p)
\end{bmatrix} \quad \text{and} \quad
B=\begin{bmatrix}
0 & -t & -N(q) & -(\tr p)N(q) \\
0 & \tr(q) & 0 & N(q) \\
1 & \tr(p) & \tr(q) & (\tr p)(\tr q)-t \\
0 & -1 & 0 & 0
\end{bmatrix}$$
of $\Mat_4(\F[t])$, and we can check that they satisfy the three identities $p(A)=0$, $q(B)=0$ and $A(\tr(q)I_4-B)+B(\tr(p) I_4-A)=t I_4$,
a routine but somewhat tedious verification.
The first two identities yield a unique homomorphism
$$\Psi : \calW_{p,q} \rightarrow \Mat_4(\F[t])$$
of $\F$-algebras
that takes $a$ and $b$ respectively to $A$ and $B$, and then the third one shows that $\Psi$ takes $\omega$ to $t\,I_4$.
In turn, this shows that $\omega$ is transcendental over $\F$. Moreover, computing that the first column of $AB$
equals $\begin{bmatrix}
0 & 0 & 0 & 1
\end{bmatrix}^T$, we also derive from $\Psi$ that $1,a,b,ab$ are linearly independent over $\F[\omega]$.
As a consequence, we find:

\begin{prop}
The element $\omega$ is transcendental over $\F$, and $\calW_{p,q}$ is a free module of rank $4$ over $\F[\omega]$,
with basis $(1,a,b,ab)$.
\end{prop}

From the viewpoint of free products, the elements $a$ and $b$ play no special role, so it is more natural to take
arbitrary $x \in \F[a] \setminus \F$ and $y \in \F[b] \setminus \F$. Writing $x=\lambda a+\lambda'$ and $y=\mu b+\mu'$ with $\lambda,\lambda',\mu,\mu'$ in $\F$, we check that
$$xy^\star+y x^\star=\lambda \mu \omega+\lambda' \tr(q)+\mu' \tr(p)+2\lambda' \mu',$$
so in replacing $(a,b)$ with $(x,y)$ the $\omega$ element is simply replaced with another generator of the $\F$-algebra $\F[\omega]$.
Moreover, the family $(1,x,y,xy)$ is valued in $\Vect_\F(1,a,b,ab)$, and its matrix in $(1,a,b,ab)$ is upper-triangular with diagonal entries $1,\lambda,\mu,\lambda\mu$,
all nonzero, so it is also a basis of the $\F[\omega]$-module $\calW_{p,q}$. Finally, its determinant
in $(1,a,b,ab)$ is $\lambda^2\mu^2$.
We say that $(1,x,y,xy)$ is the \textbf{deployed} basis associated with $(x,y)$,
and we call the move from $(a,b)$ to $(x,y)$ a \textbf{basic base change}.
\index{Deployed basis}
\index{Basic base change}

\subsection{The trace, inner product and norm}\label{section:traceetc}

Now we can introduce some key additional tools. The first one is the trace map
$$\tr : x \in \calW_{p,q} \longmapsto x+x^\star,$$
which is an endomorphism of the $\F[\omega]$-module $\calW_{p,q}$.
\label{page:trace}
Better still:
\index{Trace map}

\begin{prop}
The trace map is valued in $\F[\omega]$.
\end{prop}

\begin{proof}
To see this we only need to check that the trace map takes every vector of the deployed basis $(1,a,b,ab)$ into $\F[\omega]$.
This is obvious for $1$, $a$ and $b$, whereas for $ab$ it comes from the observation that
$(ab)^\star=b^\star a^\star$ and hence $\tr(ab)=(\tr p)(\tr q)-\omega$ as we observed earlier that $ab+b^\star a^\star=(\tr p)(\tr q)-\omega$.
\end{proof}

As a consequence, the element $x$ always commutes with its adjoint $x^\star$.
Obviously
\begin{equation}\label{eq:tracestarinvariant}
\forall x \in \calW_{p,q}, \; \tr(x^\star)=\tr(x).
\end{equation}
Finally, one derives the classical trace property
\begin{equation}\label{eq:tracecommute}
\forall (x,y)\in (\calW_{p,q})^2, \; \tr(xy)=\tr(yx).
\end{equation}
Indeed, for all $x,y$ in $\calW_{p,q}$ we can write
$$(xy^\star+y x^\star)-(y^\star x+x^\star y)=\bigl((\tr y) x-xy+(\tr x)y-yx\bigr)-\bigl((\tr y)x-yx+(\tr x)y-xy\bigr)=0,$$
which yields $\tr(xy^\star)=\tr(y^\star x)$.

\index{Inner product}
\label{page:innerproduct}
From there, we define the \textbf{inner product} of two elements $x$ and $y$ of $\calW_{p,q}$ as follows:
$$\langle x,y\rangle :=\tr(xy^\star)=xy^\star+yx^\star \in \F[\omega],$$
and in particular
$$\langle a,b\rangle=\omega.$$
Remembering again that the adjunction is an endomorphism of $\F[\omega]$-module, it turns
out that the inner product is $\F[\omega]$-bilinear. The inner product is obviously symmetric, also.
Moreover, thanks to \eqref{eq:tracecommute} we have the second expression for the inner product:
\begin{equation}\label{eq:starisometry}
\forall (x,y) \in \calW_{p,q}^2, \; \langle x,y\rangle =y^\star x+x^\star y=\langle x^\star,y^\star\rangle.
\end{equation}
Finally, the inner product appears as the polar form of the \textbf{norm} mapping
$$N : x \in \calW_{p,q} \mapsto xx^\star=x^\star x,$$
\label{page:norm}
\index{Norm mapping}
so that
$$\forall (x,y)\in \calW_{p,q}^2, \; N(x+y)=N(x)+N(y)+\langle x,y\rangle.$$
Notice the absence of division by $2$ in the definition of the inner product, which is critical to handle fields with characteristic $2$.

Better still:

\begin{prop}
The norm is valued in $\F[\omega]$.
\end{prop}

\begin{proof}
Indeed, because the polar form of $N$ is valued in $\F[\omega]$ and we have
$N(\lambda x)=\lambda xx^\star \lambda^\star=\lambda^2 N(x)$ for all $x \in \calW_{p,q}$ and all $\lambda \in \F[\omega]$,
it suffices to check that $N$ maps all four vectors of the $\F[\omega]$-basis $(1,a,b,ab)$ into $\F[\omega]$. Yet
this is straightforward as $N(1)=1$, $N(a)=aa^\star=N(p)$, $N(b)=bb^\star=N(q)$
and $N(ab)=a(bb^\star)a^\star=N(q) aa^\star=N(q)N(p)$.
\end{proof}

In turn, this shows that $N$ is multiplicative, as the centrality of its range yields
$$\forall (x,y)\in \calW_{p,q}^2, \; N(xy)=xy^\star yx^\star=x N(y) x^\star=xx^\star N(y) =N(x)N(y).$$
Finally, we can obtain several basic identities as an application of the above:
\begin{equation}\label{eq:adjeq}
\forall (x,y,z)\in \calW_{p,q}^3, \; \langle xy,z\rangle=\langle x,zy^\star \rangle \quad \text{and} \quad \langle xy,z\rangle=\langle y,x^\star z\rangle
\end{equation}
which connects formal adjoints to the adjunction with respect to the bilinear mapping $\langle -,-\rangle$:
the proof of the first identity is straightforward, and the second one is obtained by applying the first one to the triple
$(y^\star,x^\star,z^\star)$, combined with \eqref{eq:starisometry}.

Finally, as a consequence of the multiplicativity of $N$ we find by polarizing that
$$\forall (x,y,z)\in \calW_{p,q}^3, \; \langle xy,xz\rangle=N(x)\langle y,z\rangle=\langle yx,zx\rangle,$$
which identifies left- and right-multiplication with a given element as some sort of similarity with respect to the inner product.

\begin{Rem}\label{rem:algoCbase}
We can now give a simple and efficient algorithm to express a given element $x \in \calW_{p,q}$
as an $\F[\omega]$-linear combination of $(1,a,b,ab)$,
where $x$ is given as  an $\F$-linear combination of words
in $a$ and $b$ with no consecutive equal letters (such words are called vectors of the standard $\F$-basis).

The algorithm works as follows. At each step, it gives an expression of
$x$ as an $\F[\omega]$-linear combination of the vectors of the standard $\F$-basis,
where the first one is simply with coefficients in $\F$.
Now, say that we have an expression of $x$ as an $\F[\omega]$-linear combination of the vectors of the standard $\F$-basis.
Say also that at least one word with nonzero coefficient
has length greater than $2$ or is the word $ba$. In each such word, take the first occurrence of $ba$, replace it with
$\langle b,a^\star\rangle-a^\star b^\star=\tr(a)\tr(b)-\omega-a^\star b^\star$, expand, use
$b^\star b=N(b)$ and $aa^\star=N(a)$ if such subwords appear after expanding,
and eventually $b^\star=\tr(b)-b$ and $a^\star=\tr(a)-a$ if one of $b^\star$ and $a^\star$
formally remains after expanding and using the above simplifications.

Hence, after each such step, the greatest length for the words associated with nonzero coefficients in $\F[\omega]$
decreases by at least one unit if it was greater than $3$ in the first place, so after
a very limited number of steps we end up with an expression of $x$ as a $\F[\omega]$-linear combination of $(1,a,b,ab)$.
\end{Rem}

\subsection{First applications}

\index{Quadratic identity}
By writing $x^2=x(\tr(x)-x^\star)$ we get the \textbf{quadratic identity}
\begin{equation}\label{eq:quadidentity}
\forall x \in \calW_{p,q}, \quad x^2=\tr(x)\,x-N(x).
\end{equation}
Beware however that $\tr(x)$ and $N(x)$ only belong to $\F[\omega]$, not to $\F$, and do not interpret the quadratic identity as stating that every element of $\calW_{p,q}$
is quadratic over $\F$.

At this point we can prove one of the statements we announced earlier, justifying the importance of the
adjunction map. To start with, we find that it is easy to recognize the norm and trace of the
quadratic elements in $\calW_{p,q}$.

\begin{lemma}\label{lemma:quadrecognize}
Let $x \in \calW_{p,q} \setminus \F$. Assume that there exist $\lambda,\mu$ in $\F$ such that $x^2=\lambda x-\mu$. Then $\lambda=\tr(x)$ and $\mu=N(x)$.
\end{lemma}

\begin{proof}
The element $x$ cannot belong to $\F[\omega]$ because $\omega$ is transcendental over $\F$.
Because $\calW_{p,q}$ is a free $\F[\omega]$-module, it follows that $1$ and $x$ are linearly independent over $\F[\omega]$.
Since $\tr(x)x-N(x).1=x^2=\lambda x-\mu.1$, with $\lambda,\mu,\tr(x),N(x)$ all in $\F[\omega]$, we deduce that $\tr(x)=\lambda$ and $N(x)=\mu$.
\end{proof}

We deduce the following result, where the case $x \in \F$ is obvious because in that one $x$ is quadratic and $\tr(x)$ and $N(x)$ belong to $\F$:

\begin{cor}\label{lemma:quadratic}
Let $x \in \calW_{p,q}$. Then $x$ is quadratic if and only if $\tr(x) \in \F$ and $N(x) \in \F$.
\end{cor}

Here is a nice application:

\begin{prop}\label{prop:commuteadjunction}
Let $\Phi$ be an injective endomorphism of the $\F$-algebra $\calW_{p,q}$. Then $\Phi$ commutes with the adjunction.
\end{prop}

\begin{proof}
We wish to prove the identity $\forall x \in \calW_{p,q}, \; \Phi(x^\star)=\Phi(x)^\star$, which states that certain antihomomorphisms of the
$\F$-algebra $\calW_{p,q}$ are equal. Hence it suffices to check it on the generators $a$ and $b$.
Yet $p(\Phi(a))=\Phi(p(a))=0$ so $\Phi(a)^2=\tr(p) \Phi(a)-N(p)$, whereas $\Phi(a) \not\in \F$ because $\Phi$ is injective.
It follows from Lemma \ref{lemma:quadrecognize} that $\tr(\Phi(a))=\tr(p)=\tr(a)$, leading to $\Phi(a)^\star=\tr(p)-\Phi(a)=\Phi(\tr(p)-a)=\Phi(a^\star)$.
Likewise $\Phi(b)^\star=\Phi(b^\star)$, which completes the proof.
\end{proof}

Next, we are now able to answer some of the questions raised in our introduction for lay readers.
How to detect a unit or a zero divisor? Well simply by a computation of norms!

For units, the result is straightforward since $N : \calW_{p,q} \rightarrow \F[\omega]$ is multiplicative, maps $1$ to the unity of $\F[\omega]$,
and is defined as $N(x)=xx^\star=x^\star x$. Hence:

\begin{prop}\label{prop:unitdetect}
An element $x \in \calW_{p,q}$ is a unit if and only if $N(x) \in \F^\times$.
\end{prop}

For zero divisors, the result is unsurprising although the proof is more subtle:

\begin{prop}\label{prop:zerodivdetect}
A nonzero element $x \in \calW_{p,q}$ is a zero divisor if and only if $N(x)=0$.
\end{prop}

\begin{proof}
The converse implication is obvious due to the definition of the norm.

Let $x \in \calW_{p,q}$ be such that $N(x) \neq 0$.
Let $y \in \calW_{p,q}$ be such that $xy=0$ or $yx=0$. Multiplying by $x^\star$ on the left in the first case, on the right in the second case,
we find $N(x)\,y=0$ in any case. Since $\calW_{p,q}$ is a free $\F[\omega]$-module, it follows that $y=0$.
Hence $x$ is not a zero divisor.
\end{proof}

From the algorithmic viewpoint, it is then fairly easy to use the norm to test whether an element is a unit or not, or a zero divisor or not.
And of course in case $x$ is a unit we obtain its inverse as $N(x)^{-1}x^\star$.

\subsection{The determinant of the inner product}\label{section:Grammatrix}

We resume the investigation of the fundamental structure of $\calW_{p,q}$, and now we examine the inner product more closely.

The $\F[\omega]$-module $\calW_{p,q}$ is free of rank $4$, and
$\langle -,-\rangle$ is a symmetric bilinear form on it, so naturally we examine its
determinant, defined as the coset in $\F[\omega]/(\F[\omega]^\times)^2=\F[\omega]/(\F^\times)^2$
of its Gram determinant in an arbitrary basis of the $\F[\omega]$-module $\calW_{p,q}$.

Take an arbitrary deployed basis $(1,x,y,xy)$ of $\calW_{p,q}$. Then
the Gram matrix takes the rather uninspiring form:
\begin{equation}\label{eq:Gram}
\begin{bmatrix}
2 & \tr(x) & \tr(y) & \langle 1,xy\rangle \\
\tr(x) & 2N(x) & \langle x,y\rangle & N(x)\tr(y) \\
\tr(y) & \langle x,y\rangle & 2N(y) & N(y)\tr(x) \\
\langle 1,xy\rangle & N(x)\tr(y) & N(y) \tr(x) & 2N(x)N(y)
\end{bmatrix}
\end{equation}
\index{Gram matrix}
where we have used the observation that $\langle x,xy\rangle=N(x)\langle 1,y\rangle=N(x)\tr(y)$ and
$\langle y,xy\rangle=N(y)\tr(x)$ likewise. Finally, the lower-left entry $\langle 1,xy\rangle$ can be
viewed as $\langle x^\star,y\rangle$, and hence is a polynomial of degree $1$ in $\omega$.
It is at least clear from the antidiagonal that the determinant of the latter matrix is a monic polynomial of degree $4$ in $\F[\omega]$ with respect to the transcendental $\omega$,
and with tremendous courage in computing (or a lazy appeal to a formal computing software), one can obtain that for the special deployed basis $(1,a,b,ab)$,
the Gram determinant factors as follows for the polynomial
$$\Lambda_{p,q}(t):=t^2-(\tr p)(\tr q)\,t-4N(p)N(q)+(\tr p)^2 N(q)+(\tr q)^2 N(p).$$

\begin{prop}\label{prop:gram}
The Gram determinant of $\langle -,-\rangle$ in the basis $(1,a,b,ab)$ equals $\Lambda_{p,q}(\omega)^2$.
\end{prop}

\index{Fundamental polynomial}
\label{page:fundamentalpolynomial}
This is enough to justify that $\Lambda_{p,q}$ is called the \textbf{fundamental polynomial} attached to the pair $(p,q)$.
In fact, we will see Proposition \ref{prop:gram} as just a special case of computing the Gram determinant in an arbitrary deployed basis:

\begin{prop}\label{prop:gramgeneralized}
Let $x \in \F[a] \setminus \F$ and $y \in \F[b] \setminus \F$ have respective minimal polynomials $r$ and $s$.
The Gram determinant of $\langle -,-\rangle$ in the basis $(1,x,y,xy)$ equals $\Lambda_{r,s}(\langle x,y\rangle)^2$.
\end{prop}

In fact, we can give a proof that avoids almost any computation and which is based upon the observation of the roots of
$\Lambda_{r,s}$ in the splitting field $\L$ of $pq$ in $\Fbar$. Say that $r=(t-x_1)(t-x_2)$ and $s=(t-y_1)(t-y_2)$
over $\L$. Then one checks that
$$\Lambda_{r,s}(t)=\bigl(t-(x_1y_1+x_2y_2)\bigr)\bigl(t-(x_1y_2+x_2y_1)\bigr).$$
As a consequence, a transvection $x \leftarrow x+\lambda$ with $\lambda \in \F$ does not affect the end result: indeed
$\langle x+\lambda,y\rangle=\langle x,y\rangle+\lambda \tr(y)$, the minimal polynomial of $x+\lambda$ is $r(t-\lambda)$, with roots $x_1+\lambda$ and $x_2+\lambda$,
and
$$\Lambda_{r(t-\lambda),s}(t)=\bigl(t-(x_1y_1+x_2y_2)-\lambda \tr(y)\bigr)\bigl(t-(x_1y_2+x_2y_1)-\lambda \tr(y)\bigr).$$
Likewise, the end result is unaffected by performing a transvection $y \leftarrow y+\mu$ with $\mu \in \F$.
Note finally that we have already observed that the Gram determinant is unaffected by such changes of deployed bases (as we are only using transvections).

From there the proof will be much more satisfying because we can avoid brute force computation:

\begin{proof}[Proof of Proposition \ref{prop:gramgeneralized}]
We observe that we are simply stating a polynomial identity with \emph{integral} coefficients in the four variables $\tr(x),\tr(y),N(x),N(y)$. Hence for its universal validity in the field $\F$ it suffices to prove it over the field $\Q$ of rational numbers,
which allows us to avoid the traditional problems that come from the characteristic $2$ case.
Now, assume that $\F=\Q$. Then, with the above remark, we apply the transvections
$x \leftarrow x-\frac{\tr x}{2}$ and $y \leftarrow y-\frac{\tr y}{2}$ to reduce the situation to the case where $x$ and $y$ have trace zero.
Then $x^\star=-x$ and $y^\star=-y$, and the Gram determinant takes the simplified form
$$g=\begin{vmatrix}
2 & 0 & 0 & -\langle x,y\rangle \\
0 & 2N(x) & \langle x,y\rangle & 0 \\
0 & \langle x,y\rangle & 2N(y) & 0 \\
-\langle x,y\rangle & 0 & 0 & 2N(x)N(y)
\end{vmatrix}$$
which is now easily computed since it factorizes into
$$g =\begin{vmatrix}
2 & -\langle x,y\rangle \\
-\langle x,y\rangle & 2 N(x)N(y)
\end{vmatrix}\cdot \begin{vmatrix}
2 N(x) & \langle x,y\rangle \\
\langle x,y\rangle & 2 N(y)
\end{vmatrix}
=\left(4N(x)N(y)-\langle x,y\rangle^2\right)^2.$$
The result is then obtained by noting that $\Lambda_{r,s}=t^2-4N(x)N(y)$ in our reduced situation, since $\tr(x)=\tr(y)=0$.
\end{proof}

The polynomial $\Lambda_{p,q}$ might seem to come out of nowhere, but its appearance could in fact have been expected.
To see this, consider the special case where $p$ and $q$ split over $\F$, and write $p(t)=(t-x_1)(t-x_2)$ and $q(t)=(t-y_1)(t-y_2)$.
Consider a matrix representation of $\calW_{p,q}$, and denote by $A$, $B$ and $\Omega$
the matrices that correspond, respectively, to $a$, $b$ and $\omega$.
Then $p(A)=0$ and $q(B)=0$, so $A$ and $B$ are triangularizable. Now, if $A$ and $B$ have a common eigenvector $X$,
this is also an eigenvector for $\Omega$, and the corresponding eigenvalue will be $x_1y_2+x_2y_1$ or $x_1y_1+x_2y_2$, depending
on the eigenvalues of $A$ and $B$ that are attached to $X$.

In the theory of linear representation of the free Hamilton algebra, $\Lambda_{p,q}$ is connected with the dreaded \emph{exceptional} representations,
in which the image of $\omega$ is annihilated by a power of $\Lambda_{p,q}$ (see \cite{dSPsum,dSPprod}).

The polynomial $\Lambda_{p,q}$ has special resonance in Galois theory. Indeed, if $\F[a]$ and $\F[b]$ are nonisomorphic separable field
extensions of $\F$, then it can be proved that the splitting field of $\Lambda_{p,q}$ in the splitting field $\L$ of $pq$ is precisely the third quadratic
extension of $\F$ inside $\L$, where the other two are of course the respective splitting fields of $p$ and $q$. See Section \ref{section:Lambdapq} for details.

Finally, an important consequence of the previous computation is that $\langle -,-\rangle$ is non-degenerate, meaning
that its radical, defined as the set of all $x\in \calW_{p,q}$ such that $\langle x,-\rangle=0$, reduces to the zero element.
But it is not regular, meaning that $x \in \calW_{p,q} \mapsto \langle x,-\rangle \in \Hom_{\F[\omega]}(\calW_{p,q},\F[\omega])$
is not an isomorphism (it is only injective, not surjective).

\subsection{Connection with quaternion algebras}\label{section:quaternionalgebras}

At this point, we have a picture of $\calW_{p,q}$ that looks quite similar to traditional quaternion algebras over fields. The huge problem is that $\F[\omega]$ is not a field,
and worse still the determinant of the inner product is not represented by a unit of the ring $\F[\omega]$.
However, there are various ways we can connect the free Hamilton algebra to traditional quaternion algebras over fields,
and we will explain them shortly.

\index{Quaternion algebra}
Before we do so, it is essential to recall some basic facts on quaternion algebras and to remind the reader of how one can recognize
a quaternion algebra in practice.
To be short, a quaternion algebra over a field $\L$ is an $\L$-algebra that is isomorphic to the Clifford algebra of a
regular $2$-dimensional quadratic form over $\L$. Such an algebra $\calA$ is always central\footnote{I.e.\ its center is reduced to $\L$.} and simple\footnote{I.e.\ it has no nontrivial two-sided ideal.}, has dimension $4$ as an $\L$-vector space, and comes equipped with a special involution called the quaternionic conjugation $x \mapsto \overline{x}$
(which corresponds to the Clifford involution), so that $\forall x \in \calA, \; x\overline{x}=\overline{x} x \in \F$, and $\{x \in \calA : \overline{x}=-x\}$
is a $3$-dimensional $\L$-linear subspace whose elements are called the pure quaternions.
By polarizing the norm $x \mapsto x\overline{x}$ at $1$ we get the quaternionic trace $x \mapsto x+\overline{x} \in \L$. Then the quaternionic conjugation and trace are uniquely determined by the
structure of $\L$-algebra of $\calA$ by the above properties. The key is that the set $\{x \in \calA : x^2 \in \L\}$
is the union of $\L$ with the hyperplane of pure quaternions. And more globally there are two options:
\begin{itemize}
\item Either $\calA$ \emph{splits}, i.e., it is isomorphic to $\Mat_2(\L)$, in which case the Clifford involution corresponds to the standard adjunction $M \mapsto M^{\ad}$ on matrices
(the transpose of the comatrix), the norm corresponds to the matrix determinant, and the quaternionic trace corresponds to the traditional matrix trace.
\index{Split quaternion algebra}
\item Or $\calA$ is a skew field.
\end{itemize}
Whether $\calA$ splits or not can be detected from the norm: $\calA$ splits if and only if the norm is isotropic, and in that case the norm is hyperbolic.

The following theorem will help us recognize a quaternion algebra when we have an algebra equipped with a certain involution.

\begin{theo}\label{theo:recogquaternion}
Let $\calA$ be a $4$-dimensional algebra over a field $\L$, equipped with an involution $x \mapsto \overline{x}$
such that $\forall x \in \calA, \; x\overline{x} \in \L$.
Assume furthermore that the associated inner product $(x,y) \mapsto \langle x,y\rangle:= x\overline{y}+y\overline{x}$ is nondegenerate.
Then $\calA$ is a quaternion algebra, and $x \mapsto \overline{x}$ is its quaternionic conjugation.
\end{theo}

We reproduce the short argument given in \cite{dSPregular}.

\begin{proof}
Throughout, we consider orthogonality with respect to the inner product.
We consider the trace $\tr : x \mapsto \langle 1,x\rangle$ and its kernel $H$.
Note that $x^2=-x\overline{x}=-N(x)$ for all $x \in H$.
Since the inner product is nondegenerate, the subspace $H \cap H^\bot$ has dimension at most $1$, and hence we can pick
a $2$-dimensional subspace $P$ of $H$ on which the inner product is nondegenerate. It follows that $x \in P \mapsto -N(x) \in \L$
is a nondegenerate quadratic form on $P$.
Hence by the universal property of Clifford algebras we recover a homomorphism $\Phi$ of $\L$-algebras from the Clifford algebra $\calC(-N_{|P})$
to $\calA$. Since $-N_{|P}$ is nondegenerate and $\dim P=2$ the algebra $\calC(-N_{|P})$ is simple with dimension $2^2=4$, and as a consequence $\Phi$
is an isomorphism. Hence $\calA$ is a quaternion algebra over $\L$.

It remains to recognize that $x \mapsto \overline{x}$ is the quaternionic conjugation. But from the first remark we have seen that every
element of $H$ squares in $\L$, so $H$ is included in the union of $\L$ with the hyperplane of pure quaternions. Hence obviously $H$ \emph{is}
the hyperplane of pure quaternions. Then $x \mapsto \overline{x}$ and the quaternionic conjugation deduced from the above isomorphism
coincide on the pure quaternions, which is known to be a generating set of the quaternion algebra, so they are equal.
\end{proof}

We will now see that the above can be applied in two ways: \emph{globally} or \emph{locally}.

The global way works as follows.
First, we embed $\F[\omega]$ in its fraction field $\F(\omega)$.
Next, we consider the tensor product
$$\overline{\calW_{p,q}}:=\calW_{p,q} \otimes_{\F[\omega]} \F(\omega),$$
thereby obtaining an $\F(\omega)$-algebra, which we call the \textbf{extension} of $\calW_{p,q}$.
\index{Extension (of the free Hamilton algebra)}
\label{page:extension}
Because $\calW_{p,q}$ is a free $\F[\omega]$-module of rank $4$,
the resulting vector space over $\F(\omega)$ has dimension $4$, every basis of the free $\F[\omega]$-module $\calW_{p,q}$
is a basis of this vector space, and $\calW_{p,q}$ is naturally seen as a subring of $\overline{\calW_{p,q}}$.

Next, because $\omega^\star=\omega$, all our structural mappings are naturally extended
to the extension $\overline{\calW_{p,q}}$, yielding the extended
adjunction
$$x \in \overline{\calW_{p,q}} \mapsto x^\star \in \overline{\calW_{p,q}},$$
the extended norm
$$N : x \in \overline{\calW_{p,q}} \mapsto xx^\star=x^\star x \in \F(\omega),$$
its polar form
$$\langle -,-\rangle : (x,y) \in \overline{\calW_{p,q}} \mapsto xy^\star+yx^\star=y^\star x+x^\star y \in \F(\omega)$$
and finally the trace map $x \in \overline{\calW_{p,q}} \mapsto \langle 1,x\rangle=x+x^\star \in \F(\omega)$.
As $(1,a,b,ab)$ becomes a basis of $\F(\omega)$-vector space in this extension of scalars, we deduce from Proposition \ref{prop:gramgeneralized} that
the determinant of the extended inner product is now represented by $\Lambda_{p,q}(\omega)^2 \in \F(\omega)$.
Note in particular that it is a square, in full accordance with the theory of quaternion algebras (the norm of a quaternion algebra is always a Pfister form).
In particular, this determinant is nonzero in the field $\F(\omega)$, so the extended inner product is non-degenerate. Thus Theorem \ref{theo:recogquaternion} helps us conclude:

\begin{theo}
The $\F(\omega)$-algebra $\calW_{p,q} \otimes_{\F[\omega]} \F(\omega)$ is a quaternion algebra
with quaternionic conjugation $x \mapsto x^\star$, trace $x \mapsto x+x^\star$ and norm $x \mapsto xx^\star$.
\end{theo}

Of course, the next question is whether $\overline{\calW_{p,q}}$ splits or not, depending on $p$ and $q$. It will be fully answered in Section \ref{section:zerodivisors}.

\vskip 3mm
The second way to obtain quaternion algebras from the free Hamilton algebra is to look at things \emph{locally}, by specializing.
Simply, we take an irreducible (monic) polynomial $r \in \F[t]$ and consider the quotient of $\calW_{p,q}$ by the two-sided ideal $(r(\omega))$.
The quotient algebra
$$\calW_{p,q,[r]}:=\calW_{p,q}/(r(\omega))$$
now has a natural structure of vector space over the field $\L:=\F[\overline{\omega}] \simeq \F[t]/(r)$, where $\overline{\omega}$
stands for the coset of $\omega$ mod $(r(\omega))$.
And again, since $\calW_{p,q}$ is a free $\F[\omega]$-module of rank $4$ with basis $(1,a,b,ab)$, the quotient algebra
$\calW_{p,q}/(r(\omega))$ becomes a $4$-dimensional vector space over $\L$ with basis $(1,\overline{a},\overline{b},\overline{a}\,\overline{b})$,
where $\overline{a}$ and $\overline{b}$ stand for the respective cosets of $a$ and $b$ mod $(r(\omega))$.
Since the adjunction leaves $\omega$ invariant, the two-sided ideal $(r(\omega))$ is invariant under adjunction, and we obtain
an induced involution $x \mapsto x^\star$ of $\calW_{p,q}/(r(\omega))$.
Then we recover the norm form
$$N_r : x \mapsto xx^\star=x^\star x \in \L,$$
its polar form, also known as the inner product
$$\langle -,-\rangle_r : (x,y) \mapsto xy^\star+yx^\star=x^\star y+y^\star x \in \L,$$
and the trace map $x \mapsto x+x^\star=\langle 1,x \rangle_r \in \L$.
The inner product $\langle -,-\rangle_r$ is of course $\L$-bilinear.
Finally the determinant of this new inner product is now $\Lambda_{p,q}(\overline{\omega})^2$, i.e., the coset of $\Lambda_{p,q}(\omega)^2$,
which vanishes if and only if $r$ divides $\Lambda_{p,q}$.

Hence, as an application of Theorem \ref{theo:recogquaternion} two possibilities can occur:
\begin{itemize}
\item Either $\Lambda_{p,q}$ is relatively prime with $r$, in which case $\calW_{p,q}/(r(\omega))$
is a quaternion algebra over the field $\F[\omega]/(r(\omega))$;
\item Or $r$ is an irreducible factor of $\Lambda_{p,q}$, and not much can be said (yet).
\end{itemize}

Finally, in the first case we might inquire whether $\calW_{p,q}/(r(\omega))$ splits or not.
A standard case is when one of $p$ and $q$ splits: then $N(x)=0$ for some nonzero $x$ in either $\F[a] \setminus \F$ or $\F[b] \setminus \F$, and
going to the quotient yields the isotropy of the coset $\overline{x}$, which critically is nonzero
(indeed, if $\overline{x}=0$ then all the coefficients of $x$ in the deployed basis $(1,a,b,ab)$ would be multiples of $r(\omega)$, which is clearly false).
Other interesting cases include the one where $\F$ is finite, in which every specialization splits.

Let us sum up (see also theorem 2.4 in \cite{dSPregular}):

\begin{theo}\label{theo:quatlocal}
Let $r \in \Irr(\F)$ be relatively prime with $\Lambda_{p,q}$.
Then the quotient algebra $\calW_{p,q}/(r(\omega))$ is a quaternion algebra over the residue field $\F[\omega]/(r(\omega))$.
\end{theo}

Theorem \ref{theo:quatlocal} is key to understand the linear representations $\phi$ of $\calW_{p,q}$ that are regular, meaning
that the endomorphism $\phi(\omega)$ has its minimal polynomial relatively prime with $\Lambda_{p,q}$.
We also mention the following result from \cite{dSPregular} (proposition 2.6 there), which was an essential tool in solving the representation problem that was tackled there:

\begin{theo}\label{theo:quatlocaln}
Let $r \in \Irr(\F)$ be relatively prime with $\Lambda_{p,q}$.
If the quotient quaternion algebra $\calW_{p,q}/(r(\omega))$ splits,
then $\calW_{p,q}/(r(\omega)^n) \simeq \Mat_2\bigl(\F[t]/(r^n)\bigr)$ for all $n \geq 1$, where we mean an isomorphism of $\F[t]/(r^n)$-algebras.
\end{theo}

In contrast, the degeneracy of $\calW_{p,q}/(\Lambda_{p,q}(\omega))$ is intimately connected with the difficulty of understanding
the exceptional linear representations of $\calW_{p,q}$. In general, one wants to avoid considering $\calW_{p,q}/(r(\omega))$ at all when $r$ is an irreducible divisor of $\Lambda_{p,q}$ but there are a few basic remarks that can be made about it.
First of all, its inner product is degenerate and hence an interesting set is the \emph{radical} of the latter.
Using \eqref{eq:adjeq}, one sees that this radical is a two-sided ideal of $\calW_{p,q}/(r(\omega))$.
Note already how this is in contrast with the non-degenerate case because every quaternion algebra is simple.
Even more intriguing is the possibility that the radical be the whole of $\calW_{p,q}/(r(\omega))$, and judging from the Gram matrix \eqref{eq:Gram}
it is the case if and only if $\car(\F)=2$, $\tr(p)=\tr(q)=0$ and $r=t$ (note that $\Lambda_{p,q}=t^2$ in this case).

\subsection{The center and the fundamental ideal}\label{section:center}

It was already seen that $\F[\omega]$ was canonical in some way, because
replacing $a$ and $b$ with basic generators $x$ and $y$ yields an inner product $\langle x,y\rangle$ that generates the $\F$-algebra $\F[\omega]$.
The next result shows that $\F[\omega]$ is even more canonically attached to $\calW_{p,q}$: it is its center!

So far, we had refrained from proving this fact, but now we can do this without computing.

\index{Center}
\begin{theo}\label{theo:center}
The center of $\calW_{p,q}$ is $\F[\omega]$.
\end{theo}

\begin{proof}
The extended $\F(\omega)$-algebra $\calW_{p,q} \otimes_{\F[\omega]} \F(\omega)$
is a quaternion algebra. Hence its center is $\F(\omega)$. It is clear that the center of $\calW_{p,q}$ is the intersection of it with the
latter, and hence it equals $\F[\omega]$ (implicitly, this involves the fact that $\calW_{p,q}$ is a free $\F[\omega]$-module with basis $(1,a,b,ab)$).
\end{proof}

\begin{Not}
The center of $\calW_{p,q}$ is now denoted by $C$.
\end{Not}
\label{page:center}

At this point, we could get rid of the $\omega$ element as it is not canonical. Yet, it is very useful for discussing degrees of polynomials,
so sometimes we will keep the notation $\F[\omega]$ for clarity.

Although the element $\omega$ is not canonical with respect to the structure of $\F$-algebra of $\calW_{p,q}$,
a canonical object is the ideal generated by the Gram determinant of the inner product:

\index{Fundamental ideal}
\label{page:fundamentalideal}
\begin{Def}
The \textbf{fundamental ideal} $\mathfrak{F}$ of $\calW_{p,q}$ is defined as the (two-sided) ideal generated by the Gram determinant of
an arbitrary basis of the $\F[\omega]$-module $\calW_{p,q}$ for the inner product.
\end{Def}

In other words, the fundamental ideal is the two-sided ideal generated by $\Lambda_{p,q}(\omega)$.

Before we move forward, we would like to stress that the nontriviality of the center of $\calW_{p,q}$
is exceptional in the theory of free products of algebras over a field.

\begin{prop}\label{prop:center:peculiar}
Let $\calA$ and $\calB$ be nontrivial $\F$-algebras. The center of $\calA * \calB$ is nontrivial only if $\dim_\F \calA=\dim_\F \calB=2$.
\end{prop}

For the proof, which is inessential to our study of the free Hamilton algebra, we need some classical facts on free products of $\F$-algebras; they will reappear when we discuss monomial units.

Consider two non trivial $\F$-algebras $\calA$ and $\calB$ and their free product $\Pi:=\calA * \calB$. We naturally identify
$\calA$ and $\calB$ with subalgebras of $\Pi$, and the elements of $\calA \cup \calB$ are called the basic elements. For an integer $n \geq 0$, one denotes by $\Pi^{(n)}$ the linear subspace of $\Pi$ spanned by the products of at most $n$ basic elements.
Clearly $(\Pi^{(n)})_{n \geq 0}$ is a filtration of the $\F$-vector space $\Pi$ (beware that it is not a gradation!).
Another viewpoint is the following: we choose respective linear hyperplanes $H_\calA$ and $H_\calB$ of $\calA$ and $\calB$ that do not contain $1$,
and we choose respective bases $(e_{i,\calA})_{i \in I}$ and $(e_{j,\calB})_{j \in J}$ of them.
Then it can be shown that a basis of the vector space $\Pi$ is obtained by taking all the
\emph{alternating} words in letters of the form $e_{i,\calA},e_{j,\calB}$, i.e., the words in these letters
in which no two consecutive letters belong to the same basic subalgebra.
Moreover, the subspace spanned by the words with length at most $n$ is exactly $\Pi^{(n)}$.
In particular, a direct factor of $\Pi^{(n-1)}$ in $\Pi^{(n)}$ has as basis
the set of all alternating words of length $n$ in letters of the form $e_{i,\calA},e_{j,\calB}$.

\begin{proof}[Proof of Proposition \ref{prop:center:peculiar}]
Set $\Pi:=\calA*\calB$, and take respective linear hyperplanes $H_\calA$ and $H_\calB$ and respective bases $(e_i)_{i \in I}$ and $(f_j)_{j \in J}$ of them
as in the above. Denote by $\calA^{(n)}$ (respectively, $\calB^{(n)}$) the set of all alternating words of length $n$ in letters of the form
$e_i,f_j$ and that start with a vector of $\calA$ (respectively, of $\calB$).
Let $x \in \calA * \calB$ be central and nonscalar.
Denote by $n$ the greatest integer such that $x \in (\calA * \calB)^{(n)}$ (i.e., the height of $x$), and assume that $n>0$.
We will prove that $|I|=|J|=1$, i.e., that $\dim_\F \calA=\dim_\F \calB=2$.

We write $x \equiv y+z$ mod $\Pi^{(n-1)}$ for a unique $y \in \calA^{(n)}$ and a unique $z \in \calB^{(n)}$, with at least one of $y$ and $z$ nonzero.
\begin{itemize}
\item Assume first that $n$ is odd, and let $i \in I$ (which exists). Then
$e_i x \equiv e_i z$ mod $\Pi^{(n)}$, while $x e_i \equiv z e_i$ mod $\Pi^{(n)}$, and it follows that
$e_i z=z e_i$. But if $z \neq 0$ then on the left-hand side we have a nonzero element of $\calA^{(n+1)}$, and on the right-hand side a nonzero element of
$\calB^{(n+1)}$, so this is absurd. Hence $z=0$. Likewise, one would obtain $y=0$ because $J$ is nonempty, which is absurd.
\item It follows that $n$ is even. Let again $i \in I$.
Then $e_i x \equiv e_i z$ mod $\Pi^{(n)}$, while $x e_i\equiv y e_i$ mod $\Pi^{(n)}$.
Hence $e_i z=y e_i$, and it easily follows that, in the basis of $\Pi$ we have taken,
$y$ has nonzero coefficients only on the words that start with $e_i$, and likewise $z$ has nonzero coefficients only on the words
that end with $e_i$. If $|I|>1$ this yields $y=0=z$ by varying $i$. Hence $|I|=1$ and likewise we obtain $|J|=1$.
\end{itemize}
\end{proof}

\subsection{Application to Laffey's theorem}\label{section:Laffey}

In \cite{Laffey}, which seems to be the first article to mention the free Hamilton algebra (by concept if not by name), Thomas Laffey proved the following result:

\begin{theo}
The matrix algebra $\Mat_n(\F)$ is generated by two idempotents if and only if either $n=1$, or $n=2$ and $|\F|>2$.
\end{theo}

Laffey also proved that every matrix algebra over a field is generated by three idempotents.

Thanks to our previous work, we can give an enlightening proof of Laffey's theorem.
Assume that we have idempotent matrices $P$ and $Q$ that generate the algebra $\Mat_n(\F)$.
We take $p=q=t^2-t$. To $P$ and $Q$, we attach a linear representation $\varphi : \calW_{p,q} \rightarrow \Mat_n(\F)$ such that $\varphi(a)=P$ and $\varphi(b)=Q$,
and the fact that $\{P,Q\}$ generates $\Mat_n(\F)$ means that this representation is surjective.
But $\varphi(\omega)$ is a central element in $\im \varphi$, so $\varphi(\omega)=\delta\,I_n$ for some $\delta \in \F$, and hence $\varphi$ induces a surjective homomorphism of $\F$-algebras
$\calW_{p,q}/(\omega-\delta) \twoheadrightarrow \Mat_n(\F)$. Yet the source algebra now is an $\F[\omega]/(\omega-\delta) \simeq \F$ vector space of dimension $4$,
and hence $n^2 \leq 4$, so $n \leq 2$, and if $n=2$ then we have an isomorphism $\calW_{p,q}/(\omega-\delta) \overset{\simeq}{\longrightarrow} \Mat_2(\F)$ of $\F$-algebras.

Now, assume that $n=2$ and $|\F|=2$. Then $\delta \in \{0,1\}$ is a root of $\Lambda_{p,q}=t^2-t$ (where we use the specific fact that $p=q=t^2-t$)
so the inner product in $\calW_{p,q}/(\omega-\delta)$ degenerates (but not fully, as seen from the Gram matrix \eqref{eq:Gram}), to the effect that its
radical is a nontrivial two-sided ideal of $\calW_{p,q}/(\omega-\delta)$. But then this contradicts the simplicity of $\Mat_2(\F)$.
Hence $\Mat_2(\F_2)$ has no generating subset consisting of two idempotent matrices (which can, of course, also be checked by tedious verification).

Conversely, assume that $|\F|>2$, and let us prove the existence of a generating set of two idempotents of $\Mat_2(\F)$ in the abstract.
Naturally, we choose $\delta \in \F \setminus \{0,1\}$ and consider the specialization
$\calW_{p,q}/(\omega-\delta)$, which is a quaternion algebra over $\F$: this specialization splits because $t^2-t$ splits. Hence we have an isomorphism
$\Phi : \calW_{p,q}/(\omega-\delta) \overset{\simeq}{\longrightarrow} \Mat_2(\F)$ of $\F$-algebras, we simply compose it with the canonical projection to recover
a surjective homomorphism $\varphi : \calW_{p,q} \twoheadrightarrow \Mat_2(\F)$, and we conclude by taking $\{\varphi(a),\varphi(b)\}$ as our set of idempotent generators.

Clearly, a similar proof can be given to consider a much wider variety of pairs of quadratic generators, and the same result will be obtained whenever at least one specialization at a nondegenerate point renders the norm isotropic, yielding a split quaternion algebra over $\F$ (this is always the case whenever one of $p$ and $q$ splits and the roots of $\Lambda_{p,q}$ do not cover $\F$, but might fail for specific fields and choices of $p$ and $q$).

\section{Zero divisors in the free Hamilton algebra}\label{section:zerodivisors}

\subsection{The Zero Divisors Theorem: statement and comments}

\index{Zero divisors}
This section is devoted to the existence of zero divisors in $\calW_{p,q}$, and to applications of this problem.
The main result is the following one:

\begin{theo}[Zero Divisors Theorem]\label{theo:zerodivisors}
The algebra $\calW_{p,q}$ has a zero divisor if and only if one of $p$ and $q$ splits over $\F$.
\end{theo}

We claim no originality here, as we recognize a special case of Cohn's general results on free products of rings.
The originality however lies in our proof, which emphasizes the connection with quaternion algebras,
as well as in the applications of this result.

Remember from Proposition \ref{prop:zerodivdetect} that the zero divisors in $\calW_{p,q}$ are the nonzero elements with norm $0$.
In this theorem, the ``if" statement is obvious because having $p$ or $q$ split over $\F$
immediately yields a zero divisor in one of the basic subalgebras.

The difficult point is to prove that $\calW_{p,q}$ has no zero divisor when $p$ and $q$ are irreducible.

The proof is not long, but we must warn the reader of several red herrings.
For example, it is tempting to use specializations of $\calW_{p,q}$ to prove the result:
if for instance $\F=\R$ and $p$ and $q$ are irreducible, we can choose a real number
$\lambda$ such that $\calW_{p,q}/(\omega-\lambda)$ does not split (we will leave the existence of such a number as an exercise to the reader, who should first reduce the situation to the one where $p=q=t^2+1$, and then compute the resulting $\Lambda_{p,q}$ and the corresponding Gram matrix), 
and from there it is easy to derive that $\calW_{p,q}$ has no zero divisor.
However the argument surely fails for finite fields, as over such fields all quaternion algebras split (every $3$-dimensional regular quadratic form over a finite field is isotropic).

In an initial failed attempt to \emph{disprove} the Zero Divisors Theorem (when we were not aware of Cohn's work), we tried to use the Chevalley-Warning theorem
(see, e.g., theorem 3 page 5 in \cite{Serre}) by taking a fixed integer $n \geq 1$ and by searching for a nontrivial solution for the
polynomial equation $N(p_1(\omega)+p_a(\omega)a+p_b(\omega)b+p_{ab}(\omega)ab)=0$
in the polynomials $p_1,p_a,p_b,p_{ab}$ with degree less than $n$. Yet it appears that the norm is then valued
in the polynomials with degree less than $2n$, thereby resulting in $2n$ scalar equations that are homogeneous of degree $2$ in $4n$ unknowns in $\F$: this is the critical bound at which the Chevalley-Warning theorem fails!

Before we give the proof, we need to connect the problem to the structure of the extended quaternion algebra
$\overline{\calW_{p,q}}$, which is almost obvious because $\overline{\calW_{p,q}}$ splits if and only if it has a zero divisor.

\begin{prop}
The algebra $\calW_{p,q}$ has a zero divisor if and only if $\overline{\calW_{p,q}}$ splits.
\end{prop}

\begin{proof}
If $\calW_{p,q}$ has a zero divisor, then so does $\overline{\calW_{p,q}}$ and hence it splits.
Conversely, assume that $\overline{\calW_{p,q}}$ splits. Then it contains two nonzero elements $x,y$ such that $xy=0$.
Then we can find nonzero elements $\lambda,\mu$ of $\F[\omega]$ such that $\lambda x \in \calW_{p,q} \setminus \{0\}$ and $\mu y \in \calW_{p,q} \setminus \{0\}$.
Then $(\lambda x)(\mu y)=(\lambda \mu) xy=0$ and we conclude that $\lambda x$ is a zero divisor in $\calW_{p,q}$.
\end{proof}

Hence, a corollary of the Zero Divisors Theorem is:

\begin{cor}
The extended quaternion algebra $\overline{\calW_{p,q}}$ is a skew field if and only if both $p$ and $q$ are irreducible.
\end{cor}

Note that Cohn proved more generally that the free product of two skew fields over a skew field
is always embeddable in a skew field \cite{CohnFirsEmbedding}.

\subsection{Proof of the Zero Divisors Theorem}

We assume throughout that $p$ and $q$ are irreducible. We perform a \emph{reductio ad absurdum}, assuming that
$\calW_{p,q}$ has a zero divisor. This yields a vector $x \in \calW_{p,q} \setminus \{0\}$ with norm $0$.
It is tedious however to analyze the equation $N(x)=0$, because it can be viewed as an equation in four variables in $\F[\omega]$.

The first trick consists in reducing the number of variables to just three by using the properties of quaternion algebras: the norm
of a quaternion algebra is hyperbolic whenever it is isotropic, and hence in this case it has a totally isotropic $2$-dimensional subspace.
Working in the extended algebra, whose norm is isotropic, we deduce that every $3$-dimensional $\F(\omega)$-linear subspace
of $\overline{\calW_{p,q}}$ contains a nonzero isotropic vector.

And now we carefully choose the space $H:=\Vect_{\F(\omega)}(1,a,b)$, in which the norm has the relatively simple expression
$$N(\alpha-\beta a-\gamma b)=\alpha^2+N(p)\, \beta^2 +N(q)\,\gamma^2-\tr(p)\alpha \beta-\tr(q) \alpha \gamma+\omega \beta \gamma$$
for $\alpha,\beta,\gamma$ in $\F(\omega)$.
By the above, $H$ contains a nonzero element with norm $0$, and by sweeping denominators
we recover a triple $(\alpha,\beta,\gamma)$ of \emph{polynomials} in $\F[\omega]$, not all zero,
with $\gcd(\alpha,\beta,\gamma)=1$ and such that
\begin{equation}\label{eq:ZD0}
\alpha^2+N(p)\,\beta^2 +N(q)\,\gamma^2 -\tr(p)\alpha \beta-\tr(q) \alpha \gamma=-\omega \beta \gamma,
\end{equation}
which we can write alternatively as
\begin{equation}\label{eq:ZD1}
\alpha^2-\tr(p) \alpha\beta+N(p)\beta^2=\tr(q) \alpha \gamma-N(q) \gamma^2 -\omega \beta \gamma
\end{equation}
or symmetrically as
\begin{equation}\label{eq:ZD2}
\alpha^2-\tr(q) \alpha\gamma+N(q)\gamma^2=\tr(p) \alpha \beta-N(p) \beta^2 -\omega \beta \gamma.
\end{equation}
We will now perform an analysis of the degrees of $\alpha,\beta,\gamma$ in $\omega$.

First of all, the irreducibility of $p$ of $q$ means that the quadratic forms
$$Q_1 : (x,y) \mapsto x^2-\tr(p)xy+N(p)y^2 \quad \text{and} \quad Q_2 : (x,y) \mapsto x^2-\tr(q)xy+N(q)y^2$$
are nonisotropic over $\F$.
Next, we recall two important principles in the theory of quadratic forms:
\begin{itemize}
\item A regular nonisotropic quadratic form remains nonisotropic after extending the scalar field to a purely transcendental extension
(see, e.g., lemma 1.21 in \cite{Elmanetal});
\item A regular nonisotropic quadratic form remains nonisotropic after extending the scalar field to an algebraic extension of odd degree (the celebrated Artin-Springer theorem):
see, e.g., corollary 18.5 in \cite{Elmanetal}, and note that the result is known to hold over fields with any characteristic although in the literature
it is frequently stated only for fields with characteristic other than $2$.
\end{itemize}

Now all the tools are in place. First of all, if $\gamma=0$ then \eqref{eq:ZD1} would yield that $Q_1$ becomes isotropic over $\F(\omega)$,
contradicting the first principle. Hence $\gamma \neq 0$, and likewise \eqref{eq:ZD2} leads to $\beta \neq 0$.

Next, let us consider an irreducible divisor $r$ of odd degree of $\gamma$, and let us mod out the ideal $(r(\omega))$.
Equation \eqref{eq:ZD1} yields that the respective cosets $x$ and $y$ of $\alpha$ and $\beta$ in the quotient field $\F[\omega]/(r(\omega))$ satisfy $x^2-\tr(p) xy+N(p) y^2=0$.
Since $Q_1$ is nonisotropic, the Artin-Springer Theorem yields $x=y=0$, i.e., $r$ divides $\alpha$ and $\beta$. Yet this contradicts the assumption that
$\gcd(\alpha,\beta,\gamma)=1$. It follows that all the irreducible divisors of $\gamma$ have even degree, and hence
$\deg(\gamma)$ is even.

Symmetrically, applying the Artin-Springer theorem to \eqref{eq:ZD2} yields that $\deg(\beta)$ is even.

And now the conclusion is almost at hand. Denote by $d$ the greatest degree among $\alpha,\beta,\gamma$.
If $\deg(\beta)=\deg(\gamma)=d$, then the right-hand side of \eqref{eq:ZD0} has degree $2d+1$, and the left-hand side degree at most $2d$: this is absurd.
Hence, at most one of $\beta$ and $\gamma$ has degree $d$. However if none of $\beta$ and $\gamma$ has degree $d$, then the left-hand side of
\eqref{eq:ZD1} has degree $2d$, with $\alpha^2$ as the only summand contributing to the leading term, whereas the degree of the right-hand side is at most $2d-1$,
so again this is absurd. Hence exactly one of $\beta$ and $\gamma$ has degree $d$, and without loss of generality we assume that $\deg(\beta)=d$.
But then $d$ is even and $\deg(\gamma) \leq d-2$ because $\gamma$ has even degree, so now the degree of the right-hand side of \eqref{eq:ZD1} is less than $2d$.

To complete the proof, we consider the respective coefficients $\lambda$ and $\mu$ of $\alpha$ and $\beta$ on $\omega^d$ and extract the coefficient on $\omega^{2d}$
in \eqref{eq:ZD1}:
$$\lambda^2-\tr(p) \lambda\mu+N(p)\mu^2=0$$
with $\mu \neq 0$. This contradicts the nonisotropy of $Q_1$ over $\F$, and the proof is now completed.

\subsection{Application: embedding the free Hamilton algebra into matrix algebras over $\F[t]$}

The authors of \cite{SZW} proved that $\calW_{p,q}$ embeds as an $\F$-algebra into $\Mat_2(\Fbar[t])$, where $\Fbar$ stands for an algebraic closure of $\F$.
Their proof actually shows that $\calW_{p,q}$ can be embedded as an $\Fbar[\omega]$-algebra into $\Mat_2(\Fbar[\omega])$,
and naturally that $\Fbar$ can be replaced with a splitting field of $pq$.

This raises two interesting questions:
\begin{itemize}
\item Is there always an embedding of $\F[\omega]$-algebras of $\calW_{p,q}$ into $\Mat_2(\F[\omega])$?
And what are precisely the integers $n \geq 2$ for which an embedding of $\F[\omega]$-algebras into $\Mat_n(\F[\omega])$ exists?
\item Is there always an embedding of $\F$-algebras of $\calW_{p,q}$ into $\Mat_2(\F[t])$?
And what are precisely the integers $n \geq 2$ for which an embedding of $\F$-algebras into $\Mat_n(\F[t])$ exists?
\end{itemize}

Here, we will answer all these questions thanks in part to the Zero Divisors Theorem.
Here are our results:

\begin{theo}\label{theo:embed1}
\begin{enumerate}[(a)]
\item There exists an embedding of $\F[\omega]$-algebras of $\calW_{p,q}$ into $\Mat_2(\F[\omega])$
if and only if one of $p$ and $q$ splits.
\item The integers $n \geq 1$ for which there is an embedding of $\F[\omega]$-algebras of $\calW_{p,q}$ into $\Mat_n(\F[\omega])$
are the multiples of $4$ if $p$ and $q$ are irreducible, and the multiples of $2$ otherwise.
\end{enumerate}
\end{theo}

\begin{theo}\label{theo:embed2}
\begin{enumerate}[(a)]
\item There exists an embedding of $\F$-algebras of $\calW_{p,q}$ into $\Mat_2(\F[t])$ whatever the choice of $p$ and $q$.
\item The integers $n \geq 1$ for which there is an embedding of $\F$-algebras of $\calW_{p,q}$ into $\Mat_n(\F[t])$
are the multiples of $2$ if one of $p$ and $q$ is irreducible, and all the integers that are greater than $1$ otherwise.
\end{enumerate}
\end{theo}

\begin{proof}[Proof of Theorem \ref{theo:embed1}]
We directly prove the necessity in the second statement.
Let $n \geq 2$ and let $\Phi : \calW_{p,q} \hookrightarrow \Mat_n(\F[\omega])$ be an embedding of $\F[\omega]$-algebras.
We can then extend it to an embedding $\overline{\Phi} : \overline{\calW_{p,q}} \hookrightarrow \Mat_n(\F(\omega))$ of $\F(\omega)$-algebras,
which endows $\F(\omega)^n$ with a structure of left $\overline{\calW_{p,q}}$-module.
It is then known from the theory of simple algebras that any such module is the sum of simple ones, and all simple modules are isomorphic
(see e.g.\ lemma 7.2.17 of \cite{Voight}).

Finally, the theory of quaternion algebras (over fields) shows that the simple modules over a quaternion algebra
are either of dimension $2$ or $4$, whether the algebra splits or not. Hence $n$ is even in any case, and by the Zero Divisors Theorem $n$ is a multiple of $4$
if $p$ and $q$ are irreducible.

Conversely, we can view $\calW_{p,q}$ as a module over itself, and hence we recover for all $m \geq 1$ a natural embedding
$\calW_{p,q} \hookrightarrow \End_{\F[\omega]}((\calW_{p,q})^m)$, and the target $\F[\omega]$-algebra is isomorphic to
$\Mat_{4m}(\F[\omega])$.

Assume now that one of $p$ and $q$ splits, to the effect that the quaternion algebra $\overline{\calW_{p,q}}$ splits.
Thus we can take a simple $\overline{\calW_{p,q}}$-module $M$, which then has dimension $2$ over $\F(\omega)$.
Some caution is needed because $M$ is only an $\F(\omega)$-vector space, but the argument is classical: We take an arbitrary free
$\F[\omega]$-submodule $M_0$ of $M$ of rank $2$, and we consider the $\F[\omega]$-submodule $N=M_0+a M_0+b M_0+ab M_0$.
It is clear then that $N$ is invariant under multiplication with $a$ and $b$, and hence it is a $\calW_{p,q}$-submodule of $M$.
Moreover, by sweeping denominators in an arbitrary $\F(\omega)$-basis of $M$, we find that $N$ is included in a free $\F[\omega]$-submodule
of rank $2$ of $M$, and then, by the theory of modules over principal ideal domains, $N$ turns out to be  a free $\F[\omega]$-module of rank $2$.

Finally, for all $m \geq 1$, we recover a homomorphism of $\F[\omega]$-algebras $\Psi : \calW_{p,q} \rightarrow \End_{\F[\omega]}(N^m) \simeq \Mat_{2m}(\F[\omega])$.
This homomorphism must be injective because it can be extended to a homomorphism of $\F(\omega)$-algebras from $\overline{\calW_{p,q}}$ to $\Mat_{2m}(\F(\omega))$,
and the ring $\overline{\calW_{p,q}}$ is simple.
\end{proof}

Now we turn to the second embedding problem. Before we give the proof, we quickly explain the difficulty and the strategy.

Assume that we have an embedding $\Phi$ of $\F$-algebras of $\calW_{p,q}$ into $\Mat_2(\F[t])$. Note that $\omega$ is central in the $\F$-subalgebra
generated by the range of $\Phi$, and hence in the $\F(t)$-subalgebra of $\Mat_2(\F(t))$ generated by it: hence if $\Phi(\omega)$
is not central in $\Mat_2(\F(t))$, i.e., not a scalar multiple of $I_2$ over $\F(t)$, it is cyclic and its centralizer is commutative, so
$\calW_{p,q}$ would be commutative, which is false. Hence $\Phi(\omega)=r(\omega) I_2$ for some $r(t) \in \F[t]$, which of course is nonconstant,
and $r$ cannot always have degree $1$, as this would contradict Theorem \ref{theo:embed1} in the special case where $p$ and $q$ are irreducible.
Now, the question becomes whether there always exists a polynomial $r(t) \in \F[t]$ with degree at least $2$ such that the norm defined over $\F[t]$
with the Gram matrix \eqref{eq:Gram} is isotropic when $\omega$ is replaced with $r(t)$. Conversely, if such a polynomial exists then we will see that it is not difficult
to define an embedding of $\F$-algebras. Interestingly, our proof is optimal, in that $r(t)$ will be found of degree $2$.

\begin{proof}[Proof of Theorem \ref{theo:embed2}]
With the previous heuristics, we naturally come right back to the simplified equation
\eqref{eq:ZD0} from the proof of the Zero Divisors Theorem. The idea is to find an algebraic extension of
$\F(\omega)$ that is generated \emph{over $\F$} by an element $\kappa$ such that $\omega \in \F[\kappa]$.
To find such an extension, we start by considering a purely transcendental extension $\F(\omega)(u)$ of $\F(\omega)$.
With the new transcendental $u$, we remark that we can find a nonconstant polynomial $r(u) \in \F[u]$
for which some nonzero triple $(\alpha(u),\beta(u),\gamma(u)) \in \F[u]^3$ satisfies
\begin{equation}\label{eq:ZD4}
\alpha(u)^2+N(p)\,\beta(u)^2 +N(q)\,\gamma(u)^2 -\tr(p)\alpha(u) \beta(u)-\tr(q) \alpha(u) \gamma(u)=-r(u) \beta(u)\gamma(u).
\end{equation}
The choice is fairly simple : we take $\alpha(u)=u$, $\beta(u)=\gamma(u)=1$ and then
$$r(u):=-\bigl(u^2+N(p)+N(q)-\tr(p)u -\tr(q)u\bigr).$$
Next, we choose an irreducible divisor $s(u)$ of $r(u)-\omega$ in $\F(\omega)[u]$,
and we mod out the ideal $(s(u))$ in $\F(\omega)[u]$ to obtain an algebraic extension $\L:=\F(\omega)[u]/(s(u))$ of $\F(\omega)$.
Denoting by $\kappa$ the coset of $u$ in $\L$, we deduce that $\omega=r(\kappa)$, which leads to $\L=\F(\kappa)$, and in particular $\kappa$
is transcendental over $\F$.
Now, we consider the extended quaternion algebra $\calQ:=\overline{\calW_{p,q}} \otimes_{\F(\omega)} \L$ over $\L$, with its extended norm.
Identity \eqref{eq:ZD4} now tells us that in this extended quaternion algebra the norm becomes isotropic
(note that $\beta(\kappa)=1$, to ensure that we are dealing with a nonzero vector), and hence $\calQ$ splits.

Hence, we can find a $\calQ$-module $M$ that has dimension $2$ as a vector space over $\F(\kappa)$, and we can extract from it
a free $\F[\kappa]$-submodule $M_0$ of rank $2$.
Just like in the proof of Theorem \ref{theo:embed1}, we take $N:=M_0+aM_0+bM_0+(ab)M_0$, which is a free submodule of the $\F[\kappa]$-module $M$ of rank $2$.
And here we use the identity $\omega=r(\kappa)$, with $r\in \F[t]$, to see that $N$ is invariant under multiplication with $a$ and $b$.
This yields a homomorphism of $\F$-algebras
$$\Phi : \calW_{p,q} \rightarrow \End_{\F[\kappa]}(N) \simeq \Mat_2(\F[\kappa])$$
that maps $\omega$ to $r(\kappa)\, I_2$, and in particular is injective on the center $C$ of $\calW_{p,q}$.
By Proposition \ref{prop:nonzeroidealmeetscenter} in Section \ref{section:ideals} (see also the last part of the proof of theorem 4 of
\cite{SZW}, where it is proved that any nonzero ideal of $\calW_{p,q}$ contains a nonzero central element),
it follows that $\Phi$ is injective, which completes the proof of point (a).
Consequently, by taking as many copies of $N$ as necessary, we obtain an injective homomorphism of $\F$-algebras from
$\calW_{p,q}$ to $\Mat_{2n}(\F[\kappa])$ for all $n \geq 1$.

We finish with point (b). First of all, if we have an embedding of $\F$-algebras $\Phi : \calW_{p,q} \hookrightarrow \Mat_n(\F[t])$
and $p$ is irreducible, $\Phi$ must map $a$ to a matrix with minimal polynomial $p$, which remains irreducible over $\F(t)$,
and hence $n$ must be a multiple of the degree of $p$, i.e., $n$ is even.
Assume finally that both $p$ and $q$ split. Then we can choose respective roots $x_1$ and $y_1$ of $p$ and $q$ in $\F$, let $k \geq 0$ be an arbitrary integer,
choose an embedding $\Phi : \calW_{p,q} \hookrightarrow \Mat_2(\F[t])$ of $\F$-algebras and extend it to an embedding
$\calW_{p,q} \hookrightarrow \Mat_{k+2}(\F[t])$ of $\F$-algebras by taking $a$ and $b$ respectively to
the block-diagonal matrices $\Phi(a) \oplus (x_1\,I_k)$ and $\Phi(b) \oplus (y_1\,I_k)$.
\end{proof}

To close this discussion, we point out that in \cite{SZW}, embeddings into matrix algebras over polynomials rings were essentially a tool to prove
results on the free Hamilton algebra. Here, we will never use this technique as the quaternion algebras appear to be far more efficient for our needs.

\section{Units in the free Hamilton algebra (part 1)}\label{section:units1}

\index{Units}

\subsection{Introduction}

This section is the first one in which we investigate the structure of the group of units (i.e., of the invertible elements) in $\calW_{p,q}$.
Our main aim here is to prove the following result:

\begin{theo}[Weak Units Theorem]
The group $\calW_{p,q}^\times$ is generated by the basic units if and only if both $p$ and $q$ are irreducible.
\end{theo}

Better still, we will describe an algorithm that takes as entry a unit $\gamma \in \calW_{p,q}^\times$,
and returns an explicit reduced decomposition of $\gamma$ into a product of basic units if $\gamma$ is a monomial unit,
and otherwise returns an error message.
The root of this algorithm will even yield two other important results on
the Hamilton algebra in the irreducible case:
\begin{enumerate}[(i)]
\item If $p$ and $q$ are irreducible, then every automorphism of $\calW_{p,q}$ is the composite of an inner automorphism with a basic automorphism.
\item If $p$ and $q$ are irreducible, then every $2$-dimensional subalgebra of $\calW_{p,q}$ is conjugated to one of the basic subalgebras.
\end{enumerate}

This algorithm, which constitutes the main object of this section,
is dealt with in Sections \ref{section:units1:plan} to \ref{section:units1:retracing}.

We point out once more that most of the above results are special cases of theorems of Cohn on free products of skew fields
(with the notable exception of point (i) in the above). The main points here are to give a new method for proving them, a method that will serve as a warm up for
the one that will help us entirely decipher the group of units (see Section \ref{section:units2}). It introduces most of the
important technical concepts that are required for a complete understanding of the group of units.

Before we get into the algorithm, we start with elementary considerations on monomial units, i.e., on products of basic units.
To start with, we note that since the center of $\calW_{p,q}$ is $\F[\omega]$ and $\omega$ is transcendental over $\F$,
the central units are simply the elements of $\F^\times$, i.e., the \textbf{scalar units}.

It is obvious that every nonscalar monomial unit $x$ has a \emph{reduced} decomposition, i.e., a decomposition
$x=\prod_{k=1}^n x_k$ in which every $x_k$ is a basic unit and no two consecutive $x_k$'s belong to the same basic subalgebra.
We have stated in the introduction that such a decomposition is then unique up to multiplication of each factor with a nonzero scalar.
This result is not specific to free products of $2$-dimensional algebras, as it actually extends to any free product of two algebras $\calA$ and
$\calB$ over $\F$.
We quickly recall the explanation.
As in Section \ref{section:center}, we take respective families $(e_i)_{i \in I}$ and $(f_j)_{j \in J}$
that, along with the unity $1$, form bases of the vector spaces $\calA$ and $\calB$. We set $H_\calA:=\Vect(e_i)_{i \in I}$ and $H_\calB:=\Vect(f_j)_{j \in J}$.
We set $\Pi:=\calA * \calB$ and we recall the notation $\Pi^{(n)}$ for the linear subspace spanned by the words with length
at most $n$ and letters in $\calA \cup \calB$.
We notice that if we have a product $x=x_1 \cdots x_n$ in which $x_1 \in \calA \setminus \F$,
$x_2 \in \calB \setminus \F,\dots$ then $x \notin \Pi^{(n-1)}$. Indeed, in
writing each $x_k$ as a linear combination of the vectors of the corresponding bases, and in expanding the product,
we find, for each word of length $n$ in the letters $e_i,f_j$ and no two consecutive equal letters, at most one corresponding summand,
and at least one such summand comes with a nonzero coefficient, whereas all the other summands belong to $\Pi^{(n-1)}$.
This shows that $x \not\in \Pi^{(n-1)}$. In particular, $x\not\in \F$.
From there, the uniqueness, up to multiplication of each factor with a nonzero unit, of a reduced decomposition
of a monomial unit is easily proved (we leave it as an exercise for the reader).

\subsection{The plan}\label{section:units1:plan}

When we tackled the free Hamilton algebra for itself, our first motivation was to understand its involutions.
This naturally prompted us to investigate its automorphisms first, and the following strategy came quite naturally:
Take an arbitrary automorphism $\Phi$ of $\calW_{p,q}$, and then consider the pair $(\Phi(a),\Phi(b))$ (which fully characterizes $\Phi$).
Then, we can try to apply well-chosen elementary inner automorphisms to $\Phi(a)$ and $\Phi(b)$
until the resulting elements are basic. Naturally, the elementary inner automorphisms
we try to apply are the conjugations by basic units.

The reader could deem this strategy as naive because the degrees of freedom seem very small at each step, but as we shall see this very lack
of freedom might explain the miracle of the Automorphism Theorem and our various results on units.

Now, let us get deeper into the method. Instead of taking the two elements $\Phi(a)$ and $\Phi(b)$, we start from
an arbitrary element $x \in \calW_{p,q}\setminus \F$ and simply require that it be \emph{quadratic}, to the effect that $\tr(x)$ and $N(x)$
belong to $\F$ (Corollary \ref{lemma:quadratic}) and not simply to the center $C$.
The element $x$ can be split along the standard deployed basis $x=?+?a+?b+?ab$
with coefficients in $\F[\omega]$, represented by question marks, and at this point the reader must beware that
the fact that $\tr(x)$ belongs to $\F$ does \emph{not} imply that these coefficients belong to $\F$.
But we sense that the larger the degrees of the coefficients with respect to $\omega$ are, the farther $x$ is away from the basic subalgebras.
So, the idea is simply to make small steps, each step consisting of a conjugation by a basic unit, so as to make this distance decrease
at each step, until eventually the resulting vector is caught in a basic subalgebra.

In order to move forward, it is necessary to start by analyzing the consequences of having $x$ quadratic, which is measured in the trace and norm,
on the degrees of the coefficients of $x$ in the deployed basis. This will yield important measures of $x$, which we call the
\emph{distances} of $x$ with respect to the basic subalgebras. Afterwards we will investigate how a conjugation by a basic unit affects
these measures.

To simplify things and because we need to work with well-chosen deployed bases,
we will now largely forget about the canonical generators $a$ and $b$, and simply set
$$\calA:=\F[a] \quad \text{and} \quad \calB:=\F[b].$$

In what follows, we generalize the notion of a deployed basis, without reference to the special generators $a$ and $b$:

\begin{Def}
A \textbf{deployed basis} of $\calW_{p,q}$ is a quadruple of the form $(1,\alpha,\beta,\alpha\beta)$
in which $\alpha$ and $\beta$ are nonscalar basic vectors that do not belong to the same basic subalgebra.
\end{Def}
\index{Deployed basis}

\begin{prop}
Every deployed basis of $\calW_{p,q}$ is a basis of the $C$-module $\calW_{p,q}$.
\end{prop}

\begin{proof}
Let $(1,\alpha,\beta,\alpha\beta)$ be a deployed basis of $\calW_{p,q}$.
Then the result has already been proved in Section \ref{section:omega} if $\alpha \in \calA$ and $\beta \in \calB$.
Assume now that $\alpha \in \calB$ and $\beta \in \calA$.
Then $(1,\beta,\alpha,\beta\alpha)$ is a basis of the $C$-module $\calW_{p,q}$.
Yet $\beta \alpha=\langle \beta,\alpha^\star\rangle-\alpha^\star \beta^\star=-\alpha \beta+\tr(\alpha) \beta+\tr(\beta) \alpha+
\langle \beta,\alpha^\star\rangle-(\tr \alpha)(\tr \beta)$,
and from there it is easily seen that $(1,\alpha,\beta,\alpha\beta)$ is a basis of the $C$-module $\calW_{p,q}$.
\end{proof}

Crucial to our analysis will be the notion of degree in the center $C$ of $\calW_{p,q}$.
We have seen that $C=\F[\omega]$, and hence we can talk of the degree of an element of $C$ as a polynomial in $\omega$:
critically, this notion is independent of the choice of generator $\omega$ (whereas e.g., the notion of leading term for a polynomial
depends on the choice of generator).
A unifying viewpoint would be to define the degree of $r \in C \setminus \{0\}$ as the dimension of the quotient $\F$-vector space $C/(r)$, but this is not very practical,
and it is a good idea that the reader refers to the $\omega$ element as an anchor point to understand the proofs.

\begin{Not}
\index{Degree}
\label{page:degree}
Following French conventions, we denote by $\N$ the set of all non-negative integers. We also denote
by $\deg(r)$ the \textbf{degree} of an element $r \in C$, which is an element of $\{-\infty\} \cup \N$.
For all $n \in \N$, we also set
$$C_n:=\{r \in C : \; \deg(r) \leq n\}=\Vect_\F(\omega^k)_{0 \leq k \leq n},$$
We also convene that $C_{-1}=\{0\}=\{r \in C : \; \deg(r) \leq -1\}$.
\end{Not}

An important point, which was seen shortly after introducing the $\omega$ element and will be used throughout, is that
$\langle x,y\rangle$ has degree $1$ for all $x \in \calA \setminus \F$ and all $y \in \calB \setminus \F$.

\subsection{The consequences of being quadratic}

\begin{Def}
\index{Distance}
\label{page:distance}
For an element $x \in \calW_{p,q}$ and a basic subalgebra $\calC$, we set
$$d_\calC(x):=\max \{\deg \langle \alpha,x\rangle \mid \alpha \in \calC \setminus \F\} \in \N \cup \{-\infty\},$$
and call it the \textbf{distance} of $x$ to $\calC$.
The greatest value among $d_\calA(x)$ and $d_\calB(x)$ is called the \textbf{absolute distance} of $x$, denoted by $\delta(x)$.
\index{Distance (absolute)}
\end{Def}

Now, assume that $x$ is quadratic. Because of the condition $\langle 1,x\rangle=\tr(x) \in \F$,
whenever $d_\calC(x)>0$ we have $d_\calC(x)=\deg \langle \alpha,x\rangle$ whatever the choice of $\alpha \in \calC \setminus \F$.

If $x$ belongs to $\calA \setminus \F$ then its distance to $\calA$ is either $0$ or $-\infty$
(more precisely, it is $-\infty$ if and only if $\F$ has characteristic $2$ and $\calA$ degenerates),
whereas its distance to $\calB$ is $1$.
This imbalance of distances, which we have just observed at the level of the basic subalgebras, will shortly be seen to be a common feature of all the
(nonscalar) quadratic elements. We start with a basic lemma.

\begin{lemma}\label{lemma:quadraticdeployeddegree1}
Let $x \in \calW_{p,q} \setminus \F$ be quadratic, with decomposition
$x=x_1+x_\alpha \, \alpha +x_\beta \, \beta +x_{\alpha \beta }\, \alpha \beta $ in an arbitrary deployed basis $(1,\alpha ,\beta ,\alpha \beta )$ of $\calW_{p,q}$.
Denote by $n$ the greatest degree among $x_1,x_\alpha ,x_\beta ,x_{\alpha \beta }$.
Then:
\begin{enumerate}[(a)]
\item Exactly one of $x_\alpha$ and $x_\beta$ has degree $n$.
\item One has $\deg(x_{\alpha \beta})<n$ in any case.
\item If $n=0$ then $x$ is basic.
\end{enumerate}
\end{lemma}

\begin{proof}
Remember from Corollary \ref{lemma:quadratic} that $N(x)$ and $\tr(x)$ belong to $\F$.
Note first that
$$\tr(x)=2x_1+\tr(\alpha )x_\alpha +\tr(\beta )x_\beta +\langle \alpha ^\star,\beta \rangle x_{\alpha \beta }$$
and $\langle \alpha ^\star,\beta \rangle$ has degree $1$.
Then $\deg(x_{\alpha \beta })<n$, otherwise the last summand would exceed the other ones in degree.

Next, we set $y:=x-x_1$ and rewrite
\begin{equation}\label{eq:NxvsNy}
N(x)=x_1^2+x_1 \tr(y)+N(y)=x_1^2+x_1(\tr(x)-2x_1)+N(y)=(-x_1^2+\tr(x)x_1)+N(y)
\end{equation}
and
\begin{multline}\label{eq:Ny}
N(y)=N(\alpha )\,(x_\alpha )^2+N(\beta )\,(x_\beta )^2+\langle \alpha ,\beta \rangle x_\alpha  x_\beta +N(\alpha )\tr(\beta ) x_\alpha  x_{\alpha \beta }+N(\beta ) \tr(\alpha ) x_\beta  x_{\alpha \beta }\\
+N(\alpha )N(\beta ) (x_{\alpha \beta })^2.
\end{multline}
Moding out the linear subspace $C_{2n}$, we find
$$N(x) \equiv N(y) \equiv \langle \alpha ,\beta \rangle x_\alpha  x_\beta  \quad \text{mod}\; C_{2n}.$$
Yet $N(x) \in \F$, and hence $\deg(\langle \alpha ,\beta \rangle x_\alpha  x_\beta ) \leq 2n$.
Since $\langle \alpha ,\beta \rangle$ has degree $1$ we deduce that $\deg(x_\alpha )<n$ or $\deg(x_\beta )<n$.

Assume now that none of $x_\alpha $ and $x_\beta $ has degree $n$.
Then we must have $\deg(x_1)=n$, and $n>0$ (otherwise $x \in \F$).
Obviously $\deg(N(y)) \leq 2n-1$, and now we find
$$N(x) = -x_1^2+\tr(x)x_1+N(y) \equiv -x_1^2 \quad \text{mod}\; C_{2n-1}.$$
Again this yields $x_1^2 \in C_{2n-1}$, contradicting the fact that $\deg(x_1^2)=2n$.

We conclude that exactly one of $x_\alpha $ and $x_\beta $ has degree $n$.

Assume finally that $n=0$. Then from the above we deduce that if $\deg(x_\alpha )=n$ then $x \in \F[\alpha]$,
and if $\deg(x_\beta )=n$ then $x \in \F[\beta]$.
\end{proof}

We immediately draw a consequence in terms of distances to the basic subalgebras, which we will refine later:

\begin{cor}\label{cor:quadraticdistance1}
Let $x \in \calW_{p,q} \setminus \F$ be quadratic, with decomposition
$x=x_1+x_\alpha \, \alpha +x_\beta \, \beta +x_{\alpha \beta }\, \alpha \beta $ in an arbitrary deployed basis $(1,\alpha ,\beta ,\alpha \beta )$ of $\calW_{p,q}$. Set $n:=\max(\deg(x_\alpha),\deg(x_\beta))$.
\begin{enumerate}[(i)]
\item If $\deg(x_\alpha)=n$ then $d_{\F[\beta]}(x)=n+1$ and $d_{\F[\alpha]}(x) \leq n$.
\item If $\deg(x_\beta)=n$ then $d_{\F[\beta]}(x) \leq n$ and $d_{\F[\alpha]}(x) = n+1$.
\end{enumerate}
In any case $d_\calA(x) \neq d_\calB(x)$ and $\delta(x)=n+1$.
\end{cor}

\begin{proof}
We note that
$$\langle \alpha,x\rangle=\tr(\alpha)x_1+2N(\alpha) x_\alpha+\langle \alpha,\beta\rangle x_\beta+N(\alpha) \tr(\beta) x_{\alpha \beta}$$
and
$$\langle \beta,x\rangle=\tr(\beta)x_1+\langle \alpha,\beta\rangle x_\alpha+2N(\beta) x_\beta+N(\beta) \tr(\alpha) x_{\alpha \beta}.$$
If $\deg(x_\beta) <n$ we use Lemma \ref{lemma:quadraticdeployeddegree1} to see that all the summands in $\langle \alpha,x\rangle$
have degree at most $n$, while all the summands in $\langle \beta,x\rangle$ but $\langle \alpha,\beta\rangle x_\alpha$ have degree at most $n$,
and $\langle \alpha,\beta\rangle x_\alpha$ has degree $n+1$. This yields the first result. The second one is obtained symmetrically.
The third one is straightforward.
\end{proof}

\begin{Def}
Let $x \in \calW_{p,q} \setminus \F$ be quadratic.
The basic subalgebra with smaller distance to $x$ is called the \textbf{leading basic subalgebra} of $x$,
while the one with larger distance to $x$ is called the \textbf{trailing basic subalgebra} of $x$.
\index{Leading (basic subalgebra)}
\index{Trailing (basic subalgebra)}
\end{Def}

In particular, if $x$ is basic and nonscalar then its leading basic subalgebra is simply the only basic subalgebra that contains it.

\begin{lemma}
Let $x \in \calW_{p,q} \setminus \F$ be quadratic, and let $\alpha$ be a nonscalar vector in its leading basic subalgebra.
Take a nonscalar vector $\beta$ in the trailing basic subalgebra of $x$, and write
$x=x_1+x_\alpha \, \alpha +x_\beta \, \beta +x_{\alpha \beta }\, \alpha \beta $.
Then the validity of the condition $\deg(x_1)<\deg(x_\alpha)$ does not depend on the choice of $\beta$.
When it is satisfied, we say that $\alpha$ is a \textbf{leading vector} of $x$, or that it is leading for $x$.
\index{Leading (vector)}
\end{lemma}

\begin{proof}
Indeed, let us take $\beta' \in \F[\beta] \setminus \F$ and write $\beta=\lambda \beta'+\mu$ for some
$(\lambda,\mu)\in \F^\times \times \F$.
Then
$$x=(x_1+\mu x_\beta)+(x_\alpha+\mu x_{\alpha\beta}) \alpha+\lambda x_\beta \beta'+\lambda x_{\alpha\beta} \alpha \beta'.$$
Set $n:=\deg(x_\alpha)$.
Since $\alpha$ belongs to the leading subalgebra of $x$, we have $\deg(x_{\alpha\beta})<n$, and hence
$\deg(x_\alpha+\mu x_{\alpha\beta})=n$.
For the same reason $\deg(x_\beta)<n$, and hence $\deg(x_1+\mu x_\beta)<n$ if and only if $\deg(x_1)<n$.
\end{proof}

\begin{lemma}\label{lemma:leading}
Let $x \in \calW_{p,q} \setminus \F$ be quadratic.
Then $x$ has a leading vector, and it is unique up to multiplication by an element of $\F^\times$.
\end{lemma}

\begin{proof}
Let us take arbitrary nonscalar basic vectors $\alpha$ and $\beta$, with $\alpha$ (respectively, $\beta$) in the
leading (respectively, trailing) basic subalgebra of $x$.
Let us write
$x=x_1+x_\alpha \alpha+x_\beta \beta +x_{\alpha\beta} \alpha\beta $ in that basis, and set
$n:=\deg(x_\alpha)$.
Let $\alpha' \in \F[\alpha] \setminus \F$, and write $\alpha=u\alpha'+v$ with $u \in \F^\times$ and $v \in \F$.
Then
$$x=(x_1+v x_\alpha)+u x_\alpha \alpha'+ ? \beta+u \alpha' \beta.$$
Denote by $n$ the degree of $x_\alpha$, and by $\lambda$ and $\mu$ the respective coefficients of $x_\alpha$ and $x_1$ on $\omega^n$
as polynomials in $\F[\omega]$.
We still have $\deg(u x_\alpha)=n$, and now
$\deg(x_1+v x_\alpha)<n$ if and only if $\mu+\lambda v=0$, i.e., $v=-\lambda^{-1} \mu=:v_0$.
Hence, the leading vectors of $x$ are exactly the vectors of $\F^\times (\alpha-v_0)$.
\end{proof}

\begin{Def}
Let $x \in \calW_{p,q} \setminus \F$ be quadratic.
A deployed basis $(1,\alpha,\beta,\alpha\beta)$ is called \textbf{adapted to} $x$ when $\alpha$ is leading for $x$.
\index{Adapted (deployed basis)}
\end{Def}

\begin{Rem}
Assume that we have written a decomposition $x=x_1+x_\alpha \alpha+x_\beta \beta+x_{\alpha\beta} \alpha\beta$,
with $x$ quadratic and nonscalar, and $\deg(x_\beta)>\max(\deg(x_1),\deg(x_\alpha))$.
Then, and although $\F[\beta]$ is the leading basic subalgebra of $x$ in that case,
we cannot infer that $\beta$ is leading for $x$. In fact, $\beta^\star$ is leading for $x$ since
$$x=\tr(x)-x^\star=(\tr(x)-x_1)-x_\beta \beta^\star-x_\alpha \alpha^\star - x_{\alpha\beta} \beta^\star\alpha^\star.$$
\end{Rem}

Now, we move forward in our analysis of the situation. We will systematically use adapted deployed bases:

\begin{lemma}\label{lemma:quadraticdeployeddegree2}
Let $x \in \calW_{p,q} \setminus \F$ be quadratic and nonbasic, with decomposition
$x=x_1+x_\alpha \alpha+x_\beta \beta+x_{\alpha\beta} \alpha\beta$ in an \emph{adapted} deployed basis.
Set $n:=\deg(x_\alpha)$ and $\omega':=\langle \alpha,\beta \rangle$.
Writing the elements of $C$ as polynomials in $\omega'$ with coefficients in $\F$, we
denote by $L(x_\beta)$ and $L(x_{\alpha\beta})$ the respective coefficients of $x_\beta$ and $x_{\alpha\beta}$ on $(\omega')^{n-1}$,
and by $L(x_\alpha)$ the coefficient of $x_\alpha$ on $(\omega')^n$.
Then:
$$L(x_{\alpha\beta})=\tr (\alpha)\,L(x_\alpha) \quad \text{and} \quad L(x_\beta)=-N(\alpha)\,L(x_\alpha).$$
As a consequence, $\deg(x_\beta)<n-1$ if and only if $N(\alpha)=0$.
\end{lemma}

\begin{proof}
We go back to the analysis of the proof of Lemma \ref{lemma:quadraticdeployeddegree1}, with the same notation.
With $\tr(x)=2x_1+\tr(\alpha) x_\alpha+\tr(\beta) x_\beta+\langle \alpha,\beta^\star\rangle x_{\alpha\beta}$
and $\langle \alpha,\beta^\star\rangle=\tr(\beta)\tr(\alpha)-\omega'$, we find
the first identity because $\deg(x_1) <n$, $\deg(x_\beta)<n$ and $\deg(x_{\alpha\beta})<n$.
Next, since $\deg(x_1)<n$ we see that $\deg(-x_1^2+\tr(x)x_1)<2n$
and we obtain
$$N(x) \equiv N(y) \equiv N(\alpha) (x_\alpha)^2+\omega' x_\alpha x_\beta \quad \text{mod}\; C_{2n-1}.$$
Since $N(x) \in \F$, we obtain $N(\alpha) L(x_\alpha)^2+L(x_\alpha) L(x_\beta)=0$ by extracting the coefficient on $(\omega')^{2n}$, and hence
$L(x_\beta)=-N(\alpha) L(x_\alpha)$ because $L(x_\alpha) \neq 0$.
The last statement is then obvious.
\end{proof}

Let us now see how the previous result plays out in terms of the distance of $x$ to the basic subalgebras.

\begin{cor}\label{cor:differenceofdistances}
Let $x \in \calW_{p,q} \setminus \F$ be quadratic and nonbasic.
Then $|d_\calA(x)-d_\calB(x)|=1$ if and only if the leading vectors of $x$ are units.
\end{cor}

\begin{proof}
We write $x=x_1+x_\alpha \alpha+x_\beta \beta+x_{\alpha \beta} \alpha \beta$ in an adapted deployed basis, and set
$n:=\deg(x_\alpha)>0$ and $\omega':=\langle \alpha,\beta\rangle$. We have $n>0$ because $x$ is nonbasic.
We have seen in Corollary \ref{cor:quadraticdistance1} that $d_{\F[\beta]}(x)=n+1$ and $d_{\F[\alpha]}(x) \leq n$.
Now we come back to the identity
$$\langle \alpha ,x\rangle=\tr(\alpha)\,x_1+2 N(\alpha)\, x_\alpha+\langle \alpha,\beta\rangle x_\beta+N(\alpha)\tr(\beta)x_{\alpha\beta},$$
and we spot that all the summands in the right-hand side but $2N(\alpha) x_\alpha$ and $\langle \alpha,\beta\rangle x_\beta$
have degree less than $n$. Hence, with the same notation as in Lemma \ref{lemma:quadraticdeployeddegree2}, we find that the
coefficient of $\langle \alpha ,x\rangle$ on $(\omega')^n$ equals $2N(\alpha)L(x_\alpha)+L(x_\beta)=N(\alpha) L(x_\alpha)$.
Since $L(x_\alpha) \neq 0$, we deduce that $\langle \alpha ,x\rangle$ has degree $n$ if and only if $N(\alpha)\neq 0$, QED.
\end{proof}

So far, the assumption that $x$ is quadratic has only been used by noting that the trace and norm have no terms of high degree.
Hence the reader might be skeptical that the previous analysis will be sufficient to make our strategy work.
Yet this apparently narrow analysis will fully suffice, as we shall see in the next sections.

\subsection{The effect of conjugating by a basic unit}

We arrive at the critical point of the analysis.
Let us take a quadratic but nonbasic element $x \in \calW_{p,q} \setminus (\calA \cup \calB)$,
with decomposition $x=x_1+x_\alpha \alpha+x_\beta \beta+x_{\alpha\beta} \alpha\beta$ in an adapted deployed basis.

We seek to conjugate $x$ with a basic unit so as to decrease the absolute distance.
To this end, we must first decide whether we should conjugate $x$ with a basic unit in the leading or in the trailing subalgebra of $x$.
To start with, we observe thanks to identities \eqref{eq:adjeq} that for every unit $y$ (not just a basic one) and for all $z_1,z_2$ in
$\calW_{p,q}$,
$$\langle z_1,yz_2y^{-1}\rangle=N(y)^{-1} \langle z_1,yz_2y^\star \rangle=
N(y)^{-1} \langle y^\star z_1 y,z_2\rangle=\langle y^{-1} z_1 y,z_2\rangle.$$
Hence the distance of an element to a basic subalgebra is invariant under conjugation by a basic element of that subalgebra.
And as we want to have the absolute distance of $x$ decrease, the only way is to conjugate $x$ with a basic unit in $\F[\alpha]$, its leading basic
subalgebra.
It will however be useful to investigate what happens when we conjugate $x$ with a basic unit in its trailing basic subalgebra.
The next result handles both situations:

\begin{lemma}\label{lemma:conjugatebasicunit}
Let $x \in \calW_{p,q} \setminus \F$ be quadratic with leading vector $\alpha$,
leading basic subalgebra $\calC$ and trailing basic subalgebra $\calD$.
\begin{enumerate}[(a)]
\item Let $\gamma \in \calC^\times \setminus \F^\times$ and assume that $x$ is nonbasic. Then:
\begin{enumerate}[(i)]
\item $\delta(\gamma^{-1} x \gamma) \leq \delta(x)$;
\item $\delta(\gamma^{-1} x \gamma) < \delta(x)$ if and only if $\gamma \in \F^\times \alpha$;
\item If $\gamma \in \F^\times \alpha$ then
$\delta(\gamma^{-1} x \gamma)=\delta(x)-1$ and $\calC$ is the trailing basic subalgebra of $\gamma^{-1} x \gamma$.
\end{enumerate}

\item Let $\gamma \in \calD^\times \setminus \F^\times$. Then
$\delta(\gamma x \gamma^{-1})=\delta(x)+1$ and $\gamma$ is leading for $\gamma x \gamma^{-1}$.
\end{enumerate}
\end{lemma}

\begin{proof}
Let first $\gamma \in \calC \setminus \F^\times$, and assume that $x$ is nonbasic.
We write $\gamma=\lambda \alpha+\mu$ with $\lambda \in \F^\times$
and $\mu \in \F$. We extend $\alpha$ to a deployed basis $(1,\alpha,\beta,\alpha\beta)$ of $\calW_{p,q}$.
We put $x':=N(\gamma) \gamma^{-1} x \gamma=\gamma^\star x \gamma$, and write $x=x_1+x_\alpha \alpha+x_{\beta} \beta+x_{\alpha\beta} \alpha\beta$.

Let us set $n:=\deg(x_\alpha)$.
Note already that $d_\calC(x')=d_\calC(x) \leq n$ and $d_\calD(x)=n+1$.
We shall investigate the coefficients $x'_\alpha$ and $x'_{\beta^\star}$ of $x'$, respectively on $\alpha$ and $\beta^\star$,
in the deployed basis $(1,\alpha,\beta^\star,\alpha \beta^\star)$, which turns out to be more convenient than $(1,\alpha,\beta,\alpha\beta)$.

To start with, we note that $\gamma^\star 1 \gamma=N(\gamma) \in \F$ and $\gamma^\star \alpha \gamma=N(\gamma) \alpha$.
Next,
$$\gamma^\star \beta \gamma=\langle \beta,\gamma^\star\rangle \gamma^\star -(\gamma^\star)^2 \beta^\star$$
and hence, since $\alpha$ commutes with $\gamma^\star$,
$$\gamma^\star (\alpha\beta) \gamma=\alpha (\gamma^\star \beta \gamma)=\langle \beta,\gamma^\star\rangle (\alpha\gamma^\star)-\alpha(\gamma^\star)^2 \beta^\star.$$
Next, we write
$$\gamma^\star x \gamma=x_1 \gamma^\star \gamma+x_\alpha \gamma^\star \alpha \gamma+x_{\beta} \gamma^\star \beta \gamma+x_{\alpha\beta} \gamma^\star (\alpha \beta) \gamma.$$
Since $\F[\alpha]$ contains $(\gamma^\star)^2$ and $\alpha (\gamma^\star)^2$, the elements $(\gamma^\star)^2 \beta^\star$ and
$\alpha (\gamma^\star)^2 \beta^\star$ do not contribute to the coefficient of $\gamma^\star x \gamma$ on $\alpha$ in $(1,\alpha,\beta^\star,\alpha \beta^\star)$.
Hence, by using $\alpha \gamma^\star=\alpha(\lambda \alpha^\star+\mu)=\lambda N(\alpha)+\mu \alpha$, we find
$$x'_\alpha=N(\gamma) x_\alpha+\langle \beta,\gamma^\star\rangle(-\lambda x_\beta+\mu x_{\alpha\beta}).$$
Next, we put $\omega':=\langle \alpha,\beta\rangle$ and simplify
 $\langle \beta,\gamma^\star\rangle \equiv -\lambda \omega'$ mod $C_0$, which we combine with the identities in Lemma
 \ref{lemma:quadraticdeployeddegree2}. Denoting by $L(x_\alpha)$ the coefficient of $x_\alpha$ on $(\omega')^n$ as an element of $\F[\omega']$,
 we find
$$x'_\alpha \equiv (\omega')^nL(x_\alpha) \bigl(N(\gamma) -\lambda (\lambda N(\alpha)+\mu \tr(\alpha))\bigr) \quad \text{mod} \; C_{n-1},$$
which, by expanding $N(\gamma)$, is further simplified as
$$x'_\alpha \equiv (\omega')^n L(x_\alpha)\, \mu^2 \quad \text{mod} \; C_{n-1}.$$
As a first consequence $\deg(x'_\alpha) \leq n$. 
Next, if $\deg(x'_{\beta^\star}) \geq n$ then $\deg(x'_{\beta^\star}) \geq n \geq \deg(x'_\alpha)$ and hence point (ii) of Corollary \ref{cor:quadraticdistance1}
yields $d_\calC(x') \geq n+1$. Hence $\deg(x'_{\beta^\star}) <n$.
Hence, if $\deg(x'_\alpha)=n$ then $\calC$ is the leading subalgebra of $x'$ and $\delta(x')=n+1$.
Moreover
$$\deg(x'_\alpha)<n \Leftrightarrow \mu=0 \Leftrightarrow \gamma \in \F^\times \alpha,$$
to the effect that $\deg(x'_\alpha)<n$ only if $N(\alpha)\neq 0$ (because $\gamma$ must be a unit!).
Assume for a moment that $\deg(x'_\alpha)<n$. Then $N(\alpha) \neq 0$ and hence Corollary \ref{cor:differenceofdistances}
shows that $d_\calC(x)=n$. Thus $d_{\calC}(x')=d_\calC(x)=n$. Since $\deg(x'_\alpha)<n$ and $\deg(x'_{\beta^\star})<n$,
we know from Corollary \ref{cor:quadraticdistance1} that $\delta(x')\leq n$, and we conclude since 
$d_\calC(x')=n$ that $\delta(x')=n=d_\calC(x')$.
In particular $\calC$ is the trailing basic subalgebra of $x'$. Statement (a) is now entirely proven.

Finally, to prove statement (b) we can directly take $\gamma=\beta$, since $\beta$ was chosen arbitrarily
as a nonscalar element in $\calD$. Note that we no longer assume that $x$ is nonbasic.
Once more, we shall consider the coefficients of $x':=\gamma x \gamma^\star=\beta x \beta^\star$ in the deployed basis
$(1,\beta,\alpha^\star,\beta \alpha^\star)$.
Again, we compute $\beta 1 \beta^\star=N(\beta)$, $\beta \beta \beta^\star=N(\beta) \beta$,
$$\beta \alpha \beta^\star=\beta (\langle \alpha,\beta\rangle-\beta \alpha^\star)=\langle \alpha,\beta\rangle \beta -\beta^2 \alpha^\star$$
and finally
$$\beta \alpha \beta \beta^\star=N(\beta) \beta \alpha=(N(\beta) \tr(\alpha)) \beta-N(\beta) \beta \alpha^\star.$$
It follows that
$$\beta x \beta^\star=N(\beta) x_1+(\langle \alpha,\beta\rangle x_\alpha+N(\beta)x_\beta+N(\beta)\tr(\alpha)x_{\alpha\beta}) \beta
+N(\beta) x_\alpha\, \alpha^\star +? \beta \alpha^\star.$$
Obviously $\deg(N(\beta) x_1)<n+1$, $\deg(N(\beta) x_\alpha) \leq n$ and
$$\deg(\langle \alpha,\beta\rangle x_\alpha+N(\beta)x_\beta+N(\beta)\tr(\alpha)x_{\alpha\beta})=n+1$$
because $\deg(x_{\alpha\beta})<n$, $\deg(x_\alpha)=n$ and $\deg(x_\beta)<n$.
Hence $\gamma=\beta$ is a leading vector of $x'':=\gamma x \gamma^{-1}=N(\gamma)^{-1} \gamma x \gamma^\star$,
and $\delta(x'')=n+2=\delta(x)+1$.
\end{proof}

\subsection{The retracing algorithm}\label{section:units1:retracing}

\index{Retracing algorithm}
Now that the previous preparatory work has been achieved, we can swiftly fill our objectives.

We start by describing an algorithm, which we call the \textbf{retracing algorithm}: it takes as entry a quadratic element $x \in \calW_{p,q} \setminus \F$
and either reports a failure or outputs a basic vector $x'$ and a (potential empty) list $(\gamma_1,\dots,\gamma_r)$ of basic
units such that $x'=\gamma^{-1} x \gamma$ for $\gamma:=\gamma_1 \gamma_2 \cdots \gamma_r$.
Here is the procedure:
\begin{itemize}
\item Initialize $y$ as $x$ and $L$ as the empty list.
\item While $y$ is nonbasic:
\begin{itemize}
\item Compute a leading vector $\alpha$ for $y$.
\item If $N(\alpha)=0$ return ``Failure''.
\item Else update $y$ to $\alpha^{-1} y \alpha$ and append $\alpha$ to $L$.
\end{itemize}
\item Return $(y,L)$.
\end{itemize}
To see that this algorithm terminates, we note that, after each iteration that
does not return a failure, the absolute distance of the current vector $y$ decreases by exactly one unit,
so if no failure is reported then after exactly $\delta(x)-1$ iterations the current
vector $y$ satisfies $\delta(y)=1$ and hence is basic.

Hence, either the algorithm reports a failure or it outputs $(y,(\gamma_1,\dots,\gamma_n))$ where
the elements $\gamma_i$ are basic units, $y$ is a basic vector and $y=(\gamma_1 \cdots \gamma_n)^{-1} x (\gamma_1 \cdots \gamma_n)$:
in the latter case we say that the algorithm is successful.

Now, we prove that this algorithm is optimal:

\begin{prop}\label{prop:testretracing}
Let $x \in \calW_{p,q} \setminus \F$ be quadratic.
Then the retracing algorithm applied to $x$ is successful if and only if there exists a monomial unit $\gamma$
such that $\gamma^{-1} x \gamma$ is basic.
\end{prop}

\begin{proof}
The ``only if" part has just been explained. Assume conversely that there exists a monomial unit $\gamma$
such that $y:=\gamma^{-1} x \gamma$ is basic.
If $x$ is basic then the retracing algorithm is successful when applied to $x$.
In the remainder of the proof we assume that $x$ is nonbasic.

Let us consider a reduced decomposition $\gamma=\alpha_n \cdots \alpha_1$ into a product of basic units.
We define $x_i:=(\alpha_i \cdots \alpha_1) y (\alpha_i \cdots \alpha_1)^{-1}$ for all $i \in \lcro 0,n\rcro$,
so that $x_0=y$, $x_n=x$ and $x_{i+1}=\alpha_{i+1} x_i \alpha_{i+1}^{-1}$ for all $i \in \lcro 0,n-1\rcro$.

Assume that $y$ and $\alpha_1$ belong to distinct basic subalgebras.
Then, by point (b) of Lemma \ref{lemma:conjugatebasicunit}, the basic vector $\alpha_1$ is leading for $x_1$ and $\delta(x_1)=2$.
Since we have a reduced decomposition, the basic unit $\alpha_2$ belongs to the trailing basic subalgebra of $x_1$,
and hence $\alpha_2$ is leading for $x_2$ and $\delta(x_2)=3$. By finite induction, we obtain that
for all $i \in \lcro 1,n\rcro$, the vector $\alpha_i$ is leading for $x_i$ and $\delta(x_i)=i+1$.
In particular $\alpha_n$ is leading for $x$ and $\delta(x)=n+1>0$.

Hence at the first step, the retracing algorithm must find $\lambda \alpha_n$ as leading vector for $x$, for some $\lambda \in \F^\times$;
since $\lambda \alpha_n$ is a unit, no failure is reported and $x$ is updated to $x_{n-1}$. By downward induction, we find that the retracing algorithm succeeds.

Finally, if $y$ and $\alpha_1$ belong to the same basic subalgebra, then $\alpha_1$ commutes with $y$ and we can replace
$\gamma$ with $\gamma=\alpha_n \cdots \alpha_2$. Then the previous case applies and shows that the retracing algorithm succeeds
when applied to $x$.
\end{proof}

\begin{Rem}
As an application of the previous method,
let us give a new proof of the uniqueness of a reduced decomposition into a product of basic units
(up to multiplication with nonzero scalars) .

To start with, we prove that $1$ has no reduced decomposition of length greater than $1$ into a product of basic units.
So, assume on the contrary that we have such a decomposition $1=\alpha_1 \cdots \alpha_n$ with $n>1$.
Because this is a reduced decomposition, we have $\alpha_1 \not\in\F$. Let us choose
$x$ as a nonscalar basic vector of the basic subalgebra opposite to $\F[\alpha_1]$,
and consider the conjugate $y:=(\alpha_1 \cdots \alpha_n)^{-1} x (\alpha_1 \cdots \alpha_n)$.
Then, by following the line of reasoning of the previous proof, we find $\delta(y)=n+1$, which is absurd because $y=x$ and $\delta(x)=1$.

Next, consider two equal reduced decompositions $x_1 \cdots x_m \sim y_1 \cdots y_p$ into products of basic units, with $m>0$ and $p>0$
(the case $m=0$ or $p=0$ is deduced from the previous result). We proceed by induction on $m+p$.
Assume first that $x_1$ and $y_1$ belong to opposite basic subalgebras.
Then, for some $\lambda \in \F^\times$, we obtain $1=y_p^{-1} \cdots y_1^{-1} x_1 \cdots (\lambda x_m)$ as a reduced decomposition with $m+p>1$, which is not possible.
Hence $x_1$ and $y_1$ belong to the same basic subalgebra.
Assume first that $x_1 \not\sim y_1$. Then $(y_1^{-1} x_1) x_2 \cdots x_m \sim y_2 \cdots y_p$ is an equivalence of reduced decompositions with respective
lengths $m$ and $p-1$, so by induction $m=p-1$ and $y_1^{-1} x_1 \sim y_2$, which is absurd since $y_2$ and $y_1^{-1} x_1$ are nonscalar and belong to opposite
basic subalgebras. Hence $x_1 \sim y_1$. Then $x_2 \cdots x_m \sim y_2 \cdots y_p$, and by induction we deduce that 
$m=p$ and $x_i \sim y_i$ for all $i \in \lcro 2,m\rcro$. This completes the proof.
\end{Rem}

\subsection{The fruits of the retracing algorithm}\label{section:retracingfruits}

We can now collect the fruits of the retracing algorithm.
Key here is the observation that the obstructions to the success of the retracing algorithm vanish if both $p$ and $q$
are irreducible, as in that case every nonzero basic vector is a unit (we do not even need the Zero Divisors Theorem to see this).

Here is a straightforward application:

\begin{theo}\label{theo:conjsousalg2}
If $p$ and $q$ are irreducible, then every quadratic element of $\calW_{p,q}$ is conjugated to a basic vector.
\end{theo}

This is of course a special case of Cohn's theorem 3.5 in \cite{CohnFreeProductSkewFields}.

Note that the retracing algorithm only handles the case of nonscalar quadratic elements, but the remaining case is trivial.

\begin{cor}\label{cor:conjsousalg2}
If $p$ and $q$ are irreducible, then every $2$-dimensional subalgebra of $\calW_{p,q}$ is isomorphic to one of the basic subalgebras.
\end{cor}

The next applications are even more spectacular. They are all based upon the following variation of
Theorem \ref{theo:conjsousalg2}.

\begin{prop}\label{prop:doubleretracing}
Assume that $p$ and $q$ are irreducible, and let $(x,y)$ be a pair of quadratic elements of $\calW_{p,q} \setminus \F$ such that
$\deg \langle x,y\rangle=1$.
Then there exists a monomial unit $\gamma$
such that both $\gamma x \gamma^{-1}$ and $\gamma y \gamma^{-1}$ are basic.
\end{prop}

Of course then, the elements $\gamma x \gamma^{-1}$ and $\gamma y \gamma^{-1}$ belong to opposite basic subalgebras, because of
the requirement that $\deg \langle x,y\rangle=1$.

\begin{proof}
Since $p$ and $q$ are irreducible, the retracing algorithm succeeds when applied to $x$,
and hence yields a monomial unit $\gamma_1 \in \calW_{p,q}^\times$
such that $\gamma_1 x (\gamma_1)^{-1}$ is basic.

Since inner automorphisms leave the center $C$ invariant and commute with the adjunction, it is clear that they
are isometries for the inner product. In particular $\langle \gamma_1 x (\gamma_1)^{-1},\gamma_1 y (\gamma_1)^{-1}\rangle=\langle x,y\rangle$.
Hence we can reduce the situation to the one where $x$ is basic.

So, assume from now on that $x$ is basic, set $\calC:=\F[x]$ and define $\calD$ as the opposite basic subalgebra.
Of course, if $y \in \calD$ we have finished, and now we will assume that $y \not\in \calD$.

We will prove that there exists a basic unit $\alpha \in \calC^\times$ such that $\alpha y \alpha^{-1}$ is basic, which will conclude the proof
because then $\alpha x \alpha^{-1}=x$.
Simply, we note that $x \not\in \F$ otherwise $\langle x,y\rangle=x\tr(y)$ does not have degree $1$.
Combining this fact with $\deg \langle x,y\rangle=1$ leads to $d_\calC(y)=1$.
Since $p$ and $q$ are irreducible, it follows from Corollary \ref{cor:differenceofdistances} that 
$d_\calD(y) \in \{0,2\}$. If $d_\calD(y)=0$ then since $d_\calC(y)=1$ we deduce that $y \in \calD$.
Hence $d_\calD(y)=2$, and since $d_\calC(y)=1$ we obtain that $\calC$ is the leading subalgebra of $y$.
Thus the retracing algorithm gives, in just one step, a unit $\alpha \in \calC^\times$ such that $\alpha y \alpha^{-1}$ is basic. This completes the proof.
\end{proof}

Here is a straightforward application, where we recall that an automorphism of the $\F$-algebra $\calW_{p,q}$
is called basic when it maps every basic vector to a basic vector, which is equivalent to having it map $a$ and $b$ to basic vectors.

\begin{theo}\label{theo:decompositionautomorphismirreducible}
Assume that $p$ and $q$ are irreducible. Then every automorphism of the $\F$-algebra $\calW_{p,q}$
is the composite of a basic automorphism with an inner automorphism of the form $x \mapsto zxz^{-1}$, where $z$ is a monomial unit.
\end{theo}

\begin{proof}
Let $\Phi \in \Aut(\calW_{p,q})$. Set $x:=\Phi(a)$ and $y:=\Phi(b)$.
Then by Proposition \ref{prop:commuteadjunction} we find $\langle x,y\rangle=\Phi(\langle a,b\rangle)$.
We note that $\Phi$ induces an automorphism of the $\F$-algebra $C=\F[\omega]$ (because the latter is the center of $\calW_{p,q}$),
and hence this automorphism preserves the degree.
Hence the pair $(x,y)$ satisfies the assumptions of Proposition \ref{prop:doubleretracing},
and we recover a monomial unit $\gamma$ such that $\gamma x \gamma^{-1}$ and $\gamma y \gamma^{-1}$ are basic.
Hence, for $i_\gamma : z \mapsto \gamma z \gamma^{-1}$, the composite automorphism $\Psi=i_\gamma \circ \Phi$ maps $a$ and $b$
to basic vectors, and hence is a basic automorphism. Finally $\Phi=i_{\gamma^{-1}} \circ \Psi$, and
$\gamma^{-1}$ is obviously a monomial unit.
\end{proof}

In particular, we have proved that when both $p$ and $q$ are irreducible, every automorphism of the $\F$-algebra $\calW_{p,q}$
is the composite of an inner automorphism with a basic automorphism. This is still far from the Automorphisms Theorem (Theorem
\ref{theo:automorphismstheointro}) stated in the introduction (we lack the uniqueness of the decomposition, and most importantly we lack
the case where at least one of $p$ and $q$ splits), but it is a solid first step in that direction.

Our last application of the retracing algorithm is the converse implication in the Weak Units Theorem, but here we must anticipate on a result that will be proved in Section \ref{section:automorphismsII}: this result states that the only basic and inner automorphism is the identity,
which is part of the Automorphisms Theorem and
is proved fully and as an independent item as Theorem \ref{theo:uniquenessdecompauto} in Section \ref{section:uniquenessdecompauto}.

So, assume that $p$ and $q$ are irreducible.
Let $z \in \calW_{p,q}^\times$. By the previous theorem, $i_z : x \mapsto zxz^{-1}$ can be decomposed as $i_\gamma \circ \Psi$
where $\gamma$ is a monomial unit and $\Psi$ is a basic automorphism. Then $i_{\gamma^{-1} z}$ is basic,
and by Theorem \ref{theo:uniquenessdecompauto} in Section \ref{section:uniquenessdecompauto} it is the identity. Therefore $\gamma^{-1} z$ is a central unit in $\calW_{p,q}$, to the effect that $z=\lambda \gamma$ for some $\lambda \in \F^\times$. Hence $z$ is a monomial unit.

Not only does this argument prove that in the case under scrutiny (i.e., both $p$ and $q$ are irreducible)
every unit is a monomial unit, but it also gives an algorithm that takes as entry a unit $z \in \calW_{p,q}^\times$
and outputs a list $(\alpha_1,\dots,\alpha_n)$ of basic units and a scalar $\lambda$ such that $z=\lambda \alpha_1 \cdots \alpha_n$.
This algorithm runs as follows:
\begin{itemize}
\item Initialize $x$ as $z a z^\star$, $L$ as the empty list, $\gamma$ as $z$ and $\pi$ as $1$.
\item While $x$ is not basic:
\begin{itemize}
\item Compute a leading vector $\alpha$ for $x$.
\item Update $x$ to $\alpha^\star x \alpha$, $\gamma$ to $\alpha^\star \gamma$ and $\pi$ to $\pi N(\alpha)$.
\item Append $\alpha$ to $L$.
\end{itemize}
\item Append $\gamma$ to $L$. Append $\pi$ to $L$. \\
\item Return $L$.
\end{itemize}

We write the output as $L=(\alpha_1,\dots,\alpha_n,\pi)$ and we claim that
\begin{equation}\label{eq:outputalgo}
z=\pi^{-1} \alpha_1 \cdots \alpha_n.
\end{equation}
To see this, note that when the loop stops the current conjugator $\gamma$ is $\gamma=\alpha_{n-1}^\star \cdots \alpha_1^\star z=\alpha_n$ and the
current $\pi$ is $N(\alpha_1 \cdots \alpha_{n-1})$.

Next, $\alpha_1,\dots,\alpha_{n-1}$ are all basic units, and we must explain why $\alpha_n$ is also a basic unit.
Set $a':=\alpha_{n-1}^\star \cdots \alpha_1^\star (z a z^\star) \alpha_1 \cdots \alpha_{n-1}$, which is the value of $x$
after the loop stops.
So, set $b':=\alpha_{n-1}^\star \cdots \alpha_1^\star (z b z^\star) \alpha_1 \cdots \alpha_{n-1}$.
The details of the proof of Proposition \ref{prop:doubleretracing} show that there exists a basic unit
$\beta \in \F[a']$ such that
$\beta^\star b' \beta$ is basic.
Then with $\gamma':=\beta^\star \gamma$, we deduce that $\gamma' a (\gamma')^{-1}$ and $\gamma' b (\gamma')^{-1}$
are basic. Hence Theorem \ref{theo:uniquenessdecompauto} shows that $\gamma' \in \F^\times$ and we conclude that
$\beta \sim \gamma=\alpha_n$, which shows that $\alpha_n$ is a basic unit.

\begin{Rem}
The above algorithm is written so that it avoids using divisions, and the only division needed is to compute
$\pi^{-1}$ in \eqref{eq:outputalgo}. Indeed, in computing a leading vector for $x$ or $y$, it is easily seen from the proof of Lemma \ref{lemma:leading} that one can avoid dividing in $\F$, as in the notation of this lemma we can take $\alpha'=\lambda \alpha-\mu$ as leading vector.
\end{Rem}

\subsection{Counterexamples in the split case}\label{section:units1counterexamples}

We conclude by proving that the requirement that both $p$ and $q$ be irreducible
is necessary in the Weak Units Theorem. So, we assume that one of $p$ and $q$ is reducible,
and we construct a unit that is not a monomial one.

In any case, we have a nonscalar basic element $\alpha$ such that $N(\alpha)=0$.
Hence $\alpha^2=\tr(\alpha) \alpha$, and by scaling $\alpha$ we can reduce the situation to only two situations:
$\tr(\alpha)=1$ or $\tr(\alpha)=0$. In any case, we choose a nonscalar vector $\beta$ in the basic subalgebra opposite to $\F[\alpha]$.

In both cases, the idea is to pick a well-chosen element $z \in \calW_{p,q}$ such that
$\langle \alpha,z\rangle=0$, and we consider the element
$$U:=1+\alpha z^\star.$$
Then we note that
$$N(U)=1+\langle \alpha z^\star,1\rangle+N(\alpha z^\star)=1+\langle \alpha,z\rangle+N(\alpha)N(z^\star)=1.$$
Hence $U$ is a unit, and in any case we note that $U^{-1}=U^\star=1+z \alpha^\star$.
 Next, in order to know that $U$ is not a monomial unit it suffices to find
a nonscalar basic vector $y$ such that the retracing algorithm fails for $x:=U^{-1} y U$
or for $U y U^{-1}$.
Indeed, if $U$ is a monomial unit then so is $U^{-1}$, and we know from Proposition \ref{prop:testretracing}
that the retracing algorithm must succeed when applied to $x$.

Now, we need to split the discussion into two cases.

\vskip 3mm
\noindent \textbf{Case 1: $\tr(\alpha)=1$.} \\
Here we take $z:=-\langle \alpha,\beta\rangle+\beta$, which clearly satisfies $\langle \alpha,z\rangle=0$.
We choose $y:=\alpha$ and prove that the retracing algorithm fails for $x:=U^{-1} y U=U^\star \alpha U$.
Note to this end by using $\alpha^\star \alpha=0$ and $\alpha^2=\alpha$ that
$$x=\alpha (1+\alpha z^\star)=\alpha(1+z^\star)=\left[1-\langle \alpha,\beta\rangle\right] \alpha+\alpha \beta^\star.$$
By computing in the deployed basis $(1,\alpha,\beta^\star,\alpha \beta^\star)$, we deduce that
$\alpha$ is leading for $x$. Since $\alpha$ is a zero divisor, we deduce that the retracing algorithm fails for~$x$.

\vskip 3mm
\noindent
\textbf{Case 2: $\tr(\alpha)=0$.} \\
Then $\alpha^2=0$. Let us take an arbitrary nonscalar central element $z \in C \setminus \F$.
Clearly $\langle \alpha,z\rangle=z \tr(\alpha)=0$ and $U=1+\alpha z$.
This time around we take $y:=\beta$ and set $x:=U \beta U^{-1}$.
Then
\begin{align*}
x & = U \beta U^\star \\
& = U(\langle U,\beta\rangle- U \beta^\star) \\
& = \langle U,\beta\rangle U-U^2 \beta^\star \\
& = \langle U,\beta\rangle+ z  \langle U,\beta\rangle \alpha-(1+2z \alpha) \beta^\star \\
& = \langle U,\beta\rangle+ z  \langle U,\beta\rangle \alpha-\beta^\star-2z \alpha \beta^\star.
\end{align*}
Finally $\langle U,\beta\rangle=\tr(\beta)+z\langle \alpha,\beta\rangle$ has degree $1+\deg(z)$, and
$z \langle U,\beta\rangle$ has degree $1+2\deg(z)>1+\deg(z)$. It is then clear that
$\alpha$ is a leading vector for $x$. Since $\alpha$ is a zero divisor, the retracing algorithm fails for $x$,
and we conclude that $U$ is not a monomial unit.

Hence, in any case we have exhibited a unit that is not monomial.
Therefore, the Weak Units Theorem now entirely rests upon the validity of Theorem \ref{theo:uniquenessdecompauto}.

\begin{Rem}
In Case 2, starting from $y:=\alpha$ yields $U \alpha U^{-1}=\alpha$, an element for which the retracing algorithm
succeeds. Hence the criterion that, for a given nonscalar basic vector $x$, the retracing algorithm fails for the conjugate $z x z^{-1}$
is only a sufficient condition for $z$ not to be monomial, but not a necessary one.

A correct necessary and sufficient condition for $z$ to be monomial is actually that the
retracing algorithm fails for both $zaz^{-1}$ and $zbz^{-1}$, but proving this is premature at this point.
\end{Rem}

Now, the situation is clear when both polynomials $p$ and $q$ are irreducible, but many questions remain in the other cases:
\begin{itemize}
\item What are the missing generators?
\item Can we give a reasonably simple expression of the units group in that case (as an amalgamated product of two more elementary subgroups)?
\end{itemize}

Both these questions will be answered in Section \ref{section:units2}, which builds upon the present one and takes the method further.
Interestingly, the units we have just shown to be non-monomial will be very close to the missing generators.
Since this is a very technical study, we turn to other issues in the meantime.

\section{Maximal ideals in the free Hamilton algebra}\label{section:ideals}

\index{Maximal ideals}
Here and from now on, all the ideals we consider are two-sided ideals, and we will never repeat this precision.
We are uninterested here in modules over $\calW_{p,q}$, so left ideals and right ideals are irrelevant to us.

Our main aim is to determine the maximal ideals of $\calW_{p,q}$. This is considerably helped by the connection with quaternion algebras,
but will require a deeper analysis due to the degeneracy at the fundamental ideal.

Remember throughout that $\Irr(\F)$ stands for the set of all monic and irreducible polynomials in $\F[t]$
(and, as usual, $t$ is an indeterminate). Throughout $\K$ denotes a splitting field of $pq$, which we fix once and for all
(we will regularly recall the meaning of this notation, though).

\subsection{The first step}

To start with, take a proper ideal $J$ of $C$.
Then $J+Ja+Jb+Jab$ is clearly an ideal of $\calW_{p,q}$ that includes $J$, so it is the ideal generated by $J$, denoted by $\langle J\rangle$.
Because $(1,a,b,ab)$ is a $C$-basis the intersection of this ideal with $C$ equals $J$.
In particular, $\langle J\rangle \neq \calW_{p,q}$, and $\langle J\rangle$ is maximal as an ideal of $\calW_{p,q}$ only if
$J$ is maximal as an ideal of $C$ (but the converse may fail, as we shall see).

Conversely, we now consider the situation of an arbitrary non-zero ideal of $\calW_{p,q}$.
The starting result is not new (see the proof of point (ii) of theorem 4 in \cite{SZW}), but our proof is original.

\begin{prop}\label{prop:nonzeroidealmeetscenter}
Let $I$ be a non-zero ideal of $\calW_{p,q}$. Then $I \cap C \neq \{0\}$.
\end{prop}

\begin{proof}
Assume on the contrary that $I \cap C=\{0\}$.

To start with, we note that $xx^\star \in I \cap C$ for all $x \in I$, to the effect that $N$
vanishes on $I$. By polarizing, we deduce that $I$ is a totally singular $C$-linear subspace for the form $\langle -,-\rangle$.

Next, let $z \in I$ have trace $0$.
Then $z^\star=-z \in I$. For all $y \in \calW_{p,q}$, we deduce that
$\langle z,y\rangle=zy^\star+yz^\star \in I$, and since $\langle z,y\rangle \in C$ we find $\langle z,y\rangle=0$.
Therefore $z$ is in the radical of $\langle -,-\rangle$, to the effect that $z=0$. We conclude that the sole trace zero element in $I$ is $0$.

It ensues that $\tr$ is injective on $I$, and since $I \neq \{0\}$ this yields that $I$ is a free $C$-module with rank $1$, i.e.,
$I=C x_0$ for some $x_0 \in I \setminus \{0\}$.
Next, note that the orthogonal complement of $Cx_0$ for $\langle -,-\rangle$
is a free $C$-module with rank $3$. Hence we can pick an arbitrary element $y$ of it that is linearly independent of $x_0$ over $C$.
Then $y x_0^\star=-x_0y^\star \in I$. Hence $\tr(x_0)y-yx_0 \in I$ and finally $\tr(x_0) y \in I$.
It ensues that $\tr(x_0) y=\lambda x_0$ for some $\lambda \in C$, which contradicts the assumed linear independence because $\tr(x_0)\neq 0$.
\end{proof}

\begin{cor}\label{cor:maxidealcenter}
Let $I$ be a maximal ideal of $\calW_{p,q}$. Then $I \cap C$ is a maximal ideal of $C$.
\end{cor}

\begin{proof}
By Proposition \ref{prop:nonzeroidealmeetscenter}, the ideal $I \cap C$ of $C$ is nonzero, and it does not equal $C$ because $1 \not\in I$.
Hence $I \cap C=C r(\omega)$ for some nonconstant polynomial $r$. Assume that $r$ is reducible,
and consider a divisor $s$ of it such that $s$ and $\frac{r}{s}$ are non constant.

Set $\widetilde{I}:=(s(\omega))+I$.
Since $s(\omega) \not\in I$, we see that $I \subsetneq \widetilde{I}$.
Since $I$ is maximal we deduce that $\widetilde{I}=\calW_{p,q}$.
Since $s(\omega)$ is central, it follows that $1=s(\omega) x+y$ for some $x$ in $\calW_{p,q}$ and some $y \in I$.
Computing the norm, we obtain $1=s(\omega)^2 N(x)+s(\omega) \langle x,y\rangle+N(y)$.
Yet $N(y)=yy^\star \in I \cap C$, and hence $s(\omega)$ divides $1$ in $C$. This is absurd.
Therefore $r$ is irreducible, to the effect that $I \cap C$ is a maximal ideal of $C$.
\end{proof}

Now, take a maximal ideal $I$ of $\calW_{p,q}$. By Corollary \ref{cor:maxidealcenter}, the intersection $I \cap C$
equals $Cr(\omega) $ for a unique $r \in \Irr(\F)$. Then $(r(\omega)) \subseteq I$. Note that
the quotient ring $\L:=C/(I \cap C) \simeq \F[t]/(r)$ is a field. There are two cases:
\begin{itemize}
\item Either $r$ is relatively prime with $\Lambda_{p,q}$, in which case the analysis from Section \ref{section:quaternionalgebras} shows that
$\calW_{p,q}/(r(\omega))$ is a quaternion algebra over $\L$.
As every quaternion algebra is simple, it follows that $I=(r(\omega))$.
\item Or $r$ divides $\Lambda_{p,q}$, i.e., $\mathfrak{F} \subseteq I$, and in that
case we must push the analysis further.
\end{itemize}
In particular, the same argument as in the first point
shows that the ideal $(r(\omega))$ of $\calW_{p,q}$ is maximal
for every $r \in \Irr(\F)$ that is relatively prime with $\Lambda_{p,q}$.

Hence, we have a partial conclusion at this point:

\begin{theo}
The maximal ideals of $\calW_{p,q}$ are:
\begin{enumerate}[(i)]
\item The ideals of the form $(r(\omega))$ where $r \in \Irr(\F)$ is relatively prime with $\Lambda_{p,q}$;
\item The maximal ideals that include $(r(\omega))$ for some $r \in \Irr(\F)$ that divides $\Lambda_{p,q}$.
\end{enumerate}
\end{theo}

The latter are also the maximal ideals that include the fundamental ideal $\mathfrak{F}$.
It remains to understand them, which we do in the next two sections.

Our aim is the following theorem, where we recall that $\K$ is a fixed splitting field of $pq$.

\begin{theo}
Let $r \in \Irr(\F)$ be a monic irreducible divisor of $\Lambda_{p,q}$.
Then:
\begin{enumerate}[(a)]
\item For every maximal ideal $I$ of $\calW_{p,q}$ that includes $(r(\omega))$,
the quotient $\F$-algebra $\calW_{p,q}/(r(\omega))$ is isomorphic to the splitting field of $pq$.
\item There are one or two maximal ideals of $\calW_{p,q}$ that include $(r(\omega))$.
\item The following conditions are equivalent:
\begin{enumerate}[(i)]
\item There is a unique maximal ideal of $\calW_{p,q}$ that includes $(r(\omega))$.
\item Either both $p$ and $q$ have a double root in $\K$,
or one is irreducible and the other one does not split with simple roots in $\F$.
\end{enumerate}
\end{enumerate}
\end{theo}

The next section quickly discusses the nature of the fundamental polynomial $\Lambda_{p,q}$ as a function of the polynomials $p$ and $q$.

\subsection{Relationship between the basic subalgebras and the fundamental ideal}\label{section:Lambdapq}

In the prospect of the next two sections, it is crucial to understand exactly when $\Lambda_{p,q}$ is irreducible, and in case it splits,
to understand whether $\Lambda_{p,q}$ has simple roots or not.
To this end, consider the splitting field $\K$ of $pq$, and split $p$ and $q$ over $\K$ as
$$p=(t-x_1)(t-x_2) \quad \text{and} \quad q=(t-y_1)(t-y_2),$$
so that
$$\Lambda_{p,q}=\bigl(t-(x_1y_1+x_2y_2)\bigr)\bigl(t-(x_2y_1+x_1y_2)\bigr).$$
Our first observation is that the discriminant of $\Lambda_{p,q}$ equals
$$\disc(\Lambda_{p,q})=((x_1-x_2)(y_1-y_2))^2=\disc(p)\disc(q),$$
so $\Lambda_{p,q}$ has simple roots in $\K$ if and only if both $p$ and $q$ have simple roots in $\K$.
Moreover, if $p$ has a double root $x$ in $\K$, then the sole root of $\Lambda_{p,q}$ is $x\tr(q)$,
which belongs to $\F$ if and only if $x\in \F$ or $\tr(q)=0$.

Obviously, if $p$ and $q$ split over $\F$, then $\Lambda_{p,q}$ also does. For the remaining situations, we split the discussion into many cases.

\begin{itemize}
\item \textbf{Case 1.} Exactly one of $p$ and $q$ is irreducible, say $q$.
\begin{itemize}
\item \textbf{Subcase 1.1:} $p$ has a double root. Then $\Lambda_{p,q}$ has a double root in $\F$.
\item \textbf{Subcase 1.2:} $p$ has simple roots and $q$ has a double root $y$. This is possible only if $\car(\F)=2$. Then
$\Lambda_{p,q}$ has double root $y \tr(p)$, and $\tr(p) \neq 0$ because $\car(\F)=2$ and $p$ has simple roots.
Therefore $y \tr(p)\not\in \F$, and hence $\Lambda_{p,q}$ is irreducible and has the same splitting field as $q$ in $\K$.
\item \textbf{Subcase 1.3:} $p$ and $q$ have simple roots. Then the non-identity element of the Galois group of $\K$ over $\F$
exchanges the roots of $\Lambda_{p,q}$, so $\Lambda_{p,q}$ is irreducible and has the same splitting field as $q$ in $\K$.
\end{itemize}

\item \textbf{Case 2.} Both $p$ and $q$ are irreducible with the same splitting field in $\K$, which then equals $\K$.
\begin{itemize}
\item \textbf{Subcase 2.1:} $\K$ is inseparable over $\F$. Then $\Lambda_{p,q}=t^2$.
\item \textbf{Subcase 2.2:} $\K$ is separable over $\F$. Then all the elements of the Galois group of $\K$ over $\F$ fix the roots of
$\Lambda_{p,q}$, so $\Lambda_{p,q}$ splits over $\F$ with simple roots.
\end{itemize}

\item \textbf{Case 3.} Both $p$ and $q$ are irreducible, with distinct splitting fields in $\K$. \\
Denote then by $\L$ the splitting field of $\Lambda_{p,q}$ in $\K$.
\begin{itemize}
\item \textbf{Subcase 3.1:} $p$ and $q$ are separable. Then the Galois automorphism of $\K$ over $\F$ that fixes the roots of $p$ and exchanges the ones of $q$ also exchanges the roots of $\Lambda_{p,q}$. Hence $\Lambda_{p,q}$ is irreducible and $\L$
    is not the splitting field of $p$ in $\K$. Likewise $\L$ is not the splitting field of $q$ in $\K$.
    Therefore, the respective splitting fields of $p$, $q$ and $\Lambda_{p,q}$ in $\K$ are exactly the extensions of degree $2$ of $\F$ in $\K$.
 \item \textbf{Subcase 3.2:} Exactly one of $p$ and $q$ is separable, say $p$ (so that $\car(\F)=2$).
Then $\Lambda_{p,q}$ has a double root, which equals $\tr(p)y$ for the sole root $y$ of $q$, and $\tr(p) \neq 0$ because $p$ is separable,
so $\L$ is the splitting field of $q$ in $\K$, and hence $q$ splits over $\L$ but $p$ remains irreducible over $\L$.
\item \textbf{Subcase 3.2:} Both $p$ and $q$ are inseparable. Then $\Lambda_{p,q}=t^2$, and $p$ and $q$ remain irreducible over $\L=\F$.
\end{itemize}
\end{itemize}

We sum up the results in the next table for future reference, in which we denote by $n(p)$ and $n(q)$ the respective numbers of roots of $p$ and $q$, when relevant to the case under consideration, and by $\L$ the splitting field of $\Lambda_{p,q}$.

\begin{table}[H]\label{table1}
\begin{center}
\caption{The irreducibility and splitting field $\L$ of $\Lambda_{p,q}$ with respect to $p$ and $q$}
\label{figure3}
\begin{tabular}{| c | c | c || c | c |}
\hline
Type of $p$ and $q$ & $n(p)$ & $n(q)$ & Type of $\Lambda_{p,q}$ & Irreducibility of $p$ and $q$ over $\L$ \\
\hline
\hline
Both split & $?$ & $?$ & split & both split \\
\hline
$p$ splits & 1 & ? & split & $p$ splits, $q$ irreducible   \\
$q$ irreducible & & & &  \\
\hline
$p$ splits & 2 & ? & irreducible & both split   \\
$q$ irreducible & & & &  \\
\hline
$p$ and $q$ irreducible, & ? & ? & split & both irreducible \\
equal splitting fields &  &  &  &  \\
\hline
$p$ and $q$ irreducible, & 2 & 2 & irreducible & both irreducible \\
distinct splitting fields &  &  &  &  \\
\hline
$p$ and $q$ irreducible & 2 & 1 & irreducible & $p$ irreducible, $q$ splits \\
distinct splitting fields &  &  &  &  \\
\hline
$p$ and $q$ irreducible & 1 & 1 & equals $t^2$ & both irreducible \\
distinct splitting fields &  &  &  &  \\
\hline
\end{tabular}
\end{center}
\end{table}

\subsection{The maximal ideals above the fundamental ideal}

Throughout this section, we consider a maximal ideal $I$ of $\calW_{p,q}$ that includes $\mathfrak{F}$.
Hence $I$ includes $(r(\omega))$ for some monic irreducible divisor $r$ of $\Lambda_{p,q}$, and analyzing
$I$ amounts to analyzing the maximal ideals of the quotient ring $\calW_{p,q}/(r(\omega))$.

\begin{Not}
For a nonconstant polynomial $r \in \F[t]$, we denote by
$\calW_{p,q,[r]}$ the quotient ring $\calW_{p,q}/(r(\omega))$.
\end{Not}
\label{page:specialization}

From now on we fix an arbitrary $r \in \Irr(\F)$ that divides $\Lambda_{p,q}$.

Note that there are two possible situations: either $\Lambda_{p,q}$ splits and $r$ has degree $1$, or
$\Lambda_{p,q}$ is irreducible and $r=\Lambda_{p,q}$ has degree $2$. In any case the projection of the center $C$ in
$\calW_{p,q,[r]}$ is a field extension $\L$ of $\F$ (of degree $1$ or $2$), the $\F$-algebra $\L$ is isomorphic to $\F[t]/(r)$, and
$\calW_{p,q,[r]}$ is naturally seen as an $\L$-vector space. Moreover, because $(r(\omega))$ is invariant under the adjunction,
we recover the induced adjunction $x \mapsto x^\star$ on $\calW_{p,q,[r]}$, norm
$N_r : x \mapsto xx^\star$, trace $\tr_r : x \mapsto x+x^\star$, and inner product $\langle -,-\rangle_r$.
The trace is $\L$-linear, and the inner product is $\L$-bilinear.
Crucially, we have seen in Section \ref{section:Grammatrix} that the inner product $\langle -,-\rangle_r$ is degenerate
because $r$ divides $\Lambda_{p,q}$. As could be expected, the radical of $\langle -,-\rangle_r$ will play a prominent role in what follows.
\label{page:specializedtraceandnorm}

Beforehand, we start with a crucial remark:

\begin{lemma}\label{lemma:Nrvanishes}
Let $J$ be a proper ideal of $\calW_{p,q,[r]}$. Then $N_r$ vanishes on $J$.
\end{lemma}

\begin{proof}
Let $x \in J$. Then $N_r(x)=xx^\star \in J$, so $N_r(x) = 0$ otherwise $N_r(x)$ would be invertible, leading to $J=\calW_{p,q,[r]}$.
\end{proof}

\begin{Not}
We define $R_r$ as the radical of the inner product $\langle -,-\rangle_r$, and
$$\mathfrak{R}_r:=\{x \in R_r : N_r(x)=0\},$$
the latter of which we call the \textbf{radical} of $N_r$.
\index{Radical (ideal of $\calW_{p,q,[r]}$)}
\label{page:radicalmodr}
\end{Not}

Of course we have $\mathfrak{R}_r=R_r$ if $\car(\F)\neq 2$, but as we shall see this can fail if $\car(\F)=2$.

\begin{lemma}\label{lemma:caseRr=frakRr}
The sets $R_r$ and $\mathfrak{R}_r$ are ideals of $\calW_{p,q,[r]}$, and $\mathfrak{R}_r$ is a proper one.
Moreover $R_r=\mathfrak{R}_r$ if and only if $R_r \subsetneq \calW_{p,q,[r]}$.
Finally $\mathfrak{R}_r$ is invariant under the adjunction.
\end{lemma}

\begin{proof}
Classically $R_r$ is a linear subspace, and the fact that it is an ideal is immediately deduced from
identities \eqref{eq:adjeq} on page \pageref{eq:adjeq}.
Next, $N_r$ is additive on $R_r$, so $\mathfrak{R}_r$ is a subgroup of $(R_r,+)$ and hence also of $(\calW_{p,q,[r]},+)$.
Let finally $\alpha \in \mathfrak{R}_r$ and $x \in \calW_{p,q,[r]}$. We already know that $\alpha x$ and $x\alpha$ belong to $R_r$,
and finally $N_r(\alpha x)=N_r(\alpha)N_r(x)=0$ and $N_r(x\alpha)=N_r(x)N_r(\alpha)=0$, whence $\alpha x$ and $x\alpha$ belong to $\mathfrak{R}_r$.
Hence $\mathfrak{R}_r$ is an ideal of $\calW_{p,q,[r]}$.

Next $1 \not\in \mathfrak{R}_r$ because $N_r(1)=1$, and hence $\mathfrak{R}_r$ is a proper ideal.
In particular, if $R_r=\calW_{p,q,[r]}$ then $R_r \neq \mathfrak{R}_r$.
Conversely, if $R_r \subsetneq \calW_{p,q,[r]}$ then we deduce from Lemma \ref{lemma:Nrvanishes}
that $R_r=\mathfrak{R}_r$.

Finally, the invariance of $\mathfrak{R}_r$ under the adjunction is deduced from
identities \eqref{eq:starisometry} on page \pageref{eq:starisometry} and from the invariance of $N$ under the adjunction.
\end{proof}

Next, we investigate the potential dimensions of $R_r$ and $\mathfrak{R}_r$ over $\L$.

\begin{prop}\label{prop:structureofRr}
One of the following statements holds:
\begin{enumerate}[(i)]
\item The dimension of $R_r$ over $\L$ is $2$, and at least one of $p$ and $q$ has simple roots in $\K$ (a splitting field of $pq$).
\item The dimension of $R_r$ over $\L$ is $3$, $\car(\F) \neq 2$ and both $p$ and $q$ split with double roots (over $\F$).
\item The dimension of $R_r$ over $\L$ is $4$, $\car(\F)=2$ and $\tr(p)=\tr(q)=0$.
\end{enumerate}
Moreover, except in case (iii) one has $R_r=\mathfrak{R}_r$.
\end{prop}

\begin{proof}
Assume first that $\car(\F) \neq 2$.
Then we proceed as in the proof of Proposition \ref{prop:gramgeneralized}:
we take a deployed basis $(1,x,y,xy)$ such that $\tr(x)=\tr(y)=0$.
Then we see that $\langle -,-\rangle_r$
is the orthogonal direct sum of two symmetric bilinear forms that are represented by the respective matrices
$$A=\begin{bmatrix}
2 & \langle 1,\overline{x}\overline{y}\rangle_r \\
\langle 1,\overline{x}\overline{y}\rangle_r & 2N(x)N(y)
\end{bmatrix} \quad \text{and} \quad
B=\begin{bmatrix}
2N(x) & \langle \overline{x},\overline{y}\rangle_r \\
\langle \overline{x},\overline{y}\rangle_r & 2 N(y)
\end{bmatrix}.$$
As seen in the proof of Proposition \ref{prop:gramgeneralized}, both these matrices
have determinant zero, and hence both are singular. Moreover,
we see from its upper-left entry that $A \neq 0$.
Hence $\dim R_r=1+\dim (\Ker B) \in \{2,3\}$. Finally $\dim R_r=3$ if and only if $B=0$, which
implies $N(x)=N(y)=0$.

Conversely, assume that $N(x)=N(y)=0$. Then as $B$ is singular it is clear that the only option is that its
off-diagonal entries equal $0$, and hence $B=0$.
We conclude that $\dim R_r=3$ if and only if $N(x)=N(y)=0$, which amounts to $p$ and $q$ being split with double roots in $\F$
(remember that we had assumed $\tr(x)=\tr(y)=0$).

Assume finally that $\car(\F)=2$. Then $\langle -,-\rangle_r$ is a degenerate alternating form,
so its rank is even. Hence this rank is either $0$ and $2$, and we get from the Gram matrix \eqref{eq:Gram}
that this rank is zero only if $\tr(p)=\tr(q)=0$ and $\langle \overline{x},\overline{y}\rangle_r=\langle 1,\overline{x}\overline{y}\rangle_r=0$.
Assume now that $\tr(p)=\tr(q)=0$. Then $\Lambda_{p,q}=t^2$ and $r=t$,
whence $r(\omega)$ divides $\langle a,b\rangle$, i.e., $\langle \overline{a},\overline{b}\rangle_r=0$. It follows that
$\langle 1,\overline{a}\overline{b}\rangle_r=\langle \overline{a},\overline{b}\rangle_r=0$ because $a^\star=-a=a$.
Therefore $R_r=\calW_{p,q,[r]}$ if and only if $\tr(p)=\tr(q)=0$.

The last statement is readily deduced from the previous ones and from Lemma \ref{lemma:caseRr=frakRr}.
\end{proof}

\begin{lemma}
For every proper ideal $I$ of $\calW_{p,q,[r]}$, one has $I+\mathfrak{R}_r \subsetneq \calW_{p,q,[r]}$.
\end{lemma}

\begin{proof}
Let $I$ be a proper ideal of $\calW_{p,q,[r]}$.
Then $J:=I+\mathfrak{R}_r$ is an ideal of $\calW_{p,q,[r]}$.
Yet $N_r$ vanishes on both $I$ and $\mathfrak{R}_r$, which are orthogonal for the inner product $\langle -,-\rangle_r$.
Therefore $N_r$ vanishes on $J$, and in particular $1 \not\in J$.
\end{proof}

\begin{cor}\label{cor:inclusionRr}
Every maximal ideal of $\calW_{p,q,[r]}$ includes $\mathfrak{R}_r$.
\end{cor}

We deduce that the maximal ideals of $\calW_{p,q,[r]}$ are in one-to-one correspondence with the maximal ideals
of the $\L$-algebra
$$\calU_r:=\calW_{p,q,[r]}/\mathfrak{R}_{r.}$$
Their determination is simplified by the following observation:

\begin{prop}
The ring $\calU_r$ is commutative.
\end{prop}

\begin{proof}
Assume first that $\car(\F)=2$ and $\tr(p)=\tr(q)=0$.
Then $\Lambda_{p,q}=t^2$, $r=t$ and hence $ab^\star+ba^\star=\omega \in (r(\omega))$, while $a^\star=-a$ and $b^\star=-b=b$.
Hence $ab\equiv ba$ mod $(r(\omega))$, and we deduce that $\calW_{p,q,[r]}$ is commutative. Hence $\calU_r$ is also commutative.

Assume now otherwise. Then we know from Proposition \ref{prop:structureofRr} that $\dim_\L \mathfrak{R}_r \in \{2,3\}$, and hence
$\dim_\L \calU_r \in \{1,2\}$. Hence $\calU_r$ is commutative.
\end{proof}

\begin{prop}\label{prop:structureofquotient}
Let $J$ be a maximal ideal of $\calW_{p,q}$ that includes $(r(\omega))$.
Then the quotient $\F$-algebra $\calW_{p,q}/J$ is isomorphic to the $\F$-algebra $\K$ (the splitting field of $pq$).
\end{prop}

\begin{proof}
This amounts to proving that for every maximal ideal $J'$ of $\calU_r$, the quotient
$\calU_r/J'$ is isomorphic to the $\F$-algebra $\K$.

Let $J'$ be a maximal ideal of $\calU_r$. Since $\calU_r$ is commutative the quotient $\F$-algebra $\calU_r/J'$ is a field extension of $\F$.
Moreover, it is generated as an $\F$-algebra by the cosets of $\overline{a}$ and $\overline{b}$, which
are annihilated respectively by $p$ and $q$. This is the very definition of a splitting field of $pq$ over $\F$, and
we conclude because it is known that any two such splitting fields are isomorphic as $\F$-algebras.
\end{proof}

\label{page:superfundamental}
\begin{Not}
For an irreducible monic divisor $r \in \Irr(\F)$ of $\Lambda_{p,q}$, we set
$$\mathfrak{J}_r:=\{x \in \calW_{p,q} : \overline{x} \in \mathfrak{R}_r\}.$$
In other words $\mathfrak{J}_r$ consists of the vectors $x \in \calW_{p,q}$
such that $\langle x,y\rangle \equiv 0$ mod $(r(\omega))$ for all $y \in \calW_{p,q}$,
and $N(x) \equiv 0$ mod $(r(\omega))$.
\end{Not}

Note that $\mathfrak{J}_r$ is an ideal of $\calW_{p,q}$ that includes $(r(\omega))$. We have just seen that every maximal ideal of
$\calW_{p,q}$ that includes $(r(\omega))$ also includes $\mathfrak{J}_r$, and now the question remains whether
$\mathfrak{J}_r$ is maximal or not. This question is fully answered in our next proposition.

\begin{prop}\label{prop:structureofJr}
The following conditions are equivalent:
\begin{enumerate}[(i)]
\item $\mathfrak{J}_r$ is a maximal ideal of $\calW_{p,q}$;
\item $\mathfrak{J}_r$ is the sole maximal ideal of $\calW_{p,q}$ that includes $(r(\omega))$;
\item $\mathfrak{R}_r$ is the sole maximal ideal of $\calW_{p,q,[r]}$;
\item None of $p$ and $q$ splits with simple roots in $\F$.
\end{enumerate}
\end{prop}

In particular, whenever both $p$ and $q$ are irreducible,
$\mathfrak{J}_r$ is the sole maximal ideal of $\calW_{p,q}$ that includes $(r(\omega))$.

\begin{proof}
We already know that conditions (i), (ii) and (iii) are equivalent.

Assume first that both $p$ and $q$ have a double root in $\K$.
\begin{itemize}
\item Assume that $\car(\F)=2$. Then
$\mathfrak{R}_r=\{x \in \calW_{p,q,[r]} : N_r(x)=0\}$, and we deduce from Lemma \ref{lemma:Nrvanishes} that every
maximal ideal of $\calW_{p,q,[r]}$ is included in $\mathfrak{R}_r$. Hence $\mathfrak{R}_r$
is the sole maximal ideal of $\calW_{p,q,[r]}$.
\item Assume that $\car(\F) \neq 2$. Then $p$ and $q$ split and hence $\L=\F$.
It ensues from point (ii) of Proposition \ref{prop:structureofRr} that $\mathfrak{R}_r$ is maximal, and by
Corollary \ref{cor:inclusionRr} we conclude that it is the sole maximal ideal of $\calW_{p,q,[r]}$.
\end{itemize}

Assume now that at least one of $p$ and $q$ has simple roots in $\K$.
Hence $\mathfrak{R}_r=R_r$ has dimension $2$ over $\L$, so $\calU_r$ has dimension $2$ over $\L$.
Pick a maximal ideal $J$ of $\calU_r$. Then by Proposition \ref{prop:structureofquotient} we have $\dim_\L J=2-\frac{[\K:\F]}{[\L:\F]}$, and hence
$J \neq \{0\}$ if and only if $[\K:\F]=[\L:\F]$, i.e., $\L$ is a splitting field of $pq$
(recall that $\L$ is a splitting field of $\Lambda_{p,q}$, and as $\Lambda_{p,q}$ splits over $\K$ we can view $\K$ as a field extension of $\L$).
Observe also, since $\Lambda_{p,q}$ has degree $2$, that $\L$ is a splitting field of $\Lambda_{p,q}$.
Hence $\calU_r$ has no nonzero maximal ideal if and only if at least one of $p$ and $q$ remains irreducible over the splitting field of
$\Lambda_{p,q}$. Now, we can conclude by browsing Table \ref{figure3} that this happens unless one of $p$ and $q$ splits with simple roots in $\F$.
\end{proof}

It remains to understand the structure of the maximal ideals that include
$\mathfrak{J}_r$ when it is not maximal.

So, from now on we assume that $\mathfrak{J}_r$ is not maximal.
In particular, by combining Propositions \ref{prop:structureofRr} and \ref{prop:structureofJr} we obtain that
$\mathfrak{R}_r=R_r$ and $\dim_\L R_r=2$.

We are left with finding the maximal ideals of
$\calW_{p,q,[r]}$, and we can then go to the quotient $\L$-algebra $\calU_r=\calW_{p,q,[r]}/\mathfrak{R}_r$.
Note that $\dim_\L \calU_r=2$.
The maximal ideals of $\calW_{p,q,[r]}$ are then in one-to-one correspondence with the maximal ideals of $\calU_r$
through the canonical projection of $\calW_{p,q,[r]}$ onto $\calU_r$. In particular, $\calU_r$ has a nonzero maximal ideal.

A potential lead here is to note that $\calU_r$ is a $2$-dimensional $\L$-algebra with a nonzero maximal ideal,
so either it splits or it is degenerate. In the first case $\calU_r$ has exactly two $1$-dimensional ideals, each of which generated
by an idempotent, and in the second case $\calU_r$ has a unique $1$-dimensional ideal.

We will now discard the second possibility, but this requires that we use the quadratic structures.
Since $\mathfrak{R}_r$ is invariant under conjugation, we obtain
an induced adjunction $x \mapsto x^\star$ on $\calU_r$.
Moreover, in light of the definition of $\mathfrak{R}_r$,
the norm and inner product respectively induce a norm $N_{r,c} : \calU_r \rightarrow \L$ and an inner product
$\langle -,-\rangle_{r,c} : (\calU_r)^2 \rightarrow \L$, the latter of which is the polar form of the former.
Moreover $\langle -,-\rangle_{r,c}$ is now non-degenerate, whence $N_{r,c}$ is a regular quadratic form on the $\L$-vector space
$\calU_r$.

The quotient norm is simply $N_{r,c} : x \mapsto xx^\star$, and it is multiplicative.
Now, with the same line of reasoning as in Lemma \ref{lemma:Nrvanishes}, we obtain that every proper ideal $J$ of $\calU_r$
is included in the isotropy cone of $N_{r,c}$. Yet we know that $\calU_r$ has a nonzero maximal ideal,
so $N_{r,c}$ is hyperbolic. It follows that $N_{r,c}$ has exactly two isotropic lines, and hence at most
two maximal ideals.

Now, we will complete the proof by obtaining that the two isotropic lines are ideals.
To do so, we take a nonzero element $z \in \calU_r$ such that $N_{r,c}(z)=0$.
Since the norm is multiplicative, we get $N_{r,c}(zx)=0$ for all $x \in \calU_r$.
Hence $\calU_r z$ is a totally $N_{r,c}$-isotropic $\L$-linear subspace that includes $\L z$,
and hence it equals $\L z$ because $N_{r,c}$ is hyperbolic with rank $2$.
Likewise $z \calU_r=\L z$, and hence $\L z$ is an ideal of $\calU_r$.
We conclude that $\calU_r$ has exactly $2$ maximal ideals, to the effect that it splits.

Let us finally observe that the adjunction $x \mapsto x^\star$ exchanges the two maximal ideals of $\calU_r$.
To see this, let $z$ be idempotent in $\calU_r$ (such an element exists because $\calU_r$ splits).
Then $\L z$ is a maximal ideal, to the effect that $zz^\star=0$ (indeed otherwise $zz^\star$, which is a scalar element, is invertible in $\calU_r$).
Hence $z^\star \neq z$. It follows that $x \mapsto x^\star$ is not the identity of $\calU_r$, yet it is an involution of this $\L$-algebra.
Hence it is its nontrivial involution. The claimed result follows by observing that the nontrivial involution of the $\L$-algebra
$\calU_r \simeq \L \times \L$ exchanges its two maximal ideals.

Let us conclude:

\begin{prop}\label{prop:numberofmaxiideals}
If $\mathfrak{J}_r$ is not a maximal ideal of $\calW_{p,q}$, then exactly two maximal ideals of $\calW_{p,q}$
include $(r(\omega))$, they are exchanged by $x \mapsto x^\star$, and finally
$\calW_{p,q}/\mathfrak{J}_r \simeq \calW_{p,q,[r]}/\mathfrak{R}_r \simeq \L^2$ as $\F$-algebras, where $\L$ stands for a splitting field
of $\Lambda_{p,q}$.
\end{prop}

Note also that in this case the two maximal ideals of $\calW_{p,q,[r]}$ are the two totally $N_r$-isotropic $\L$-linear hyperplanes of it.

This concludes our study of the maximal ideals of $\calW_{p,q}$ that include $(r(\omega))$, thereby closing our study of the maximal ideals of $\calW_{p,q}$.

\subsection{Application to the Zero Divisors Theorem}\label{section:ZDnewproof}

As an application of the previous study, we give here an alternative proof of the Zero Divisors Theorem
(Theorem \ref{theo:zerodivisors}).
The proof uses a similar reduction to a simple case, but goes much further in simplifying the situation, allowing
us to completely avoid using nontrivial results on quadratic forms (such as the Artin-Springer theorem).

Our new proof is based upon the following critical, yet very simple observation, which will be reused in the determination of the automorphism group.
Here, we choose a divisor $r \in \Irr(\F)$ of $\Lambda_{p,q}$.
As in Section \ref{section:units1}, we will use the degrees of the elements of the center $C$ as polynomials in $\omega$.

\begin{lemma}\label{lemma:liftinglemma}
Set $d:=\deg(r)$.
Let $x \in \calW_{p,q}$.
Assume that the coefficients of $x$ in the deployed basis $(1,a,b,ab)$
all have degree less than $d$, and that $N(x)\equiv 0$ mod $(r(\omega)^2)$. Then $N(x)=0$.
\end{lemma}

\begin{proof}
Since $N(x)\equiv 0$ mod $(r(\omega)^2)$, either $N(x)=0$ or $\deg(N(x)) \geq 2d$.
Write $x=x_1 +x_a  a+x_b b+x_{ab} ab$ with $x_1,x_a,x_b,x_{ab}$ in $C$.
In expanding $N(x)$, we see that all the summands have degree at most $1+2(d-1)<2d$.
Hence $N(x)=0$.
\end{proof}

Let us still denote by $d$ the degree of $r$.
A vector $x$ of $\calW_{p,q}$ is called normalized when it is normalized with
respect to the structure of free $C$-module of $\calW_{p,q}$, meaning that the coefficients of $x$ in the $C$-basis
$(1,a,b,ab)$ have no nonconstant common divisor. Clearly, if we have a zero divisor $x$, we can
consider a normalization $x=s y$ with $s\in C^\times$ and $y$ normalized, and $y$ is then a normalized zero divisor.

There are four main steps in this proof of the Zero Divisors Theorem. Throughout, we assume that $p$ and $q$ are irreducible.
\begin{itemize}
\item \textbf{Step 1.} We prove that every zero divisor of $\calW_{p,q}$ belongs to $\mathfrak{J}_r$.
\item \textbf{Step 2.} We prove that if a zero divisor exists, a zero divisor also exists in $H:=C+Ca+Cb$.
\item \textbf{Step 3.} We prove that if a zero divisor exists, then there exists a zero divisor of the form $x_1+x_a a+x_b b$
where $x_1,x_a,x_b$ belong to $C$ and have degree less than $d$.
\item \textbf{Step 4.} Finally, we will easily discard the latter possibility.
\end{itemize}

Let us start with Step 1. Let $x \in \calW_{p,q}$ be a zero divisor.
Denote by $x_r$ its coset in $\calW_{p,q,[r]}$.
We have $x_r x_r^\star=0$. Because both $p$ and $q$ are irreducible we know from Propositions \ref{prop:structureofquotient} and \ref{prop:structureofJr} that
$\calW_{p,q,[r]}/\mathfrak{R}_r$ is a field, and because $\mathfrak{R}_r$ is invariant under adjunction we deduce that
$x_r \in \mathfrak{R}_r$. Hence $x \in \mathfrak{J}_r$.

For the second step, the quickest route is to use the quaternionic structure of $\overline{\calW_{p,q}}$, just like in Section \ref{section:zerodivisors}.
We will however give an alternative proof that avoids using quaternions at all and relies only upon the consideration of $\calW_{p,q,[r]}$,
but we will postpone the explanation of it after the next steps are completed (see Remark \ref{rem:altroutezerodivisor} below).

We move on to Step 3. Assume that $\calW_{p,q}$ has a zero divisor. By Step 2, it has a zero divisor in $H$,
and by normalizing we can find a normalized zero divisor $x \in H$.
By the Euclidean division, we find an element $y \in C+Ca+Cb$ all whose coefficients have degree less than $d$, and a vector $z \in \calW_{p,q}$ such that
$x=y+r(\omega) z$. Note that $y \neq 0$ because $x$ is normalized. Then we observe that
$$N(y)=N(x)-r(\omega) \langle x,z\rangle+r(\omega)^2 N(z)=-r(\omega) \langle x,z\rangle+r(\omega)^2 N(z).$$
Because $x \in \mathfrak{J}_r$ we deduce that $N(y) \equiv 0 \; (r(\omega)^2)$, and we derive from Lemma \ref{lemma:liftinglemma} that $N(y)=0$.
Hence Step 3 is completed.

Let us finish with Step 4. Again, let us assume that there is a zero divisor $x=x_1 +x_a a+x_b b$, where $x_1,x_a,x_b$ are elements of $C$
all with degree less than $d$, and hence less than $2$. Once more we can assume that $x$ is normalized.
Assume that $\deg(x_b)=1$. Then by specializing at the root $\lambda \in \F$ of $x_b$, we find
that
$$N\left(x_1(\lambda)+x_a(\lambda) a\right)=x_1(\lambda)^2+x_1(\lambda)x_a(\lambda) \tr(a)+x_a(\lambda)^2 N(a)
=N(x)[\lambda]=0.$$
Since $p$ is irreducible we deduce that $x_1(\lambda)=x_a(\lambda)=0$.
Hence $x_b$ divides $x_1$ and $x_a$, thereby contradicting the fact that $x$ is normalized.
It follows that $x_b$ is constant, and symmetrically we find that $x_a$ is constant.
Then, in expanding $N(x_1+x_a a+x_b b)$ we find that $N(x)=x_1^2+\mu$ for some $\mu \in C$ with degree at most $1$.
Hence $x_1$ is constant. Therefore $x_1,x_a,x_b$ are all constant.
Finally by expanding $N(x)$ one last time we find $N(x)=x_ax_b \omega +\mu$ for some constant $\mu$,
and hence $x_ax_b=0$. If $x_b=0$ then again $x_1=x_a=0$ because $p$ is irreducible, which is absurd, and likewise
$x_a=0$ leads to a contradiction. We have a contradiction in any case, and we conclude that $\calW_{p,q}$ contains no zero divisor.

\begin{Rem}\label{rem:altroutezerodivisor}
Let us give an alternative approach to Step 2, avoiding using the extension $\overline{\calW_{p,q}}$ and relying only upon the analysis of
$\calW_{p,q,[r]}$.
So, we start from a normalized zero divisor $x$ in $\calW_{p,q}$.

As in the previous section, we put $\L:=\F[t]/(r)$, which is isomorphic to the splitting field of $\Lambda_{p,q}$ in $\K$.
Next, since $p$ and $q$ are irreducible we know from Table \ref{figure3}
that at least one of $p$ and $q$ remains irreducible over $\L$.
Without loss of generality, we assume that $p$ remains irreducible over $\L$.

Now, we consider the $C$-submodule $C[a] x$, and we prove that its rank over $C$ is (at least) $2$.
To do this, we use a \emph{reductio ad absurdum}, and assume that $\rk_C(C[a] x)=1$.
Denote by $a_r$ the coset of $a$ modulo $(r(\omega))$.
Then $x_r,a_r x_r$ are linearly dependent over $\L$, i.e.,
$a_r x_r=\lambda x_r$ for some $\lambda \in \L$.
Note that $a_r-\lambda$ is nonzero because it is the coset of $a-\mu$ for some $\mu \in C$,
whose coefficient on $a$ is not a multiple of $r(\omega)$.
Hence $a_r-\lambda$ is a zero divisor in $\calW_{p,q,[r]}$.
Yet it is clear that the projection of $C[a]$ in $\calW_{p,q,[r]}$
is isomorphic to $C[a]/(r(\omega))$ and hence to $\L[t]/(p)$. Because $p$ remains irreducible over $\L$, the
algebra $\L[t]/(p)$ is a field, whence $a_r-\lambda$ is invertible in
$\calW_{p,q,[r]}$. This is absurd. As a consequence, $\rk_C(C[a] x) \geq 2$.

Next, $\rk_C H=3$, and since $\calW_{p,q}$ is a free $C$-module of rank $4$ this is enough to see that $H \cap (C[a]x) \neq \{0\}$.
To conclude, we pick $x' \in (H \cap (C[a]x)) \setminus \{0\}$ and observe that it is a zero divisor, which is obvious by the multiplicativity of $N$.
\end{Rem}

\subsection{Additional results on the ideals above the fundamental ideal}

Here, we will dive deeper into the study of the ideals of $\calW_{p,q,[r]}$ and will essentially complete their description,
leaving aside only the special case where both $p$ and $q$ split with a double root.
These results are for use in the proof of the structure theorems for the automorphism group in Section \ref{section:automorphismsII}, so
we invite the reader to skip them at first reading and move to any one of the next sections.

\begin{prop}\label{prop:allidealsonesplits}
Assume that one of $p$ and $q$ splits, but that $p$ and $q$ do not both split with a double root.
Let $r \in \Irr(\F)$ be an irreducible divisor of $\Lambda_{p,q}$, and for $x \in \calW_{p,q}$
denote by $x_r$ its coset in $\calW_{p,q,[r]}$. Set $\L:=\F[t]/(r)$, and consider $\calW_{p,q,[r]}$
with its induced structure of vector space over $\L$.
Then:
\begin{enumerate}[(a)]
\item $\mathfrak{R}_r$ is the sole $2$-dimensional ideal of $\calW_{p,q,[r]}$.
\item All the $1$-dimensional ideals of $\calW_{p,q,[r]}$ are included in $\mathfrak{R}_r$.
\item There is a $1$-dimensional ideal in $\calW_{p,q,[r]}$ if and only if
$p$ and $q$ split over $\L$. In that case, there are exactly two such ideals,
each one is of the form $\L \alpha \beta^\star$ with $\alpha \in \L[a_r] \setminus \L$ and $\beta \in \L[b_r] \setminus \L$ such that $N_r(\alpha)=N_r(\beta)=0$, and each one is invariant under the adjunction.
\item If $q$ has a double root in $\L$, then $\mathfrak{R}_r=\beta \calW_{p,q,[r]}=\calW_{p,q,[r]} \beta=(\beta)$ for every $\beta \in \L[b_r] \setminus \L$ such that $N(\beta)=0$.
\end{enumerate}
\end{prop}

\begin{proof}
First of all, we know that $(1,a_r,b_r,a_rb_r)$ is a basis of the $\L$-vector space $\calW_{p,q,[r]}$,
and it easily follows, as in Section \ref{section:omega}, that $(1,\alpha,\beta,\alpha\beta)$ is also a basis of it for all
$\alpha \in \L[a_r] \setminus \L$ and $\beta \in \L[b_r] \setminus \L$.

We assume first that none of $p$ and $q$ splits with simple roots (in $\F$).
Then without loss of generality, we assume that $q=t^2$ and that $p$ does not split with simple roots (in $\F$).
The starting assumption yields that $p$ is irreducible. Here $\L=\F$ and $[\K:\F]=2$, and we know from Propositions \ref{prop:structureofquotient}
and \ref{prop:structureofJr} that $\mathfrak{R}_r$ is the sole maximal ideal of $\calW_{p,q,[r]}$ and that its dimension equals $2$.
In particular, every proper ideal of $\calW_{p,q,[r]}$ is included in $\mathfrak{R}_r$.
Moreover, here $N(b)=0$, $\tr(b)=0$, $r(\omega)=\langle a,b\rangle$ and $\langle 1,ab\rangle=\tr(a)\tr(b)-\langle a,b\rangle=-r(\omega)$.
Hence the Gram matrix of Section \ref{section:Grammatrix}
helps us see that $b_r$ and $a_rb_r$ belong to $\mathfrak{R}_r$. Hence $\mathfrak{R}_r=\F[a_r] b_r$.
Since $\mathfrak{R}_r$ is an ideal we conclude that $\mathfrak{R}_r=(b_r)$. Moreover,
for all $\alpha \in \F[a_r] \setminus \{0\}$, we see that the ideal $(\alpha b_r)$ must include $(b_r)$ because $\alpha$ is invertible
(remember that $\F[a_r]$ is a field because $p$ is irreducible).
Hence $\mathfrak{R}_r$ has no $1$-dimensional sub-ideal. We conclude in this case that $\mathfrak{R}_r$ is the sole nontrivial ideal of
$\calW_{p,q,[r]}$ and that it is generated -- both as a left ideal and as a right ideal -- by any $\beta \in \L[b_r] \setminus \L$ such that $N(\beta)=0$
(because these elements are the ones of $\F^\times b_r$).

From now on, we assume that at least one of $p$ and $q$ splits with simple roots.
Without loss of generality, we assume that $p$ splits with simple roots.
Then Table \ref{figure3} shows that $q$ splits over $\L$. Hence we can pick
$\alpha \in \L[a_r] \setminus \L$ such that $\alpha^2=\alpha$, and
$\beta \in \L[b_r] \setminus \L$ such that either $\beta^2=0$ or $\beta^2=\beta$.
In particular $N_r(\alpha)=0$ and $N_r(\beta)=0$. Hence
the Gram matrix of the basis $(1,\alpha,\beta,\alpha\beta)$ for $\langle -,-\rangle_r$ equals
$$\begin{bmatrix}
2 & \tr(\alpha) & \tr(\beta) & \langle 1,\alpha \beta\rangle_r \\
\tr(\alpha) & 0 & \langle \alpha,\beta\rangle_r & 0 \\
\tr(\beta) & \langle \alpha,\beta\rangle_r & 0 & 0 \\
\langle 1,\alpha\beta\rangle_r & 0 & 0 & 0
\end{bmatrix}.$$
Yet $\langle -,-\rangle_r$ is singular, so this matrix is also singular, and we deduce that
$\langle 1,\alpha\beta\rangle_r=0$ or $\langle \alpha,\beta\rangle_r=0$.
The first identity yields $\langle \alpha^\star,\beta\rangle_r=0$. Hence, by replacing $\alpha$ with the idempotent $\alpha^\star$ if necessary,
we lose no generality in assuming that $\langle \alpha,\beta\rangle_r=0$, which we will now do.

Now we have $\alpha \beta^\star+\beta \alpha^\star=0$.
We already note that $\alpha \beta^\star$ belongs to $\mathfrak{R}_r$. Indeed:
\begin{itemize}
\item To start with $N_r(\alpha \beta^\star)=N_r(\alpha)N_r(\beta^\star)=0$;
\item Next, $\langle \alpha \beta^\star,\alpha \beta^\star\rangle_r=2 N_r(\alpha \beta^\star)=0$,
$\langle \alpha,\alpha \beta^\star\rangle_r=N_r(\alpha) \tr_r(\beta^\star)=0$,
$\langle \beta^\star,\alpha \beta^\star\rangle_r=N_r(\beta^\star) \tr_r(\alpha)=0$,
and finally $\langle 1,\alpha\beta^\star\rangle_r=\langle \alpha,\beta\rangle_r=0$.
\end{itemize}
Likewise we obtain $\alpha^\star \beta \in \mathfrak{R}_r$.
Observe also that $\alpha \beta^\star=\tr(\beta) \alpha-\alpha \beta$ and $\alpha^\star \beta=\beta -\alpha \beta$
are linearly independent over $\L$, as seen from their coefficients in the deployed basis $(1,\alpha,\beta,\alpha \beta)$.

Now, let us take an ideal $I$ of $\calW_{p,q,[r]}$ with dimension $d \in \{1,2\}$ (over $\L$, of course). We consider the quotient $\L$-vector
space $V:=\calW_{p,q,[r]}/I$ and the regular representation
$$\Phi : x \in \calW_{p,q,[r]} \longmapsto [y \mapsto xy] \in \End_\L(V),$$
which is a homomorphism of $\L$-algebras.
As $I$ is a two-sided ideal, we see that $\Ker \Phi=I$.
Recall that $\alpha \beta^\star+\beta \alpha^\star=0$.
Then $\Phi(\alpha) \neq 0$, for otherwise $\Phi(\alpha^\star)=\id$ and then $\Phi(\beta)=0$, and finally
$\Phi(\alpha\beta)=\Phi(\alpha)\Phi(\beta)=0$, so $\Vect_\L(\alpha,\beta,\alpha \beta) \subseteq \Ker \Phi=I$,
contradicting the assumption that $\dim_\L I \leq 2$.
If $\Phi(\alpha^\star)=0$, we would find likewise that $I$ contains the $\L$-linearly independent vectors $\alpha^\star,\beta^\star,\alpha^\star\beta^\star$,
which again would contradict $\dim_\L I \leq 2$.
Hence $\Phi(\alpha)$ is a nontrivial idempotent of $\End_\L(\calW_{p,q,[r]}/I)$.

Now, we set $n:=\dim_\L V$ and $r:=\rk \Phi(\alpha) \in \lcro 1,n-1\rcro$, we choose a basis of $V$ that is adapted to the decomposition
$V=\im \Phi(\alpha) \oplus \Ker\Phi(\alpha)$, and we consider the respective matrices $A$ and $B$ of $\Phi(\alpha)$ and $\Phi(\beta)$ in it.
Of course, we also set $A^\star:=\tr(\alpha) I_n-A=I_n-A$ and $B^\star:=\tr(\beta) I_n-B$.
The identity $\alpha \beta^\star=-\beta \alpha^\star$ yields that $\Phi(\beta^\star)$ maps $\Ker \Phi(\alpha^\star)=\im \Phi(\alpha)$
into $\Ker \Phi(\alpha)$, and likewise that $\Phi(\beta)$ maps $\im \Phi(\alpha^\star)=\Ker \Phi(\alpha)$ into $\im \Phi(\alpha)$.
It follows that
$$A=\begin{bmatrix}
I_r & [0]_{r \times (n-r)} \\
[0]_{(n-r) \times r} & [0]_{(n-r) \times (n-r)}
\end{bmatrix} \quad \text{and} \quad
B=\begin{bmatrix}
\tr(\beta) I_r & C_2 \\
C_1  & [0]_{(n-r) \times (n-r)}
\end{bmatrix},$$
where $C_1 \in \Mat_{n-r,r}(\L)$ and $C_2 \in \Mat_{r,n-r}(\L)$.
We observe that
$$ABA^\star=\begin{bmatrix}
[0]_{r \times r} & C_2 \\
[0]_{(n-r)\times r}  & [0]_{(n-r) \times (n-r)}
\end{bmatrix} \quad \text{and} \quad
A^\star BA =\begin{bmatrix}
[0]_{r \times r} & [0]_{r \times (n-r)} \\
C_1  & [0]_{(n-r) \times (n-r)}
\end{bmatrix}.$$
Hence having $C_1 \neq 0$ and $C_2 \neq 0$ would lead to having $\rk \Phi=4$
(by observing that $A,A^\star,ABA^\star,A^\star BA$ are linearly independent), thereby contradicting the non-injectivity of $\Phi$.

With the same argument, we see that if one of $C_1$ and $C_2$ is nonzero then $\dim_\L I=1$.
And if both $C_1$ and $C_2$ are zero then we observe that $1,\alpha,\beta,\alpha \beta$ are mapped into $\Vect(I_n,A)$ under $\Phi$,
which yields $\dim_\L I\geq 2$.

Conversely, assume that $\dim_\L I=2$. Then $C_1=0$ and $C_2=0$, and we observe that
$\Phi(\alpha \beta^\star)=0$ and $\Phi(\alpha^\star \beta)=0$ because obviously $A B^\star=0=A^\star B$.
We have seen earlier that $\alpha \beta^\star$ and $\alpha^\star \beta$ are linearly independent over $\L$,
and hence $I=\Vect_\L(\alpha \beta^\star,\alpha^\star \beta)$.
This proves that $\Vect_\L(\alpha \beta^\star,\alpha^\star \beta)$ is the only possible $2$-dimensional ideal of $\calW_{p,q,[r]}$.
Yet because $p$ splits with simple roots in $\F$ we know from Proposition \ref{prop:structureofRr} that $\mathfrak{R}_r$
is such a $2$-dimensional ideal. Hence point (a) is proven, along with the identity
\begin{equation}\label{eq:descriptionRr}
\mathfrak{R}_r=\Vect_\L(\alpha \beta^\star,\alpha^\star\beta).
\end{equation}

Now, we assume $\dim_\L I=1$. Hence exactly one of $C_1$ and $C_2$ equals $0$.

\begin{itemize}
\item If $C_1=0$, then $A^\star B=0$, so $I$ contains the nonzero element $\alpha^\star \beta$,
and we deduce that $I=(\alpha^\star \beta) \subseteq \mathfrak{R}_r$.
\item If $C_2=0$, then $A B^\star=0$, so $I$ contains the nonzero element $\alpha \beta^\star$,
and  we deduce that $I=(\alpha \beta^\star) \subseteq \mathfrak{R}_r$.
\end{itemize}
This proves that the only possible $1$-dimensional ideals of $\calW_{p,q,[r]}$ are $\L \alpha \beta^\star$ and
$\L \alpha^\star \beta$, and that both are included in $\mathfrak{R}_r$. We also know that they are distinct.
Now, we must check that $\L \alpha \beta^\star$ and $\L \alpha^\star \beta$ are actually
ideals and that they are invariant under the adjunction, which will conclude the proof of point (c).
The invariance under the adjunction is clear because each one of $\alpha \beta^\star$ and $\alpha^\star \beta$
has trace zero (this has been assumed for $\alpha \beta^\star$, whereas $\tr(\alpha^\star \beta)=\tr((\alpha^\star \beta)^\star)=\tr(\beta^\star \alpha)
=\tr(\alpha \beta^\star)=0$).

Next, for $\L \alpha \beta^\star$
we use the double-writing $\L \alpha \beta^\star=\L \beta \alpha^\star$
to see that it is invariant under right-multiplication by $\beta$ (with the first expression) and by $\alpha$ (with the second one).
Since $\alpha$ and $\beta$ generate the $\L$-algebra $\calW_{p,q,[r]}$, we deduce that
$\L \alpha \beta^\star$ is a right ideal. We proceed likewise to see that $\L \alpha \beta^\star$
is invariant under left-multiplication by $\alpha^\star$ and $\beta^\star$, and deduce that it is a left ideal.
Hence $\L \alpha \beta^\star$ is an ideal, as claimed. The proof is similar for $\L \alpha^\star \beta$.

Hence, point (c) is now entirely proven.

Now, we add the assumption that $q$ has a double root in $\L$ (and we still assume that $p$ splits). Then $\L[b_r]$ is degenerate and
$\tr(\beta)=0$ and $N(\beta)=0$. As seen earlier $\L \alpha \beta^\star=\L \alpha \beta$ and $\L \alpha^\star \beta$ are distinct $1$-dimensional
subspaces over $\L$, and from \eqref{eq:descriptionRr} we deduce that
$\mathfrak{R}_r=\L \alpha \beta^\star + \L \alpha^\star \beta=\L \alpha \beta + \L \alpha^\star \beta$.
Since $\beta^2=0$ it is easy to deduce that $\mathfrak{R}_r=\calW_{p,q,[r]} \beta$.
Hence $\mathfrak{R}_r=\calW_{p,q,[r]} \beta$, and by applying the adjunction we also obtain $\mathfrak{R}_r=\beta \calW_{p,q,[r]}$
because $\beta^\star=-\beta$. It follows that $\mathfrak{R}_r$ is the ideal generated by $\beta$
(and it is also the left ideal generated by $\beta$, and the right ideal generated by $\beta$).
We have taken a specific $\beta \in \L[b_r]$, but the result is unchanged in replacing it with an arbitrary zero divisor in $\L[b_r]$ because
these zero divisors all belong to $\L^\times \beta$. This completes the proof of point (d).
\end{proof}

We finish with the study of the ideals in $\calW_{p,q}$ that include $(r(\omega))$ in the case where $p$ and $q$ are irreducible.
This is another manifestation of the rigidity of $\calW_{p,q}$ in the ``double-irreducible" case.

\begin{prop}\label{prop:allidealsbothirr}
Assume that $p$ and $q$ are irreducible, and let $r \in \Irr(\F)$ divide $\Lambda_{p,q}$.
Then the only proper ideals of $\calW_{p,q}$ that include $(r(\omega))$ are
$\mathfrak{J}_r$ and $(r(\omega))$, and they are equal if and only if $p$ and $q$ are inseparable with distinct splitting fields.
\end{prop}

\begin{proof}
This amounts to proving that the only proper ideals of $\calW_{p,q,[r]}$ are
$\mathfrak{R}_r$ and $\{0\}$, and that $\mathfrak{R}_r=\{0\}$ if and only if $p$ and $q$ are inseparable with distinct splitting fields.

We have already seen that $\mathfrak{R}_r$ is the sole maximal ideal in $\calW_{p,q,[r]}$ (Proposition \ref{prop:structureofJr}), and hence it remains to
determine the ideals that are included in $\mathfrak{R}_r$.

Set $\L:=\F[t]/(r)$, and assume first that $\dim_\L \mathfrak{R}_r=2$ (which holds whenever $\dim_\L R_r=2$).
By Table \ref{table1}, at least one of $p$ and $q$ remains irreducible over $\L$, and without loss of generality we will assume that $p$ does.
Let then $x \in \mathfrak{R}_r \setminus \{0\}$. The ideal $(x)$ contains $a_r x$, where $a_r$ stands for the coset of $a$ in $\calW_{p,q,[r]}$.
Just like in the proof of the Zero Divisors Theorem given in Section \ref{section:ZDnewproof}, and more precisely in Remark \ref{rem:altroutezerodivisor},
we find that $x$ and $a_r x$ are linearly independent over $\L$, whence $(x)=\mathfrak{R}_r$.
This shows that no ideal of $\calW_{p,q,[r]}$ lies strictly between $\mathfrak{R}_r$ and $\{0\}$.

Assume now that $\dim_\L R_r \neq 2$. Then by Proposition \ref{prop:structureofRr} the polynomials $p$ and $q$ are inseparable, so $r=t$ and $\L=\F$.
From Proposition \ref{prop:structureofquotient} we find
$\dim_\L \mathfrak{R}_r=4-[\K:\F]$, and hence either
$p$ and $q$ have the same splitting field in $\K$, in which case $\dim_\L \mathfrak{R}_r=2$
and the previous study shows that the only nontrivial ideal of $\calW_{p,q,[r]}$ is $\mathfrak{R}_r$,
or $p$ and $q$ have distinct splitting fields in $\K$, in which case $\mathfrak{R}_r=\{0\}$
and $\calW_{p,q,[r]}$ has no nontrivial ideal.
\end{proof}

\section{Finite-dimensional subalgebras of the free Hamilton algebra}\label{section:finitedimalg}

\index{Subalgebra}
In this section, we study the finite-dimensional $\F$-subalgebras of $\calW_{p,q}$.

Our starting point will be the observation that all the algebraic elements in $\calW_{p,q}$ are actually quadratic.
Then we will investigate the possible quadratic subalgebras of $\calW_{p,q}$, thanks to the use of the maximal ideals that include
the fundamental ideal. Finally, we will be able to understand the internal structure of the finite-dimensional subalgebras,
which uses the fact that the extended quaternion algebra $\overline{\calW_{p,q}}$ splits whenever one of $p$ and $q$ splits.

The finite-dimensional subalgebras will be considered again in Section \ref{section:misc:finitedimsubalg}, where we will
study their orbits under conjugation in $\calW_{p,q}$.

Note throughout that in case $p$ and $q$ are irreducible, the $2$-dimensional subalgebras of $\calW_{p,q}$
are already known to be the conjugates of the basic subalgebras (Theorem \ref{theo:conjsousalg2}), but to emphasize the new methods we will work as if this result were not already known.

\subsection{Every algebraic element of the free Hamilton algebra is quadratic}

\index{Algebraic (element)}
\begin{prop}\label{prop:algebraicimpliesquadratic}
Every element of $\calW_{p,q}$ is either quadratic or transcendental over $\F$.
\end{prop}

\begin{proof}
Let $x \in \calW_{p,q}$.
We view $x$ as a vector of the extended $\F(\omega)$-algebra $\overline{\calW_{p,q}}$.
We start from the observation that  $x^2=\tr(x)\,x-N(x)$, with $\tr(x)$ and $N(x)$ in $\F[\omega]$.
Hence $x$ is integral over the ring $\F[\omega]$, with degree at most $2$, and since $\F[\omega]$
is a unique factorization domain it is classical
(combine, e.g., corollary 1.6 and proposition 1.7 from chapter VII of \cite{Lang})
that the minimal polynomial $\mu_{\F(\omega)} \in \F(\omega)[t]$ of $x$ over the field $\F(\omega)$ has its coefficients in $\F[\omega]$ and degree at most $2$.

Assume now that $x$ is algebraic over $\F$, and denote by $\mu_\F \in \F[t]$ its minimal polynomial over $\F$.
Then $\mu_{\F(\omega)}$ divides $\mu_\F$ in $\F(\omega)[t]$, and hence in $\F[\omega][t]$ because $\mu_{\F(\omega)}$ is monic with coefficients in $\F[\omega]$. Writing $\mu_\F=\mu_{\F(\omega)} r$ for some $r \in \F[\omega][t]$,
we compare the degrees in $\omega$ (seeing all three polynomials as polynomials in the indeterminate $\omega$ with coefficients in the
domain $\F[t]$), we note that $\deg_\omega(\mu_\F)=0$ and deduce that $\deg_\omega(\mu_{\F(\omega)})=0$.
In other words, $\mu_{\F(\omega)} \in \F[t]$, and we conclude that $x$ is quadratic over $\F$.
\end{proof}

\begin{cor}\label{cor:basiclemmasubalgebra}
Let $\calA$ be a finite-dimensional subalgebra of $\calW_{p,q}$.
Then:
\begin{enumerate}[(a)]
\item All the elements of $\calA$ are quadratic;
\item $\calA$ is invariant under the adjunction;
\item $\langle x,y\rangle \in \F$ for all $x,y$ in $\calA$.
\end{enumerate}
\end{cor}

\begin{proof}
All the elements of $\calA$ are algebraic over $\F$, and hence the first point readily follows from Proposition
\ref{prop:algebraicimpliesquadratic}. As a consequence $\tr(x) \in \F$ and $N(x) \in \F$ for all $x \in \calA$,
due to Lemma \ref{lemma:quadrecognize}.
Let $x \in \calA$. Then $x^\star=\tr(x)-x \in \calA$, yielding the second point.
As a consequence, for all $x,y$ in $\calA$, the algebra $\calA$ contains $xy^\star$ and we conclude that
$\langle x,y\rangle=\tr(xy^\star) \in \F$ (alternatively, we could start from $\forall x \in \calA, \; N(x) \in \F$ and polarize).
\end{proof}

\subsection{The possible $2$-dimensional subalgebras}\label{section:dim2subalg}

Our next step is the identification of the structure, up to isomorphism, of the $2$-dimensional subalgebras of $\calW_{p,q}$,
i.e., of the subalgebras generated by quadratic elements. Our aim is to prove the following result:

\begin{theo}\label{theo:classdim2}
Let $\calA$ be a $2$-dimensional subalgebra of $\calW_{p,q}$.
Then one of the following two statements holds:
\begin{enumerate}[(i)]
\item $\calA$ is isomorphic to one of the basic subalgebras;
\item One of $p$ and $q$ splits, and $\calA$ is degenerate.
\end{enumerate}
Moreover, whenever one of $p$ and $q$ splits with simple roots there exists a degenerate quadratic subalgebra of $\calW_{p,q}$.
\end{theo}

\begin{proof}
We choose an irreducible monic divisor $r \in \Irr(\F)$ of $\Lambda_{p,q}$ and
a maximal ideal $I$ of $\calW_{p,q}$ that includes $(r(\omega))$. We consider the residue field $\L:=\F[t]/(r)$.
Remember from Proposition \ref{prop:structureofquotient} that
$\calW_{p,q}/I$ is isomorphic to the splitting field $\K$ of $pq$.

The main idea is to consider the homomorphism
$$\Phi : \calA \rightarrow \calW_{p,q}/I$$
of $\F$-algebras induced by the canonical projection.

\noindent
\textbf{Step 1.} \textbf{The case when $p$ and $q$ are irreducible.} \\
For this case, we could of course refer to Corollary \ref{cor:conjsousalg2}, but we will give a completely different proof
that can be adapted to the other cases.

To start with, the Zero Divisors Theorem yields that $\calA$ is a field.
Hence $\Phi$ is injective, to the effect that $\calA$ is isomorphic (as an $\F$-algebra) to a subfield of $\K$.
If $p$ and $q$ have the same splitting field, this is enough to see that $\calA \simeq \K \simeq \F[a]$.

In the remainder of this step, we assume that $p$ and $q$ have distinct splitting fields in $\K$.

\vskip 3mm
\noindent \textbf{Subcase 1.} Exactly one of $p$ and $q$ is separable. \\
Then there are only two subfields of $\K$ with degree $2$ over $\F$, they are isomorphic to $\F[a]$ and $\F[b]$,
and hence $\calA$ is isomorphic to either $\F[a]$ or $\F[b]$.

\vskip 3mm
\noindent \textbf{Subcase 2.} Both $p$ and $q$ are inseparable. \\
In particular $\car(\F)=2$ and $r=t$. In that case we will not use the projection to $\calW_{p,q}/I$ any further.
Take $x \in \calA \setminus \F$. The field $\calA$ is inseparable so its quadratic trace is zero,
and we deduce that $\tr(x)=0$. Then $1,a,b,x$ all have trace zero, and by decomposing $x$ in the deployed basis $(1,a,b,ab)$
we infer that $x \in \Vect_C(1,a,b)$.
Next, we write $x=s_1(\omega)+s_2(\omega)\,a+s_3(\omega)\,b$ and compute the norm to get
$$N(x)=s_1(\omega)^2+s_2(\omega)^2 N(a)+s_3(\omega)^2 N(b)+s_2(\omega)s_3(\omega) \omega.$$
Remembering that $\car(\F)=2$, we deduce that the first three summands on the right-hand side
belong to $\F[\omega^2]$. Since $N(x)\in \F$, computing the coefficient on $\omega$ of the polynomials yields $s_2(0)s_3(0)=0$.
If $s_3(0)=0$, then $s_1(0)^2+s_2(0)^2 N(a)=N(x)$ and we deduce that $t^2-\tr(x) t+N(x)$, which is the minimal polynomial of $x$, also annihilates
$s_1(0)+s_2(0) a$, which must then belong to $\F[a] \setminus \F$. It follows that $\F[x] \simeq \F[a]$.
Likewise, if $s_2(0)=0$ then we obtain $\F[x] \simeq \F[b]$.

\vskip 3mm
\noindent \textbf{Subcase 3.} Both $p$ and $q$ are separable. \\
Here $r=\Lambda_{p,q}$ and we know from Table \ref{figure3} that the only quadratic subextensions of $\F - \K$ are the respective splitting fields
of $p$, $q$ and $\Lambda_{p,q}$. Now, we assume that $\calA$ is not isomorphic to one of the first two, and seek to find a contradiction.
Then $\Phi$ induces an isomorphism from $\calA$ to $\L$, identified with an $\F$-subalgebra of $\calW_{p,q}/I$.
Moreover $I=\mathfrak{J}_r$ by Proposition \ref{prop:structureofJr}.
Hence for every $x \in \calA$ there exists a unique $\varphi(x) \in \F[\omega]$ of degree at most $1$ such that
$x \equiv \varphi(x)$ mod $\mathfrak{J}_r$, and $\varphi(x) \neq 0$ whenever $x \neq 0$.
Moreover, $\varphi : x \mapsto \varphi(x)$ is clearly $\F$-linear, so it is injective.

Remembering the definition of $\mathfrak{J}_r$ in terms of inner product, we deduce that
$$\forall y \in \calW_{p,q}, \; \langle x,y\rangle \equiv \langle \varphi(x),y\rangle \quad \text{mod} \; (\Lambda_{p,q}(\omega))$$
i.e.,
\begin{equation}\label{eq:modeqtrace}
\forall y \in \calW_{p,q}, \; \tr(x y^\star) \equiv \varphi(x)\,\tr(y) \quad \text{mod} \; (\Lambda_{p,q}(\omega)).
\end{equation}
Remember also from Corollary \ref{cor:basiclemmasubalgebra} that $\calA$ is invariant under adjunction.
Then, we apply \eqref{eq:modeqtrace} to an arbitrary pair $(x,y)\in \calA^2$, and we note thanks to point (c) of Corollary \ref{cor:basiclemmasubalgebra}
that the right and left-hand sides of \eqref{eq:modeqtrace} are polynomials in $\omega$ with degree at most $1$, whereas $\deg(\Lambda_{p,q}(\omega))=2$.
Therefore
$$\forall (x,y)\in \calA^2, \; \tr(x y^\star)=\varphi(x)\,\tr(y).$$
Finally, since $\calA$ is separable we can pick $y \in \calA$ such that $\tr(y) \neq 0$.
It ensues that $\varphi$ is valued in $\F$, i.e., it is an $\F$-linear form. Yet we have seen that
$\varphi$ is injective from the start. This yields a final contradiction.

\vskip 3mm

\noindent
\textbf{Step 2.} \textbf{The case when at least one of $p$ and $q$ splits.}\\
Then $[\K:\F] \leq 2$. Assume that $\calA$ is a field.
Then $\Phi$ is an isomorphism and $[\K:\F]=2$. It follows that exactly one of $p$ and $q$
is irreducible with splitting field $\K$, so $\calA \simeq \F[a]$ or $\calA \simeq \F[b]$.

Assume that $\calA$ splits but none of $p$ and $q$ splits with simple roots.
Then we note from Proposition \ref{prop:structureofJr} that $I=\mathfrak{J}_r$, and the definition of
$\mathfrak{J}_r$ yields that $\tr(x)=\langle 1,x\rangle \in (r(\omega))$ for all $x \in \mathfrak{J}_r$.
Choose a nontrivial idempotent $x \in \calA$. Then $\tr(x)=1$, and we deduce from the previous remark that $x \not\in \mathfrak{J}_r$,
whence $\Phi(x) \neq 0$. Likewise $\Phi(x^\star) \neq 0$, and hence $\Phi(x)$ is a nontrivial idempotent in $\K$.
This is absurd because $\K$ is a field.

Hence, we have proved the first part of Theorem \ref{theo:classdim2}.

In order to conclude, it suffices to prove that if at least one of $p$ and $q$ splits with simple roots then $\calW_{p,q}$ includes a $2$-dimensional
degenerate algebra. We use the same elementary construction as in Section \ref{section:units1counterexamples}.
Say that $p$ splits, and choose an idempotent $\alpha$ in $\F[a] \setminus \F$.
Then we can take $y \in \calW_{p,q}$ such that $\alpha y^\star \neq 0$ and $\langle \alpha,y\rangle=0$
(see the construction of such an element in  Section \ref{section:units1counterexamples}; alternatively, one can use
elementary considerations on free $C$-modules, by noting that $\{y \in \calW_{p,q} : \langle \alpha,y\rangle=0\}$ is a free $C$-submodule of dimension $3$, 
whereas $\{y \in \calW_{p,q} : \alpha y^\star=0\}$ is a free $C$-submodule with dimension at most $2$ since clearly the range of 
$y \in \calW_{p,q} \mapsto \alpha y^\star$ contains $\alpha$ and $\alpha \beta^\star$, which are linearly independent over $C$).
It follows that $\tr(\alpha y^\star)=0$ and $N(\alpha y^\star)=0$, to the effect that $\calB:=\F[\alpha y^\star]$ is a degenerate $2$-dimensional subalgebra.
\end{proof}

\subsection{Application: Isomorphisms between free Hamilton algebras}

Before we proceed with the study of finite-dimensional subalgebras,
we can already give an application of the study of $2$-dimensional subalgebras,
as now we can obtain the Rigidity Theorem (Theorem \ref{theo:isomorphismbetween} in the introduction) almost effortlessly. Let us restate it:

\begin{theo}[Rigidity Theorem]
Let $p_1,q_1,p_2,q_2$ be monic polynomials with degree $2$ in $\F[t]$.
Then the $\F$-algebras $\calW_{p_1,q_1}$ and $\calW_{p_2,q_2}$ are isomorphic if and only if one of the following conditions holds:
\begin{itemize}
\item There are isomorphisms $\F[t]/(p_1) \simeq \F[t]/(p_2)$ and $\F[t]/(q_1) \simeq \F[t]/(q_2)$ of $\F$-algebras.
\item There are isomorphisms $\F[t]/(p_1) \simeq \F[t]/(q_2)$ and $\F[t]/(q_1) \simeq \F[t]/(p_2)$ of $\F$-algebras.
\end{itemize}
\end{theo}

If we take two $2$-dimensional algebras $\calA$ and $\calB$, we have just proved in Theorem \ref{theo:classdim2} that
the isomorphism types of the $2$-dimensional subalgebras of the free product $\calA * \calB$ are
the ones of $\calA$, of $\calB$ and in addition of $\F[\varepsilon]/(\varepsilon^2)$ if at least one of $\calA$ and $\calB$ splits and none is degenerate.
This is clearly enough to determine the unordered pair $\{\calA,\calB\}$ up to isomorphism,
with the only possible remaining ambiguity residing in the case where none of $\calA$ and $\calB$ is a field and at least
one of them splits.

Theorem \ref{theo:classdim2} is insufficient to differentiate
$\F^2 * (\F[\varepsilon]/(\varepsilon^2))$ from $\F^2 * \F^2$ up to isomorphism; indeed, in both algebras the possible isomorphism types of the $2$-dimensional subalgebras
are the ones of $\F^2$ and $\F[\varepsilon]/(\varepsilon^2)$.

There are several ways to deal with this special case, and at this point of the analysis the quickest one is to
consider the maximal ideals. Indeed, if $I$ is a maximal ideal of $\calW_{p,q}$, then either $I+\mathfrak{F}=\calW_{p,q}$
in which case $\calW_{p,q}/I$ is noncommutative (it is a quaternion algebra over the field $C/(I \cap C)$),
or $I$ includes $\mathfrak{F}$ and $\calW_{p,q}/I$ is a field (isomorphic to the splitting field of $pq$).
Hence a relevant invariant under isomorphism is the \emph{number} of maximal ideals $I$ such that $\calW_{p,q}/I$
is a field. Now, assume that both $p$ and $q$ split, and at least one has simple roots.
Then we know from Propositions \ref{prop:structureofJr} and \ref{prop:numberofmaxiideals} that for every root $z$ of $\Lambda_{p,q}$, exactly two maximal ideals of $\calW_{p,q}$ include the ideal $(\omega-z)$. Hence, there are exactly four maximal ideals of $\calW_{p,q}$ that include $\mathfrak{F}$
if $\Lambda_{p,q}$ has simple roots, and exactly two if $\Lambda_{p,q}$ has a double root.
Finally $\Lambda_{p,q}$ has a double root if and only if at least one of $p$ and $q$ has a double root (this is explained at the start of Section \ref{section:Lambdapq}). Hence we conclude that $\F^2 * (\F[\varepsilon]/(\varepsilon^2))$ and $\F^2 * \F^2$ are not isomorphic,
which completes the proof of Theorem \ref{theo:isomorphismbetween}.

\begin{Rem}
Another way to differentiate $\F^2 * (\F[\varepsilon]/(\varepsilon^2))$ from $\F^2 * \F^2$ up to isomorphism
is to look at the effect of a potential isomorphism on the central elements.
This anticipates the method of the start of Section \ref{section:automorphismsI}, and we briefly sketch the ideas.
Assume that there exists an isomorphism $\Phi : \calW_{p_1,q_1} \overset{\simeq}{\rightarrow} \calW_{p_2,q_2}$ of $\F$-algebras, and denote by
$\langle -,-\rangle_i$ the inner product on $\calW_{p_i,q_i}$, by $C_i$ its center and by $\mathfrak{F}_i$ its fundamental ideal.
With the same line of reasoning as for Proposition \ref{prop:commuteadjunction}, one proves that
$\Phi(x^\star)=\Phi(x)^\star$ for all $x \in \calW_{p,q}$, and one deduces that
$\langle \Phi(x),\Phi(y)\rangle_2=\Phi(\langle x,y\rangle_1)$ for all $x,y$ in $\calW_{p_1,q_1}$.
Then, by analyzing the Gram determinant of a deployed basis, one proves that
$\Phi$ maps $\mathfrak{F}_1$ onto $\mathfrak{F}_2$. By considering the intersection with the centers,
it follows that the $\F$-algebras $\F[t]/(\Lambda_{p_1,q_1})$ and $\F[t]/(\Lambda_{p_2,q_2})$
are isomorphic. If both $p$ and $q$ split with simple roots, then $\F[t]/(\Lambda_{p,q})$ splits, whereas
if both split and at least one of them has a double root, then $\F[t]/(\Lambda_{p,q})$ is degenerate.
Therefore $\F^2 * (\F[\varepsilon]/(\varepsilon^2))$ and $\F^2 * \F^2$ are nonisomorphic as $\F$-algebras.
\end{Rem}

\subsection{Finite-dimensional subalgebras: the irreducible case}

At this point, the reader might be worried that we spent so much time analyzing the subalgebras of dimension $2$. Surely, things will get even more difficult when we get to arbitrary dimensions? Well, not really.

To simplify things, we tackle two cases separately: the case where both polynomials are irreducible, and the case where at least one splits.
In each case, we note that every element in a finite-dimensional subalgebra is algebraic, and hence quadratic by Proposition \ref{prop:algebraicimpliesquadratic}.

\begin{theo}
If $p$ and $q$ are irreducible, then every nontrivial finite-dimensional subalgebra of $\calW_{p,q}$ is isomorphic to one of the basic subalgebras.
\end{theo}

\begin{proof}
Assume that $p$ and $q$ are irreducible. Let $\calA$ be a finite-dimensional subalgebra of $\calW_{p,q}$ with $\dim_\F \calA>1$.
If $\dim_\F \calA=2$ then we already know that $\calA$ is isomorphic to one of the basic subalgebras.

Now, we assume that $\dim_\F \calA \geq 3$ and seek to find a contradiction. First of all, we use the Zero Divisors Theorem:
for all $x \in \calA \setminus \{0\}$, we know that $N(x) \neq 0$ and hence $x$ is invertible in $\calA$
(e.g., remember from Corollary \ref{cor:basiclemmasubalgebra} that $\calA$ is invariant under adjunction).
Hence $\calA$ is a skew field.

Next, once more we take a maximal ideal $I$ of $\calW_{p,q}$, so that $\calW_{p,q}/I$ is isomorphic to the splitting field $\K$ of $pq$.
Then the standard projection induces a homomorphism $\Phi : \calA \rightarrow \calW_{p,q}/I$ of $\F$-algebras.
Because $\calA$ is a skew field this homomorphism is injective, and hence $\calA$ is isomorphic to a subfield of $\K$,
and in particular $\dim_\F \calA$ divides $4$. Hence $[\K:\F]=4$ and $\Phi$ is an isomorphism, and in particular $\calA$
is a field.

To conclude, we use another projection. We note first that there exists a scalar $\lambda \in \F$
that is not a root of $\Lambda_{p,q}$. Indeed, either at least one of $p$ and $q$ is separable, and then it is known from Table \ref{figure3}
that $\Lambda_{p,q}$ is irreducible, or both $p$ and $q$ are inseparable and $\Lambda_{p,q}=t^2$.
In any case we can take $\lambda=1$.
Then the quotient algebra $\calW_{p,q}/(\omega-\lambda)$ is a quaternion algebra over $\F$,
and the projection onto it induces a homomorphism $\Psi : \calA \rightarrow \calW_{p,q}/(\omega-\lambda)$ of $\F$-algebras.
Then again $\Psi$ is injective because $\calA$ is a field, and hence it is an isomorphism because $\dim_\F \calA=4=\dim_\F \bigl( \calW_{p,q}/(\omega-\lambda)\bigr)$. This contradicts the fact that $\calW_{p,q}/(\omega-\lambda)$ is noncommutative.
\end{proof}

\subsection{Finite-dimensional subalgebras: two key examples}\label{section:newsubalgebras}

In order to deal with the case where at least one of $p$ and $q$ splits,
we shall construct two possible kinds of subalgebras with dimension greater than $2$.

The following notation will be useful:

\index{Nilpotent cone}
\label{page:nilpotentcone}
\begin{Not}
For an $\F$-algebra $\calA$, we denote by $\calN(\calA)$ the set of all its nilpotent elements,
called the \textbf{nilpotent cone} of $\calA$.
\end{Not}

In general $\calN(\calA)$ is not an ideal of $\calA$, but it is the case if $\calA$ has dimension $2$.
We will see that all the finite-dimensional subalgebras of $\calW_{p,q}$ share the property that their nilpotent cone is
an ideal (although some of them are non-commutative).

\label{page:Hn+2}
\begin{Not}
For every integer $n \geq 0$, we denote by
$\calH_{n+2}$ the set of all matrices of $\Mat_2(\F[t])$ of the form
$$\begin{bmatrix}
\lambda & s(t) \\
0 & \mu
\end{bmatrix}$$
where $\lambda,\mu$ belong to $\F$ and $s(t) \in \F[t]$ has degree less than $n$.
We also denote by $\calU_{n+1}$ the set of all such matrices in which $\lambda=\mu$.
\end{Not}

One easily checks that $\calH_{n+2}$ and $\calU_{n+1}$ are $\F$-subalgebras of $\Mat_2(\F[t])$,
with respective dimensions $n+2$ and $n+1$. In both cases, the nilpotent cone is the
set of all matrices of the form
$\begin{bmatrix}
0 & s(t) \\
0 & 0
\end{bmatrix}$ where $s(t) \in \F[t]$ has degree less than $n$:
it is clearly an ideal in both $\calH_{n+2}$ and $\calU_{n+1}$.

\begin{Rem}
It is easily checked that $\calU_{n+1}$ is isomorphic to the quotient algebra $\F[t_1,\dots,t_n]/(t_1^2,\dots,t_n^2)$.
\end{Rem}

Now, we return to $\calW_{p,q}$.
In the discussion, we need notions that are related to the structure of the $C$-module $\calW_{p,q}$.
\index{Normalized (element)}
A nonzero element $x$ of $\calW_{p,q}$ is called \textbf{normalized}
when its coefficients in an arbitrary $C$-basis have $1$ as greatest common divisor, or in other words
if the quotient $C$-module $\calW_{p,q}/Cx$ is torsion-free.
Every nonzero element $x$ splits as $x=r y$ where $y$ is normalized and $r \in C$, and
$y$ is uniquely determined up to multiplication with a nonzero scalar.
In particular, there is a unique such splitting with $r$ monic with respect to $\omega$, and then we say that
$r$ is the \textbf{modular norm} of $x$ (as opposed to the norm $N(x)$), and $y$ is called the \textbf{normalization} of $x$ (with respect to $\omega$).
\index{Modular norm}
\index{Normalization (of a nonzero element)}

Now, assume that we have a normalized element $\beta \in \calW_{p,q} \setminus \{0\}$ such that $\beta^2=0$.
We define
$$\calU(\beta):=\F \oplus C \beta.$$
Clearly $\calU(\beta)$ is a commutative subalgebra of $\calW_{p,q}$, and $\calN(\calU(\beta))=C \beta$
is an ideal of it (and obviously a maximal one).
Now, consider an arbitrary $\F$-linear subspace $V$ of $C$ with finite dimension $n>0$. Then
\label{page:degeneratesubalg}
$$\calU(\beta,V):=\F \oplus V \beta$$
is clearly a subalgebra of $\calU(\beta)$ with dimension $n+1$ and nilpotent cone $V \beta$.
Hence $V$ is uniquely determined by $\calU(\beta,V)$, and $\beta$ is uniquely determined by
$\calU(\beta,V)$ up to multiplication with a nonzero element of $\F$ (both of these properties follow from the requirement that $\beta$ be normalized).
We say that these subalgebras are of \textbf{degenerate type}.
\index{Degenerate type (of subalgebras)}

Note that the existence of $\beta$ is equivalent to the one of an element $x \in \calW_{p,q} \setminus \{0\}$
such that $x^2=0$ (it then suffices to take $\beta$ as the normalization of $x$), and the existence of $\beta$ is granted
whenever one of $p$ and $q$ splits, as we have seen in Section \ref{section:dim2subalg}.
Now, we prove that the structure of the algebra $\calU(\beta,V)$ up to isomorphism depends only on the dimension of $V$:

\begin{prop}
Let $V$ be an $n$-dimensional linear subspace of $C$, and $\beta$ be a normalized element of $\calW_{p,q} \setminus \{0\}$ such that $\beta^2=0$.
Then $\calU(\beta,V) \simeq \calU_{n+1}$.
\end{prop}

\begin{proof}
Denote by $\F_{<n}[t]$ the vector space of all polynomials with degree less than $n$.
Consider an isomorphism $\varphi : V \overset{\simeq}{\rightarrow} \F_{<n}[t]$ of $\F$-vector spaces.
Then one easily checks that $\lambda+u \beta \in \calU(\beta,V) \mapsto \begin{bmatrix}
\lambda & \varphi(u) \\
0 & \lambda
\end{bmatrix} \in \calU_{n+1}$ is an isomorphism of $\F$-algebras. Details are left to the reader.
\end{proof}

Our next family of examples requires that we can find a nontrivial idempotent $\alpha \in \calW_{p,q}$.
Note that the existence of such an element is not guaranteed: as seen in Theorem \ref{theo:classdim2}, it holds
if and only if at least one of $p$ and $q$ splits with simple roots.

We start with a lemma:

\label{page:alphadiese}
\begin{lemma}\label{lemma:idempotentlemma}
Let $\alpha \in \calW_{p,q}$ be a nontrivial idempotent.
Let $x \in \calW_{p,q}$. The following conditions are equivalent:
\begin{enumerate}[(i)]
\item $x=\alpha z \alpha^\star$ for some $z \in \calW_{p,q}$;
\item $x=\alpha y^\star$ for some $y \in \calW_{p,q}$ such that $\langle \alpha,y\rangle=0$;
\item $\alpha^\star x=0$ and $x \alpha=0$.
\end{enumerate}
Moreover, the set $\alpha^\sharp$ of all $x \in \calW_{p,q}$ that satisfy these conditions is a rank $1$ $C$-submodule that is generated
by a normalized vector.

Finally $x^2=0$ for all $x \in \alpha^\sharp$.
\end{lemma}

Note also that $\alpha^\sharp$ is invariant under the adjunction (this is clear by property (i)).

\begin{proof}
Assume that (i) holds, and take an associated $z$.
Then $y:=\alpha z^\star$ satisfies $\langle \alpha,y\rangle=N(\alpha) \tr(z^\star)=0$.
Hence (i) implies (ii).

Assume that (ii) holds, and take an associated $y$. It is clear that $\alpha^\star(\alpha y^\star)=0$, and
$\alpha y^\star \alpha=\alpha(\langle y,\alpha\rangle-\alpha^\star y)=0$.
Hence (ii) implies (iii).

Assume finally that (iii) holds. Since $\alpha$ is a nontrivial idempotent, it is nonscalar and hence $\tr(\alpha)=1$ by Lemma \ref{lemma:quadrecognize}.
Then
$$x=(\alpha+\alpha^\star)x(\alpha+\alpha^\star)=\alpha x \alpha^\star,$$
and so (i) holds, and then $x^2=\alpha x \alpha^\star \alpha x \alpha^\star=0$ since $\alpha^\star \alpha=0$.

Next, the existence of $\alpha$ shows that the extended quaternion algebra $\overline{\calW_{p,q}}$ splits.
Hence, we have an isomorphism of $\F(\omega)$-algebras $\Phi : \overline{\calW_{p,q}} \overset{\simeq}{\rightarrow} \Mat_2(\F(\omega))$.
Set $L:=\{x \in \calW_{p,q} : \alpha^\star x=0\}$ and $R:=\{x \in \calW_{p,q} : x \alpha=0\}$.
Classically, for every field $\K$ and every pair $(M,N) \in \Mat_2(\K)^2$ of rank $1$ matrices (i.e., of non-zero singular matrices),
$\{A \in \Mat_2(\K) : MA=0 \; \text{and} \; AN=0\}$ is a $\K$-linear subspace
of dimension $1$. Hence $L \cap R$ is the intersection of $\calW_{p,q}$ with a
$1$-dimensional $\F(\omega)$-linear subspace of $\overline{\calW_{p,q}}$. Since $\calW_{p,q}$ is a free $\F[\omega]$-module, it is then clear
that $\alpha^\sharp=L \cap R$ is a free submodule of rank $1$ of $\calW_{p,q}$ that is generated by a normalized element.
\end{proof}

For every nontrivial idempotent $\alpha \in \calW_{p,q}$, we have $\calU(z)=\F \oplus \alpha^\sharp$ for any normalized element $z$ of $\alpha^\sharp$, and obviously
$\calU(z)$ does not contain $\alpha$ (because the only idempotents in $\calU(z)$ are $0$ and $1$).

\label{page:degeneratesubalgidem}
\begin{Def}
Let $\alpha$ be a nontrivial idempotent in $\calW_{p,q}$.
Choose a normalized element $z$ in $\alpha^\sharp$ and
set
$$\calH(\alpha):=\F \oplus \F \alpha \oplus \alpha^\sharp = \F \alpha \oplus \calU(z)=\F \alpha^\star \oplus \calU(z).$$
More generally, for every linear subspace $V$ of $C$, we set
$$\calH(\alpha,V):=\F \alpha \oplus \calU(z,V)=\F \alpha^\star \oplus \calU(z,V).$$
\end{Def}

Using Lemma \ref{lemma:idempotentlemma}, we see that $\alpha^\star (\F \alpha^\star \oplus \calU(z,V))=\F \alpha^\star$
and $(\F \alpha \oplus \calU(z,V)) \alpha=\F\alpha$.
Combining this with the fact that $\calU(z,V)$ is an $\F$-subalgebra of $\calW_{p,q}$,
we deduce that $\calH(\alpha,V)$ is an $\F$-subalgebra of $\calW_{p,q}$, and if $V$ is finite-dimensional $\calH(\alpha,V)$ has finite
dimension $\dim V+2$.
Such subalgebras will be called of \textbf{idempotent type}.
\index{Idempotent type (of subalgebras)}
In particular $\calH(\alpha)=\calH(\alpha,C)$ is a subalgebra, and from point (iii) in Lemma \ref{lemma:idempotentlemma}
it is easily seen that $\alpha^\sharp$ is an ideal of $\calH(\alpha)$.

Once again, we identify the internal structure of $\calH(\alpha,V)$ as a function of the dimension of $V$:

\begin{prop}
Let $V$ be an $n$-dimensional $\F$-linear subspace of $C$.
Then $\calH(\alpha,V) \simeq \calH_{n+2}$.
\end{prop}

\begin{proof}
Denote by $\F_{<n}[t]$ the vector space of all polynomials with degree less than $n$, and choose a generator $z$ of $\alpha^\sharp$.
Consider an isomorphism $\varphi : V \overset{\simeq}{\rightarrow} \F_{<n}[t]$ of $\F$-vector spaces.
Then one easily checks that $\lambda \alpha+\mu \alpha^\star +u z \mapsto \begin{bmatrix}
\lambda & \varphi(u) \\
0 & \mu
\end{bmatrix}$ is an isomorphism of $\F$-algebras from $\calH(\alpha,V)$ to $\calH_{n+2}$. Details are left to the reader.
\end{proof}

As for $\calH(\alpha)$, it is easily seen that it is isomorphic to the $\F$-subalgebra of $\Mat_2(\F[t])$
consisting of all the matrices of the form
$\begin{bmatrix}
\lambda & u \\
0 & \mu
\end{bmatrix}$ where $u \in \F[t]$ and $(\lambda,\mu)\in \F^2$.

Note that for all $x \in \alpha^\sharp \setminus \{0\}$, $(\alpha+x)^2=\alpha+\alpha x+x \alpha+x^2=\alpha+x$,
and obviously $\alpha+x \not\in \{0,1\}$. It is then easy to check that $(\alpha+x)^\sharp=\alpha^\sharp$:
note indeed that both contain the nonzero element $x$ since
$(\alpha+x)^\star x=N(x)=0$ and $x(\alpha+x)=x^2=0$, hence if we denote by $x'$ the normalization of $x$
then $x'$ is a normalized vector in both $(\alpha+x)^\sharp$ and $\alpha^\sharp$, to the effect that 
$(\alpha+x)^\sharp= C x'=\alpha^\sharp$.
Hence $\calH(\alpha+x,V)=\calH(\alpha,V)$. This shows that $\calH(\alpha,V)$ does not determine $\alpha$.
The following lemmas will help us be more precise:

\begin{lemma}\label{lemma:nilpotentcone}
Let $\alpha$ be a nontrivial idempotent in $\calW_{p,q}$.
Then the nilpotent cone of $\calH(\alpha)$ is $\alpha^\sharp$, and
the nontrivial idempotents in $\calH(\alpha)$ are the elements of $(\alpha+\alpha^\sharp) \cup (\alpha^\star+\alpha^\sharp)$.
\end{lemma}

\begin{proof}
We have seen earlier that $\alpha^\sharp$ is an ideal of $\calH(\alpha)$, and it is then clear from $\calH(\alpha)=\F \alpha\oplus \F \alpha^\star\oplus \alpha^\sharp$
that $\calH(\alpha)/\alpha^\sharp$ is a split $2$-dimensional algebra whose only nontrivial idempotents are the cosets of $\alpha$ and $\alpha^\star$.
In particular, every nilpotent element of $\calH(\alpha)$ belongs to $\alpha^\sharp$;
the converse is obvious.

Let $x \in \calH(\alpha)$ be a nontrivial idempotent. Then it cannot belong to
$\alpha^\sharp$, and neither does $x^\star$, so the coset of $x$ in $\calH(\alpha)/\alpha^\sharp$
is a nontrivial idempotent. Hence it equals the coset of $\alpha$ or the one of $\alpha^\star$.
This proves that every nontrivial idempotent in $\calH(\alpha)$ belongs to
$(\alpha+\alpha^\sharp) \cup (\alpha^\star+\alpha^\sharp)$.
For the converse inclusion, we have seen right before stating this lemma that $\alpha+\alpha^\sharp$ consists only of nontrivial idempotents, and by applying the adjunction we deduce that every element of $\alpha^\star+\alpha^\sharp$ is a nontrivial idempotent
(recall indeed from Lemma \ref{lemma:idempotentlemma} that $\alpha^\sharp$ is invariant under the adjunction).
\end{proof}

\begin{lemma}
Let $\alpha$ and $\beta$ be nontrivial idempotents in $\calW_{p,q}$, and
$V$ and $W$ be nonzero linear subspaces of $C$. Choose a generator $z$ of $\alpha^\sharp$.
Assume that $\calH(\alpha,V)=\calH(\beta,W)$. Then $\beta \in \alpha+Vz$ and $V=W$.
\end{lemma}

\begin{proof}
Take $z'$ as a generator of $\beta^\sharp$.
By extracting the nilpotent cone, we find $V z=W z'$, and since $V \neq \{0\}$
we deduce that $z \in \F^\times z'$ and $V=W$. In turn, this shows that $\alpha^\sharp=\beta^\sharp$.
Finally $\beta$ is a nontrivial idempotent in $\calH(\alpha,V)$, so either $\beta \in \alpha+V z$ or $\beta \in \alpha^\star+Vz$.
Assume that the second case holds. Then $\beta=\alpha^\star +u z$ for some $u \in V$.
As $z \in \alpha^\sharp=\beta^\sharp$ we find $z \beta=0$.
Hence
$$0=z\alpha^\star+\underbrace{u z^2}_0=z (1-\alpha)=z,$$
which is absurd. Hence $\beta \in \alpha+V z$.
\end{proof}

\subsection{Finite-dimensional subalgebras: the case where one of $p$ and $q$ splits}

Now, we can state our main result:

\begin{theo}\label{theo:subalgdim>2split}
Assume that one of $p$ and $q$ splits.
Let $\calA$ be a finite-dimensional subalgebra of $\calW_{p,q}$, with $\dim \calA>2$.
Then:
\begin{itemize}
\item Either $\calA=\calU(\beta,V)$ for some normalized $\beta \in \calW_{p,q}$ such that $\beta^2=0$,
and some finite-dimensional linear subspace $V$ of $C$.
\item Or $\calA=\calH(\alpha,V)$ for some nontrivial idempotent $\alpha \in \calW_{p,q}$
and some finite-dimensional linear subspace $V$ of $C$.
\end{itemize}
\end{theo}

Combining this with the results of the previous paragraphs, we can conclude:

\begin{cor}
Let $n>2$ be an integer.
\begin{enumerate}[(i)]
\item If both $p$ and $q$ are irreducible, then $\calW_{p,q}$ has no $n$-dimensional subalgebra.
\item If one of $p$ and $q$ splits, but none splits with simple roots, there is an $n$-dimensional subalgebra of
$\calW_{p,q}$, and every such subalgebra is isomorphic to~$\calU_n$.
\item If one of $p$ and $q$ splits with simple roots, then there are two isomorphism classes of $n$-dimensional subalgebras of
$\calW_{p,q}$, and each such subalgebra is isomorphic to $\calH_n$ or to $\calU_n$.
\end{enumerate}
\end{cor}

Now, we prove Theorem \ref{theo:subalgdim>2split}.
So, we assume that one of $p$ and $q$ splits, and we take a subalgebra $\calA$ of $\calW_{p,q}$ with finite dimension $n>2$.

As is now customary, we choose a divisor $r \in \Irr(\F)$ of $\Lambda_{p,q}$, and we consider the residue field $\L:=\F[t]/(r)$. Note that the splitting field $\K$ of $pq$
has degree at most $2$ over $\F$.
There are two main steps in the proof. The first step consists in proving that $\calA \cap \mathfrak{J}_r \neq \{0\}$.

\vskip 3mm
\noindent \textbf{Step 1:} Proof that $\calA \cap \mathfrak{J}_r \neq \{0\}$. \\
We use a \emph{reductio ad absurdum}. Assume that $\calA \cap \mathfrak{J}_r=\{0\}$.
Moding out $(r(\omega))$ and then $\mathfrak{R}_r$, we obtain
an injective homomorphism $\Phi : \calA \rightarrow \calW_{p,q,[r]}/\mathfrak{R}_r$ of $\F$-algebras.
We claim that $\dim_\L \mathfrak{R}_r \geq 2$.
If at least one of $p$ and $q$ has simple roots in $\K$ or if $\car(\F) \neq 2$ then this directly follows from Proposition \ref{prop:structureofRr};
if both $p$ and $q$ split with double roots in $\K$ and $\car(\F)=2$ then we observe that $\L=\F$, $r=t$, and by taking
$\alpha \in \F[a] \setminus \F$ and $\beta \in \F[b] \setminus \F$ such that $N(\alpha)=N(\beta)=0$
we see that $(\alpha_r,\beta_r)$ is $\F$-linearly independent and $\alpha_r,\beta_r$ belong to $\mathfrak{R}_r$.

As a consequence of $\dim_\L \mathfrak{R}_r \geq 2$, we find $\dim_\F \calA \leq 2\, [\L:\F]$ and we deduce that $[\L:\F]=2$ and $\dim_\L \mathfrak{R}_r=2$.
It follows that exactly one of $p$ and $q$ is irreducible,
and that the other one splits with simple roots. By Propositions \ref{prop:structureofJr} and \ref{prop:numberofmaxiideals}, we know that $\calW_{p,q,[r]}/\mathfrak{R}_r$
is isomorphic to the $\F$-algebra $\L^2$, and we deduce that $\calA$ is isomorphic to an $\F$-subalgebra $\calB$ of $\L^2$
with dimension at least $3$.

We shall conclude by noting that this contradicts the fact that all the elements of $\calA$ are quadratic over $\F$.
Note indeed that given $\alpha \in \L \setminus \F$, the element $(\alpha,0)$ is not quadratic in the $\F$-algebra $\L$.
Indeed, if $(\alpha,0)^2=\lambda (\alpha,0)+\mu (1,1)$ for some $(\lambda,\mu) \in \F^2$, then $\mu=0$ by taking the second component,
and it successively follows that $\alpha^2=\lambda \alpha$ and that $\alpha=\lambda \in\F$.
Likewise $(0,\beta)$ is non-$\F$-quadratic for all $\beta \in \L \setminus \F$.

Next, the intersection $(\L \times \{0\}) \cap \calB$ must have dimension at least $1$ since $\dim_\F \L^2=4$ and $\dim_\F \calB \geq 3$.
Since all its elements are in $\calB$, they are $\F$-quadratic and we deduce that $(1,0) \in \calB$. Then $(0,1)=1_{\L^2}-(1,0) \in \calB$.
Finally, let $(\alpha,\beta) \in \calB$. Then $(\alpha,0)=(1,0)(\alpha,\beta)$ and $(0,\beta)=(0,1)(\alpha,\beta)$ belong to $\calB$,
hence they are $\F$-quadratic, which leads to $\alpha \in \F$ and $\beta \in \F$. Therefore, $\calB\subseteq \F^2$ and we have contradicted the assumption that
$\dim_\F \calB=3$.

We conclude that $\calA \cap \mathfrak{J}_r \neq \{0\}$.

\vskip 3mm
\noindent \textbf{Step 2:} Using a nonzero element of $\calA \cap \mathfrak{J}_r$. \\
Now, we choose a nonzero element $z \in \calA \cap \mathfrak{J}_r$.
We claim that $N(z)=0$ and $\langle x,z\rangle=0$ for all $x \in \calA$.
To see this, note that $N(z) \in \F$, whereas $N(z) \equiv 0$ mod $(r(\omega))$ due to
the definition of $\mathfrak{J}_r$.
Likewise, for all $x \in \calA$, we know from Corollary \ref{cor:basiclemmasubalgebra} that $\langle x,z\rangle \in \F$
whereas $\langle x,z\rangle \in (r(\omega))$, and we deduce that $\langle x,z\rangle=0$.
In particular $\tr(z)=\langle 1,z\rangle=0$ and we deduce that $z^2=0$.
Note that for all $x \in \calA$ we have $x z^\star+z x^\star=0$, i.e., $xz=zx^\star$.

Now, it will be convenient to use the extended quaternion algebra $\overline{\calW_{p,q}}$, which splits because one of $p$ and $q$ splits.
We choose an isomorphism of $\F(\omega)$-algebras
$$\Psi : \overline{\calW_{p,q}} \overset{\simeq}{\longrightarrow} \Mat_2(\F(\omega))$$
and consider the matrix $\Psi(z)$, whose square is zero. Hence there exists $P \in \GL_2(\F(\omega))$
such that $\Psi(z)=P E_{1,2} P^{-1}$, where $E_{1,2}:=\begin{bmatrix}
0 & 1 \\
0 & 0
\end{bmatrix}$. By composing $\Psi$ with the conjugation $M \mapsto P^{-1}MP$, we
reduce the situation to the one where $\Phi(z)=E_{1,2}$.
Now, let $x \in \calA$. Then $\Phi(x)E_{1,2}=E_{1,2} \Phi(x^\star)$
shows that the matrix $\Phi(x)$ leaves the column space of $E_{1,2}$ invariant, i.e.,
$$\Phi(x)=\begin{bmatrix}
f_1(x) & ? \\
0 & f_2(x)
\end{bmatrix} \quad \text{for some $f_1(x)$ and $f_2(x)$ in $\F(\omega)$.}$$
Next, since $\Phi(x)$ is algebraic over $\F$, it follows that $f_1(x)$ and $f_2(x)$ are algebraic over $\F$,
and hence they belong to $\F$.
This yields a homomorphism of $\F$-algebras
$$\Theta : x \in \calA \mapsto (f_1(x),f_2(x)) \in \F^2.$$
Finally, we take a normalization $\beta$ of $z$ in the $C$-module $\calW_{p,q}$.
Then $\Ker \Theta$ consists of the elements $x \in \calA$ that are
$\F(\omega)$-scalar multiples of $z$, and hence $\Ker \Theta = \calA \cap (\F[\omega]\beta)$,
which equals $V \beta$ for some $\F$-linear subspace $V$ of $\F[\omega]$ with dimension $n-\rk \Theta$.

Hence if the range of $\Theta$ has dimension $1$, then clearly
$\calA=\F \oplus \Ker \Theta=\calU(\beta,V)$.

Assume finally that $\Theta$ is surjective.
Then we can pick $\alpha \in \calA$ such that $\Theta(\alpha)=(1,0)$. By matrix computation it is clear that
$\Phi(\alpha)$ is a nontrivial idempotent in $\Mat_2(\F[\omega])$, and hence $\alpha$ is a nontrivial idempotent in $\calW_{p,q}$.
Then we see that $(I_2-\Phi(\alpha)) \Phi(\beta)=0$ and $\Phi(\beta)\Phi(\alpha)=0$ by matrix computation,
and it follows that $\alpha^\star \beta=0=\beta \alpha$. Using Lemma \ref{lemma:idempotentlemma}, we deduce that
$\beta$ is a generator of the $C$-module $\alpha^\sharp$.
Finally, since $\Theta$ maps $\F+\F \alpha$ onto $\F^2$ we find $\calA=\F +\F \alpha+ \Ker \Theta=\calH(\alpha,V)$.
This completes the proof of Theorem \ref{theo:subalgdim>2split}.

\vskip 3mm
At this point, we have a constructive description of all the finite-dimensional subalgebras of $\calW_{p,q}$,
and a clear understanding of the internal structure of such algebras.
Two main questions remain:
\begin{enumerate}[(i)]
\item What are the orbits of the finite-dimensional subalgebras of $\calW_{p,q}$ under conjugation
(i.e., under the standard action of the group of inner automorphisms)?
\item What are the orbits of the finite-dimensional subalgebras of $\calW_{p,q}$ under the action of the automorphism group
$\Aut(\calW_{p,q})$?
\end{enumerate}
The second issue has already been solved in Section \ref{section:units1} in the case where both $p$ and $q$ are irreducible
(i.e., there is a single orbit for each isomorphism type of finite-dimensional subalgebras).
Solving both in the general case is premature at this point: we will wait until the very last section, which combines deep
result on the automorphism group and on the group of units.

\section{Automorphisms of the free Hamilton algebra (part 1): The action on the center}\label{section:automorphismsI}

\subsection{Introduction}

The aim of this section and the next one is to analyze the automorphisms of the $\F$-algebra $\calW_{p,q}$.

Our first observation is that in most cases $\calW_{p,q}$ has outer automorphisms.
Indeed, every inner automorphism leaves $\omega$ invariant, yet we have seen in Section \ref{section:omega} that the automorphism
that leaves $a$ invariant and maps $b$ to $b^\star$ takes $\omega$ to $(\tr p)(\tr q)-\omega$,
which differs from $\omega$ unless $\car(\F)=2$ and $(\tr p)(\tr q)=0$.

\label{page:Cautomorphisms}
Hence, remembering that $C$ denotes the center of $\calW_{p,q}$,
we can already sense that a big part in our study will be played by the normal subgroup $\Aut_C(\calW_{p,q}) \trianglelefteq \Aut(\calW_{p,q})$ of all
the automorphisms of the $C$-algebra $\calW_{p,q}$.
Such automorphisms will be called \textbf{$C$-automorphisms} (as opposed to plain automorphisms of the $\F$-algebra $\calW_{p,q}$,
and which we simply call automorphisms).
\index{Automorphisms}
\index{$C$-automorphisms}

Another big part of course will also be played by the \emph{basic} automorphisms
that we have introduced earlier and whose definition we quickly recall.
Remember that an $\F$-automorphism $\Phi$ of $\calW_{p,q}$ is called basic when it maps every basic vector to a basic vector,
which by linearity (and injectivity) amounts to have it either leave both basic subalgebras $\F[a]$ and $\F[b]$ invariant (in which case we say that it is \textbf{positive}), 
\index{Positive (basic automorphism)}
\index{Negative (basic automorphism)}
or maps each basic subalgebra to the opposite one (in which case we say that it is \textbf{negative}).
In the first case, the automorphism induces respective automorphisms of $\F[a]$ and $\F[b]$, and in the second one the algebras $\F[a]$ and $\F[b]$ are isomorphic. In any case the set of all basic automorphisms of $\calW_{p,q}$ is a subgroup of $\Aut(\calW_{p,q})$, denoted by
$\BAut(\calW_{p,q})$, and it contains the normal subgroup $\BAut_+(\calW_{p,q})$ of all positive basic automorphisms, which is
a proper subgroup of $\BAut(\calW_{p,q})$ only if $\F[a]$ and $\F[b]$ are isomorphic, in which case it has index $2$.
\label{page:positivebasic}

Let us now recall our main goal, by restating the Automorphisms Theorem that was already announced in the introduction:

\begin{theo}[Automorphisms Theorem]\label{theo:maintheoautomorphisms}
The subgroup $\BAut(\calW_{p,q})$ is a semi-direct factor $\Inn(\calW_{p,q})$ in $\Aut(\calW_{p,q})$.
In other words, every automorphism of the $\F$-algebra $\calW_{p,q}$ splits uniquely as the composite of an inner automorphism
followed by a basic automorphism.
\end{theo}

Now, let us briefly discuss the $C$-automorphisms of $\calW_{p,q}$.
Obviously $\Aut_C(\calW_{p,q})$ includes $\Inn(\calW_{p,q})$ as a normal subgroup.
One might think that $\Aut_C(\calW_{p,q})$ equals $\Inn(\calW_{p,q})$, but in general
we will see that $\Inn(\calW_{p,q})$ is a proper subgroup of $\Aut_C(\calW_{p,q})$ (but with very low index, except in a very special case).

\index{Pseudo-adjunction}
\label{page:pseudoadjunction}
A classical example of basic $C$-automorphism is the \textbf{pseudo-adjunction},
which is the involutory automorphism $\Phi_\star$ that swaps $a$ and $a^\star$ and swaps $b$ and $b^\star$
(it should not be confused with the adjunction, which is an antiautomorphism, although it coincides with the adjunction on basic vectors). Clearly $\Phi_\star(\omega)=a^\star b+b^\star a=\omega$, yet it will be seen that $\Phi_\star$ is not inner unless it equals the identity
(this amounts to having $a^\star=a$ and $b^\star=b$, which happens in the rare situation where $\car(\F)=2$ and $\tr(p)=\tr(q)=0$).

Our study is naturally split into two main parts. In the first part, which constitutes the remainder of the present section and is the shorter one, we
analyze the gap between $\Aut_C(\calW_{p,q})$ and $\Aut(\calW_{p,q})$, and we prove that it is filled in some sense by basic automorphisms.
The second part, which is dealt with in Section \ref{section:automorphismsII}, analyzes the gap between $\Inn(\calW_{p,q})$
and $\Aut_C(\calW_{p,q})$. The most difficult part is the second one, and the two parts involve largely different tools
(with the notable exception of the use of the maximal ideals above the fundamental ideal, which are used in the first part to deal with some very special situations
for fields with characteristic $2$, and are ubiquitous in the second part).

Since the existence part of the Automorphisms Theorem has already been partly proved in Section
\ref{section:units1} when both $p$ and $q$ are irreducible, we could have decided to leave this case aside.
However, it will turn out that, with the sole exception of two tedious special cases in the analysis of the gap between $\Aut_C(\calW_{p,q})$ and
$\Aut(\calW_{p,q})$, all connected with the characteristic $2$ situation, very limited relief is obtained by using Theorem
\ref{theo:decompositionautomorphismirreducible}. Instead, we have chosen to keep the treatment general throughout, which makes sense because we will use techniques that are completely different from the ones featured in Section \ref{section:units1}. Hence the following two sections are completely independent of Section \ref{section:units1}.

\subsection{Additional considerations on the basic automorphisms}

If we have respective automorphisms $f$ and $g$ of $\F[a]$ and $\F[b]$, then $p(f(a))=f(p(a))=0$ and $q(g(b))=g(q(b))=0$,
so we can extend them to a whole endomorphism $\Phi_{f,g}$ of $\F$-algebra, and since it is clear that
$\Phi_{f,g} \circ \Phi_{f^{-1},g^{-1}}=\id= \Phi_{f^{-1},g^{-1}} \circ \Phi_{f,g}$ on generators, $\Phi_{f,g}$ is an $\F$-automorphism.
Hence $\BAut_+(\calW_{p,q})$ is naturally isomorphic to $\Aut_\F(\F[a]) \times \Aut_\F(\F[b])$.

Assume now that $\F[a]$ and $\F[b]$ are isomorphic. Then, for every pair $(h,j)$ consisting of an isomorphism $h$ from $\F[a]$ to $\F[b]$
and an isomorphism $j$ from $\F[b]$ to $\F[a]$, there is a unique automorphism $\Psi_{h,j}$ of $\calW_{p,q}$ that takes $a$ to $h(a)$
and $b$ to $j(b)$ (its inverse is $\Psi_{j^{-1},h^{-1}}$).

As a consequence, if the algebras $\F[a]$ and $\F[b]$ are isomorphic, then
$\BAut_+(\calW_{p,q})$ has index $2$ in $\BAut(\calW_{p,q})$, and it is easy to find a direct factor of it (take $j=h^{-1}$ in the preceding construction).

As a consequence, and unless one of $p$ and $q$ splits with a double root, $\BAut(\calW_{p,q})$ has order $1,2,4$ or $8$ and it is then an easy exercise to detect the isomorphism class of this group depending on the respective choices of $p$ and $q$. For example, this group is trivial exactly when both $p$ and $q$ are inseparable and the algebra $\F[a]$ and $\F[b]$ are nonisomorphic, and it has cardinality $8$ if both $p$ and $q$ have simple roots (but not necessarily split) and the algebras $\F[a]$ and $\F[b]$ are isomorphic (e.g., in the idempotent case where $p=q=t^2-t$). In the latter case it is easy to recognize that $\BAut(\calW_{p,q}) \simeq D_4$ by noting that $\BAut(\calW_{p,q})$ is naturally embedded as a subgroup of the permutation group $\mathfrak{S}(\{a,b,a^\star,b^\star\})$ and by remembering that the only groups of order $8$ in the latter, which are the $2$-Sylow subgroups, are known to be isomorphic to $D_4$.

Now, at this point we could immediately inquire about the intersection $\BAut(\calW_{p,q}) \cap \Inn(\calW_{p,q})$,
yet although it can be done by entirely elementary means, we believe it is premature to do so. When the time comes, we will have all the necessary tools to obtain this with as little computation as possible, and the result will be easy to understand.
At this point though, the fruits are not ripe yet.

\subsection{The first invariant: the action on the center}\label{section:actiononcenter}

Let $\Phi$ be an automorphism of $\calW_{p,q}$. Then $\Phi$ induces an $\F$-automorphism of the center
$\F[\omega]$, which then reads $r(\omega) \mapsto r(\lambda \omega+\mu)$ for a unique pair $(\lambda,\mu)\in \F^\times \times \F$.
Note that $\Phi(\omega)=\omega$ if and only if $\lambda=1$ and $\mu=0$.
As opposed to what one might think at first glance, not every pair $(\lambda,\mu)$ is possible, and a key invariance will now be observed.

\begin{prop}
Let $\Phi$ be an automorphism of $\calW_{p,q}$. Then there is a nonzero scalar $\delta \in \F^\times$ such that
$$\Phi(\Lambda_{p,q}(\omega))=\delta \Lambda_{p,q}(\omega).$$
\end{prop}

\begin{proof}
We use the relationship between the fundamental polynomial $\Lambda_{p,q}$ and the determinant of the inner product $\langle -,-\rangle$.
First of all, remember from Proposition \ref{prop:commuteadjunction} that $\Phi$ commutes with the adjunction, and hence
$$\forall (x,y) \in (\calW_{p,q})^2, \; \Phi(\langle x,y\rangle)=\langle \Phi(x),\Phi(y)\rangle.$$
It follows that the Gram matrix of the quadruple $(\Phi(1),\Phi(a),\Phi(b),\Phi(ab))$ with respect to $\langle -,-\rangle$
is obtained by applying $\Phi$ entrywise to the Gram matrix of $(1,a,b,ab)$ (with respect to $\langle -,-\rangle$), and hence
the Gram determinant of $(\Phi(1),\Phi(a),\Phi(b),\Phi(ab))$ is the image under $\Phi$ of the one
in $(1,a,b,ab)$, i.e., of $\Lambda_{p,q}(\omega)^2$.

Because $\Phi$ is an automorphism of $\F$-algebra that leaves $C$ (globally) invariant, the family $(\Phi(1),\Phi(a),\Phi(b),\Phi(ab))$
is a basis of the $C$-module $\calW_{p,q}$, and hence the two corresponding Gram determinants
are equal up to multiplication with the square of an element of $C^\times=\F^\times$. This yields a nonzero scalar $\delta \in \F^\times$ such that
$$\Phi(\Lambda_{p,q}(\omega)^2)=\delta^2 \Lambda_{p,q}(\omega)^2,$$
and we conclude that $\Phi(\Lambda_{p,q}(\omega))=\pm \delta \Lambda_{p,q}(\omega)$ because $C$ has no zero divisor.
\end{proof}

Hence, applying the previous result to $\Phi$, we recover that, for some $\delta \in \F^\times$,
$$\Lambda_{p,q}(\lambda \omega+\mu)=\delta \Lambda_{p,q}(\omega).$$
As in the previous sections, denote by $\K$ the splitting field of $pq$.
It ensues from the above that the affine mapping
$$\calA(\Phi) : z \in \K \mapsto \lambda z+\mu \in \K$$
permutes the roots of $\Lambda_{p,q}$ in $\K$. Note in any case that the trace of $\Lambda_{p,q}$ is $(\tr p)(\tr q)$,
and that an affine mapping on a line is determined by its values at two different points.
For a subset $X$ of $\K$, we will denote by $\GA(\F,X)$ the group of all affine automorphisms of the line $\F$ whose
(algebraic) extension to $\K$ permutes the set $X$.
Hence, in denoting by $\Root(\Lambda_{p,q})$ the set of all roots of $\Lambda_{p,q}$ in $\K$,
we recover a group homomorphism
$$\Theta : \Aut(\calW_{p,q}) \longrightarrow \GA(\F,\Root(\Lambda_{p,q})),$$
whose kernel is $\Aut_C(\calW_{p,q})$.

Let us briefly discuss the precise nature of $\GA(\F,\Root(\Lambda_{p,q}))$ according to $(p,q)$.
\begin{itemize}
\item If $\Lambda_{p,q}$ has two distinct roots in $\K$, then
$\GA(\F,\Root(\Lambda_{p,q}))$ has exactly two elements: the identity and $z \mapsto (\tr p)(\tr q)-z$.
Note that if $\car(\F)=2$ these are really distinct mappings otherwise $\Lambda_{p,q}$ would have a double root.
\item If $\Lambda_{p,q}$ has a double root $z_0$ in $\K \setminus \F$, then $\GA(\F,\Root(\Lambda_{p,q}))$
contains only the identity mapping. Indeed, given $(\lambda,\mu) \in \F^\times \times \F$, the equality
$\lambda z_0+\mu=z_0$ leads to  $\lambda=1$ and $\mu=0$.
\item Finally, if $\Lambda_{p,q}$ has a double root $z_0$ in $\F$, then $\GA(\F,\Root(\Lambda_{p,q}))$ is the group of all homotheties $z \mapsto \lambda(z-z_0)+z_0$ (with $\lambda \in \F^\times$) around the point $z_0$, and it is isomorphic to $\F^\times$.
\end{itemize}

Remember finally from Section \ref{section:Lambdapq} that $\Lambda_{p,q}$ has a double root if and only if at least one of $p$ and $q$
has a double root in $\K$, in which case the root of $\Lambda_{p,q}$ belongs to $\F$ unless one of $p$ and $q$ splits with simple roots and the other one is inseparable.
Let us conclude:

\begin{lemma}
Let $\Phi$ be an automorphism of the $\F$-algebra $\calW_{p,q}$.
Then $\Phi(\omega) \in \{\omega,(\tr p)(\tr q)-\omega\}$ with the possible exception of the following special cases:
\begin{itemize}
\item At least one of $p$ and $q$ has a double root in $\F$, in which case $\Lambda_{p,q}$ has a double root $z$ in $\F$,
and $\Phi(\omega)=\lambda(\omega-z)+z$ for some $\lambda \in \F^\times$;
\item Both $p$ and $q$ are irreducible and inseparable.
\end{itemize}
\end{lemma}

Here is our next step:

\begin{prop}\label{prop:analysisautomorphismcenter}
Unless both $p$ and $q$ are irreducible and inseparable, the restriction
$\BAut(\calW_{p,q}) \rightarrow \GA(\F,\Root(\Lambda_{p,q}))$
of $\Theta$ to the subgroup of basic automorphisms is surjective.
\end{prop}

\begin{proof}
To start with, remember that the basic automorphism $\Phi$ that fixes $a$ and takes $b$ to $b^\star$ satisfies
$$\Phi(\omega)=ab+b^\star a^\star=(\tr p)(\tr q)-\omega.$$
This is enough to conclude unless the first special case cited in the previous lemma is encountered.

So, assume now that $p$ has a double root $x$ in $\F$. Then $z:=\tr(q)\,x$ is the double root of $\Lambda_{p,q}$ in $\F$.
Let $\lambda \in \F^\times$, and consider the basic automorphism that takes
$a-x$ to $\lambda (a-x)$ and leaves $b$ invariant.
Then
$$\langle \Phi(a-x),\Phi(b)\rangle=\langle \lambda (a-x),b\rangle=\lambda (\omega-x\tr(q))=\lambda (\omega-z)$$
and hence
$$\Phi(\omega)=\langle \Phi(a-x),\Phi(b)\rangle+\langle \Phi(x),\Phi(b)\rangle=\lambda (\omega-z)+\langle x,b\rangle=\lambda (\omega-z)+z.$$
We proceed likewise if $q$ has a double root in $\F$.
\end{proof}

\subsection{Analyzing the remaining special cases}\label{section:centerspecialcases}

Here, we complete the previous study by considering the special case where
$\car(\F)=2$ and both $p$ and $q$ are inseparable (and hence irreducible).
Then $\Lambda_{p,q}(\omega)=\omega^2$ and there are at most two basic automorphisms of $\calW_{p,q}$.
Yet, the group $\GA(\F,\Root(\Lambda_{p,q}))$ is isomorphic to $\F^\times$, which has no element of order $2$ because $\car(\F)=2$.
Hence in that case the mapping $\Theta$ is trivial (i.e., constant) on $\BAut(\calW_{p,q})$.
Thus $\Theta$ is useless here.

Fortunately we can prove the following:

\begin{prop}\label{prop:analysisautomorphismcenterinseparable}
If both $p$ and $q$ are inseparable then every automorphism of $\calW_{p,q}$ is a $C$-automorphism.
\end{prop}

At this very point, we could get rid of this difficulty by relying on Theorem \ref{theo:decompositionautomorphismirreducible},
but we will actually take the challenge and give a direct proof of
Proposition \ref{prop:analysisautomorphismcenterinseparable}, at least one that does not use a technique that is similar to the one featured in Section \ref{section:units1}. The reader is invited to skip this part at first reading and move directly to the next section.

\begin{proof}
Assume throughout the proof that both $p$ and $q$ are inseparable, in which case $\car(\F)=2$
and $\Lambda_{p,q}=t^2$. We set $r:=t$ throughout, and we take an arbitrary automorphism $\Phi$ of $\calW_{p,q}$.
It will suffice to prove that $\Phi(\omega)=\omega$.
First of all, we deduce from Section \ref{section:actiononcenter} that $\Phi(\omega)=\lambda \omega$
for some $\lambda \in \F^\times$. So, we need to prove that $\lambda=1$.

There are two subcases that must be dealt with separately: the one where $p$ and $q$ have distinct splitting fields in $\K$
(i.e., $[\K:\F]=4$), and the one where they have the same splitting field in $\K$ (i.e., $[\K:\F]=2$). The latter is the more difficult,
so we will start with the former. In each case the idea is to compute $\lambda$ by analyzing the
$\F$-linear mappings induced by $\Phi$ on the quotient $\F$-vector spaces $\calW_{p,q}/\mathfrak{J}_r$ and
$\mathfrak{J}_r/(\omega)$. Remember to this end that since both $p$ and $q$ are irreducible
$\mathfrak{J}_r$ is the sole maximal ideal of $\calW_{p,q}$ that includes $(\omega)$
(see Proposition \ref{prop:structureofJr}),
and as $\Phi$ leaves $(\omega)$ invariant (because $\Phi(\omega)=\lambda \omega$) it must also leave $\mathfrak{J}_r$ invariant.

Note also that $\mathfrak{J}_r$ is the subset of all $x \in \calW_{p,q}$ such that $N(x) \equiv 0 \; (\omega)$ (see the last case in Proposition \ref{prop:structureofRr}). Finally
$$\forall (x,y)\in \calW_{p,q}^2, \quad \langle x,y\rangle \equiv 0 \quad \text{mod} \; (\omega),$$
again by the last case in Proposition \ref{prop:structureofRr}. The latter identity will be used repeatedly.

As a consequence, the norm $N$ induces a non-isotropic and totally degenerate (i.e., with polar form zero) quadratic form $\overline{N}$ on the $\F$-vector space $\calW_{p,q}/\mathfrak{J}_r$. In particular $\overline{N}$ is an injective homomorphism from the additive group $(\calW_{p,q}/\mathfrak{J}_r,+)$
to $(\F,+)$. Next, observe that $\Phi$ induces an $\F$-linear endomorphism $\overline{\Phi}$ of $\calW_{p,q}/\mathfrak{J}_r$.
Letting $x \in \calW_{p,q}$, we write $N(x) \equiv \alpha \quad \text{mod} \; (\omega)$ for some $\alpha \in \F$, and we use the invariance of $(\omega)$ under $\Phi$
to deduce that $\Phi(N(x)) \equiv \Phi(\alpha) \quad \text{mod} \; (\omega)$, which yields $N(\Phi(x)) \equiv \alpha \quad \text{mod} \; (\omega)$.
This shows that $\overline{N}(\overline{\Phi}(\overline{x}))=\overline{N}(\overline{x})$ for all $\overline{x} \in \calW_{p,q}/\mathfrak{J}_r$.
Since $\overline{N}$ is injective, we deduce that $\overline{\Phi}$ is the identity, i.e.,
\begin{equation}\label{eq:identitymod}
\forall x \in \calW_{p,q}, \quad \Phi(x) \equiv x \quad \text{mod} \; \mathfrak{J}_r.
\end{equation}
Now, we need to split the discussion into two subcases.

\vskip 3mm
\noindent \textbf{Case 1:} $[\K:\F]=4$. \\
Here $\mathfrak{J}_r=(\omega)$. By \eqref{eq:identitymod} we have $\Phi(a)=a+\omega a'$ and $\Phi(b)=b+\omega b'$ for some $a',b'$ in $\calW_{p,q}$.
Then
$$\Phi(\omega)=\Phi(\langle a,b\rangle)=\langle \Phi(a),\Phi(b)\rangle=\langle a,b\rangle+\omega \langle a',\Phi(b)\rangle+\omega \langle a,b'\rangle$$
and hence $\Phi(\omega) \equiv \omega \; (\omega^2)$. Hence $\lambda=1$ and we conclude that $\Phi(\omega)=\omega$.

\vskip 3mm
\noindent \textbf{Case 2:} $[\K:\F]=2$. \\
Here the $\F$-algebras $\F[a]$ and $\F[b]$ are isomorphic. Without loss of generality we can then assume that $p=q=t^2+\delta$ for some $\delta \in \F$ which is not a square in $\F$. Then we see that $N(a+b)=2\delta+\langle a,b\rangle=\omega$ and $N(\delta+ab)=\delta^2+\delta^2+\delta \langle 1,ab\rangle=\delta \omega$ because $b^\star=b$.
Next, the respective cosets $\overline{a}$ and $\overline{b}$ in $\calW_{p,q}/\mathfrak{J}_r$
are such that $(1,\overline{a},\overline{a}+\overline{b},\delta+\overline{a}\overline{b})$
is a basis of the $\F$-vector space $\calW_{p,q}/(\omega)$, and from the previous computation we observe that $(\overline{a}+\overline{b},\delta+\overline{a}\overline{b})$ is a basis of $\mathfrak{J}_r/(\omega)=\mathfrak{R}_r$.

We set $V:=\Vect_\F(a+b,\delta+ab)$ and we note that $N$ is totally degenerate on $V$: Indeed,
$\langle a+b,\delta+ab\rangle=\delta(\tr(a)+\tr(b))+N(a) \tr (b)+N(b)\tr(a)=0$ since $\tr(a)=\tr(b)=0$.
Since $\Phi$ leaves $\mathfrak{J}_r$ invariant, we can write
$$\Phi(\delta+ab)\equiv \mu (a+b)+\nu (\delta+ab) \quad \text{ mod $(\omega)$}$$
for some $(\mu,\nu)\in \F^2$.
We also use \eqref{eq:identitymod} to obtain $\Phi(a)=a+\omega a'$ for some $a' \in \calW_{p,q}$.
Now, from $\langle a,\delta+ab\rangle=0$, we deduce that $\langle \Phi(a),\Phi(\delta+ab)\rangle=0$,
and since $\langle a',\mu(a+b)+\nu(\delta+ab)\rangle \equiv 0$ mod $(\omega)$ this yields
\begin{equation}\label{eq:tordu}
\langle a,\mu(a+b)+\nu(\delta+ab)\rangle \equiv 0 \quad \text{mod} \; (\omega^2).
\end{equation}
As the left-hand side in \eqref{eq:tordu} equals $\mu \delta+\mu \omega$, we deduce that $\mu=0$.

Now we remember that $\Phi(\omega)=\lambda \omega$.
Noting that $\langle 1,\delta+ab\rangle=\omega$, we deduce that
$\langle \Phi(1),\Phi(\delta+ab)\rangle=\lambda \omega$ and hence $\lambda=\nu$.
Finally,
$$\Phi(N(\delta+ab))=N(\Phi(\delta+ab)) \equiv N(\lambda(\delta+ab)) \quad \text{mod} \; (\omega^2),$$
and hence $\Phi(\delta \omega) \equiv \lambda^2 \delta \omega$  mod $(\omega^2)$.
This yields $\lambda=\lambda^2$, whence $\lambda=1$ and we conclude that $\Phi(\omega)=\omega$.
\end{proof}

\subsection{A milestone in the proof of the Automorphisms Theorem}

A straightforward consequence of Propositions \ref{prop:analysisautomorphismcenter} and \ref{prop:analysisautomorphismcenterinseparable}
is the following milestone for our proof of the Automorphisms Theorem:

\begin{cor}\label{cor:analysisautomorphismcenter}
For every $\Phi \in \Aut(\calW_{p,q})$, there exist $\Phi_1 \in \Aut_C(\calW_{p,q})$
and $\Phi_2 \in \BAut(\calW_{p,q})$ such that $\Phi=\Phi_1 \circ \Phi_2$.
\end{cor}

Clearly, this entirely reduces the Automorphisms Theorem to the corresponding statement for $C$-automorphisms.

\section{Automorphisms of the free Hamilton algebra (part 2): Analyzing the $C$-automorphisms}\label{section:automorphismsII}

Recall that $C$ denotes the center of $\calW_{p,q}$.
This section focuses on the group $\Aut_C(\calW_{p,q})$ consisting of the automorphisms of the $C$-algebra
$\calW_{p,q}$, also called the $C$-automorphisms.
We observe that this group contains all the inner automorphisms of $\calW_{p,q}$, but not all the basic ones in general.
To this end we denote by $\BAut_C(\calW_{p,q})$ the subgroup of all $C$-automorphisms that are also basic.

Keeping in mind that our ultimate aim is to prove the Automorphisms Theorem (Theorem \ref{theo:automorphismstheointro}), we observe thanks to Corollary
\ref{cor:analysisautomorphismcenter} that this theorem is now reduced to the following statement:

\begin{theo}\label{theo:automorphismstheoC}
Every $C$-automorphism of $\calW_{p,q}$ splits uniquely as the composite of a basic $C$-automorphism followed by an inner automorphism.
\end{theo}

Our main starting point to study the $C$-automorphisms is the observation that the extension $\overline{\calW_{p,q}}$ is a quaternion algebra
over the field of fractions of $C$, along with the Skolem-Noether theorem. Remember that the latter states that every automorphism of a central simple algebra over a field is an inner automorphism (see e.g.\ \cite{Pierce} \S 12.6), and in the special case of $\overline{\calW_{p,q}}$ this yields that every automorphism
of the $\F(\omega)$-algebra $\overline{\calW_{p,q}}$ is inner. A major difficulty however is that it will not always be true that
such an automorphism that leaves $\calW_{p,q}$ invariant is the conjugation by a unit of $\calW_{p,q}$.

Before we start studying the problem in depth, it is useful that we quickly review the basic $C$-automorphisms.
Of course, throughout we keep the notation $\K$ for a fixed splitting field of $pq$.

Some simple observations before we truly start. Let $\Phi$ be a $C$-automorphism of $\calW_{p,q}$.
Since $\Phi$ is in particular an $\F$-automorphism, we have seen in the previous section that
it satisfies
$$\forall x \in \calW_{p,q}, \; N(\Phi(x))=\Phi(N(x)) \quad \text{and} \quad \forall (x,y)\in \calW_{p,q}^2, \; \langle \Phi(x),\Phi(y)\rangle=\Phi(\langle x,y\rangle).$$
Hence, because it is a $C$-automorphism, $\Phi$ satisfies the identities
$$\forall x \in \calW_{p,q}, \; N(\Phi(x))=N(x) \quad \text{and} \quad \forall (x,y)\in \calW_{p,q}^2, \; \langle \Phi(x),\Phi(y)\rangle=\langle x,y\rangle,$$
and we deduce from the latter identity that $\Phi$ also preserves the trace.

\subsection{The basic $C$-automorphisms}\label{section:basicCauto}

Leaving aside the identity, we identify three main types of basic automorphisms that fix the elements of $C$.

\begin{enumerate}[(i)]
\item The \textbf{pseudo-adjunction}, denoted by $\Phi_\star$, is the automorphism of $\calW_{p,q}$ that takes $a$ to $a^\star$ and $b$ to $b^\star$ (and hence also $a^\star$ to $a$ and $b^\star$ to $b$). It should not be confused with the adjunction, as it is an automorphism
rather than an antiautomorphism. The pseudo-adjunction is the identity only if $\car(\F)=2$ and $p$ and $q$ have a double root in $\K$.
In any case, it coincides with the adjunction on basic vectors.

\index{Swap}
\item The \textbf{swaps} are defined when the algebras $\F[a]$ and $\F[b]$ are isomorphic.
They are the \emph{involutory} basic automorphisms that exchanges $\F[a]$ and $\F[b]$.
Let us immediately check that a swap is a $C$-automorphism. Take a swap $\Phi$,
and set $\beta:=\Phi(a)$. Then $\Phi$ exchanges $a$ and $\beta$
therefore $\Phi(\langle a,\beta\rangle)=\langle \Phi(a),\Phi(\beta)\rangle=\langle \beta,a\rangle=\langle a,\beta\rangle$,
and since $\langle a,\beta\rangle$ generates the $\F$-algebra $C$ this proves that $\Phi$ is a $C$-automorphism.

If say $p=q$ and $p$ does not have a double root in $\F$, the swap automorphisms are the automorphism that exchanges $a$ and $b$ (and therefore also $a^\star$ and $b^\star$), and the automorphism that exchanges $a$ and $b^\star$ (and therefore also $a^\star$ and $b$).
Moreover, these two automorphisms are different unless $p$ and $q$ are inseparable.
If both $p$ and $q$ have a double root, the set of all swaps has the cardinality of $\F^\times$.

\index{Hyperbolic automorphism}
\item The \textbf{hyperbolic automorphisms} are defined when both algebras $\F[a]$ and $\F[b]$ are degenerate. \\
They are the basic automorphisms that leave $\F[a]$ and $\F[b]$ invariant and fix the $\omega$ element.
Now, say that $a^2=b^2=0$ to simplify things.
In that case, it is clear that the automorphisms that leave
$\F[a]$ and $\F[b]$ invariant are the ones $\Phi$ for which there exists a pair $(\lambda,\mu) \in (\F^\times)^2$ such that
$\Phi(a)=\lambda a$ and $\Phi(b)=\mu b$ (of course for each such pair $(\lambda,\mu)$ there is a corresponding automorphism),
and we note that $\Phi(\omega)=\lambda \mu \omega$, so $\Phi$ is a $C$-automorphism if and only if $\lambda\mu=1$.
In that situation, given $\lambda \in \F^\times$ we will denote by $H_\lambda$ the automorphism such that
$$H_\lambda(a)=\lambda a \quad \text{and} \quad H_\lambda(b)=\lambda^{-1} b.$$
Obviously $\{H_\lambda \mid \lambda \in \F^\times\}$ is a subgroup of $\BAut_C(\calW_{p,q})$ and it is isomorphic to $\F^\times$.
\end{enumerate}

Now, we must of course prove that these are the \emph{only} basic $C$-automorphisms besides the identity, which is not entirely obvious.

\begin{prop}
The basic $C$-automorphisms of $\calW_{p,q}$ are the identity, the pseudo-adjunction, the swaps and the hyperbolic automorphisms.
\end{prop}

\begin{proof}
Let $\Phi$ be a basic $C$-automorphism.

Assume first that $\Phi$ is positive, and start with the case where $\Phi(a)=a$. Then
$$\langle a,b\rangle=\langle \Phi(a),\Phi(b)\rangle=\langle a,\Phi(b)\rangle$$
and hence $\langle a,b-\Phi(b)\rangle=0$. Since $b-\Phi(b) \in \F[b]$
we must have $\Phi(b)-b=\lambda$ for some $\lambda \in \F$, and if $\lambda \neq 0$ then $\tr(a)=0$.
Assume now that $\lambda \neq 0$. Then analyzing $\tr(\Phi(b))=\tr(b)$ and $N(\Phi(b))=N(b)$
leads to $2\lambda=0$ and $\lambda \tr(b)+\lambda^2=0$, and hence $\car(\F)=2$ and $\lambda=\tr(b)$.
Then $\Phi(b)=b^\star$ and $\Phi(a)=a=a^\star$ because $\tr(a)=0$, and we conclude that $\Phi=\Phi_\star$.
Hence, we have proved that if $\Phi$ fixes at least one of $a$ and $b$, then it is the identity or the pseudo-adjunction.

Assume now that $\Phi(a) \neq a$ and $\Phi(b) \neq b$, while keeping the starting assumption that $\Phi$ is positive.
Then both $p$ and $q$ are separable. If $p$ has simple roots in $\K$, then $\Phi(a)=a^\star$, hence $\Phi_\star \circ \Phi$ fixes $a$; by the previous case
$\Phi_\star \circ \Phi \in \{\id,\Phi_\star\}$, and we conclude that $\Phi \in \{\Phi_\star,\id\}$.
Symmetrically, the same conclusion holds if $q$ has simple roots in $\K$.
Hence, the only remaining case is when both $p$ and $q$ have a double root in $\F$, and by the previous study this
implies that $\Phi$ is a hyperbolic automorphism.

Assume now that $\Phi$ is negative.
Then we must prove that $\Phi^2=\id$.
Consider the swap $\Psi$ that coincides with $\Phi$ on $\F[a]$. Then
$\Psi \circ \Phi$ leaves $\F[a]$ and $\F[b]$ invariant, and fixes $a$.
By the previous study $\Psi \circ \Phi$ is the identity or the pseudo-adjunction.
In the second case, we must have $a^\star=a$, so $p$ has a double root in $\K$ and hence $b^\star=b$ because $\F[a]$ and $\F[b]$ are isomorphic.
Hence $\Psi \circ \Phi=\id$ in any case, and we deduce that $\Phi=\Psi^{-1}=\Psi$, whence $\Phi$ is a swap.
\end{proof}

\subsection{The conjugators of a $C$-automorphism}\label{section:conjugatorsintroduction}

Now, we take an arbitrary $C$-automorphism $\Phi$ and start analyzing it.

By tensoring, we extend $\Phi$ to an $\F(\omega)$-algebra automorphism
$$\overline{\Phi} : \overline{\calW_{p,q}} \overset{\simeq}{\longrightarrow} \overline{\calW_{p,q}}.$$
Since $\overline{\calW_{p,q}}$ is a quaternion algebra over $\F(\omega)$, and in particular
a central simple algebra, the Skolem-Noether theorem (\cite{Pierce} \S 12.6)
yields that
$$\overline{\Phi} : x \mapsto \gamma x \gamma^{-1}$$
for some invertible element $\gamma \in \overline{\calW_{p,q}}^\times$, which we call a \textbf{conjugator} of $\Phi$.
\index{Conjugator (of a $C$-automorphism)}

The problem now is that $\gamma$ does not necessarily belong to $\calW_{p,q}$, and even if it does, in which case $N(\gamma) \neq 0$,
it might not be an invertible element of $\calW_{p,q}$ (i.e., $N(\gamma) \in \F^\times$).
To get closer to the solution, we introduce a ``right" choice of conjugator.
First of all, because $\overline{\calW_{p,q}}$ is a central $\F(\omega)$-algebra, the set of all conjugators of $\Phi$ is $\F(\omega)^\times \gamma$.
Then $\F(\omega) \gamma \cap \calW_{p,q}$ is a $C$-submodule of rank $1$ of $\calW_{p,q}$
and it is spanned as such by a \emph{normalized} element $\gamma_1$, to the effect that the conjugators of $\Phi$
that belong to $\calW_{p,q}$ are the nonzero elements of $\F[\omega] \gamma_1$.
We say that $\gamma_1$ is a \textbf{normalized conjugator} of $\Phi$, and we note that it is uniquely determined by $\Phi$

up to multiplication with an element of $\F^\times$.

From now on we will systematically take $\gamma$ as a normalized conjugator of $\Phi$.
Note that $\Phi$ is an inner automorphism if and only if $\gamma$ is a unit in $\calW_{p,q}$,
i.e., $N(\gamma) \in \F^\times$. Obviously, the key will lie in the analysis of the norm of $\gamma$.

We must now warn the reader of two difficulties. First of all, there is no easy converse statement here:
if we start from a normalized vector $\gamma \in \calW_{p,q}$ such that $N(\gamma) \neq 0$,
the mapping $x \mapsto \gamma x \gamma^{-1}$ is in general \emph{not} an automorphism of $\calW_{p,q}$,
because we cannot guarantee that it leaves $\calW_{p,q}$ invariant.
Even it did leave $\calW_{p,q}$ invariant, one could doubt that it is really an automorphism of $\calW_{p,q}$,
as it could fail to be surjective, but actually it is not difficult to prove that it is always surjective if it maps
$\calW_{p,q}$ into itself (hint: use the Gram matrices of the inner product). We will however focus
solely on the invariance of $\calW_{p,q}$.

Another potential source of misunderstanding lies in the problem of composition of $C$-automorphisms.
In composing two $C$-automorphisms $\Phi_1$ and $\Phi_2$, with associated normalized conjugators $\gamma_1$ and $\gamma_2$,
the product $\gamma_1\gamma_2$ is a conjugator of $\Phi_1 \Phi_2$, but not a normalized one in general.
Here is a simple example: assuming that there exists a $C$-automorphism $\Phi$ that is not inner, then
$\Phi^{-1}$ is not inner either; take normalized conjugators $\gamma$ and $\gamma'$ associated with $\Phi$ and $\Phi^{-1}$. Then
$\gamma \gamma'$ is a conjugator of $\Phi \circ \Phi^{-1}=\id$,
but it cannot be normalized otherwise it would belong to $\F^\times$, yet $N(\gamma \gamma')=N(\gamma)N(\gamma')$
has positive degree in $C$.

In general the following result holds:

\begin{lemma}\label{lemma:conjugatormultiplicativity}
Let $\Phi_1$ and $\Phi_2$ be $C$-automorphisms of $\calW_{p,q}$, and let $\gamma_1$ and $\gamma_2$ be associated
normalized conjugators. Let $\gamma$ be a normalized commutator of $\Phi_1 \circ \Phi_2$.
Then $N(\gamma_1)N(\gamma_2)=s^2 N(\gamma)$ for some $s \in C \setminus \{0\}$.
\end{lemma}

\begin{proof}
Indeed, as $\gamma_1\gamma_2$ is a conjugator of $\Phi_1 \circ \Phi_2$, we have
$\gamma_1\gamma_2=s \gamma$ for some $s \in C \setminus \{0\}$.
\end{proof}

The possibility that $\gamma_1 \gamma_2$ becomes non-normalized, which we call the \textbf{collapsing phenomenon}, will be
given more scrutiny in the later stages of our proof, and will play a critical part.
\index{Collapsing phenomenon}

Let us come back to $\Phi \in \Aut_C(\calW_{p,q})$, and let us take an associated normalized conjugator $\gamma$.
The key is to use the fact that the automorphism $x \mapsto \gamma x \gamma^{-1}$ of $\overline{\calW_{p,q}}$
leaves $\calW_{p,q}$ invariant. Simply, we rewrite the fact that $\gamma$ is a conjugator of $\Phi$ as
\begin{equation}\label{eq:keyidstar}
\forall x \in \calW_{p,q}, \; \gamma x \gamma^\star =N(\gamma)\, \Phi(x)
\end{equation}
and, forgetting $\Phi$ itself and focusing entirely on the conjugator $\gamma$, we will systematically interpret \eqref{eq:keyidstar}
as follows:
\begin{equation}\label{eq:keyidstarmodular}
\forall x \in \calW_{p,q}, \; \gamma x \gamma^\star \equiv 0 \quad \text{mod} \; (N(\gamma)).
\end{equation}
We immediately state an important consequence of this:

\begin{lemma}\label{lemma:idealandconjugator}
Let $I$ be a nontrivial (two-sided) ideal of $\calW_{p,q}$ that contains $N(\gamma)$ and is invariant under the adjunction,
and for $x \in \calW_{p,q}$ denote by $x_I$ its coset in the quotient ring
$\calW_{p,q}/I$. Then for the (two-sided) ideal $(\gamma_I)$ the following identity holds:
$$\forall (x,y) \in (\gamma_I)^2, \; \forall z \in \calW_{p,q}/I, \; x zy^\star=0.$$
In particular $(\gamma_I)$ is a proper ideal of $\calW_{p,q}/I$.
\end{lemma}

\begin{proof}
Thanks to \eqref{eq:keyidstarmodular} we have $\gamma_I z (\gamma_I)^\star=0$ for all $z \in \calW_{p,q}/I$.
Let $z \in \calW_{p,q}/I$.
Let $x_1,x_2,y_1,y_2$ in $\calW_{p,q}/I$. Then
$$x_1 \gamma_I x_2 z (y_1 \gamma_I y_2)^\star=x_1 \underbrace{\gamma_I (x_2 z y_2^\star) \gamma_I^\star}_{0} y_1^\star=0.$$
From there, the first claimed statement is readily deduced.

If $(\gamma_I)=\calW_{p,q}/I$ then taking $x=y=z=1$ in the identity we have just proved yields $1=0$ in $\calW_{p,q}/I$,
which is absurd.
\end{proof}

Here is a key application, which is the next milestone in our proof.

\begin{lemma}\label{lemma:conjugatorandquaternions}
Let $r \in \Irr(\F)$ be such that $r(\omega)$ divides $N(\gamma)$ in $C$.
Then $r$ divides~$\Lambda_{p,q}$.
\end{lemma}

\begin{proof}
Assume on the contrary that $r$ is relatively prime with $\Lambda_{p,q}$, and consider the
maximal ideal $I=(r(\omega))$. By Lemma \ref{lemma:idealandconjugator}, $(\gamma_I)$ is a proper ideal of $\calW_{p,q}/I$.
Moreover it is nonzero, otherwise $\gamma_I$ would not be normalized (all its coefficients in the deployed basis $(1,a,b,ab)$
would be multiples of $r(\omega)$).
Yet we know from Section \ref{section:quaternionalgebras} that $\calW_{p,q}/I$ is a quaternion algebra over the residue field $\F[\omega]/(r(\omega))$,
and as a consequence it is simple. This contradicts the previous observation that $(\gamma_I)$ is a nontrivial ideal of it.
\end{proof}

\subsection{Exponents and signatures}

Let us continue our analysis of an arbitrary automorphism $\Phi \in \Aut_C(\calW_{p,q})$ and of
an associated normalized conjugator $\gamma$.

\begin{Not}
We denote by $\Irr_{\Lambda_{p,q}}$
the set of all $r \in \Irr(\F)$ that divide $\Lambda_{p,q}$.
\label{page:irrdivisors}
\end{Not}

Remember that, for two elements $s_1,s_2$ of $\calW_{p,q}$, the notation $s_1 \sim s_2$ means that there exists $\lambda \in \F^\times$ such that
$s_1 =\lambda s_2$ (this defines an equivalence relation on $\calW_{p,q}$).

\begin{Def}
For $r \in \Irr_{\Lambda_{p,q}}$, we denote by
$n_r(\Phi)$ the valuation of $N(\gamma)$ with respect to the irreducible $r(\omega)$,
and call it the \textbf{$r$-exponent} of $\Phi$. Note that $n_r(\Phi)$ does not depend on the choice of the normalized conjugator $\gamma$.
\end{Def}
\label{page:rexponents}
\index{$r$-exponent}

The $r$-exponent of $\Phi$ is a nonnegative integer, and in theory it can be nonzero.
Indeed, if the contrary held for all $r \in \Irr_{\Lambda_{p,q}}$ then $\Phi$ would be an inner automorphism.

Now, Lemma \ref{lemma:conjugatorandquaternions} can be reinterpreted as saying that
$$N(\gamma) \sim \prod_{r \in \Irr_{\Lambda_{p,q}}} r(\omega)^{n_r(\Phi)}.$$
Lemma \ref{lemma:conjugatormultiplicativity} makes it relevant, for all $r \in \Irr_{\Lambda_{p,q}}$,
\index{$r$-signature}
to also introduce the \textbf{$r$-signature} of $\Phi$, defined as the parity
$\varepsilon_r(\Phi):= \overline{n_r(\Phi)} \in \Z/2$ of the exponent $n_r(\Phi)$, and
\label{page:rsignature}
it guarantees at least that
$$\varepsilon_r : \Phi \in \Aut_C(\calW_{p,q}) \mapsto \varepsilon_r(\Phi) \in \Z/2$$
is a group homomorphism.

We gather these homomorphisms to form the \textbf{full signature homomorphism}
\index{Full signature automorphism}
$$\varepsilon : \Phi \in \Aut_C(\calW_{p,q}) \rightarrow (\varepsilon_r(\Phi))_r \in \prod_{r \in \Irr_{\Lambda_{p,q}}} \Z/2,$$
\label{page:fullsignature}
whose kernel we denote by $\Aut_{C,0}(\calW_{p,q})$.
Obviously this new group is inserted in the chain of normal subgroups
$$\Inn(\calW_{p,q}) \trianglelefteq \Aut_{C,0}(\calW_{p,q}) \trianglelefteq \Aut_{C}(\calW_{p,q})
\trianglelefteq \Aut(\calW_{p,q}).$$
The following questions are then natural:
\begin{enumerate}[(i)]
\item Is the full signature homomorphism surjective?
\item Must an automorphism with full signature zero be inner?
\end{enumerate}
A reasonable bet is that both questions have positive answers, yet it turns out that
the answer depends on the specific pair $(p,q)$ under consideration, with negative answers occurring only when
at least one of $p$ and $q$ has a double root.
These questions will occupy the remainder of this study, and their full answer is deeply connected with
the Automorphisms Theorem.

We start with the first question. As we are unable to provide other examples of elements of $\Aut_C(\calW_{p,q})$
beyond the inner automorphisms, some basic automorphisms and their composites, we will compute the exponents and the full signature of all the basic $C$-automorphisms. This will yield a partial answer to the first problem, and a negative answer to the second one in some cases.
We devote the next section to this study.

\subsection{Exponents and signatures of basic $C$-automorphisms}\label{section:basicCautoconjugator}

We have reviewed the basic $C$-automorphisms in Section \ref{section:basicCauto}.
For each type, excluding the identity of course, we will now compute a corresponding normalized conjugator, and as a result we will obtain the
exponents as well as the full signature.

\subsubsection{The pseudo-adjunction}\label{section:pseudoadjunction}

Here the quaternionic structures will give the heuristics for finding a conjugator.
Say for a moment that $\car(\F) \neq 2$. Then, with a harmless basic base change, we can reduce the situation to the one where $\tr(a)=\tr(b)=0$.
The extension $\overline{\Phi_\star}$ then acts on the hyperplane of pure quaternions of
$\overline{\calW_{p,q}}$ as an element of its special orthogonal group (for the norm quadratic form) that
takes $a$ to $-a$ and $b$ to $-b$. The only known rotation of the pure quaternions that acts in this way
must fix the vectors of the orthogonal complement of $\{a,b\}$ in this hyperplane. We are then looking for a nonzero trace zero element $x$ in
$\overline{\calW_{p,q}}$ such that $\tr(a^\star x)=\tr(b^\star x)=0$, that is $\tr(ax)=\tr(bx)=0$,
and naturally we look no further than to the Lie commutator $[a,b]:=ab-ba$.

With these heuristics, it becomes perfectly natural to try and prove that $\gamma:=[a,b]$ is
a normalized conjugator for $\Phi_\star$ in general (without assuming $\car(\F) \neq 2$ and $\tr(a)=\tr(b)=0$, that is),
because basic base changes leave the commutator invariant up to multiplication with an element of $\F^\times$.
Now, we must check this, and we compute
\begin{multline*}
[a,b]a^\star-a[a,b]=aba^\star-baa^\star-a^2b+aba=ab\tr(a)-b N(a)-a^2 b \\
= -(a^2-\tr(a)a+N(a))\,b=0.
\end{multline*}
Symmetrically $[a,b]b^\star-b[a,b]=0$. In order to conclude that $[a,b]$
is a conjugator of $\Phi_\star$, we must now compute its norm and check that it is nonzero.
We start from the observation that $N([a,b])=N(ab)-\langle ab,ba\rangle+N(ba)=2N(a)N(b)-\langle ab,ba\rangle$.
Next,
\begin{align*}
\langle ab,ba\rangle & =\langle b,a^\star ba\rangle=\langle b,(\omega-b^\star a)a\rangle \\
& =\omega^2-\langle b^2,a^2\rangle \\
& =\omega^2-\langle \tr(b)b-N(b),\tr(a)a-N(a)\rangle \\
& =\omega^2 - \tr(b)\tr(a)\omega+(\tr b)^2N(a)+(\tr a)^2 N(b)-2N(a)N(b).
\end{align*}
This yields the lovely identity
\begin{equation}\label{eq:normcommutator}
N([a,b])=-\Lambda_{p,q}(\omega).
\end{equation}
If $\Lambda_{p,q}$ is irreducible or has simple roots in $\F$, identity \eqref{eq:normcommutator} readily yields that $[a,b]$ is normalized.
However, because $\Lambda_{p,q}$ might have a double root in $\F$, we must resort to a different argument in general.
We simply note that
\begin{equation}\label{eq:decompcommutator}
[a,b]=ab-(\tr b)a+b^\star a=ab-(\tr b)a+\omega-a^\star b=\omega+2\,ab-(\tr b)\,a-(\tr a)\,b.
\end{equation}
Hence $[a,b]$ is normalized unless $\car(\F)=2$ and $\tr(a)=\tr(b)=0$, in which case $a^\star=a$ and $b^\star=b$
and $\Phi_\star$ is obviously the identity.
In light of \eqref{eq:normcommutator}, it is tempting to think that $[a,b]$ belongs to the ideal
$\mathfrak{J}_r$, and it turns out that such is the case. This property will be useful to us later:

\begin{lemma}\label{lemma:commutator}
For every $r \in \Irr_{\Lambda_{p,q}}$ the commutator $[a,b]$ belongs to $\mathfrak{J}_r$ and its norm equals $-\Lambda_{p,q}(\omega)$.
\end{lemma}

\begin{proof}
We have just proved the second part of the statement.
Let now $r \in \Irr_{\Lambda_{p,q}}$.
Since $N([a,b]) \equiv 0$ mod $(r(\omega))$,
the condition $[a,b] \in \mathfrak{J}_r$ is equivalent to
having  $\langle x,[a,b]\rangle \in (r(\omega))$, i.e., $\tr(x^\star [a,b]) \equiv 0$ mod $(r(\omega))$, for every $x$ in a well-chosen
basis of the $C$-module $\calW_{p,q}$.
We already have $\tr([a,b])=0$, $\tr(a[a,b])=\tr([a,a]b)=0$ and $\tr(b[a,b])=-\tr(b[b,a])=0$.
If $(1,a^\star,b^\star,[a,b])$ is a $C$-basis of $\calW_{p,q}$, then we use $\tr([a,b]^\star [a,b])=2N([a,b])=-2 \Lambda_{p,q}(\omega)$
to conclude. Yet $(1,a^\star,b^\star,[a,b])$ fails to be a $C$-basis if $\car(\F)=2$, since in that case
$[a,b] \in \Vect_C(1,a,b)$.

Hence, we will use the $C$-basis $(1,a^\star,b^\star,(ab)^\star)$ instead.
We will directly check that $\langle ab,[a,b]\rangle \equiv 0$ mod $(r(\omega))$.
Simply, we can observe, because $N(ab)=N(ba)$, that
$$\langle ab,ab-ba\rangle=2 N(ab)-\langle ab,ba\rangle=N(ab-ba)=-\Lambda_{p,q}(\omega),$$
which yields the claimed statement. Hence $[a,b] \in \mathfrak{J}_r$.
\end{proof}

\subsubsection{Swaps}\label{section:swaps}

Here, we consider a swap $S$ of $\calW_{p,q}$.
We take $\beta:=S(a)$. Without essential loss of generality, we perform a basic base change and reduce the situation to the one where
$\beta=b^\star$, in which case $p=q$ and in particular $\tr(a)=\tr(b)$. Note then that $S(b)=a^\star$.

As in the case of the pseudo-adjunction, the heuristics from orthogonal groups will help us find
a normalized conjugator that is associated with $S$.
Again, if $\car(\F) \neq 2$ we can perform the basic base changes $a\leftarrow a-\frac{\tr a}{2}$ and
 $b\leftarrow b-\frac{\tr b}{2}$ to reduce the situation to the one where $\tr(a)=\tr(b)=0$, in which case
 $S(a)=-b$; the rotation of the hyperplane of pure quaternions that is induced by $\overline{S}$
 has fixed vector $a-b$, and hence $a-b$ should be one of its conjugators.
 Reverting the basic base change does not modify the situation because $\tr(a)=\tr(b)$ here.

 Hence, in the general case we should try to check that the vector $a-b$, which is obviously normalized, is
 a conjugator for $S$.
 Simply, we observe, thanks to $\tr(a)=\tr(b)$, that
 $$(a-b)\,b^\star=a b^\star - N(b) \quad \text{while} \quad a(a-b)=a(b^\star-a^\star)=ab^\star-N(a)$$
 and hence $(a-b)b^\star=a (a-b)$ because $N(a)=N(b^\star)=N(b)$.
 Symmetrically $(a-b) a^\star=-(b-a)a^\star=-b(b-a)=b(a-b)$.
 Finally, we compute
$$N(a-b)=N(a)+N(b)-\langle a,b\rangle=2N(a)-\omega.$$
Hence $a-b$ is a normalized conjugator for $S$, with norm $2N(a)-\omega$.
This norm has degree $1$, and by Lemma \ref{lemma:conjugatorandquaternions} it must equal $-r(\omega)$ for some $r \in \Irr_{\Lambda_{p,q}}$.
Remembering however that we had initially reduced the situation to the one where $p=q$,
going back to the general case (i.e., not assuming that $\Phi(a)=b^\star$ anymore)
proves the following result:

\begin{lemma}
For every swap automorphism $S$ of $\calW_{p,q}$,
and every normalized conjugator $\gamma$ of $S$,
the norm $N(\gamma)$ is a divisor of degree $1$ of $\Lambda_{p,q}(\omega)$ in $C$.
\end{lemma}

Note in any case that the existence of a swap automorphism implies that $\Lambda_{p,q}$ splits over $\F$
(see also Table \ref{figure3}), and if $\Lambda_{p,q}$ has a double root then the previous lemma gives us all the information we were
seeking. However, we need to dig deeper if $\Lambda_{p,q}$ has simple roots, i.e., when $p$ has
simple roots in the splitting field $\K$. In that case indeed, we must examine whether the norms
of the normalized conjugators of all swap automorphisms correspond to the same monic irreducible divisor of $\Lambda_{p,q}(\omega)$.
Fortunately, we can solve this issue without computing.

Assume indeed that $p$ and $q$ are split with simple roots in $\K$,
and have the same splitting field in $\K$. Then there are exactly two isomorphisms
from $\F[a]$ to $\F[b]$, and these extend (uniquely) to two swaps $S_1$ and $S_2$.
Then the composite $S_2 \circ S_1$ is the pseudo-adjunction $\Phi_\star$
because it is a non-identity positive basic $C$-automorphism.
Now, take respective normalized conjugators $\gamma_1$ and $\gamma_2$ of $S_1$ and $S_2$.
Then by Lemma \ref{lemma:conjugatormultiplicativity} and identity \eqref{eq:normcommutator}
we have $N(\gamma_1)N(\gamma_2) \sim s^2 \Lambda_{p,q}(\omega)$ for some $s \in C \setminus \{0\}$.
As both $N(\gamma_1)$ and $N(\gamma_2)$ have degree $1$, we see that $s$ is constant, and we conclude that
$N(\gamma_1)$ and $N(\gamma_2)$ are relatively prime.
Hence the full signatures of $S_1$ and $S_2$ are the families in $\prod_{r \in \Irr_{\Lambda_{p,q}}} \Z/2$
with exactly one component equal to $1$.
And finally $S_1$ and $S_2$ are the only two swaps here.

\subsubsection{Hyperbolic automorphisms}\label{section:hyperbolicautomorphisms}

This situation is relevant only when $p$ and $q$ split with a double root.
To simplify the analysis, we will consider only the special case $p=q=t^2$.

Let $\lambda \in \F^\times$.
Once again, we need to find a normalized conjugator for the hyperbolic automorphism $H_\lambda$
(see the notation in Section \ref{section:basicCauto}), compute its norm, and
of course we only need to do this when $\lambda \neq 1$. Assume now that $\lambda \neq 1$.
Here quaternion geometry is of poor help to find the results, and we will turn to the matrix viewpoint instead.
Since $\overline{\calW_{p,q}}$ splits, we are invited to consider the embedding of $\calW_{p,q}$ into $\Mat_2(\F[\omega])$ that takes
$a$ and $b$ to the respective matrices
$$A=\begin{bmatrix}
0 & 0 \\
1 & 0
\end{bmatrix} \quad \text{and} \quad
B=\begin{bmatrix}
0 & -\omega \\
0 & 0
\end{bmatrix}$$
(it is clear indeed that $N(a)=0=\det A$, $N(b)=0=\det B$ and $\langle a,b\rangle=\omega=\langle A,B\rangle$).

Hence $H_\lambda$ corresponds to the conjugation $M \mapsto D_\lambda M D_\lambda^{-1}$ of $\Mat_2(\F[\omega])$ with
the diagonal matrix $D_\lambda:=\begin{bmatrix}
1 & 0 \\
0 & \lambda
\end{bmatrix}$. Yet it can be checked that $D_\lambda$ does not belong to the $\F[\omega]$-subalgebra of $\Mat_2(\F[\omega])$
generated by $A$ and $B$, so we naturally replace it with
$-\omega D_\lambda=BA+\lambda AB$.
We deduce that $\gamma:=ba+\lambda ab$ is a conjugator of $H_\lambda$, and two things remain to be done.
We must check that $\gamma$ is normalized, and compute its norm.
For the first point, we note that
$$\gamma=-ba^\star+\lambda ab=-\omega+ab^\star+\lambda ab=-\omega+(\lambda-1)ab,$$
and since $\lambda-1 \neq 0$ we see that $\gamma$ is normalized.
Finally, thanks to the expression we have just obtained, we compute
\begin{align*}
N(\gamma)=\omega^2-\omega (\lambda-1)\langle 1,ab\rangle+(\lambda-1)^2 N(ab)=\omega^2-\omega(\lambda-1)
\langle a^\star,b\rangle\\
=\omega^2+\omega^2(\lambda-1)=\lambda \omega^2.
\end{align*}
Noting that $\omega^2=\Lambda_{p,q}(\omega)$ in the present case, we observe that the full signature of $H_\lambda$
is zero but the exponent for $r=t$ equals $2$. In particular $H_\lambda$ is not an inner automorphism
(remember here that $\lambda \neq 1$).
Note finally that with the above, for all $\lambda,\mu$ in $\F^\times$ such that $\lambda \neq \mu$,
the element $\lambda ab+\mu ba$ is a normalized conjugator for a nontrivial hyperbolic automorphism
(namely, for $H_{\lambda \mu^{-1}}$).

\subsection{First consequences of the study of signatures}\label{section:uniquenessdecompauto}

We can now draw several important consequences of the previous study of signatures of basic $C$-automorphisms, the first of which is the uniqueness
part in the Automorphisms Theorem.

\begin{theo}\label{theo:uniquenessdecompauto}
The only automorphism of $\calW_{p,q}$ that is both basic and inner is the identity.
\end{theo}

\begin{proof}
As seen earlier, it suffices to take an arbitrary $\Phi \in \BAut_C(\calW_{p,q}) \setminus \{\id\}$ and prove that it is not inner.

Let $\gamma$ be a normalized conjugator of $\Phi$.

If $\Phi$ is a swap, then the previous study shows that $N(\gamma)$ has degree $1$.
If $\Phi$ is the pseudo-adjunction or is a hyperbolic automorphism, then the previous study shows that $N(\gamma)$
has degree $2$.

In any case, $N(\gamma)$ is nonconstant, and hence $\Phi$ is not inner.
\end{proof}

This theorem was the last missing key to the Weak Units Theorem, whose proof is now complete.

Now, we can also give partial answers to the previous questions on the range and kernel of the full signature homomorphism.

\begin{prop}\label{prop:fullsignaturesurjective}
In each one of the following cases, the full signature homomorphism
is surjective and its range is also the range of its restriction to $\BAut_C(\calW_{p,q})$:
\begin{itemize}
\item Both $p$ and $q$ have simple roots in $\K$;
\item Both $p$ and $q$ have a double root in $\K$, with the same splitting field in $\K$.
\item One of $p$ and $q$ is inseparable, and the other one has simple roots in $\K$.
\end{itemize}
\end{prop}

\begin{proof}
Assume first that $\Lambda_{p,q}$ is irreducible. Then the full signature homomorphism essentially maps to $\Z/2$,
and the full signature of the pseudo-adjunction is $1$ unless $\car(\F)=2$ and both $p$ and $q$ have a double root in $\K$.
But since $\Lambda_{p,q}$ is irreducible, an inspection of Table \ref{figure3} shows that this exception cannot take place here.

If $\Lambda_{p,q}$ splits with a double root, then again the full signature homomorphism essentially maps to $\Z/2$,
and the full signature of any swap is $1$. But there might be no swap!

Finally, if $\Lambda_{p,q}$ splits with simple roots then the full signature homomorphism essentially maps to $(\Z/2)^2$,
both polynomials $p$ and $q$ have simple roots in $\K$, and if at least one of them is not split then $p$ and $q$ must be irreducible with the same
splitting field in $\K$ (see Table \ref{figure3}). Hence in any case the algebras $\F[a]$ and $\F[b]$ are isomorphic,
there are two swaps of $\calW_{p,q}$, and by the previous study of swaps the respective full signatures of these swaps are
the pairs $(1,0)$ and $(0,1)$, which is enough to prove the surjectivity of the restriction to $\BAut_C(\calW_{p,q})$ of the full signature homomorphism.

It remains to see that any of the three cases listed in the proposition falls into one of the previous three cases.
If $p$ and $q$ have simple roots in $\K$, then the third case occurs if $p$ and $q$ have the same splitting field in $\K$,
and the first one occurs otherwise. If both $p$ and $q$ have a double root in $\K$ with the same splitting field in $\K$,
the second case occurs and there is a swap automorphism of $\calW_{p,q}$.
Finally, if one of $p$ and $q$ is inseparable and the other one has simple roots in $\K$,
then the first case occurs. This concludes the proof.
\end{proof}

Hence, there remains two main cases for the analysis of the surjectivity of the full signature
on $\Aut_C(\calW_{p,q})$ and on $\BAut_C(\calW_{p,q})$:
the case where exactly one of $p$ and $q$ splits with a double root in $\F$,
and the case where both $p$ and $q$ are inseparable, with distinct splitting fields in $\K$.
In the second case there is no nontrivial basic automorphism, and naturally we will prove that
the full signature of every $C$-automorphism is $0$ (see Section \ref{section:auto:case1}).
The other case is much more problematic. In that one the only nontrivial basic $C$-automorphism
is the pseudo-adjunction, and its full signature is $0$. It is then very tempting
to think that there should be an easy way to prove that
every $C$-automorphism has full signature $0$ in that situation, but although this result will ultimately be proved
it will only be after a very deep analysis of the situation.

\subsection{Preparing the last part of the proof}\label{section:preparation}

It is now time for a review of where we stand and what remains to be done to prove the Automorphisms Theorem.
We already know that this theorem has been reduced to Theorem \ref{theo:automorphismstheoC}, i.e., to its
analogue for $C$-automorphisms, thanks to the analysis performed in Section \ref{section:automorphismsI}.
Moreover, the uniqueness statement from the Automorphisms Theorem as been proved as Theorem \ref{theo:uniquenessdecompauto}.

Hence, all that remains is to prove that every $C$-automorphism splits as the composite of a basic $C$-automorphism with an inner automorphism.
Moreover, in the event where the range of the full signature automorphism is the range of its restriction to
$\BAut_C(\calW_{p,q})$, every $C$-automorphism $\Phi$ splits into
$\Phi=\Phi_2 \circ \Phi_1$ where $\Phi_1 \in \BAut_C(\calW_{p,q})$
and $\Phi_2 \in \Aut_{C,0}(\calW_{p,q})$ (i.e., $\Phi_2$ is a $C$-automorphism with full signature $0$).

In particular, when both $p$ and $q$ have simple roots in $\K$ our only viable option now is to prove
that $\Aut_{C,0}(\calW_{p,q})=\Inn(\calW_{p,q})$. The same goes if both $p$ and $q$ are inseparable with the same splitting field.
But if both $p$ and $q$ split with a double root we have seen earlier that there are nontrivial basic automorphisms
in $\Aut_{C,0}(\calW_{p,q})$, and we must still prove that every automorphism in $\Aut_{C,0}(\calW_{p,q})$
is the product of a basic automorphism (necessarily of hyperbolic type) with an inner automorphism.
Of course there also remain the two main cases where we do not know yet that
the full signature automorphism has the same range on $\Aut_C(\calW_{p,q})$ and on $\BAut_C(\calW_{p,q})$,
i.e., the case where exactly one of $p$ and $q$ splits with a double root, and the case where both $p$ and $q$
are inseparable, with distinct splitting fields.

Now, we must split the discussion into several cases.
There is a common thread to all these cases, nevertheless, and we will now briefly explain it.
In any case we will take $\Phi \in \Aut_C(\calW_{p,q})$, with associated normalized conjugator $\gamma$,
take an irreducible divisor $r \in \Irr_{\Lambda_{p,q}}$, and assume that
$r$ divides $N(\gamma)$.
In all the situations where we already know that the analysis of $\Aut_{C,0}(\calW_{p,q})$
is sufficient to conclude, we will actually directly assume that $r^2$ divides $N(\gamma)$,
seeking to find a contradiction unless both $p$ and $q$ split with a double root.
In the remaining cases where the range of the full signature automorphisms is not known yet,
we will only assume that $r$ divides $N(\gamma)$, and we will either directly obtain a contradiction
(in the case where both $p$ and $q$ are inseparable, with distinct splitting fields), or
prove that $r^2$ divides $N(\gamma)$ (which is a first step in proving that the full signature homomorphism
vanishes).

In any case, we will analyze the consequences of the assumptions on the residue $\gamma_r$
of $\gamma$ modulo $(r(\omega))$ (i.e., its coset in $\calW_{p,q,[r]}$). In some cases, the simple assumption that $r(\omega)$ divides $N(\gamma)$
is enough to obtain very precise information on $\gamma_r$, mainly by using Lemma \ref{lemma:idealandconjugator}, but in most cases we will require that
$r(\omega)^2$ divides $N(\gamma)$, and we will use our thorough study of the ideals above $(r(\omega))$ from Section \ref{section:ideals}.
Once such information is obtained, and assuming that
$r(\omega)^2$ divides $N(\gamma)$, we will try to find relevant information on the residue
$\gamma_{r^2}$ of $\gamma$ mod $(r(\omega)^2)$. In many cases, we will end up with a contradiction,
and in some cases we will have enough information on the residue $\gamma_{r^2}$
to conclude that $\Phi$ is the product of a basic $C$-automorphism with an inner automorphism.

The study will be split into the following five cases, almost in increasing order of difficulty:
\begin{itemize}
\item Both $p$ and $q$ are irreducible.
\item One of $p$ and $q$ splits with simple roots, and the other one has simple roots in $\K$.
\item Both $p$ and $q$ split with a double root.
\item One of $p$ and $q$ splits with simple roots, and the other one is inseparable.
\item Exactly one of $p$ and $q$ splits with a double root.
\end{itemize}

The following simple observation will be regularly used:

\begin{lemma}\label{lemma:normmodsquare}
Let $r \in \Irr_{\Lambda_{p,q}}$. Let $x \in \calW_{p,q}$ and $y \in \mathfrak{J}_r$ be such that $x \equiv y$ mod $(r(\omega))$.
Then $N(x) \equiv N(y)$ mod $(r(\omega)^2)$.
\end{lemma}

\begin{proof}
We write $x=y+r(\omega)z$ for some $z \in \calW_{p,q}$ and note that $\langle y,z\rangle \in (r(\omega))$ because $y \in \mathfrak{J}_r$.
Hence $N(x)=N(y)+r(\omega) \langle y,z\rangle+r(\omega)^2 N(z)$ and the conclusion follows.
\end{proof}

\subsection{Case I: When both $p$ and $q$ are irreducible}\label{section:auto:case1}

Here we assume that both $p$ and $q$ are irreducible.

To start with, we consider the very special case where $p$ and $q$ are inseparable with
distinct splitting fields. In this case we do not know yet whether the full signature automorphism, which goes to $\Z/2$,
is surjective. Here we have $\Lambda_{p,q}=t^2$, so we set $r:=t$.
Let $\Phi \in \Aut_C(\calW_{p,q})$, with normalized conjugator $\gamma$.
Assume that the full signature of $\Phi$ is nonzero. In particular
$\omega$ divides $N(\gamma)$. Hence by Lemma \ref{lemma:idealandconjugator} the ideal $(\gamma_r)$ of $\calW_{p,q}/(r(\omega))$
is nontrivial. This directly contradicts Proposition \ref{prop:structureofJr}, because here $\mathfrak{R}_r=\{0\}$.
Hence the full signature homomorphism vanishes on $\Aut_C(\calW_{p,q})$.

Now, we return to the general case and conclude that all that remains to be done is prove that
$\Inn(\calW_{p,q})=\Aut_{C,0}(\calW_{p,q})$. So, assume on the contrary that there exists
$\Phi \in \Aut_{C,0}(\calW_{p,q}) \setminus \Inn(\calW_{p,q})$ and take an associated
normalized conjugator $\gamma$. Then there exists $r \in \Irr_{\Lambda_{p,q}}$ such that
$r(\omega)^2$ divides $N(\gamma)$.
Applying Lemma \ref{lemma:idealandconjugator}, we directly see that the ideal of $\calW_{p,q,[r]}$
generated by $\gamma_r$ is nontrivial. As a consequence of Theorem \ref{prop:structureofJr}, we deduce that
this ideal is included in $\mathfrak{R}_r$, i.e., $\gamma \in \mathfrak{J}_r$.
By Euclidean division, we can find $x \in \calW_{p,q}$ of the form $x=x_1+x_a a+x_{b} b+x_{ab} b$,
with all coefficients $x_1,x_a,x_b,x_{ab}$ of degree less than the one of $r$,
and such that $\gamma \equiv x$ mod $(r(\omega))$.
Then $N(\gamma) \equiv N(x)$ mod $(r(\omega)^2)$ by Lemma \ref{lemma:normmodsquare}, and we conclude by Lemma
\ref{lemma:liftinglemma} that $N(x)=0$.
Yet the Zero Divisors Theorem then yields $x=0$, and in turn we conclude that $\gamma_r=0$, contradicting the fact that
$\gamma$ is normalized.

This contradiction shows that $\Inn(\calW_{p,q})=\Aut_{C,0}(\calW_{p,q})$. Hence our proof of Theorem \ref{theo:maintheoautomorphisms}
is now complete for the special case where both $p$ and $q$ are irreducible. This very result was already proved in
Section \ref{section:units1} by a completely different method,
but treating this case did not cost us much, with the exception of several difficulties that were all located in the characteristic $2$ case
(see Section \ref{section:centerspecialcases}).

\subsection{Case II: When one of $p$ and $q$ splits with simple roots, and both have simple roots in $\K$}

Here we assume without loss of generality that $p$ splits with simple roots and $q$ has simple roots in $\K$.
In that situation, we know from Proposition \ref{prop:fullsignaturesurjective} that
it will suffice to prove that $\Inn(\calW_{p,q})=\Aut_{C,0}(\calW_{p,q})$.
Hence, just like in the previous section we assume that there exists
$\Phi \in \Aut_{C,0}(\calW_{p,q}) \setminus \Inn(\calW_{p,q})$ and we try to find a contradiction.
To this end we take an associated normalized conjugator $\gamma$ and an irreducible $r \in \Irr_{\Lambda_{p,q}}$
such that $r(\omega)^2$ divides $N(\gamma)$. Throughout, we denote by $d$ the degree of $r$.
Beware at this point that $\Lambda_{p,q}$ can be irreducible, so we must bear in mind that $d \in \{1,2\}$.

Now, we analyze $\gamma_r$. First of all, we apply Lemma \ref{lemma:idealandconjugator} to find that $(\gamma_r)$ is a nontrivial ideal
of $\calW_{p,q,[r]}$.

Assume first that $(\gamma_r)$ has dimension $2$ or $3$ over the residue field $\L:=\F[\omega]/(r(\omega))$.
In the first case, Proposition \ref{prop:allidealsonesplits} shows that $(\gamma_r)$ equals $\mathfrak{R}_r$, and in the second case
Corollary \ref{cor:inclusionRr} shows that $(\gamma_r)$ includes $\mathfrak{R}_r$.
Hence, in any case $(\gamma_r)$ includes $\mathfrak{R}_r$, and now we use the magical trick that
the residue $[a,b]_r$ of the commutator $[a,b]$ also belongs to $\mathfrak{R}_r$, along with
$N([a,b])=-\Lambda_{p,q}(\omega)$ (see Lemma \ref{lemma:commutator}).
Hence there exists $z \in (\gamma)$ such that $[a,b] \equiv z$ mod $(r(\omega))$.
But now Lemma \ref{lemma:idealandconjugator} applied to the ideal $(r(\omega)^2)$ shows that $N(z)=zz^\star \equiv 0$ mod $(r(\omega)^2)$.
Since $[a,b] \in \mathfrak{J}_r$, we deduce from Lemma \ref{lemma:normmodsquare} that $N([a,b]) \equiv 0$ mod $(r(\omega)^2)$,
whence $r(\omega)^2$ divides $\Lambda_{p,q}(\omega)$. Hence $\Lambda_{p,q}$ would have a double root in $\K$,
which is not the case here because both $p$ and $q$ have simple roots in $\K$. Hence we find a contradiction.

We deduce from the previous analysis that $(\gamma_r)$ has dimension $1$ over $\L$. We can then use point (c) of Proposition
\ref{prop:allidealsonesplits} to find elements $x \in \L[a_r] \setminus \L$ and $y \in \L[b_r] \setminus \L$ such that
$N_r(x)=N_r(y)=0$ and $\gamma_r=xy^\star$, and the $1$-dimensional ideals of  $\calW_{p,q,[r]}$
are $(xy^\star)$ and $(x^\star y)$.

Since $p$ splits with simple roots, we have $x = s\,\alpha_r$ for some nontrivial idempotent $\alpha \in \F[a]$
and some $s \in \L^\times$. Replacing $x$ with $\alpha$ and $y$ with $s^{-1}y$ then reduces the situation to the one where $x=\alpha_r$.
Next, by Euclidean division applied to a lifting of $y$ in $\calW_{p,q}$, we find a list $(\beta_0,\dots,\beta_{d-1}) \in \F[b]^d$ such that the element
$$\beta:=\sum_{k=0}^{d-1} \omega^k \beta_k$$
is nonzero and $\beta_r=y$, so
$$\gamma \equiv \alpha \beta^\star \quad \text{mod} \; (r(\omega)) \quad \text{and} \quad N(\beta) \equiv 0 \quad \text{mod} \; (r(\omega)).$$
Note that $(\alpha_r \beta_r^\star)=(x y^\star)$ and $(\alpha_r^\star \beta_r)=(x^\star y)$ are the $1$-dimensional ideals of $\calW_{p,q,[r]}$.

Next, we observe that $\deg(N(\beta)) \leq 2(d-1)$, and since $d \leq 2$ the fact that $r(\omega)$ divides $N(\beta)$ yields
$$N(\beta)=\lambda\, r(\omega) \quad \text{for some $\lambda \in \F$,}$$
with $\lambda=0$ if $d=1$ (i.e., if $q$ splits).
Note that in fact $\lambda \neq 0$ if $q$ is irreducible, because in that case $q$ remains irreducible over the purely transcendental extension $\F(\omega)$.
Moreover, since $\gamma_r \in \mathfrak{R}_r$ we have in particular $\tr_r(\gamma_r)=0$, which translates into
the fact that $r(\omega)$ divides $\langle \alpha,\beta\rangle$. Again $\deg \langle \alpha,\beta\rangle \leq \deg(r)$, and hence
$$\langle \alpha,\beta\rangle =\mu\, r(\omega) \quad \text{for some $\mu \in \F$.}$$
This time around, we observe that if $d=1$ then $\mu \neq 0$, because in that case $\beta$ is an element in $\F[b]$ and a nonscalar one
(being nonzero yet with norm $0$).

Next, we introduce an element $u \in \calW_{p,q}$ such that $\gamma=\alpha \beta^\star+r(\omega) u$ mod $(r(\omega)^2)$.
Expanding the identity
\begin{equation}\label{eq:keyidstarmodularcaseI}
\forall z \in \calW_{p,q}, \; \gamma z \gamma^\star \equiv 0 \quad \text{mod} \; (r(\omega)^2)
\end{equation}
thus yields
\begin{equation}\label{eq:expandedmod2id}
\forall z \in \calW_{p,q}, \; (\alpha \beta^\star) z (\alpha \beta^\star)^\star
+r(\omega)(u z (\alpha \beta^\star)^\star+(\alpha \beta^\star) z u^\star)
 \equiv 0 \quad \text{mod} \quad (r(\omega)^2).
\end{equation}

We start by examining how close the approximation $\alpha \beta^\star$ is to satisfying \eqref{eq:keyidstarmodularcaseI} in place of $\gamma$,
and for this we take $z$ among the vectors of the $C$-basis $(1,\alpha,\beta,\alpha \beta^\star)$.
We immediately observe that
$$(\alpha \beta^\star)(\alpha \beta^\star)^\star=N(\alpha \beta^\star)=N(\alpha)N(\beta^\star)=0$$
and hence
$$(\alpha \beta^\star)(\alpha \beta^\star)(\alpha \beta^\star)^\star=0.$$
However
$$(\alpha \beta^\star)\beta(\alpha \beta^\star)^\star=N(\beta) \alpha \beta \alpha^\star=\lambda r(\omega)\, \alpha \beta \alpha^\star$$
and
$$(\alpha \beta^\star)\alpha(\alpha \beta^\star)^\star=\alpha \beta^\star \alpha \beta \alpha^\star
=\langle \alpha,\beta\rangle  \alpha \beta \alpha^\star-\beta \alpha^\star \alpha \beta \alpha^\star
=\mu r(\omega)\, \alpha \beta \alpha^\star.$$
Finally, by noting that $\alpha \beta \alpha^\star=\langle \alpha,\beta\rangle \alpha-\alpha^2 \beta^\star=\mu r(\omega)\alpha-\alpha \beta^\star$
we deduce that
$$(\alpha \beta^\star)\beta(\alpha \beta^\star)^\star \equiv -\lambda r(\omega)\,\alpha \beta^\star  \quad \text{mod} \; (r(\omega)^2)$$
and
$$(\alpha \beta^\star)\alpha(\alpha \beta^\star)^\star \equiv -\mu r(\omega)\,\alpha \beta^\star \quad \text{mod} \; (r(\omega)^2)$$
Noting that $x \mapsto (\alpha \beta^\star)x(\alpha \beta^\star)^\star$ is $C$-linear, we recover a $C$-linear form
$\theta : \Vect_C(1,\alpha,\beta,\alpha\beta^\star) \rightarrow C$ such that
$$\forall z \in \Vect_C(1,\alpha,\beta,\alpha\beta^\star), \; (\alpha \beta^\star)z(\alpha \beta^\star)^\star \equiv -\theta(z)\,r(\omega)\,\alpha \beta^\star \quad \text{mod} \; (r(\omega)^2),$$
with $\theta(\alpha)=\mu$, $\theta(\beta)=\lambda$ and $\theta(1)=\theta(\alpha \beta^\star)=0$.
Now, the trick is to apply \eqref{eq:expandedmod2id} to $z:=\alpha^\star \beta$.
Recall that the $1$-dimensional ideals of the $\L$-algebra $\calW_{p,q,[r]}$ are
$(\alpha_r \beta_r^\star)$ and $(\alpha_r^\star \beta_r)$, and that we know from Proposition \ref{prop:allidealsonesplits}
that they are invariant under adjunction.
As a consequence, $u_r z_r (\alpha_r \beta_r^\star)^\star$ belongs to both, leading to
$u_r z_r (\alpha_r \beta_r^\star)^\star=0$. Likewise $(\alpha_r \beta_r^\star) z_r u_r^\star=0$.
Hence identity \eqref{eq:expandedmod2id} leads to $\theta(z)\alpha \beta^\star \equiv 0$ mod $(r(\omega))$,
i.e., $r(\omega)$ divides $\theta(z)$.

Finally $z=\alpha^\star \beta=\beta-\alpha \beta=\beta-\tr(\beta) \alpha+\alpha \beta^\star$, so by the linearity of $\theta$ we recover that
$r(\omega)$ divides $\lambda-\mu \tr(\beta)$. Yet $\deg(\tr(\beta)) \leq d-1$, and hence
\begin{equation}\label{eq:keyidstarmodularcaseIfinal}
\lambda-\mu \tr(\beta)=0.
\end{equation}
The conclusion is near, and now we have to discuss whether $r$ has degree $1$ or $2$, i.e., whether $q$ splits or not.

\vskip 3mm
\noindent \textbf{Case 1. $q$ splits.}
Then, as pointed out earlier we have $\lambda=0$ and $\mu \neq 0$, so here $\tr(\beta)=0$.
Yet $N(\beta)=0$ and $\beta$ is a nonzero element of $\F[b]$. This would yield that $\F[b]$ is degenerate, contradicting the assumption that $q$
has simple roots in $\K$.

\vskip 3mm
\noindent \textbf{Case 2. $q$ is irreducible.}
Here $d=2$ and we will use \eqref{eq:keyidstarmodularcaseIfinal} in a different way. We know that $\lambda \neq 0$, so $\mu \neq 0$ and we derive that
$\tr(\beta)$ is constant as a polynomial in $\omega$. However this information is still insufficient to obtain a contradiction.
Now the trick is to apply this observation not only to $\beta$, but to $v^\star \beta$ for any $v \in \F[b]^\times$.
Indeed, let us right-compose $\Phi$ with the inner automorphism $x \mapsto v x v^{-1}$ for an arbitrary $v \in \F[b]^\times$.
By doing so we obtain an automorphism with normalized conjugator $\gamma v$, merely replacing $\beta$ with $v^\star \beta$,
and this new automorphism cannot be inner.
All the other assumptions are preserved here, and we derive from the previous analysis
that $\tr(v^\star \beta)$ is constant. Since $\F[b]$ is a field, this means that $\langle v,\beta \rangle$ is constant whatever the
choice of $v \in \F[b]$. Remembering the splitting $\beta=\beta_0+\omega \beta_1$, we infer that
$\langle v,\beta_1\rangle=0$ for all $v \in \F[b]$.
Since $q$ is separable we deduce that $\beta_1=0$. Hence $\beta \in \F[b]$.
And now $N(\beta)$ is constant, whereas we had seen earlier that $N(\beta)=\lambda r(\omega)$ with $\lambda \neq 0$.
This final contradiction completes the proof of the case where both $p$ and $q$ have simple roots in $\K$, and at least one splits.

\subsection{Case III: When both $p$ and $q$ split with a double root}

Here we assume that both $p$ and $q$ split with a double root. Without loss of generality, we assume that $p=q=t^2$.
So here $\Lambda_{p,q}=t^2$ and we will take $r:=t$.

For a $C$-automorphism $\Phi$, we simply put
$n(\Phi):=n_r(\Phi)$.

As seen in Section \ref{section:preparation}, it only remains in that case to prove that every $C$-automorphism with full signature zero
is the composite of a basic $C$-automorphism with an inner automorphism.

Our first solid step is the following result:

\begin{lemma}\label{lemma:resolutiontwodouble}
Let $\Phi \in \Aut_C(\calW_{p,q})$ satisfy $n(\Phi) \geq 2$, with normalized conjugator $\gamma$.
Then:
\begin{enumerate}[(a)]
\item $\gamma \equiv \lambda ab$ mod $(\omega)$ for a unique $\lambda \in \F^\times$.
\item There is a unique $\mu \in \F$ such that $\gamma \equiv \lambda ab+\mu \omega$ mod $\left((\omega^2)+\omega \Vect_\F(a,b,ab)\right)$,
and the inequality $n(\Phi)>2$ is equivalent to $\mu \in \{0,\lambda\}$.
\end{enumerate}
\end{lemma}

Before we prove this result, it is useful to make an observation that will be used repeatedly:
$\F(a_r b_r)$ is an ideal of $\calW_{p,q,[r]}$
and more precisely it is left-annihilated and right-annihilated by all the elements of $\Vect_\F(a_r,b_r,a_rb_r)$.
This is obtained by noting that $bab=-b^\star a b=-\omega b+a^\star b^2=-\omega b$ for left-multiplication,
and that $aba=-\omega a$ likewise for right-multiplication.

\begin{proof}[Proof of Lemma \ref{lemma:resolutiontwodouble}]
By assumption $\omega^2$ divides $N(\gamma)$.
Once again, we will take advantage of Lemma \ref{lemma:idealandconjugator} to obtain a precise form for the residue
$\gamma_r$. There are two main steps: firstly we will prove that $\gamma_r$ belongs to $(a_r)$ or $(b_r)$,
and then we will show that $\gamma_r$ belongs to $(a_r b_r)$. Afterwards, little computation will be required.

First of all, we deduce from Lemma \ref{lemma:idealandconjugator} that $(\gamma_r)$ is a nontrivial ideal of $\calW_{p,q,[r]}$.
We see from Proposition \ref{prop:structureofJr} that $\mathfrak{R}_r$ is the sole maximal ideal of  $\calW_{p,q,[r]}$, and hence
$\gamma_r \in \mathfrak{R}_r$, i.e., $\gamma \in \mathfrak{J}_r$.
Next, the Gram matrix of the deployed basis $(1,a,b,ab)$ for $\langle -,-\rangle$ equals
\begin{equation}\label{eq:lastgram}
\begin{bmatrix}
2 & 0 & 0 & -\omega \\
0 & 0 & \omega & 0 \\
0 & \omega & 0 & 0 \\
-\omega & 0 & 0 & 0
\end{bmatrix}.
\end{equation}
Hence, the one of $(1,a_r,b_r,a_rb_r)$ for $\langle -,-\rangle_r$ equals
$$\begin{bmatrix}
2 & 0 & 0 & 0 \\
0 & 0 & 0 & 0 \\
0 & 0 & 0 & 0 \\
0 & 0 & 0 & 0
\end{bmatrix},$$
while $N_r(1) \neq 0$ and $N_r(a_r)=N_r(b_r)=N_r(a_r b_r)=0$. From there
it is clear that $\mathfrak{R}_r=\Vect_\F(a_r,b_r,(ab)_r)$.

Now, we introduce the unique $z \in \Vect_\F(a,b,ab)$ such that $z_r=\gamma_r$.
By Lemma \ref{lemma:normmodsquare} we have $N(z) \equiv N(\gamma)$ mod $(\omega^2)$ because $\gamma \in \mathfrak{J}_r$, and hence
$N(z)=0$ by Lemma \ref{lemma:liftinglemma}. Writing $z=\lambda a+\mu b+\nu ab$ with $\lambda,\mu,\nu$ in $\F$, we deduce from the Gram matrix
\eqref{eq:lastgram} that $\lambda \mu \omega=0$, and hence $z \in \Vect_\F(a,ab)$ or $z \in \Vect_\F(b,ab)$.
As we can swap $\F[a]$ for $\F[b]$, no generality is lost in assuming that $z \in \Vect_\F(a,ab)$, an assumption we will now make.
In other words $z=a \beta^\star$ for some $\beta \in \F[b] \setminus \{0\}$.

Now, we take an element $u \in \calW_{p,q}$ such that $\gamma\equiv a\beta^\star+\omega u$ mod $(\omega^2)$.

The next step consists in proving that $\beta \in \F b$.
For this, the key is to apply identity \eqref{eq:keyidstarmodular} to a well-chosen $x \in \calW_{p,q}$.
Let first $x \in \calW_{p,q}$ be arbitrary.
Then we use \eqref{eq:keyidstarmodular} to find
$$(a \beta^\star) x (a \beta^\star)^\star+\omega u x (a \beta^\star)^\star+\omega (a\beta^\star) x u^\star \equiv 0 \quad \text{mod} \; (\omega^2).$$
Moreover, we can write
$$u x (a \beta^\star)^\star+(a\beta^\star) x u^\star=\langle ux,a \beta^\star\rangle+(a \beta^\star) (x-x^\star)u^\star,$$
and note that $\langle ux,a \beta^\star\rangle \equiv 0\; (\omega)$ since $a\beta^\star \in \mathfrak{J}_r$.
Hence
\begin{equation}\label{eq:doublerootsx2}
(a \beta^\star) x (a \beta^\star)^\star \equiv -\omega (a \beta^\star) (x-x^\star)u^\star  \quad \text{mod} \; (\omega^2).
\end{equation}
Assume first that $\beta \in \F^\times$. Then we observe that
$(a \beta^\star) b (a \beta^\star)^\star=N(\beta) aba^\star$ and $a b a^\star=a(\omega-ab^\star)=\omega a$,
whence $(a \beta^\star) b (a \beta^\star)^\star=N(\beta) \omega a$.
However $(a \beta^\star) (b-b^\star)u^\star=2 \beta^\star a(b u^\star)$ and $ab u^\star \equiv \lambda ab$ mod $(\omega)$ for some $\lambda \in \F$
because $(a_r b_r)$ is a $1$-dimensional ideal of $\calW_{p,q,[r]}$.
Hence $N(\beta) a \equiv - 2 \lambda \beta^\star ab$ mod $(\omega)$, which is contradictory because $N(\beta) \neq 0$.

Hence $\beta \not\in \F^\times$, so from now on we write $\beta=\lambda b+\mu$ for some $\lambda \in \F^\times$ and some $\mu \in \F$. Note that
$\langle a,\beta\rangle=\lambda \omega$.
Next, we apply \eqref{eq:doublerootsx2} to $x=\beta$.
The left-hand side then equals
$$N(\beta) a \beta a^\star=N(\beta)\langle a,\beta\rangle a-N(\beta) a^2 \beta^\star=
N(\beta)\langle a,\beta\rangle a=\lambda \omega N(\beta)\,a,$$
whence
$$\lambda N(\beta)\,a \equiv -(a \beta^\star) (\beta-\beta^\star)u^\star  \quad \text{mod} \; (\omega).$$
The right-hand side in this congruence equals $-(a \beta^\star) (2\lambda b) u^\star=-2\lambda \mu ab u^\star$ with $\lambda \in \F^\times$.
Remembering that $(a_rb_r)$ is an ideal of $\calW_{p,q,[r]}$  we deduce that
$-(a \beta^\star) (2\lambda b) u^\star  \equiv \nu ab$ mod $(\omega)$ for some $\nu \in \F$.
Hence we have found $\lambda N(\beta) a \equiv \nu ab$ mod $(\omega)$ with $\lambda \neq 0$, and it follows that $N(\beta)=0$
by working in the deployed basis $(1,a,b,ab)$.

We conclude that $\beta \sim b$, which completes the proof of point (a).

\vskip 3mm
In point (b), the existence and uniqueness of $\mu$ are obvious.
Next, we write
$$u \equiv \mu +u' \quad \text{mod} \; (\omega) \quad \text{for some $u' \in \Vect_\F(a,b,ab)$.}$$
We also introduce $v \in \calW_{p,q}$ such that $\gamma=\lambda ab+\omega u+\omega^2 v$.
Since $ab \in \mathfrak{J}_r$, we find that $N(\gamma) \equiv N(\lambda ab+\omega u)$ mod $(\omega^3)$.
Next,
$$N(\lambda ab+\omega u)=\lambda \omega \langle ab,u\rangle+\omega^2 N(u)$$
and $N(u) \equiv N(\mu)$ mod $(\omega)$ because $u' \in \mathfrak{J}_r$.
Hence
$$N(\gamma) \equiv \lambda \omega \langle ab,u\rangle+\mu^2 \omega^2 \quad \text{mod} \; (\omega^3).$$
Finally, since $ab \in \mathfrak{J}_r$ we have $\langle ab,u\rangle \equiv \langle ab,\mu+u'\rangle$ mod $(\omega^2)$;
moreover $\langle ab,u'\rangle=0$ since $u'\in \Vect_\F(a,b,ab)$, so $\langle ab,u\rangle \equiv \mu \langle ab,1\rangle=-\mu \omega$ mod $(\omega^2)$.
Hence $N(\gamma) \equiv \mu(\mu-\lambda)\omega^2$ mod $(\omega^3)$,
and we conclude that $n(\Phi)>2$ if and only if $\mu=\lambda$ or $\mu=0$.
\end{proof}

Now, we will make subtle use of the collapsing phenomenon that was mentioned in Section
\ref{section:conjugatorsintroduction} when composing two $C$-automorphisms.
To start with, we introduce the swap automorphism $S$ that exchanges $a$ and $b=-b^\star$,
and we recall that $a+b$ is an associated conjugator because
$\tr(a)=\tr(b)=0$ (see Section \ref{section:swaps}).

\begin{lemma}
Let $\Phi \in \Aut_C(\calW_{p,q})$.
Then $n(\Phi \circ S) \leq n(\Phi)+1$, and if $n(\Phi) \geq 2$ then
$n(\Phi \circ S) \leq n(\Phi)-1$.
\end{lemma}

\begin{proof}
Take a normalized conjugator $\gamma$ for $\Phi$ and a normalized conjugator $\gamma'$ for $\Phi \circ S$.
By Lemma \ref{lemma:conjugatormultiplicativity} there exists $s \in \F[\omega] \setminus \{0\}$ such that
$s^2 N(\gamma') \sim N(\gamma)N(a+b)$.
Moreover $N(a+b) \sim \omega$ as seen in Section \ref{section:swaps}.
Hence $n(\Phi \circ S) \leq n(\Phi)+1$, and $n(\Phi \circ S) \leq n(\Phi)-1$ if $s$ is nonconstant,
i.e., if $\gamma (a+b)$ is not normalized.
To conclude, we will simply check that $\gamma (a+b)$ is not normalized if $n(\Phi) \geq 2$.
So, assume that $n(\Phi) \geq 2$. By Lemma \ref{lemma:resolutiontwodouble} we have $\gamma\equiv \lambda ab$ mod $(\omega)$
for some $\lambda \in \F^\times$.
Moreover, as recalled before the proof of Lemma \ref{lemma:resolutiontwodouble}, the element $(a+b)_r$ belongs to the right-annihilator of the ideal $(a_rb_r)$ in $\calW_{p,q,[r]}$, which means that $ab(a+b) \equiv 0$ mod $(\omega)$. Hence $\gamma (a+b) \equiv 0$ mod $(\omega)$,
and hence $\gamma (a+b)$ is not normalized.
\end{proof}

\begin{cor}\label{cor:nPhimoinsque2}
Let $\Phi \in \Aut_C(\calW_{p,q})$. Then $n(\Phi) \leq 2$.
\end{cor}

\begin{proof}
Assume on the contrary that $n(\Phi) > 2$. Then $n(\Phi \circ S) < n(\Phi)$ by
the previous lemma. If $n(\Phi \circ S) \geq 2$ then $n((\Phi \circ S) \circ S) < n(\Phi \circ S) < n(\Phi)$ by another
round of this lemma, which is absurd because $S^2=\id$.
Hence $n(\Phi \circ S) \leq 1$, and then $n(\Phi)=n(\Phi \circ S^2) \leq n(\Phi \circ S)+1 \leq 2$, another contradictory statement.
\end{proof}

Now we can complete the proof of the Automorphisms Theorem in the case under consideration. Let $\Phi \in \Aut_{C,0}(\calW_{p,q}) \setminus \Inn(\calW_{p,q})$.
Then $n(\Phi)$ is even and nonzero, and it ensues from Corollary \ref{cor:nPhimoinsque2} that $n(\Phi)=2$.
Choose a normalized conjugator $\gamma$ for $\Phi$. By Lemma \ref{lemma:resolutiontwodouble}, we have scalars
$\lambda \in \F^\times$ and $\mu \in \F \setminus \{0,\lambda\}$, and a vector
$u \in \Vect_\F(a,b,ab)=:G$ such that $\gamma\equiv \lambda ab+\mu \omega +\omega u$ mod $(\omega^2)$.
The latter congruence can be rewritten
$$\gamma \equiv (\lambda-\mu) ab-\mu ba+\omega u \quad \text{mod} \; (\omega^2),$$
and we will exploit the observation that the right-hand side is reminiscent of the conjugators of
the hyperbolic automorphisms (see Section \ref{section:hyperbolicautomorphisms}).

So, naturally we put $\lambda':=\lambda-\mu$ and $\mu':=-\mu$, noting
$\lambda'$ and $\mu'$ are non zero and distinct.
We recall that the hyperbolic automorphism $H_\delta$ for $\delta:=\mu' (\lambda')^{-1}$
(which takes $a$ to $\delta a$ and $b$ to $\delta^{-1}b$) has
$\mu' ab+\lambda' ba$ as normalized conjugator.
It is reasonable to hope that $\Phi \circ H_\delta$ is inner.
To confirm this we consider the product $\pi:=\gamma (\mu' ab+\lambda' ba) $ of conjugators.
Thanks to $abab=ab^\star a b^\star=a(\omega-a^\star b)b^\star=\omega ab^\star$ and likewise
$baba=\omega ba^\star$, we obtain
$$\pi \equiv \lambda'\mu'\omega^2+\omega u(\mu' ab+\lambda' ba) \quad \text{mod} \; (\omega^2).$$
Finally $u (\mu' ab+\lambda' ba) \equiv 0$ mod $(\omega)$ because $u \in G$ and $(\mu' ab+\lambda' ba)_r \in (a_r b_r)$
(see the remark preceding the proof of Lemma \ref{lemma:resolutiontwodouble}).
Hence $\pi \equiv 0$ mod $(\omega^2)$.

Therefore, if we take a normalized conjugator $\gamma'$ of $\Phi \circ H_\delta$,
then $\omega^2 s \gamma'=\gamma (\mu' ab+\lambda' ba)$ for some $s \in \F[\omega] \setminus \{0\}$.
Taking the norm yields $\omega^4 s^2 N(\gamma') \sim \omega^2 N(\gamma)$ and hence $n(\Phi \circ H_\delta) \leq n(\Phi)-2=0$.
Thus $n(\Phi \circ  H_\delta)=0$ and we conclude that $\Phi \circ H_{\delta}$ is inner.
Hence $\Phi= (\Phi \circ H_{\delta})\circ H_{\delta^{-1}}$ is the composite of an inner automorphism with a basic automorphism.
This case is now closed.

\subsection{Preliminary work for the remaining two cases}

Only two cases remain at this point:
The case where exactly one of $p$ and $q$ splits with a double root,
and the case where one of $p$ and $q$ splits with simple roots and the other one is inseparable.

Here, we assume to be in either one of these cases. Then $\Lambda_{p,q}$ has a double root in $\K$,
and $\Irr_{\Lambda_{p,q}}$ consists of a single element $r$ (of degree $1$ in the first case, and of degree $2$ in the second one).

Here we prove the following result, which is common to these cases:

\begin{lemma}\label{lemma:normalizedconjugatorinJr}
Let $\Phi \in \Aut_C(\calW_{p,q})$ be such that $n_r(\Phi) \geq 1$.
Then every normalized conjugator $\gamma$ of $\Phi$ belongs to $\mathfrak{J}_r$.
\end{lemma}

\begin{proof}
Let $\gamma$ be a normalized conjugator of $\Phi$.
By Lemma \ref{lemma:idealandconjugator},
we find that $(\gamma_r)$ is a nontrivial ideal of $\calW_{p,q,[r]}$.
Assume that $\gamma_r \not\in \mathfrak{R}_r$. Then in particular $\mathfrak{R}_r$ is not the sole maximal ideal of $\calW_{p,q,[r]}$.
By Proposition \ref{prop:structureofJr}, if one of $p$ and $q$ splits with a double root then the other one must split with simple roots.

Hence, because of our initial restriction on the possibilities for $(p,q)$, exactly one of $p$ and $q$ splits with simple roots, and the other one
has a double root in $\L:=\F[\omega]/(r(\omega))$. Moreover $\K$ is the splitting field of $\Lambda_{p,q}$.
Applying points (a) and (b) of Proposition \ref{prop:allidealsonesplits}, we deduce that $I:=(\gamma_r)$ is a $3$-dimensional ideal of $\calW_{p,q,[r]}$
(the dimension here refers to the structure of $\L$-linear subspace).

Without loss of generality, we assume that $p$ splits with simple roots, which yields a nontrivial idempotent
$\alpha$ in $\F[a]$. Then $\alpha_r$ is a nontrivial idempotent in $\calW_{p,q,[r]}$,
and since $\calW_{p,q,[r]}/I$ is a field we deduce that $I$ contains one of $\alpha_r$ and $\alpha_r^\star$.
Replacing $\alpha_r$ by $\alpha_r^\star$ if necessary, we can assume without loss of generality that
$\alpha_r \in I$.
Finally, because $q$ has a double root in $\L$ we deduce from point (d) of Proposition \ref{prop:allidealsonesplits}
that there exists $\beta_r \in \L[b_r] \setminus \L$ such that $N_r(\beta_r)=0$ and $\mathfrak{R}_r=(\beta_r)$.
Note that $\tr_r(\beta_r)=0$.

Then we deduce from Lemma \ref{lemma:idealandconjugator} that
$\alpha_r \beta_r \alpha_r^\star=0$.
Yet
$$\alpha_r \beta_r \alpha_r^\star=\langle \alpha_r,\beta_r\rangle \alpha_r -\alpha_r^2 \beta_r^\star
=\alpha_r \beta_r$$
because $\langle \alpha_r,\beta_r\rangle=0$ (remember that $\beta_r \in \mathfrak{R}_r$).
Hence $\alpha_r \beta_r=0$. However $\beta_r=s_1+s_2 b_r$ for some $s_1 \in \L$ and some $s_2 \in \L^\times$,
and hence $\alpha_r \beta_r=s_1 \alpha_r+s_2 \alpha_r b_r$ has at least one nonzero coordinate in the basis $(1,\alpha_r,b_r,\alpha_r b_r)$
of the $\L$-vector space $\calW_{p,q,[r]}$.

We deduce from this \emph{reductio ad absurdum} that $\gamma_r \in \mathfrak{R}_r$, i.e., $\gamma \in \mathfrak{J}_r$.
\end{proof}

\subsection{Case IV: When one of $p$ and $q$ splits with simple roots, and the other one is inseparable}

Here we assume that $p$ splits with simple roots and $q$ is inseparable.
In that situation $\Lambda_{p,q}$ is irreducible, and the sole basic $C$-automorphism is the pseudo-adjunction $\Phi_\star$.
So here we take $r:=\Lambda_{p,q}$.

Here is our first key observation:
\begin{equation}\label{eq:produitdansromega}
\forall (x,y)\in (\mathfrak{J}_r)^2, \; xy \in (r(\omega)).
\end{equation}
Indeed, let $x,y$ in $\mathfrak{R}_r$. Here $q$ splits over $\L:=\F[\omega]/(r(\omega))$ with a double root,
which yields an element $\beta_r \in \L[b_r] \setminus \L$ such that $N_r(\beta_r)=0$ and $\tr_r(\beta_r)=0$.
Then by the last point in Proposition \ref{prop:allidealsonesplits}
there are elements $x_1$ and $y_1$ of $\calW_{p,q,[r]}$ such that $x=x_1 \beta_r$ and $y=\beta_r y_1$.
Then $xy=0$ because $\beta_r^2=-N_r(\beta_r)=0$.
This yields the stated result.

Next, we use this observation as follows:

\begin{lemma}
Let $\Phi \in \Aut_C(\calW_{p,q})$ be such that $n_r(\Phi)>0$.
Then $n_r(\Phi \circ \Phi_\star)< n(\Phi)$.
\end{lemma}

\begin{proof}
Consider indeed a normalized conjugator $\gamma$ associated with $\Phi$.
Here $\Phi_\star \neq \id$ because $p$ splits with simple roots.

Remember that $[a,b]$ is a normalized conjugator of $\Phi_\star$. Moreover
$n_r(\Phi_\star)=1$ because $\Lambda_{p,q}=r$.
By Lemma \ref{lemma:normalizedconjugatorinJr} we find that both
$\gamma$ and $[a,b]$ belong to $\mathfrak{J}_r$. Then identity \eqref{eq:produitdansromega} yields
$ \gamma [a,b] \in (r(\omega))$.
Now, taking a normalized conjugator $\gamma_1$ of $\Phi \circ \Phi_\star$, we
deduce that $\gamma [a,b]= s\gamma_1$ for some $s \in C$ that is a multiple of $r(\omega)$.
Hence $N(\gamma_1)$ divides $\frac{N(\gamma [a,b])}{r(\omega)^2}\cdot$
Remembering that $N([a,b])=-\Lambda_{p,q}(\omega)=-r(\omega)$, we deduce that
$n_r(\Phi \circ \Phi_\star) < n_r(\Phi)$.
\end{proof}

Now, we can easily conclude.
Let $\Phi \in \Aut_C(\calW_{p,q}) \setminus \Inn(\calW_{p,q})$.
Then $n_r(\Phi)>0$ and we deduce from the previous lemma that
$n_r(\Phi \circ \Phi_\star)< n_r(\Phi)$. If
$\Phi \circ \Phi_\star$ were non-inner, then applying the same principle once more would yield
$n_r((\Phi \circ \Phi_\star) \circ \Phi_\star)<n_r(\Phi \circ \Phi_\star)<n_r(\Phi)$, which is absurd because $(\Phi_\star)^2=\id$.
Hence $\Phi \circ \Phi_\star$ is inner and $\Phi=(\Phi \circ \Phi_\star) \circ \Phi_\star$.
The proof is therefore complete in that case.

\subsection{Case V: When exactly one of $p$ and $q$ splits with a double root}

We finish the proof with the last case, which is the most difficult one.
Without loss of generality, we assume that $q=t^2$ and that $p$ does not split with a double root.
Hence $\Lambda_{p,q}=t^2$. Throughout, we set $r:=t$, and for a $C$-automorphism $\Phi$
we denote by $n(\Phi)$ the valuation of $\omega$ in $N(\gamma)$ for an arbitrary normalized conjugator $\gamma$ of $\Phi$.
Here we know from Proposition \ref{prop:allidealsonesplits} that $\mathfrak{R}_r$ is the $2$-sided ideal generated by $b_r$,
but also the left-ideal and the right-ideal generated by it.

At this point, the only valid information on the possible values of $n(\Phi)$ is that if $\Phi$
is the pseudo-adjunction and $p$ is separable, then $n(\Phi)=2$.
However, we will also use Lemma \ref{lemma:normalizedconjugatorinJr} to obtain the following information:

\begin{lemma}\label{lemma:nphi>=2}
Let $\Phi \in \Aut_C(\calW_{p,q})$ be such that $n(\Phi) > 0$.
Then $n(\Phi) \geq 2$.
\end{lemma}

\begin{proof}
Let $\gamma$ be a normalized conjugator of $\Phi$. We have seen in Lemma \ref{lemma:normalizedconjugatorinJr}
that $\gamma_r \in \mathfrak{R}_r$.
Hence $\gamma \equiv z b$ mod $(\omega)$ for some $z \in \calW_{p,q}$.
Yet $N(z b)=N(z) N(b)=0$ and $\gamma \in \mathfrak{J}_r$.
It follows from Lemma \ref{lemma:normmodsquare} that $N(\gamma) \equiv 0$ mod $(\omega^2)$,
and hence $n(\Phi) \geq 2$.
\end{proof}

Next, an important property of the ideal $\mathfrak{J}_r$ will be used repeatedly:

\begin{lemma}\label{lemma:productJr}
One has $xy\in \omega \mathfrak{J}_r$ for all $x,y$ in $\mathfrak{J}_r$.
\end{lemma}

\begin{proof}
Indeed we can write $\mathfrak{J}_r=(\omega)+b \F[a]$ and $\mathfrak{J}_r=(\omega)+\F[a] b$.
Let then $x$ and $y$ belong to $\mathfrak{J}_r$. Then $x \equiv x'b$ mod $(\omega)$ and
$y \equiv by'$ mod $(\omega)$ for some $x'$ and $y'$ in $\F[a]$.
Hence $x y \equiv x'bby'$ mod $\omega \mathfrak{J}_r$, which completes the proof because $b^2=0$.
\end{proof}

Now, let $\Phi \in \Aut_C(\calW_{p,q})$ be such that $n(\Phi) > 0$, to the effect that $n(\Phi) \geq 2$ (by Lemma \ref{lemma:nphi>=2}),
and let $\gamma$ be an associated normalized conjugator.
By Lemma \ref{lemma:normalizedconjugatorinJr} we can find $\alpha \in \F[a]^\times$ such that $\gamma \equiv b \alpha^\star$ mod $(\omega)$.
Then we can choose $u \in \F[a]+b \F[a]$ such that
$$\gamma \equiv b\alpha^\star+\omega u \quad \text{mod} \quad (\omega^2)$$
and we can further split
$$u=u_1+b u_2 \quad \text{with $u_1 \in \F[a]$ and $u_2 \in \F[a]$.}$$
Now we can better analyze the identity
$$\forall x \in \calW_{p,q}, \; \gamma x \gamma^\star \equiv 0 \quad \text{mod} \; (\omega^2).$$
Like in the previous cases, we successively reduce this identity to
$$\forall x \in \calW_{p,q}, \; (b \alpha^\star) x (b \alpha^\star)^\star
+\omega(u x (b \alpha^\star)^\star
+(b \alpha^\star) x u^\star) \equiv 0 \quad \text{mod} \; (\omega^2)$$
and to
\begin{equation}\label{eq:doublerootsx1}
\forall x \in \calW_{p,q}, \; (b \alpha^\star) x (b \alpha^\star)^\star \equiv \omega u(x-x^\star) \alpha b \quad \text{mod} \; (\omega^2)
\end{equation}
by using the fact that $\langle b\alpha^\star, u x^\star\rangle \equiv 0$ mod $(\omega)$ and that
$b^\star=-b$.
Finally, \eqref{eq:doublerootsx1} only requires an analysis for $x \in \F[a]+b \F[a]$
, which is a direct summand of $(\omega)$ in the additive group $\calW_{p,q}$
(indeed, by Lemma \ref{lemma:productJr} we directly have $(ba^\star) z (ba^\star)^\star \in (\omega)$ for all $z \in \calW_{p,q}$).

We note that \eqref{eq:doublerootsx1} gives us no information for $x=1$ (the left-hand side is then $N(b\alpha^\star)=0$, and the right-hand side vanishes, obviously). It neither gives us any information for $x \in b \F[a]$, as in that case $x$ and $x^\star$ belong to $\mathfrak{J}_r$,
and by Lemma \ref{lemma:productJr} (applied twice for the left-hand side) we have $(b \alpha^\star) x (b \alpha^\star)^\star \equiv 0$ mod $(\omega^2)$,
and by the same lemma $\omega u(x-x^\star) \alpha b \equiv 0 \quad \text{mod} \quad (\omega^2)$.
Hence our only hope to retrieve meaningful information is to apply \eqref{eq:doublerootsx1} to $x=a$
(and not even to $\alpha$, which at this point can very well belong to $\F$!).
Now, on the one hand
\begin{multline*}
(b \alpha^\star) a (b \alpha^\star)^\star=
b(\alpha^\star a \alpha) b^\star=b(N(\alpha)a) b^\star=N(\alpha) (bab^\star) \\
=N(\alpha)(\langle a,b\rangle b-b^2 a^\star)=N(\alpha) \omega b,
\end{multline*}
and on the other hand
$$u(a-a^\star) \alpha b \equiv u_1 (a-a^\star) \alpha b \quad \text{mod} \; (\omega)$$
by using Lemma \ref{lemma:productJr} once more. Because $u_1(a-a^\star) \alpha \in \F[a]$, we extract the identity
\begin{equation}\label{eq:doublerootsx1refined}
u_1(a-a^\star)\alpha=N(\alpha).
\end{equation}
This immediately discards the possibility that $p$ be inseparable.
Indeed, in that case $\alpha$ would be invertible, leading to $N(\alpha) \neq 0$, whereas $a-a^\star=0$.

Hence $p$ is separable, and in particular $a-a^\star$ is a nonzero element of $\F[a]$ with trace $0$,
and again because $p$ is separable and does not have a double root in $\F$ (thanks to our starting assumption) the element
$a-a^\star$ is invertible.

Now, we split the discussion in two cases, whether $\alpha$ is invertible or not.

\vskip 3mm
\noindent \textbf{Case 1: $\alpha$ is invertible.} \\
Since $a-a^\star$ is invertible, equation \eqref{eq:doublerootsx1refined} fully determines $u_1$ as a function
of $\alpha$. However, we cannot draw any condition on $u_2$ from \eqref{eq:doublerootsx1}.

Since $p$ has simple roots in $\K$ the pseudo-adjunction automorphism $\Phi_\star$ is not the identity,
to the effect that $[a,b]$ is an associated normalized conjugator.
We shall use the collapsing phenomenon once more to prove that $\Phi \circ \Phi_\star$ is inner.
For this, almost everything hinges on the observation that $u_1$ was fully determined by $\alpha$.
To start with, we remember from identity \eqref{eq:decompcommutator} in Section \ref{section:pseudoadjunction} that
$$[a,b]=-[b,a]=-\omega+\tr(a)\,b-2ba=b(a^\star-a)-\omega=b(a-a^\star)^\star-\omega.$$
Hence $\Phi_\star$ satisfies the very assumptions
we have just analyzed! Now we can replace $\Phi$ with $\Phi':=\Phi \circ i_{\alpha (a^\star-a)}$,
which does not change the problem of decomposing $\Phi$ as the product of an inner automorphism with $\Phi_\star$,
but reduces the situation to the one where $\alpha=a-a^\star$
because $\gamma \alpha (a^\star-a) \equiv b\alpha^\star \alpha (a-a^\star)^\star \equiv N(\alpha) b (a-a^\star)^\star$ mod $(\omega)$,
and we can replace $\gamma$ with $N(\alpha)^{-1} \gamma$.

Hence, we lose no generality in assuming that $\alpha=a-a^\star$, to the effect that
$\gamma \equiv [a,b]$ mod $(\omega)$. And now, since we have seen that the component $u_1$ is fully determined by $\alpha$,
by applying this observation to $\Phi$ and $\Phi_\star$ we can confidently
infer that $\gamma \equiv [a,b]$ mod $\omega \mathfrak{J}_r$.

We will now conclude by taking advantage of the collapsing phenomenon.
Since $[a,b]$ belongs to $\mathfrak{J}_r$ (see Lemma \ref{lemma:commutator}), we deduce from Lemma \ref{lemma:productJr} that
$\gamma [a,b] \equiv [a,b]^2$ mod $\omega^2 \mathfrak{J}_r$.
Since $\tr([a,b])=0$, we have $[a,b]^2=-N([a,b])=\Lambda_{p,q}(\omega)=\omega^2$.
Hence $\gamma [a,b]=\omega^2(1+z)$ for some $z \in \mathfrak{J}_r$, and we infer that
$1+z$ is a conjugator of $\Phi \circ \Phi_\star$.
Moreover since $\omega^4 N(1+z)=N(\gamma)N([a,b])=-\omega^2 N(\gamma)$ we find that
$N(1+z)$ divides $N(\gamma)$, and hence $N(1+z)$ is a power of $\omega$ multiplied by a nonzero scalar.
Remembering that $z \in \mathfrak{J}_r$, we see that $N(1+z) \equiv 1$ mod $(\omega)$ and conclude that
$N(1+z)=1$. In particular $1+z$ is normalized, and we conclude that $\Phi \circ \Phi_\star$ is inner.
This completes the proof in the case where $\alpha$ is invertible.
Note in particular that $n(\Phi)=2$.

\vskip 3mm
\noindent \textbf{Case 2: $\alpha$ is not invertible.} \\
In that case the polynomial $p$ splits with simple roots, and without loss of generality we will assume that $\alpha=a$
and that $a^2=a$.
We will prove that this situation leads to a contradiction,
but reaching such a contradiction will require congruences modulo $(\omega^3)$, and hence
a substantial dose of additional computation.

First of all, we go right back to identity \eqref{eq:doublerootsx1refined}, which now simplifies as
$u_1 a=0$. Hence $u_1=\lambda a^\star$ for some $\lambda \in \F$.

A critical observation now is that $n(\Phi) \geq 3$. Indeed, as before we note that
$ba^\star \in \mathfrak{J}_r$ to obtain the congruence
$$N(\gamma)\equiv N(ba^\star+\lambda \omega a^\star+\omega b u_2) \quad \text{mod} \; (\omega^3).$$
Yet, thanks to $N(a)=0$ and $N(b)=0$, we find by expanding that
$$N(ba^\star+\lambda \omega a^\star+\omega bu_2)=
\lambda \omega^2 \langle a^\star,bu_2\rangle \equiv 0 \quad \text{mod} \; (\omega^3),$$
where the latter congruence comes from the observation that $bu_2 \in \mathfrak{J}_r$.
Combining the previous two congruences yields that $\omega^3$ divides $N(\gamma)$, i.e., $n(\Phi) \geq 3$.
Hence we now have the reinforced statement
\begin{equation}\label{eq:keyid3}
\forall x \in \calW_{p,q}, \; \gamma x \gamma^\star \equiv 0 \quad \text{mod} \; (\omega^3)
\end{equation}
and we will see that it yields a final contradiction.
Note that we could try using a collapsing argument, but unfortunately one can check that $\Phi \circ \Phi_\star$
satisfies exactly the same assumptions as $\Phi$ (the reader will easily compute that $\gamma[a,b]=\omega \gamma'$ for some normalized $\gamma'$).

The contradiction will actually come from applying identity \eqref{eq:keyid3} to $x=b$.
For this, we introduce $v \in \calW_{p,q}$ such that $\gamma=ba^\star+\omega u +\omega^2 v$, and we expand the left-hand side
of \eqref{eq:keyid3} to obtain
\begin{equation}\label{eq:ultimatecongruencemodomega3}
s_1+s_2+s_3+s_4 \equiv 0 \quad \text{mod} \; (\omega^3)
\end{equation}
for
$$s_1:=(ba^\star) b (ba^\star)^\star ; \quad s_2:=\omega (u b (ba^\star)^\star+(ba^\star) b u^\star);$$
$$s_3:=\omega^2 u b u^\star \quad \text{and} \quad s_4:=\omega^2 (vb (ba^\star)^\star+(ba^\star) b v^\star),$$

Let us analyze each summand separately.
\begin{itemize}
\item To start with, $s_1=ba^\star b a b^\star=\omega bab^\star=\omega^2 b$ thanks to $b^2=0$.
\item That $s_4 \equiv 0$ mod $(\omega^3)$ directly follows from Lemma \ref{lemma:productJr}.
\item Lemma \ref{lemma:productJr} also yields $s_3 \equiv \omega^2 u_1 b u_1^\star \equiv \lambda^2 \omega^2 a^\star ba$ mod $(\omega^3)$.
Yet $a^\star ba=\omega a-b^\star a^2=\omega a-b^\star a=\omega a+ba$.
Hence
$$s_3 \equiv  \lambda^2 \omega^2 ba \quad \text{ mod $(\omega^3)$.}$$
\item Two additional rounds of Lemma \ref{lemma:productJr} yield
$s_2 \equiv \omega (u_1 b (ba^\star)^\star+(ba^\star) b u_1^\star)$ mod $(\omega^3)$.
Now we compute
$$u_1 b (ba^\star)^\star=\omega u_1 b-u_1 b(b a^\star)=\lambda \omega a^\star b$$
and hence
$$u_1 b (ba^\star)^\star=\lambda \omega^2-\lambda \omega b^\star a=\lambda \omega^2+\lambda \omega b a.$$
Besides $(ba^\star) b u_1^\star=\omega b u_1^\star=\lambda \omega ba$.
Hence,
$$s_2 \equiv 2\lambda \omega^2\,ba \quad \text{mod} \; (\omega^3).$$
\end{itemize}
Hence we derive from \eqref{eq:ultimatecongruencemodomega3} that
$$\omega^2\,b+\omega^2 (\lambda^2+2\lambda)\, ba \equiv 0 \quad \text{mod} \; (\omega^3).$$
By extracting coefficients in the deployed basis $(1,b,a,ba)$, we end up with $\omega^2 \equiv 0$ mod $(\omega^3)$,
which is absurd.

This final contradiction completes the proof of the Automorphisms Theorem when exactly one of $p$ and $q$ splits with a double root.
The proof is therefore complete in all cases, at last!

\section{Units in the free Hamilton algebra (part 2)}\label{section:units2}

In this section, we pick up the study of the group of units $\calW_{p,q}^\times$
where we left it at the end of Section \ref{section:units1}. There, we had seen
that the methods were insufficient to fully understand the group $\calW_{p,q}^\times$ when at least one of $p$ and $q$ splits,
i.e., in that case some units are not products of basic units, and we must therefore find a larger generating subset of units.
Very simply, this will involve a generalization of the examples of units we have considered in Section \ref{section:units1counterexamples}.

\subsection{Semi-basic units}\label{section:semibasicunits}

To define our larger generating subset, we start by constructing
special units that are not basic in general.

Let $\alpha$ be a nonzero basic element with norm $0$ (i.e., a zero divisor in a basic subalgebra).
For all $x \in \calW_{p,q}$ we note that
$$N(1+\alpha x^\star)=N(1)+\langle 1,\alpha x^\star\rangle+N(\alpha)N(x^\star)=1+\langle \alpha,x\rangle,$$
and in particular $1+\alpha x^\star$ has norm $1$ if and only if $\langle \alpha,x\rangle=0$.
We say that $1+\alpha x^\star$ is a \textbf{semi-basic} unit attached to $\alpha$ whenever $\langle \alpha,x\rangle=0$.
\index{Semi-basic unit}
In that case, we can write its inverse indifferently as $1+x\alpha^\star$ or as $1-\alpha x^\star$.

\begin{Not}
Let $\alpha$ be a basic zero divisor.
We denote by $\SB(\alpha)$ the set of all semi-basic units attached to $\alpha$.
\label{page:semibasicsubgroupelement}
\end{Not}

Obviously $\SB(\lambda \alpha)=\SB(\alpha)$ for all $\lambda \in \F^\times$, so the set
$\SB(\alpha)$ is really a function of the line $\F \alpha$, and in what follows we will almost always consider only the
situations where $\alpha$ satisfies $\alpha^2=\alpha$ or $\alpha^2=0$.

Noting that
$$(1+\alpha x^\star)(1+\alpha y^\star)=1+\alpha (x+y)^\star+\alpha (\langle x,\alpha\rangle-\alpha^\star x)y^\star =1+\alpha(x+y)^\star$$
for all $x,y$ in $\calW_{p,q}$ such that $\langle \alpha,x\rangle=\langle \alpha,y\rangle=0$, we find that
$\SB(\alpha)$ is a commutative subgroup of $\calW_{p,q}^\times$.

In any case, we shall recognize in $\SB(\alpha)$ a subgroup of the group of units of one of the
algebras that were introduced in Section \ref{section:newsubalgebras}.

\begin{itemize}
\item First of all, if $\alpha$ is idempotent we recognize that $\SB(\alpha)=1+\alpha^\sharp$, which can be seen as a subgroup of
the group of units of the algebra $\calH(\alpha,C)$,
and also as a subgroup of the group of units of $\calU(z)$ where $z$ stands for an arbitrary normalized vector in $\alpha^\sharp$.
Moreover, in choosing an arbitrary nonscalar element $\beta$ in the opposite basic subalgebra of $\F[\alpha]$, one sees
that $z:=\alpha y^\star$ is a correct choice if we take $y:=\langle \alpha,\beta\rangle-\beta$.

\item Assume now that $\alpha^2=0$. Denote by $\calD$ the basic subalgebra opposite to $\F[\alpha]$.
Let $y \in \calW_{p,q}$ and split $y=y'+y'' \alpha$ where $y'$ and $y''$ belong to $\Vect_C(\calD)$.
Then $\langle \alpha,y\rangle=\langle \alpha,y'\rangle$, and $\langle \alpha,y'\rangle=0$ if and only if $y' \in C$.
Finally $\alpha y^\star=\alpha (y')^\star$.
Hence, in that case the semi-basic units attached to $\alpha$ are the elements of $1+C \alpha$,
which is a subgroup of the group of units of $\calU(\alpha)$.
\end{itemize}

In any case, the group of semi-basic units attached to $\alpha$ is isomorphic to $(C,+)$ and hence to $(\F[t],+)$.

\begin{Rem}\label{remark:leadingsemibasic}
Recalling the terminology of leading vectors from Section \ref{section:units1},
it will be useful at some point to observe that $\alpha$ is leading for every \emph{nonbasic}
vector of $\SB(\alpha)$, which is easily seen from the previous study.
Note to this end that every element $x$ of $\SB(\alpha)$ satisfies $(x-1)^2=0$ and hence is quadratic.
\end{Rem}

We can already state one of our main results, but proving it is premature:

\begin{theo}\label{theo:refinedunits}
The group of units of $\calW_{p,q}$ is generated by the basic units and the semi-basic units.
\end{theo}

In fact, we will prove a much more powerful result on the structure of $\calW_{p,q}^\times$,
see Theorems \ref{theo:strongunits} and \ref{theo:semibasicunits} in the next section.

\subsection{Semi-basic subgroups}

Now, we attach a subgroup to each basic subalgebra, as follows, so as to extend the basic units.

\label{page:semibasicsubgroupalgebra}
\begin{Def}
Let $\calC$ be a basic subalgebra of $\calW_{p,q}$.
The \textbf{semi-basic subgroup} attached to $\calC$, denoted by $\SB(\calC)$, is defined as
\index{Semi-basic subgroup}
the subgroup of units generated by:
\begin{itemize}
\item the basic units in $\calC$;
\item the semi-basic units that are associated with zero divisors in $\calC$.
\end{itemize} By using only the generators of the second kind, we obtain a subgroup of $\SB(\calC)$
that we call the \textbf{special semi-basic subgroup} and denote by $\SSB(\calC)$.
\index{Special semi-basic subgroup}
\end{Def}
\label{page:specialsemibasicsubgroup}

Of course, $\SB(\calC)$ equals $\calC^\times$ when $\calC$ is a field, but in any other case it is larger than $\calC$,
as seen in Section \ref{section:units1counterexamples}.

Assume now that $\calC$ is not a field.
There are two subcases.

\begin{itemize}
\item If $\calC$ degenerates then it has a unique zero divisor $\alpha$ up to multiplication with a non-zero scalar,
and all its semi-basic units belong to $\Vect_C(\calC)$. It is then easily seen that $\SB(\calC)$
is the set of all elements of the form $\lambda+r \alpha$ with $r \in C$ and $\lambda \in \F^\times$,
and $\SB(\calC)$ is simply the group of units of the subring $\Vect_C(\calC)$, as well as the group
of units of the subalgebra $\calU(\alpha)$.

\item In contrast, if $\calC$ splits then it has, up to multiplication with nonzero scalars, exactly two zero divisors.
Then $\SB(\calC)$ is defined by three distinct subsets of generators, and it is easily seen that it is noncommutative.
\end{itemize}
In any case, $\SB(\calC)$ includes $\F^\times$.

Finally, we note that $\SSB(\calC)$ is a normal subgroup of $\SB(\calC)$. To see this, it suffices to check
that $\SSB(\calC)$ is invariant under conjugation by the basic units of $\calC$. But this is easy: take a zero divisor $\alpha$ in $\calC$, and take a basic unit $\gamma$ in $\calC^\times$.
Let $x \in \calW_{p,q}$ be such that $\langle \alpha,x\rangle=0$.
Then $\gamma(1+\alpha x^\star) \gamma^{-1}=1+\alpha (\gamma x^\star \gamma^{-1})=1+\alpha ((\gamma^{-1})^\star x \gamma^\star)^\star$
because $\gamma$ centralizes $\calC$, and since $\gamma(1+\alpha x^\star) \gamma^{-1}$ has norm $1$ it belongs to $\SB(\alpha)$.

We can now state the two ultimate structural results on the group of units:

\begin{theo}[Strong Units Theorem]\label{theo:strongunits}
The inclusions of $\SB(\F[a])$ and $\SB(\F[b])$ into $\calW_{p,q}^\times$ induce
an isomorphism between $\calW_{p,q}^\times$ and the amalgamated product $\SB(\F[a]) \underset{\F^\times}{*} \SB(\F[b])$.
\end{theo}

\begin{theo}[Semi-Basic Units Theorem]\label{theo:semibasicunits}
Let $\calC$ be a split basic subalgebra of $\calW_{p,q}$, and let $\alpha$ be a zero divisor in it.
Then:
\begin{enumerate}[(i)]
\item The basic subgroup $\calC^\times$ is a semi-direct factor of $\SSB(\calC)$ in $\SB(\calC)$.
\item The inclusions of $\SB(\alpha)$ and $\SB(\alpha^\star)$ in $\SSB(\calC)$ induces an
isomorphism between $\SSB(\calC)$ and the free product $\SB(\alpha) * \SB(\alpha^\star)$.
\end{enumerate}
\end{theo}

Combining the previous two theorems gives us a complete picture of $\calW_{p,q}^\times$, and the situation is even more remarkable
in terms of the group $\Inn(\calW_{p,q})$ of inner automorphisms:

\begin{cor}
The inclusions of $\mathrm{P}\SB(\F[a])$ and $\mathrm{P}\SB(\F[b])$ into $\Inn(\calW_{p,q})$ induce
an isomorphism between $\Inn(\calW_{p,q})$ and the free product
$\mathrm{P}\SB(\F[a]) * \mathrm{P}\SB(\F[b])$.
\end{cor}

Moreover, if $\calC$ is a degenerate basic subalgebra of $\calW_{p,q}$ then
$\mathrm{P}\SB(\calC)$ is naturally isomorphic to $\SB(\alpha)$ for an arbitrary zero divisor $\alpha \in \calC$,
and hence $\mathrm{P}\SB(\calC)$ is commutative (and isomorphic to $(\F[t],+)$). Finally, the situation of a split basic subalgebra is described below:

\begin{cor}
Let $\calC$ be a split basic subalgebra of $\calW_{p,q}$, and $\alpha$ be a zero divisor in it.
Then:
\begin{enumerate}[(i)]
\item The projective basic units subgroup $\mathrm{P}\calC^\times$ is a semi-direct factor of $\mathrm{P}\SSB(\calC)$ in $\mathrm{P}\SB(\calC)$.
\item The projective subgroup $\mathrm{P}\SB(\alpha)$ is naturally isomorphic to $\SB(\alpha)$, and the
projective subgroup $\mathrm{P}\SB(\alpha^\star)$ is naturally isomorphic to $\SB(\alpha^\star)$.
\item The inclusions of $\mathrm{P}\SB(\alpha)$ and $\mathrm{P}\SB(\alpha^\star)$ in $\mathrm{P}\SB(\calC)$ induce
an isomorphism between $\mathrm{P}\SB(\calC)$ and the free product $\mathrm{P}\SB(\alpha) * \mathrm{P}\SB(\alpha^\star)$.
\end{enumerate}
\end{cor}

Combining these results with the Automorphisms Theorem and the description of the group of basic automorphisms,
we obtain a complete picture of $\Aut(\calW_{p,q})$. We will give several applications of this description in Section \ref{section:misc}.

Our strategy to prove the above results entirely relies upon an adaptation of the retracing algorithm from Section \ref{section:units1}.
There, we discovered that the obstructions consisted in the situations where the leading vector is a zero divisor,
which forbade us to use basic units. The patch seems obvious: when we meet such an obstruction, we will replace
the basic units with semi-basic units that are associated to the said zero divisor.
Hence we must investigate the effect of conjugating a quadratic vector with a semi-basic unit.

\begin{Def}
A non-void sequence $(x_1,\dots,x_n)$ (of mathematical objects) is called \textbf{strongly $2$-periodical}
when:
\index{Strongly $2$-periodical (sequence)}
\begin{enumerate}[(i)]
\item For all $i \in \lcro 1,n-2\rcro$, one has $x_{i+2}=x_i$;
\item If $n>1$ then $x_2 \neq x_1$.
\end{enumerate}
\end{Def}

\begin{Def}\label{def:reduceddecomposition}
Given a nonscalar unit $\gamma \in \calW_{p,q}^\times \setminus \F^\times$, a \textbf{reduced decomposition} of $\gamma$
is a decomposition $\gamma=\gamma_1 \cdots \gamma_n$
\index{Reduced decomposition (of a unit)}
in which there is a strongly $2$-periodical sequence $(\calC_1,\dots,\calC_n)$ of basic subalgebras
such that $\gamma_i \in \SB(\calC_i) \setminus \F^\times$ for all $i \in \lcro 1,n\rcro$.
\end{Def}

The Refined Units Theorem yields that every nonscalar unit has a reduced decomposition, and the Strong Units
Decomposition that such a reduced decomposition is unique up to multiplication of each factor with a nonzero scalar.

\begin{Def}\label{def:specializeddecomposition}
Let $\calC$ be a split basic subalgebra of $\calW_{p,q}$.
Given a nontrivial unit $\gamma \in \SSB(\calC) \setminus \{1\}$, a \textbf{specialized decomposition} of $\gamma$
is a decomposition $\gamma=\gamma_1 \cdots \gamma_n$ in which
\index{Specialized decomposition}
there is a strongly $2$-periodical sequence $(\alpha_1,\dots,\alpha_n)$ of nontrivial idempotents of $\calC$
such that $\gamma_i \in \SB(\alpha_i) \setminus \{1\}$ for all $i \in \lcro 1,n\rcro$.
\end{Def}

By the Semi-Basic Units Theorem, every nontrivial element of $\SSB(\calC)$
has a unique specialized decomposition.

\begin{Def}\label{def:specializeddecomposition2}
Let $\calC$ be a split basic subalgebra of $\calW_{p,q}$.
Given a nonscalar unit $\gamma \in \SB(\calC) \setminus \F^\times$, a
\textbf{specialized decomposition} of $\gamma$ is a decomposition $\gamma \sim \gamma_1 \cdots \gamma_n$ in which:
\index{Specialized decomposition}
\begin{itemize}
\item Either $\gamma_1 \in \calC^\times \setminus \F^\times$ and $n=1$.
\item Or $\gamma_1 \in \calC^\times \setminus \F^\times$, $n>1$ and $(\gamma_2,\dots,\gamma_n)$ is a specialized decomposition
of some element of $\SSB(\calC) \setminus \{1\}$.
\item Or $(\gamma_1,\dots,\gamma_n)$ is a specialized decomposition
of some element of $\SSB(\calC) \setminus \{1\}$.
\end{itemize}
\end{Def}

By the Semi-Basic Units Theorem, every nontrivial element of $\SSB(\calC)$
has a unique specialized decomposition, up to multiplication of the first factor by a nonzero scalar if the first factor is a basic unit.

\subsection{The effect of conjugating with a semi-basic unit}

We start by recalling a result that we have obtained in Section \ref{section:units1}
(see point (b) of Lemma \ref{lemma:conjugatebasicunit}):

\begin{lemma}\label{lemma:effectofbasic}
Let $x \in \calW_{p,q} \setminus \F$ be a nonscalar quadratic vector, with trailing subalgebra denoted by $\calD$.
Let $\beta \in \calD^\times \setminus \F$.
Then $\beta x \beta^{-1}$ is nonbasic and $\beta$ is a leading vector for it.
\end{lemma}

One of our aims is to prove the following analogue for semi-basic units.

\begin{lemma}\label{lemma:effectofsemibasic}
Let $x \in \calW_{p,q}  \setminus \F$ be a nonscalar quadratic vector, with leading basic vector $y$.
Let $\alpha$ be a basic zero divisor, and let $\gamma \in \SB(\alpha)$ be nonbasic.
In any of the following cases, $\gamma x \gamma^{-1}$ is nonbasic, $\alpha$ is leading for it and $\delta(\gamma x \gamma^{-1})>\delta(x)$:
\begin{enumerate}[(a)]
\item $\F[\alpha]$ is trailing for $x$;
\item $\alpha \notin \F y$ and $\alpha^2 \neq 0$;
\item $\alpha \not\in \F y$ and $x$ is nonbasic.
\end{enumerate}
\end{lemma}

We also need a critical step for enhancing the retracing algorithm.
It involves the following special case:

\begin{Def}
Let $\alpha$ be a basic zero divisor.
A vector $x \in \calW_{p,q}$ is called \textbf{special degenerate} attached to $\alpha$ when
$x \in \F+\SB(\alpha)$ and $x$ is nonbasic.
\index{Special degenerate (element)}
\end{Def}

It is easily seen that every such vector $x$ is quadratic and that the subalgebra $\F[x]$ is degenerate.
Here is what we shall prove:

\begin{lemma}\label{lemma:retracingstepspecial}
Let $x \in \calW_{p,q}$ be a nonbasic quadratic vector, with leading vector $\alpha$.
Assume that $\alpha$ is a zero divisor and that $x$ is not special degenerate.
Then there exists $\gamma \in \SB(\alpha)$ such that $\delta(\gamma x \gamma^{-1})<\delta(x)$.
\end{lemma}

\begin{Rem}
If $x \in \F + \SB(\alpha)$ then as $\F+\SB(\alpha)$ is a commutative subalgebra of $\calW_{p,q}$
it is clear that there is no $\gamma \in \SB(\alpha)$ such that $\delta(\gamma x \gamma^{-1})<\delta(x)$.
Moreover, in that case it is not difficult to see that $\alpha$ is leading for $x$ if in addition $x$ is nonscalar,
and hence by Lemma \ref{lemma:effectofsemibasic} there is no other reasonable choice of semi-basic unit
to decrease the absolute distance after conjugating.
\end{Rem}

This result, along with other ones, will be proved in the course of the next subsections.

\subsubsection{The effect of conjugating with a semi-basic unit: split case}\label{section:effectsemibasicsplit}

We start with semi-basic units associated with a split basic subalgebra.
We let $\alpha$ be a nontrivial basic idempotent, and we choose a nonscalar element $\beta$ in the basic subalgebra opposite to $\F[\alpha]$.
We will consider the conjugation by a semi-basic unit in $\SB(\alpha) \setminus \{1\}$.
To this end, we set $\omega':=\langle \alpha,\beta\rangle$.

So, let us take $\gamma=1+\alpha y^\star$, where $y=r\,(\omega'-\beta)$ for some $r \in C \setminus \{0\}$. We will use $(1,\alpha,\beta^\star,\alpha \beta^\star)$ as our deployed basis.
Let $x \in \calW_{p,q}$, which we write
$$x=x_1+x_\alpha \alpha+x_{\beta^\star} \beta^\star+x_{\alpha \beta^\star} \alpha \beta^{\star}.$$
Since $\gamma$ commutes with $\alpha y^\star$, computations are easier by changing the representation and using the four vectors $1,\alpha,y^\star,\alpha y^\star$
instead of the deployed basis. Since $\beta^\star=\omega'-r^{-1} y^\star$, this leads to
$$x=(x_1+\omega' x_{\beta^\star})\,1+(x_\alpha+\omega' x_{\alpha \beta^\star})\,\alpha+(-r^{-1} x_{\beta^\star})\, y^\star+(-r^{-1} x_{\alpha \beta^\star})\, \alpha y^\star.$$
Next, we compute the four conjugates $\gamma 1 \gamma^{-1}$, $\gamma \alpha \gamma^{-1}$, $\gamma y^\star \gamma^{-1}$
and $\gamma \alpha y^\star \gamma^{-1}$.
The first and fourth are obvious, being equal to $1$ and $\alpha y^\star$, respectively.
Next,
$$\gamma \alpha \gamma^{-1}=(1-y\alpha^\star)\alpha (1-\alpha y^\star)=\alpha(1-\alpha y^\star)=\alpha-\alpha y^\star.$$
Finally
\begin{align*}
\gamma y^\star \gamma^{-1} & = (1+\alpha y^\star)y^\star(1+y \alpha^\star) \\
& = (1+\alpha y^\star)y^\star+N(y)\,(1+\alpha y^\star) \alpha^\star \\
& = y^\star+\alpha (\tr(y^\star)y^\star-N(y^\star))+N(y)\alpha^\star-N(y) y (\alpha^\star)^2 \\
& = N(y)-2N(y)\alpha+y^\star +\tr(y)\,\alpha y^\star-N(y)\,y \alpha^\star \\
& = N(y)-2N(y)\alpha+y^\star +\bigl(\tr(y)+N(y)\bigr)\,\alpha y^\star.
\end{align*}
Hence we obtain
\begin{align*}
\gamma x \gamma^{-1} = & \left[ x_1+\omega' x_{\beta^\star}-r^{-1}  N(y) x_{\beta^\star}\right]\ \\
 & + \left[(x_\alpha+\omega' x_{\alpha \beta^\star})+2 r^{-1} N(y) x_{\beta^\star}\right]\,\alpha \\
 & +\left[-r^{-1} x_{\beta^\star}\right] \,y^\star \\
 & +\left[-(x_\alpha+\omega' x_{\alpha \beta^\star})-r^{-1} x_{\beta^\star} (\tr(y)+N(y))-r^{-1} x_{\alpha \beta^\star}\right]\,\alpha y^\star.
\end{align*}
Finally, by using $y^\star=r \omega'-r\beta^\star$, we recover that the first three coefficients of $x':=\gamma x \gamma^{-1}$ in the deployed basis $(1,\alpha,\beta^\star,\alpha \beta^\star)$ are, respectively,
$$\begin{cases}
x'_1 = x_1-r^{-1} N(y) x_{\beta^\star} \\
x'_\alpha = x_\alpha-\omega' r (x_\alpha+\omega' x_{\alpha \beta^\star})+\bigl(2r^{-1}N(y)-\omega' (\tr(y)+N(y))\bigr) x_{\beta^\star} \\
x'_{\beta^\star} = x_{\beta^\star.}
\end{cases}$$
Notice in particular the invariance of the coefficient on $\beta^\star$ when replacing $x$ with $x'$.

From now on, we assume that $x$ is quadratic and nonscalar.

\vskip 2mm
\noindent \textbf{Situation 1:} $\F[\beta]$ is the leading subalgebra of $x$. \\
Then $\deg(x_\alpha) < \deg(x_{\beta^\star})$.
Denote by $n$ the degree of $x_{\beta^\star}$. Then we know from Lemma \ref{lemma:quadraticdeployeddegree1} that
$\deg(x_{\alpha \beta^\star}) \leq n-1$ and $\deg(x_1) \leq n$.
Denote by $s$ the degree of $r$.
We observe that $r^{-1} N(y)=r ((\omega')^2-\omega' \tr(\beta)+N(\beta))$ has degree $s+2$ and that $\deg(\tr (y)) \leq s+1$.
Hence $\deg(x'_1) \leq n+s+2$, $\deg(\omega' r (x_\alpha+\omega' x_{\alpha \beta^\star})) \leq n+s+1$
and $\deg((2r^{-1} N(y)-\omega' \tr(y)) x_{\beta^\star}) \leq n+s+2$.
However $\deg(-\omega' N(y) x_{\beta^\star})=n+2s+3>n+s+2$.
Hence $\deg(x'_\alpha)=n+2s+3$, $\deg(x'_1) \leq n+s+2$ and $\deg(x'_{\beta^\star})=n<n+2s+3$.
Therefore $x'$ is nonbasic and $\alpha$ is leading for it, with $\delta(x')>\delta(x)$.

\vskip 2mm
For the remaining two situations, we come back to the analysis undertaken in Lemma \ref{lemma:quadraticdeployeddegree1} and reorganize the expression of
the norm of $x$, thanks to $N(\alpha)=0$ and $\tr(\alpha)=1$, as
\begin{equation}\label{eq:devnorme2}
N(x)=-x_1^2+\tr(x) x_1+N(\beta^\star)  x_{\beta^\star}^2+\langle \alpha,\beta^\star\rangle x_\alpha x_{\beta^\star}
+N(\beta^\star) x_{\beta^\star} x_{\alpha \beta^\star.}
\end{equation}
From now on, we set $n:=\deg(x_\alpha)$.

\vskip 2mm
\noindent \textbf{Situation 2:} $\F[\alpha]$ is the leading subalgebra of $x$ but $\alpha$ is not leading for $x$. \\
In that case we must prove that $\alpha$ is leading for $x'$ and that $\delta(x')>\delta(x)$, whatever the choice of $r$.
Note that $n=\deg(x_1)$, $\deg(x_{\beta^\star})<n$ and $\deg(x_{\alpha \beta^\star})<n$.
If $n>0$ then $\deg(\tr(x) x_1+N(\beta^\star)  x_{\beta^\star}^2+N(\beta^\star) x_{\beta^\star} x_{\alpha \beta^\star})<2n$,
whereas $x_1^2$ has degree $2n$; remembering that $N(x) \in \F$ we deduce from identity \eqref{eq:devnorme2} that $\deg(\langle \alpha,\beta^\star\rangle x_\alpha x_{\beta^\star})=2n$,
and hence $\deg(x_{\beta^\star})=n-1$.
Now we split the discussion into two subcases.
\begin{itemize}
\item Assume first that $n>0$.
Set $s:=\deg(r)$. Then, as in Situation 1 we find that
$\bigl(2r^{-1}N(y)-\omega' (\tr(y)+N(y))\bigr) x_{\beta^\star}$ has degree $(n-1)+2s+3=n+2s+2$,
whereas $\omega' r (x_\alpha+\omega' x_{\alpha \beta^\star})$ has degree at most $n+s+1$, and $\deg(x_\alpha)=n$.
Hence $\deg(x'_\alpha)=n+2s+2$. Likewise we note that $\deg(r^{-1} N(y)x_{\beta^\star})=s+2+(n-1)=n+s+1>\deg(x_1)$
and hence $\deg(x'_1)=n+s+1$. Since $\deg(x'_{\beta^\star})=\deg(x_{\beta^\star})=n-1<(n+2s+2)-1$
and $\deg(x'_1)=n+s+1<n+2s+2=\deg(x'_\alpha)$, we deduce that $\alpha$ is leading for $x'$ and that
$\delta(x')=n+2s+3>n+1=\delta(x)$.
\item Assume finally that $n=0$. Then $x_{\beta^\star}=0$ and $x_{\alpha \beta^\star}=0$,
so $x'_1=x_1$ and $x'_\alpha=(1-\omega'r) x_\alpha$, whence $\deg(x'_\alpha)=n+\deg(r)+1>0$ and $\deg(x'_\alpha)>n \geq \deg(x'_1)$, so
$x'$ is nonbasic and $\alpha$ is leading for $x'$.
\end{itemize}

\vskip 2mm
\noindent \textbf{Situation 3:} $x$ is nonbasic and $\alpha$ is leading for $x$. \\
In that case we must prove that $r$ can be chosen so that $\deg(x'_\alpha)<\deg(x_\alpha)$ unless $x$ is special degenerate.
To do this, the trick will be to try to have $\deg(x'_1)<\deg(x_1)$, if possible
(this seems easier, considering the complexity of the expression of $x'_\alpha$).

Set $d:=\deg(x_1)$.
Since $\alpha$ is leading for $x$ we have $d<n$.
Now, it is obvious that either $\deg(\langle \alpha,\beta^\star\rangle x_\alpha x_{\beta^\star})>\deg(N(\beta^\star) x_{\beta^\star} x_{\alpha \beta^\star})$
or both degrees equal $-\infty$, and in any case
the degree of $N(\beta^\star)  x_{\beta^\star}^2+\langle \alpha,\beta^\star\rangle x_\alpha x_{\beta^\star}+N(\beta^\star) x_{\beta^\star} x_{\alpha \beta^\star}$ equals $1+n+\deg(x_{\beta^\star})$,
and in particular it equals $-\infty$ if and only if $x_{\beta^\star}=0$.
Remember finally that $N(x)$ and $\tr(x)$ are constant because $x$ is quadratic.
\begin{itemize}
\item Assume first that $d>0$. Then $\deg(-x_1^2+\tr(x) x_1)=2d$, and we deduce from \eqref{eq:devnorme2} that we must have
$x_{\beta^\star} \neq 0$ and then $1+n+\deg(x_{\beta^\star})=2d$, whence $\deg(x_{\beta^\star})=2d-n-1$ and in particular
$\deg(x_{\beta^\star})<n-1$.
\item If $d \leq 0$ then we obtain likewise $1+n+\deg(x_{\beta^\star}) \leq 0$, which leads to
$x_{\beta^\star}=0$. In that case $x_1$ is constant.
\end{itemize}
In particular, and this is critical to our proof,
$$x_{\beta^\star} \neq 0 \Rightarrow \deg(x_\alpha)=2\deg(x_1)-\deg(x_{\beta^\star})-1.$$
Assume first that $x_{\beta^\star} \neq 0$. Then $2d-n-1\geq 0$ by the previous analysis, and in particular $d>0$.
We observe that $r^{-1}N(y)=r((\omega')^2-\omega' \tr(\beta)+N(\beta))$ belongs to $C$, has degree $\deg(r)+2$ and the same leading
coefficient as $r$ with respect to $\omega'$.
We also note that $2+\deg(x_{\beta^\star})=2d-n+1 \leq d$ because $d \leq n-1$,
and hence we can choose $r$ such that $r^{-1}N(y) x_{\beta^\star}$ has degree $d$ and the same leading coefficient as $x_1$ with respect to $\omega'$.
With this choice we find $\deg(x'_1)<\deg(x_1)$.
Now, with this choice we can be assured that $\deg(x'_\alpha)<\deg(x_\alpha)$: indeed, if $\deg(x'_\alpha) \geq \deg(x_\alpha)$ we are exactly in the previous position for $x'$
(having $\deg(x'_1)<\deg(x'_\alpha)$, $\deg(x'_{\beta^\star})<\deg(x'_\alpha)$ and $x'_{\beta^\star} \neq 0$, which in particular suffices to see that $x'$ is nonbasic),
to the effect that $\deg(x'_\alpha)=2 \deg(x'_1)-\deg(x'_{\beta^\star})-1<2 \deg(x_1)-\deg(x_{\beta^\star})-1=\deg(x_\alpha)$, and this is absurd.
Hence, with that precise choice of $r$ our aim is fulfilled, and we have $\delta(x')\leq n <\delta(x)$.

Assume next that $x_{\beta^\star}=0$.
Then $x_1$ is constant and no generality is lost in replacing $x$ with $x-x_1$, i.e., in assuming that $x_1=0$.
Observing that $\tr(x)=2x_1+x_\alpha+\omega' x_{\alpha \beta^\star}=x_\alpha+\omega' x_{\alpha \beta^\star}$,
we have the simplified identity $x'_\alpha=x_\alpha-\omega' r \tr(x)$.
Moreover, $x=\alpha (x_\alpha+x_{\alpha \beta^\star} \beta^\star)$, so $N(x)=0$.
Since $x$ is not basic we must have $n>0$, i.e., $\deg(x_\alpha)>0$.
\begin{itemize}
\item If $\tr(x) \neq 0$ we see again that $r$
can be adjusted so that $\deg(x'_\alpha)<\deg(x_\alpha)$, and hence $\delta(x') \leq n<\delta(x)$ (in that case, it is actually not difficult to see that
we can directly adjust $r$ so that $x' \in \F[\alpha]$).

\item Assume finally that $\tr(x)=0$. Then $x=\alpha z^\star$ for some $z \in \F[\beta]$ such that $\langle \alpha,z\rangle=0$, and
hence $x \in \alpha^\sharp$. In that case $x$ is special degenerate, attached to $\alpha$ (remember that we have assumed from the start that $x$ is nonbasic).
\end{itemize}
Hence, in releasing the assumption that $x_1 \neq 0$, we confirm the conclusion of Lemma \ref{lemma:retracingstepspecial}
for the case $\alpha^2 \neq 0$.

Combining the results of Situations 1 and 2, we obtain the results of Lemma \ref{lemma:effectofsemibasic}
in the special case where $\tr(\alpha) \neq 0$.

\subsubsection{The effect of conjugating with a semi-basic unit: degenerate case}\label{section:effectsemibasicdegenerate}

Now, we take a nonscalar basic vector $\alpha$ such that $\alpha^2=0$, and
we take a semi-basic unit in $\SB(\alpha)$, which we write $\gamma=1+r\alpha$
for some $r \in C$. We assume that $\gamma$ is nonbasic, to the effect that $\deg(r) \geq 1$.
We choose an arbitrary nonscalar vector $\beta$ in the basic subalgebra opposite to $\F[\alpha]$.
We will use $(1,\alpha,\beta,\alpha \beta)$ as our deployed basis.
Let $x \in \calW_{p,q}$, which we write
$$x=x_1+x_\alpha \alpha+x_\beta \beta+x_{\alpha \beta} \alpha \beta.$$
We set $\omega':=\langle \alpha,\beta\rangle$.

Here the computations are different from the ones of the previous section, but fortunately they are substantially simpler.
We note that $\gamma$ commutes with $1$ and $\alpha$. Next we compute
$$\gamma \beta \gamma^{-1}=\gamma \beta \gamma^\star=\langle \gamma,\beta\rangle \gamma-\gamma^2 \beta^\star
=\langle \gamma,\beta\rangle \gamma-\gamma^2 \tr(\beta)+\gamma^2 \beta.$$
Noting that $\gamma^2=1+2r\alpha$ and $\langle \gamma,\beta\rangle=\tr (\beta)+\omega' r$, we end up with
$$\gamma \beta \gamma^{-1}=(r \omega').1+r(r\omega'-\tr(\beta))\, \alpha +\beta+(2r)\,\alpha\beta.$$
Finally
$$\gamma (\alpha \beta)\gamma^{-1}=(\gamma \alpha \gamma^{-1})(\gamma \beta \gamma^{-1})=\alpha (\gamma \beta \gamma^{-1})
=(r \omega')\,\alpha+\alpha \beta.$$
Hence
the first three coefficients of $x':=\gamma x \gamma^{-1}$ in $(1,\alpha,\beta,\alpha \beta)$ are
$$x'_1=x_1+r\omega' x_\beta, \; x'_\alpha=x_\alpha+r(r\omega'-\tr(\beta))\,x_\beta+r\omega' x_{\alpha \beta} \quad \text{and} \quad x'_\beta=x_\beta.$$
Notice, just like in the previous case, the invariance of the coefficient on $\beta$.

Now, we assume that $x$ is quadratic and nonscalar, and we consider three situations separately.

\vskip 2mm
\noindent \textbf{Situation 1:} $\F[\beta]$ is the leading subalgebra of $x$. \\
Set $s:=\deg(r)$, and recall that $s>0$.
Set also $n:=\deg(x_\beta)$, so that $\deg(x_\alpha)<\deg(x_\beta)$ and $\deg(x_{\alpha \beta})<n$. Then we see that $r(r\omega'-\tr (\beta)))$ has degree $2s+1$,
whereas $r\omega' x_{\alpha+\beta}$ has degree at most $s+1+(n-1)$, which is less than $n+2s+1$.
Hence $\deg(x'_\alpha)=n+2s+1$. Moreover $\deg(x'_1)=n+s+1<\deg(x'_\alpha)$ and $\deg(x'_\beta)=\deg(x_\beta)= n < \deg(x'_\alpha)$.
Hence $x'$ is nonbasic, $\alpha$ is leading for $x'$ and $\delta(x')>\delta(x)$.

\vskip 2mm
For the remainder of the proof, we set $n:=\deg(x_\alpha)$.
As in the preceding section, we find the identity
\begin{equation}\label{eq:devnorme3}
N(x)=-x_1^2+\tr(x) x_1+N(\beta)  x_{\beta}^2+\langle \alpha,\beta\rangle x_\alpha x_{\beta}
\end{equation}
thanks to $N(\alpha)=0$ and $\tr(\alpha)=0$.

\vskip 2mm
\noindent \textbf{Situation 2:} $x$ is nonbasic, $\F[\alpha]$ is the leading subalgebra of $x$ but $\alpha$ is \emph{not} leading for $x$. \\
Here $n>0$ because $x$ is nonbasic.
Just like in the corresponding situation from Section \ref{section:effectsemibasicsplit}, we deduce from identity \eqref{eq:devnorme3}
that $\deg(x_{\beta})=n-1$, $\deg(x_1)=n$ and $\deg(x_{\alpha \beta}) \leq n-1$.
Remembering that $\deg(r) \geq 1$, we see that $\deg(x'_1)=n+\deg(r)>n>\deg(x'_\beta)$
and $\deg(x'_\alpha)=2\deg(r)+n$ (with $r^2 \omega' x_\beta$ as the only summand of degree at least $2\deg(r)+n$ in the previous expression of
$x'_\alpha$). Again $\deg(x'_\alpha)>\deg(x'_1)>\deg(x'_\beta)$ and we conclude that $x'$ is nonbasic and that $\alpha$ is leading for it,
with $\delta(x')>\delta(x)$.

\vskip 2mm
\noindent \textbf{Situation 3:} $x$ is nonbasic and $\alpha$ is leading for $x$. \\
Here, we try to adjust $r$ so that $\delta(x')<\delta(x)$.
To do so, we examine the $x'_1$ coefficient. Set $s:=\deg(x_1)$ to this end.
With exactly the same method as in Section \ref{section:effectsemibasicsplit} (situation 3), we find the following results:
\begin{itemize}
\item If $x_1$ is nonconstant, then $x_\beta \neq 0$ and $\deg(x_\beta)=2s-n-1$.
\item If $x_1$ is constant, then $x_\beta=0$, and we compute the trace of $x$ to find that $\omega' x_{\alpha \beta}+2 x_1$ is constant,
which yields $x_{\alpha \beta}=0$. Hence in that case we recognize that $x \in \F+ C \alpha=\F +\SB(\alpha)$, and hence $x$
is special degenerate attached to $\alpha$.
\end{itemize}
Assume now that $x$ is not special degenerate attached to $\alpha$. Then the former case holds. Then $\deg(x_1)\geq 1+\deg(x_\beta)=\deg(\omega' x_\beta)$
and we deduce that $r$ can be adjusted so that $\deg(x'_1)<\deg(x_1)$.
So, assume that we have adjusted $\gamma$ in this way.
Then, since $x'_\beta \neq 0$ we cannot have $x'_1$ constant, and with the same line of reasoning as in Situation 3 of Section \ref{section:effectsemibasicsplit}
we deduce that $\deg(x'_\alpha)<\deg(x_\alpha)$. Hence $\delta(x')<\delta(x)$.

Combining the results of Situations 1 and 2, we obtain the result of Lemma \ref{lemma:effectofsemibasic}
in the special case where $\tr(\alpha) = 0$, whereas the study of Situation 3 completes the proof of Lemma \ref{lemma:retracingstepspecial}
for this case.

Now, both Lemmas \ref{lemma:effectofsemibasic} and \ref{lemma:retracingstepspecial} are established.

\subsubsection{The behaviour of special degenerate elements}

\begin{lemma}\label{lemma:specialdegenerate}
Let $x \in \calW_{p,q}$ be a special degenerate element.
Then:
\begin{enumerate}[(a)]
\item If $x$ is attached to some $\alpha$ such that $\tr(\alpha) = 0$,
then there is no automorphism $\Phi$ of $\calW_{p,q}$ such that $\Phi(x)$ is basic.
\item $x$ is not conjugated to a basic vector.
\end{enumerate}
\end{lemma}

\begin{proof}
Assume first that $x$ is attached to some $\alpha$ such that $\tr(\alpha) = 0$.
Then $x=\lambda +r \alpha$ for some $\lambda \in \F$ and some $r \in C \setminus \F$.
Then $x-\lambda$ is non-normalized and non-zero.
Let $\Phi \in \Aut(\calW_{p,q})$. Then as $\Phi(C)=C$ we see that $\Phi$ maps every normalized vector to a normalized vector.
Hence $\Phi(x-\lambda)$ is nonbasic, and so $\Phi(x)$ is nonbasic. This proves point (a), and point (b) immediately follows
in the same situation (where $x$ is attached to a trace zero element $\alpha$).

In the remainder of the proof, we assume that $x$ is attached to a basic (nontrivial) idempotent $\alpha$.
Subtracting a scalar from $x$ leaves both the assumptions and the conclusion unchanged, so we lose no generality in
directly assuming that $x=\alpha u^\star$ for some $u \in \calW_{p,q}$ such that $\langle \alpha,u\rangle=0$. In particular
$x^2=0$.

Assume now that there is a unit $\gamma \in \calW_{p,q}^\times$ such that
$c:=\gamma^{-1} x \gamma$ is basic. Note then that $c^2=0$.

If $c \in \F[\alpha]$, then $\F[\alpha]$ would be degenerate, which is not true because we have assumed that $\alpha$ is idempotent.
Hence $c$ must belong to the basic subalgebra opposite to $\F[\alpha]$.
In particular $x=r \alpha(\langle \alpha,c\rangle-c)$ for some $r \in C \setminus \{0\}$ (see Section \ref{section:semibasicunits}).

Let us set $\calE:=C+C\alpha$, which is a subalgebra of $\calW_{p,q}$.
We can then write $\gamma=y+zc$ for some (unique) $y,z$ in $\calE$. We note that
$x =\lambda \gamma c \gamma^\star$ for some $\lambda \in \F^\times$.
Then $\gamma c \gamma^\star=y c y^\star=\langle c,y\rangle y-y^2 c^\star=\langle c,y\rangle y+y^2 c$, and $y^2 \in \calE$.
Remember now that $x=r \langle \alpha,c\rangle \alpha-r \alpha c$ for some $r \in C$.
If $\langle c,y\rangle=0$ we would find the vanishing of the second coefficient of $\gamma c \gamma^\star$ in the basis 
$(1,\alpha,c,\alpha c)$ of the $C$-module $\calW_{p,q}$, contradicting the fact that it is a nonzero scalar multiple of $r \langle \alpha,c\rangle$.
Hence $\langle c,y\rangle \neq 0$, and by extracting the first coefficient in the basis $(1,\alpha,c,\alpha c)$
we find $y \in C \alpha$.
It follows that $N(y)=0$. Then
$$N(\gamma)=N(y)+\langle y,zc\rangle+N(z)N(c)=\langle z^\star y,c\rangle$$
with $z^\star y \in \calE$. But since
$\tr(c)=0$ we observe that $\langle u,c\rangle \in C \langle \alpha,c\rangle$ for all $u \in \calE$, and in particular
$N(\gamma) \not\in \F^\times$ because $\langle \alpha,c\rangle$ has degree $1$. This is a contradiction, and the proof is completed.
\end{proof}

\subsection{The refined retracing algorithm}

\index{Retracing algorithm}

Now, we can refine the retracing algorithm of Section \ref{section:units1}
so as to take semi-basic units into account, thanks to the results of the previous section.
The algorithm takes as entry a nonscalar quadratic element $x \in \calW_{p,q}$, and runs as follows:
\begin{itemize}
\item Initialize $y$ as $x$ and $L$ as the empty list.
\item While $y$ is nonbasic:
\begin{itemize}
\item Compute a leading vector $\alpha$ for $y$.
\item Write $y=y_1 +y_\alpha \alpha+y_\beta \beta+y_{\alpha\beta} \alpha\beta$ in a corresponding deployed basis.
\begin{itemize}
\item If $N(\alpha)\neq 0$, put $\gamma:=\alpha$.
\item If $N(\alpha)=0$ and $y_1$ is nonconstant, apply the process described in Sections \ref{section:effectsemibasicsplit} and \ref{section:effectsemibasicdegenerate} to compute a semi-basic unit $\gamma'$ such that
    $\delta(\gamma' y (\gamma')^{-1})<\delta(y)$, then put $\gamma:=(\gamma')^{-1}$.
\item If $N(\alpha)=0$, $y_1$ is constant and $\tr(y) \neq 2y_1$, apply the process
described in Sections \ref{section:effectsemibasicsplit} and \ref{section:effectsemibasicdegenerate} to compute a semi-basic unit $\gamma'$ such that
    $\delta(\gamma' y (\gamma')^{-1})<\delta(y)$, then put $\gamma:=(\gamma')^{-1}$.
\item Else, return ``Failure''.
\end{itemize}
\item Update $y$ to $\gamma^{-1} y \gamma$ and append $\gamma$ to $L$.
\end{itemize}
\item Return $(y,L)$.
\end{itemize}

After each iteration that does not return a failure, the absolute distance of the current vector $y$
decreases by at least one unit, so the algorithm terminates in any situation.

Unless it reports a failure, the refined retracing algorithm outputs a basic vector $y$ and a list $L=(\gamma_1,\dots,\gamma_n)$
of basic or semi-basic units such that $y=(\gamma_1 \cdots \gamma_n)^{-1} x (\gamma_1\cdots \gamma_n)$.
Moreover, by Lemma \ref{lemma:specialdegenerate}, no failure is reported if the starting vector
is conjugated to a basic vector. As a consequence, this algorithm can be used to detect whether a given vector is
conjugated to a basic vector (of course, it should only be used after checking that the given vector is quadratic).

Here is an application, which reinforces Corollary \ref{cor:conjsousalg2} in two ways: it takes all possible polynomials $p$ and $q$
into account, and it only leaves out the situation of $2$-dimensional degenerate subalgebras.

\begin{theo}\label{theo:conjugateofquadratictobasic}
Let $\calC$ be a $2$-dimensional nondegenerate subalgebra of $\calW_{p,q}$.
Then $\calC$ is conjugated to a basic subalgebra.
\end{theo}

\begin{proof}
Choose $x \in \calC \setminus \F$ and apply the refined retracing algorithm to $x$.
At each step the algebra $\F[y]$ is nondegenerate because it is isomorphic to $\calC$. Hence
in the algorithm the current vector is never a special degenerate element.
Therefore, the refined retracing algorithm yields a unit $\gamma \in \calW_{p,q}^\times$ such that $\gamma^{-1} x \gamma$ is basic,
to the effect that $\gamma^{-1} \calC \gamma$ is one of the basic subalgebras.
\end{proof}

The necessity to leave out degenerate subalgebras is easy to see by noting that every nonzero basic vector is normalized.
In contrast, if one of $p$ and $q$ splits there is an element $\beta \in \calW_{p,q} \setminus \{0\}$ such that $\beta^2=0$.
Then $\omega \beta$ is quadratic and non-normalized, so it is not conjugated to a basic vector, and hence
$\F[\omega\beta]$ is not conjugated to a basic subalgebra. We will complete the study of the conjugacy classes of quadratic elements
in Section \ref{section:conjclassquadratic}.

\subsection{The fruits of the refined retracing algorithm}

Now, we will apply the refined algorithm to obtain two results: the existence part
of the Automorphisms Theorem, and the Refined Units Theorem.
This is done by proving the following result:

\begin{theo}\label{theo:superdecompositiontheorem}
Let $\Phi \in \Aut(\calW_{p,q})$. Then there exist an element $\gamma \in \calW_{p,q}^\times$
that a product of basic and semi-basic units, and a basic automorphism $\Psi \in \BAut(\calW_{p,q})$
such that $\Phi \circ \Psi$ is the inner automorphism $x \mapsto \gamma x \gamma^{-1}$.
\end{theo}

\begin{proof}
For $\gamma \in \calW_{p,q}^\times$, denote by $i_\gamma$ the inner automorphism $x \mapsto \gamma x \gamma^{-1}$.
By exchanging $p$ and $q$ if necessary, and by performing a basic base change if necessary, we can
assume that one of the following three situations holds:
\begin{enumerate}[(i)]
\item $\F[a]$ is a field;
\item $a$ is idempotent;
\item $a^2=b^2=0$.
\end{enumerate}
We observe that in any case the refined retracing algorithm succeeds when applied to $\Phi(a)$.
Indeed $\F[\Phi(a)]$ is isomorphic to $\F[a]$, and hence the algorithm fails only if $\F[a]$ is degenerate.
Yet $\F[a]$ is degenerate only in case (iii), and in that case all the special degenerate elements
are associated with elements with trace $0$: In that case, point (a) of Lemma \ref{lemma:specialdegenerate}
shows that no failure can be reported because at each step we work with the image of $a$ under an automorphism of $\calW_{p,q}$.

Hence the refined retracing algorithm yields a unit $\gamma_1$ which is a product of units and of semi-basic units
and such that the composite $i_{\gamma_1} \circ \Phi$ maps $a$ to a basic vector. Then
we set $\Phi':=i_{\gamma_1} \circ \Phi$, $\alpha:=\Phi'(a)$ (which may very well belong to $\F[b]$),
$\calA:=\F[\alpha]$, and finally we define $\calB$ as the opposite basic subalgebra.

Now, we examine $x:=\Phi'(b)$, which is quadratic and nonscalar. As in the proof of Proposition \ref{prop:doubleretracing},
we observe that $d_{\calA}(x)=1$. Recalling at this point that $d_{\calB}(x) \neq d_{\calA}(x)$, there are two options to consider.
If $d_{\calB}(x) \leq 0$ then $x \in \calB$ and we immediately conclude that
$\Psi:=i_{\gamma_1} \circ \Phi$ is a basic automorphism, which completes the proof by taking $\gamma:=(\gamma_1)^{-1}$.

For the remainder of the proof, we assume that $d_{\calB}(x) \geq 2$, to the effect that $x$ is nonbasic and with leading subalgebra $\calA$.
We take a leading vector $\alpha'$ for $x$.

Assume first that $\alpha'$ is a unit.
Then we deduce from Corollary \ref{cor:differenceofdistances} that $d_{\calB}(x)=2$, and we conclude that
$(\alpha')^{-1} x \alpha'$ is basic. Since $\alpha' \in \calA$, we have $(\alpha')^{-1} \alpha \alpha'=\alpha$
and we conclude that $i_{(\alpha')^{-1}} \circ \Phi'$ is basic. Hence, the conclusion is validated in that case.

Assume finally that $\alpha'$ is a zero divisor. Then $d_{\calB}(x)>2$ by Corollary \ref{cor:differenceofdistances}. This rules case (i) out because
$\calA$ is isomorphic to $\F[a]$. It remains to consider cases (ii) and (iii).
Now, we choose an arbitrary vector $\beta \in \calB \setminus \F$
and write $x=x_1+x_\alpha \alpha+x_\beta \beta+x_{\alpha \beta} \alpha \beta$.
The key is to observe that, since $\Phi'$ is an automorphism of the algebra $\calW_{p,q}$, the family
$(\Phi'(1),\Phi'(a),\Phi'(b),\Phi'(ab))=(1,\alpha,x,\alpha x)$ is a $C$-basis of $\calW_{p,q}$.
We also note that $\alpha^2=\tr(\alpha) \alpha$ since $N(\alpha)=0$.
Hence
$$\alpha x=(x_1+\tr(\alpha)x_\alpha)\,\alpha+(x_\beta+\tr(\alpha) x_{\alpha\beta})\, \alpha \beta.$$
Thus the matrix of $(1,\alpha,x,\alpha x)$ in the $C$-basis $(1,\alpha,\beta,\alpha \beta)$ is
$$\begin{bmatrix}
1 & 0 & x_1  & 0 \\
0 & 1 & x_\alpha & x_1+\tr(\alpha) x_\alpha \\
0 & 0 & x_\beta & 0 \\
0 & 0 & x_{\alpha\beta} & x_\beta+\tr(\alpha) x_{\alpha\beta}
\end{bmatrix}.$$
The determinant of the latter is $x_\beta(x_\beta+\tr(\alpha) x_{\alpha\beta})$,
and we infer that $x_\beta$ and $x_\beta+\tr(\alpha) x_{\alpha\beta}$ are constant and nonzero.
It follows in particular that $\tr(\alpha) x_{\alpha \beta}$ is constant.

Now, we split the discussion into two cases.
Assume first that case (ii) holds, so that $\alpha$ is idempotent.
Then $\tr(\alpha)=1$, and since $\alpha'$ is a zero divisor the only possibilities are that $\alpha' \sim \alpha$ or $\alpha' \sim \alpha^\star$.
By the previous analysis, $x_{\alpha\beta}$ is constant.
Then $\tr(x)=2x_1+x_\alpha+\tr(\beta)x_\beta+\langle \alpha,\beta\rangle x_{\alpha\beta}$
and $\deg(x_\alpha) \geq 2$ because $d_\calB(x) > 2$. Since $x_{\beta}$ and $x_{\alpha\beta}$ are constant,
it follows from the computation of $\tr(x)$ that $\deg(2x_1+x_\alpha) \leq 1$.
Yet this is not possible. Indeed:
\begin{itemize}
\item Either $\alpha' \sim \alpha$, in which case $\deg(x_1)<\deg(x_\alpha)$
and hence $\deg(2x_1+x_\alpha)=\deg(x_\alpha) \geq 2$;
\item Or $\alpha' \sim \alpha^\star$, in which case
we rewrite $x=(x_1+x_\alpha) -x_\alpha\, \alpha^\star+(x_\beta+x_{\alpha\beta})\, \beta-x_{\alpha \beta}\, \alpha^\star \beta$
to see that $\deg(x_1+x_\alpha)<\deg(-x_\alpha)=\deg(x_\alpha)$, and once more we find a contradiction by writing
$2x_1+x_\alpha=2(x_1+x_\alpha)-x_\alpha$.
\end{itemize}
Therefore case (ii) is ruled out in the present situation.

Hence, only case (iii) is possible. Then $\F[\alpha]$ is degenerate
and its zero divisors are the vectors of $\F^\times \alpha$. Hence $\alpha$ is leading for $x$.
Then we apply the refined retracing algorithm to $x$, which succeeds because, just like at the start of the proof,
no obstruction can come here from the special degenerate elements.
The trick here is to observe by induction that at each step the current vector $x'$ has leading subalgebra $\calA$.
Let indeed $\gamma' \in \SB(\calA)$ be arbitrary. Since $\calA$ is degenerate
the elements of the group $\SB(\calA)$ commute with all the elements of $\calA$, and in particular
$(\gamma')^{-1} \alpha \gamma'=\alpha$. Hence, with the same argument as for $x$ we find that
$x':=(\gamma')^{-1} x \gamma'$ has leading subalgebra $\calA$ unless $x' \in \calB$.
Hence, after finitely many steps we find $\gamma' \in \SB(\calA)$ such that
$(\gamma')^{-1} x \gamma'$ is basic, while $(\gamma')^{-1} \alpha \gamma'=\alpha$. This completes the proof.
\end{proof}

A straightforward consequence of the previous theorem is of course the existence part of the decomposition in the
Automorphisms Theorem. Another immediate consequence is the Refined Units Theorem, which ensues from Theorem \ref{theo:superdecompositiontheorem}
and from the uniqueness part in the Automorphisms Theorem. The proof is a word-for-word adaptation of the one
of the Weak Units Theorem given in Section \ref{section:retracingfruits}, and hence no details are necessary at this point.

We finish by noting that the proof of Theorem \ref{theo:superdecompositiontheorem} gives an explicit algorithm
to decompose an arbitrary automorphism $\Phi$ of $\calW_{p,q}$ as the product of a basic automorphism followed by an inner automorphism,
an algorithm that also gives a decomposition of a conjugator associated with the inner automorphism into a product of basic and semi-basic units,
taking as only entries the elements $\Phi(a)$ and $\Phi(b)$.
We note that the algorithm also provides a test of whether $\Phi$, considered as an endomorphism only and defined by the datum of
$\Phi(a)$ and $\Phi(b)$, is really an automorphism. Of course, for efficiency one should, prior to using the algorithm, first test whether
$\Phi(a)$ and $\Phi(b)$ are quadratic with the same trace and norm as $a$ and $b$, respectively, and one should also test
whether $\langle \Phi(a),\Phi(b)\rangle$ has degree $1$.

\subsection{Uniqueness of decompositions}\label{section:uniquenessfull}

We conclude by proving the Strong Units Theorem and the Semi-Basic Units Theorem.
Now that the Refined Units Theorem has been proved, all that remains is to prove the uniqueness of certain decompositions,
and all the statements can be reduced to just one lemma, which we state below (Lemma \ref{lemme:ultimateuniqueness}).

Before we can state this lemma, some additional terminology is in order.
Let $\calC$ be a basic subalgebra.
If $\calC$ splits, a (potentially void) list $(x_1,\dots,x_k)$ is called a \textbf{reduced chain of semi-basic units} of $\calC$
whenever there exists a zero divisor $\alpha \in \calC$ such that
$x_i \in \SB(\alpha) \setminus \{1\}$ for every odd $i$, and
$x_i \in \SB(\alpha^\star) \setminus \{1\}$ for every even $i$, and the $k$-list
$(\alpha,\alpha^\star,\alpha,\dots)$ is then said to be adapted to it.
\index{Reduced chain (of semi-basic units in a split basic subalgebra)}

In the general case, a \textbf{reduced chain} attached to $\calC$ is a nonvoid list
\index{Reduced chain}
$((x_1,y_1),\dots,(x_k,y_k))$ of pairs in which every $x_i$ is an element of $\SB(\calC) \setminus \{1\}$, every
$y_i$ is either a basic unit in $\calC \setminus \F$ or a zero divisor in $\calC$ and:
\begin{itemize}
\item Either $\calC$ is a field, $k=1$, $x_1$ is a nonscalar basic unit in $\calC^\times$ and $y_1=x_1^{-1}$;
\item Or $\calC$ degenerates, $x_1$ is a nonscalar basic unit in $\calC^\times$, $y_1=x_1^{-1}$ and if $k\geq 2$ then $k=2$,
$x_2$ is a semi-basic unit that is nonbasic and $y_2$ is a zero divisor in $\calC$;
\item Or $\calC$ degenerates, $k=1$, $x_1$ is a nonbasic semi-basic unit in $\SB(\calC)$, and $y_1$ is a zero divisor in $\calC$;
\item Or $\calC$ splits and $(x_1,\dots,x_k)$ is a reduced chain of semi-basic units of $\calC$ (as defined earlier), and
$(y_1,\dots,y_k)$ is an adapted chain of zero divisors in $\calC$;
\item Or $\calC$ splits, $x_1 \in \calC^\times \setminus \F^\times$, $y_1=x_1^{-1}$, $(x_2,\dots,x_k)$ is a reduced chain of semi-basic units of $\calC$ and $(y_2,\dots,y_k)$ is an adapted chain of zero divisors in $\calC$.
\end{itemize}
In particular, in such a reduced chain no two consecutive vectors of $(y_1,\dots,y_k)$ are scalar multiples of one another.

We can now state our main result, which will yield all the uniqueness results:

\begin{lemma}\label{lemme:ultimateuniqueness}
Let $N \geq 1$ be an integer, $(\calC_i)_{1 \leq i \leq N}$ be a strongly $2$-periodical sequence valued in $\{\F[a],\F[b]\}$.
Let, for each $i \in \lcro 1,N\rcro$ be a reduced chain
$((w_{i,1},z_{i,1}),\dots,(w_{i,n_i},z_{i,n_i}))$ that is attached to the basic subalgebra $\calC_i$, and set $w_i:=\prod_{k=1}^{n_i} w_{i,k} \in \SB(\calC_i)$.
Then $\prod_{i=1}^N w_i \not\in \F$.
\end{lemma}

\begin{proof}
We rewrite the concatenated lists $(w_{1,1},\dots,w_{1,n_1},w_{2,1},\dots,w_{2,n_2},\dots,w_{N,n_N})$ and
$(z_{1,1},\dots,z_{1,n_1},z_{2,1},\dots,z_{2,n_2},\dots,z_{N,n_N})$
respectively as $(\gamma_1,\dots,\gamma_M)$ and $(\alpha_1,\dots,\alpha_M)$ for some integer $M \geq 1$.
We set $\Gamma:=\prod_{k=1}^M \gamma_k$. We shall deduce from Lemmas \ref{lemma:effectofbasic} and \ref{lemma:effectofsemibasic}
that at least of one of $\Gamma^{-1} a \Gamma$ and $\Gamma^{-1} b \Gamma$ is nonbasic.

Whether we take $a$ or $b$ depends on $\gamma_1$.
If $\gamma_1 \in \SB(\F[a])$ we take $x_0:=b$, whereas if $\gamma_1 \in \SB(\F[b])$ we take $x_0:=a$.
Then, according as $\gamma_1$ is a basic or a semi-basic unit, one of Lemmas \ref{lemma:effectofbasic} and \ref{lemma:effectofsemibasic}
shows that $\gamma_1^{-1} x_0 \gamma_1$ is nonbasic and $\alpha_1$ is leading for it.

The initialization is done. Now, we define $(x_0,\dots,x_M)$ inductively by $x_i:=\gamma_i^{-1} x_{i-1} \gamma_i$ for all $i\in \lcro 1,M\rcro$.
We have just shown that $x_1$ is nonbasic, with leading vector $\alpha_1$.

Now, the key is to note that in the sequence $(\alpha_1,\dots,\alpha_M)$ the following properties are satisfied for every $k \in \lcro 2,M\rcro$:
\begin{enumerate}[(i)]
\item If $\alpha_k$ is a unit then $\alpha_{k-1}$ and $\alpha_k$ belong to opposite basic subalgebras;
\item If $\alpha_k$ is a zero divisor then $\alpha_{k-1} \not\sim \alpha_k$.
\end{enumerate}
Indeed, for all $k \in \lcro 2,M\rcro$:
\begin{itemize}
\item If $\alpha_{k-1}$ and $\alpha_{k}$ appear as two consecutive terms in a subsequence $(z_{j,1},\dots,z_{j,n_j})$, then
$\alpha_k$ is not a basic unit (so point (i) trivially holds), and point (ii) has already been observed after defining reduced chains;
\item If $\alpha_k$ is at the start of a subsequence $(z_{j,1},\dots,z_{j,n_j})$ and $\alpha_{k-1}$ at the end of the preceding one
$(z_{j-1,1},\dots,z_{j-1,n_{j-1}})$, then by our assumptions $\alpha_k$ and $\alpha_{k-1}$ belong to opposite basic subalgebras,
and hence both points are obvious.
\end{itemize}

From there, by induction we can combine Lemmas \ref{lemma:effectofbasic} and \ref{lemma:effectofsemibasic}
to find that, for all $i \in \lcro 1,M\rcro$ the vector $x_i$ is nonbasic with leading vector $\alpha_i$.
Indeed, if for some $i \in \lcro 1,M-1\rcro$ we know that $x_i$ is nonbasic with leading vector $\alpha_i$, then:
\begin{itemize}
\item Either $\alpha_{i+1}$ is a zero divisor, in which case by point (ii) it is not a scalar multiple of $\alpha_i$, and hence
Lemma \ref{lemma:effectofsemibasic} yields that $x_{i+1}$ is nonbasic with leading vector $\alpha_{i+1}$;
\item Or $\alpha_{i+1}$ is a unit, in which case point (i) shows that $\alpha_i$ and $\alpha_{i+1}$ lie in opposite subalgebras,
and we deduce from Lemma \ref{lemma:effectofbasic} that $x_{i+1}$ is nonbasic with leading vector $\alpha_{i+1}$.
\end{itemize}
In particular $x_M=\Gamma^{-1} x_0 \Gamma$ is nonbasic and hence $\Gamma \not\in \F$. This completes the proof.
\end{proof}

We finish by deriving the Strong Units Theorem and the Semi-Basic Units Theorem from the previous lemma.

\begin{proof}[Proof of the Semi-Basic Units Theorem]
Let $\calC$ be a split basic subalgebra. We choose a nontrivial idempotent $\alpha$ in $\calC$.

We have already observed that $\SB(\calC)=\calC^\times \cdot \SSB(\calC)$ because $\calC^\times$ normalizes $\SSB(\calC)$
(and by its very definition every element of $\SB(\calC)$ is a product of elements of $\calC^\times \cup \SSB(\calC)$).

It remains to prove that $\calC^\times \cap \SSB(\calC)=\{1\}$ and that $\SSB(\calC)$
is the internal free product of the two subgroups of semi-basic units $\SB(\alpha)$ and $\SB(\alpha^\star$).

Of course, every element of $\SSB(\calC)$ can be written in the form of a product
$\prod_{k=1}^n x_k$, where there is a strongly $2$-periodical sequence $(\beta_k)_{1 \leq k \leq n}$ valued in $\{\alpha,\alpha^\star\}$
such that $x_k \in \SB(\beta_k)$ for all $k \in \lcro 1,n\rcro$.

So, we let $k \geq 1$ be an integer and $(\alpha_i)_{1 \leq i \leq k}$ be a strongly $2$-periodical sequence valued in $\{\alpha,\alpha^\star\}$,
and we let $(x_1,\dots,x_k) \in (\SB(\alpha_1) \setminus \{1\}) \times \cdots \times (\SB(\alpha_k) \setminus \{1\})$, and let $y \in \calC^\times$ be a basic unit.
We assume that $\prod_{i=1}^k x_i=y$ and prove that it leads to a contradiction, which will clearly yield all the results we want to prove here
(indeed, simply having $\prod_{i=1}^k x_i \neq 1$ in particular will yield point (b) in the Semi-Basic Units Theorem).

If $y \not\in \F$ we set $x_0:=y^{-1}$ and recognize that $((x_0,y),(x_1,\alpha_1),\dots,(x_k,\alpha_k))$ is a reduced
chain that is attached to $\calC$, and $\prod_{i=0}^k x_i=1$. This contradicts Lemma \ref{lemme:ultimateuniqueness}.
Hence $y \in \F$, and then $((x_1,\alpha_1),\dots,(x_k,\alpha_k))$ is a reduced
chain that is attached to $\calC$, and once more we obtain a contradiction with Lemma \ref{lemme:ultimateuniqueness}.
\end{proof}

\begin{proof}[Proof of the Strong Units Theorem]
We already know from the Refined Units Theorem that $\SB(\F[a]) \cup \SB(\F[b])$ generates the group $\calW_{p,q}^\times$.

Next, let $k \geq 1$ be a positive integer.
Let $(\calC_i)_{1 \leq i \leq k}$ be a strongly $2$-periodical sequence valued in $\{\F[a],\F[b]\}$, and
let $(x_1,\dots,x_k) \in (\SB(\calC_1) \setminus \F) \times \cdots \times (\SB(\calC_k) \setminus \F)$.
All we need to do is prove that the product $\prod_{i=1}^k x_i$ does not belong to $\F$.

Let $i \in \lcro 1,k\rcro$. By the definition of the group $\SB(\calC_i)$,
we note that there exists a reduced chain $((w_{i,1},\alpha_{i,1}),\dots,(w_{i,N_i},\alpha_{i,N_i}))$
that is attached to $\calC_i$ and such that $x_i=w_{i,1} \cdots w_{i,N_i}$.
This is obvious if $\calC_i$ is a field; if $\calC_i$ splits we have shown this through the
Semi-Basic Units Theorem; finally if $\calC_i$ degenerates then this follows from the observation that $\calC_i^\times$ normalizes $\SB(\alpha)$
for an arbitrary zero divisor $\alpha$ in $\calC_i$ (in fact, every element of $\SB(\alpha)$ commutes with every element of $\calC_i$).

Then Lemma \ref{lemme:ultimateuniqueness} directly yields that $\prod_{i=1}^k x_i \not\in \F$, and our proof is complete.
\end{proof}

Now the group of units $\calW_{p,q}^\times$ is entirely deciphered.

\subsection{Application to the action of inner automorphisms on basic vectors}

As an application of the previous results and techniques, we will obtain the following result, which
reinforces the uniqueness part in the Automorphisms Theorem.

\begin{prop}\label{prop:conjugatetoatmostonebasic}
Let $\gamma \in \calW_{p,q}^\times$ and $x$ be a \emph{nonscalar} basic vector.
If $\gamma x \gamma^{-1}$ is basic then $\gamma x \gamma^{-1}=x$.
\end{prop}

\begin{proof}
Set $\calA:=\F[x]$ and denote by $\calB$ the opposite basic subalgebra.

We shall prove that $\gamma x \gamma^{-1}=x$ or $\gamma x \gamma^{-1}$ is nonbasic.

The result is obvious if $\gamma \in \F^\times$, so we assume that $\gamma \not\in \F^\times$ from now on.

We can now use the previous description of units to decompose $\gamma^{-1}$ in reduced form
$\gamma^{-1} =\gamma_1 \cdots \gamma_n$
where there is a strongly $2$-periodical sequence
$(\calC_1,\dots,\calC_n)$ valued in $\{\calA,\calB\}$ such that $\gamma_i \in \SB(\calC_i) \setminus \F$
for all $i \in \lcro 1,n\rcro$.
We can then further decompose each $\gamma_i$ into a product
$\prod_{k=1}^{N_i} x_{i,k}$ where $(x_{i,1},\dots,x_{i,N_i})$ can be extended to a reduced chain $((x_{i,1},y_{i,1}),\dots,(x_{i,N_i},y_{i,N_i}))$
attached to $\calC_i$.

If $\calC_1=\calB$ then we can follow the line of reasoning of the proof of Lemma \ref{lemme:ultimateuniqueness}
to see that $\gamma x \gamma^{-1}$ is nonbasic.
The same holds if $\calA$ splits and $x_{1,1}$ is not a basic unit.
Assume from now on that $\calC_1=\calA$.

\begin{itemize}
\item Assume that $\calA$ is a field or splits, and $x_{1,1}$ is a basic unit.
Then $x_{1,1}$ commutes with $x$, and hence
$\gamma x \gamma^{-1}=(\gamma') x (\gamma')^{-1}$ for
$\gamma':=(x_{1,2} \cdots x_{1,N_1} \gamma_2 \cdots \gamma_n)^{-1}$.
If $\gamma' \not\in \F$ we proceed as before to see that $(\gamma') x (\gamma')^{-1}$ is nonbasic:
indeed, either $N_1>1$ and then $x_{1,2}$ is not a basic unit, or $N_1=1$ and then we conclude because $\calC_2=\calB$.
If $\gamma' \in \F$ we directly conclude that $(\gamma') x (\gamma')^{-1}=x$.

\item Assume finally that $\calA$ degenerates. Then $\gamma_1$ commutes with $x$,
and hence $\gamma x \gamma^{-1}=(\gamma') x (\gamma')^{-1}$ for
$\gamma':=(\gamma_2 \cdots \gamma_n)^{-1}$, and we conclude like in the previous case.
\end{itemize}
\end{proof}

\subsection{Addendum: conjugating semi-basic units with basic units}

We complete the study by considering the case of a
split basic subalgebra $\calA$. We have not yet explained precisely how the
conjugation of the elements of $\SSB(\calA)$ by the basic units of $\calA^\times$ works in practice.
The following proposition answers this question:

\begin{prop}\label{prop:actionbasiconsemibasic}
Let $\calA$ be a split subalgebra and let $\alpha \in \calA$ be a nontrivial idempotent.
Let $\delta \in \F^\times$ and $\delta' \in \F^\times$, and set $\gamma:=\delta \alpha+\delta' \alpha^\star$.
Let $y \in \calW_{p,q}$ be such that $\langle \alpha,y\rangle=0$.
Then $\gamma (1+\alpha y^\star) \gamma^{-1}=1+\delta (\delta')^{-1} \alpha y^\star$.

In particular if $\delta=-\delta'$ then the conjugation by $\gamma$ in $\SB(\alpha)$ is the inversion $x \mapsto x^{-1}$.
\end{prop}

\begin{proof}
Remember from Lemma \ref{lemma:idempotentlemma} that $(\alpha y^\star)\alpha=0$ and $\alpha^\star (\alpha y^\star)=0$,
and the first identity also yields $(\alpha y^\star)\alpha^\star=\alpha y^\star$.
Hence $\gamma (\alpha y^\star) \gamma^\star=\delta^2 (\alpha y^\star)$, and since
$N(\gamma)=\delta \delta'$ we deduce that $\gamma (\alpha y^\star) \gamma^{-1}=\delta (\delta')^{-1} \alpha y^\star$,
which yields the claimed result.
\end{proof}

\section{Miscellaneous Issues}\label{section:misc}

In this final part, we solve various questions on the structure of $\calW_{p,q}$ that require a deep understanding of the units group
and the automorphism group.
The first two sections are devoted to the related study of the conjugacy classes of quadratic elements (Section \ref{section:conjclassquadratic})
and of the conjugacy classes of finite-dimensional subalgebras (Section \ref{section:misc:finitedimsubalg}).
The last two sections deal with the structure of the automorphisms group.
We shall determine its center (Section \ref{section:centerofautomorphismgroup}),
and then we will classify the conjugacy classes of the elements of order $2$
provided that $\car(\F) \neq 2$ (Section \ref{section:involutions}).

Throughout, we denote by
$$i_\gamma : x \in \calW_{p,q} \mapsto \gamma x \gamma^{-1} \in \calW_{p,q}$$
the inner automorphism associated with the unit $\gamma$.

\subsection{Conjugacy classes of quadratic elements}\label{section:conjclassquadratic}

Here, we complete the study of conjugacy classes of quadratic elements of $\calW_{p,q} \setminus \F$, which we initiated in Section
\ref{section:units1}. So, let $x \in \calW_{p,q} \setminus \F$ be quadratic. An invariant of $x$ under conjugation is its minimal polynomial $r$,
which amounts to the datum of $\tr(x)$ and $N(x)$.
By Proposition \ref{prop:conjugatetoatmostonebasic}, $x$ is conjugated to at most one basic vector.

If $r$ is irreducible or splits with simple roots, then the refined retracing algorithm shows that $x$ is conjugated to a basic vector.
Hence, in any of those two cases the conjugacy class of $x$ contains a unique basic vector.

It remains to examine the case where $r=(t-\lambda)^2$ for some $\lambda \in \F$.
Then, by replacing $x$ with $x-\lambda$, we lose no generality in assuming that $x^2=0$.
Another invariant is needed: the \emph{modular norm} of $x$.
Remember that $x$ splits uniquely as $x=s\,x_n$ for some normalized $x_n \in \calW_{p,q}$ and some monic $s \in \F[\omega]$
(monic with respect to $\omega$), and $s$ is called the \textbf{modular norm} of $x$, while $x_n$ is called the \textbf{normalization} of $x$.
The modular norm is invariant under any automorphism that fixes the elements of $C$, and in particular under every
inner automorphism. Hence the problem is now entirely reduced to the one of determining the
conjugacy classes of the normalized elements $x$ such that $x^2=0$.
Now, we take such an element $x$ and we try to apply the refined retracing algorithm to it.
\begin{itemize}
\item Either the algorithm succeeds, which proves that $x$ is conjugated to a basic vector (necessarily unique).
\item Or the algorithm fails, in which case we know at least that $x$ is conjugated to a normalized special degenerate vector
$x'$. Then attached to $x'$ is a nontrivial basic idempotent $\alpha$
(because the case of a trace zero element would yield that $x'$ is not normalized),
and because $x'$ is normalized it must be a spanning vector for the $C$-module $\alpha^\sharp$.
In that case, it has been proved in Lemma \ref{lemma:specialdegenerate} that $x'$ is not conjugated to a basic vector.
Alternatively, one can also prove this fact by applying the Strong Units Theorem to the unit $1+x$.
\end{itemize}
It remains to understand when two special degenerate vectors of the above type are conjugated.

Throughout our study, the reader must beware that, given nontrivial idempotents $\alpha$ and $\alpha'$,
the equality of $\alpha^\sharp$ and $(\alpha')^\sharp$ does not imply that $\alpha=\alpha'$.
In fact, it can be proved, e.g., by using an injective homomorphism of $\calW_{p,q}$ into $\Mat_2(\F(\omega))$,
that $\alpha^\sharp=(\alpha')^\sharp$ if and only if $\alpha' \in \alpha +\alpha^\sharp$.
However this result is not convenient enough for our study, and we prefer to leave it as an exercise.
Fortunately, we will only be concerned with the situation where $\alpha$ and $\alpha'$ are \emph{basic} idempotents,
and in this case we have a more useful result.

\begin{lemma}\label{lemma:equalityofsharp}
Let $\alpha$ and $\alpha'$ be nontrivial basic idempotents such that
$\alpha^\sharp \cap (\alpha')^\sharp \neq \{0\}$. Then $\alpha=\alpha'$.
\end{lemma}

\begin{proof}
The quickest proof is through the consideration of leading vectors (see Section \ref{section:units1}).
Choose $z \in \alpha^\sharp$ with $z \neq 0$. Then, for some nonscalar element $\beta$ of the basic subalgebra opposite to $\alpha$,
we have $z=r \alpha (\langle\alpha,\beta\rangle-\beta)^\star$ for some $r \in C \setminus \{0\}$. Then
$z=r\langle \alpha,\beta\rangle \alpha-r \alpha \beta^\star$, and it is clear that the quadratic element $z$ has $\alpha$ as leading vector.
Since leading vectors are unique up to multiplication with nonzero scalars (Lemma \ref{lemma:leading}), the assumption $\alpha^\sharp \cap (\alpha')^\sharp$ yields
$\alpha' \sim \alpha$, and hence $\alpha'=\alpha$ since $\alpha$ and $\alpha'$ are idempotents.
\end{proof}

Next, we observe thanks to characterization (iii) in Lemma \ref{lemma:idempotentlemma} (combined with the fact that every automorphism of
$\calW_{p,q}$ commutes with the adjunction) that for any automorphism $\Phi$ of $\calW_{p,q}$ and any nontrivial idempotent $\alpha$,
the equality $\Phi(\alpha^\sharp)=\Phi(\alpha)^\sharp$ holds.

\begin{lemma}
Let $\alpha$ and $\alpha'$ be nontrivial basic idempotents, and let
$z \in \alpha^\sharp \setminus \{0\}$ and $z' \in (\alpha')^\sharp \setminus \{0\}$.
Then $z$ and $z'$ are conjugated in $\calW_{p,q}$ if and only if $\alpha=\alpha'$ and $z' \sim z$.
\end{lemma}

\begin{proof}
We start with the converse implication. Assume that $\alpha'=\alpha$ and $z'=\lambda z$ for some $\lambda \in \F^\times$.
Set $\gamma:=\lambda \alpha+\alpha^\star$, which is a unit with $\gamma^{-1}=\lambda^{-1} \alpha+\alpha^\star$. Then clearly $\gamma \alpha \gamma^{-1}=\alpha$.
We can write $z=\alpha y \alpha^\star$ for some $y \in \calW_{p,q}$, and hence
$\gamma z \gamma^{-1}=(\gamma \alpha) y (\alpha^\star \gamma^{-1})=(\lambda \alpha) y \alpha^\star=z'$.

Conversely, assume that there exists $\gamma \in \calW_{p,q}^\times$ such that $\gamma z \gamma^{-1}=z'$.
Then $\gamma (1+z) \gamma^{-1}=1+z'$. Put $\calA:=\F[\alpha]$.
If $\gamma \in \F^\times$ then $z=z'$, so Lemma \ref{lemma:equalityofsharp} readily yields $\alpha=\alpha'$.

From now on, we assume that $\gamma$ is nonscalar and we use the
Strong Units Theorem to find a reduced decomposition
$\gamma=\gamma_1 \cdots \gamma_n$ (see Definition \ref{def:reduceddecomposition}).
Denote by $\calC$ the basic subalgebra such that $\gamma_n \in \SB(\calC) \setminus \F^\times$.
If $\calC \neq \calA$ then the Strong Units Theorem allow us to see that
$\gamma_1 \cdots \gamma_n (1+z) \gamma_n^{-1} \cdots \gamma_1^{-1}$ is not a semi-basic unit
(the decomposition cannot be simplified).
Therefore $\calC=\calA$.

If in addition $n>1$ then as $\gamma_n (1+z) \gamma_n^{-1} \in \SB(\calA) \setminus \F$ 
we see that $\gamma_1 \cdots \gamma_n (1+z) \gamma_n^{-1} \cdots \gamma_1^{-1}$ is not a semi-basic unit.
Hence $\gamma \in \SB(\calA)$ and it follows that $1+z' \in \SB(\calA)$.

Now, we take \emph{specialized} decomposition
$\gamma=\gamma'_1 \cdots \gamma'_m$ (see Definition \ref{def:specializeddecomposition})

Assume first that $\gamma'_1$ is not a basic unit, and denote by $\beta$ the nontrivial idempotent of $\calA$ such that $\gamma'_m \in \SB(\beta)$.
If $\beta \neq \alpha$ then we observe that
$\gamma'_1 \cdots \gamma'_m (1+z) (\gamma'_m)^{-1} \cdots (\gamma'_1)^{-1}$ is a specialized decomposition
with length greater than $1$, which contradicts our assumptions.
Otherwise, because $\SB(\alpha)$ is commutative we see that
$\gamma (1+z) \gamma^{-1}=\gamma'_1 \cdots \gamma'_{m-1}(1+z)  (\gamma'_{m-1})^{-1} \cdots (\gamma'_1)^{-1}$
is a specialized decomposition, and it has length greater than $1$ unless $\gamma=\gamma'_m$, in which case
$\gamma (1+z) \gamma^{-1}=1+z$ and again we conclude that $z=z'$ and $\alpha=\alpha'$.

Assume finally that $\gamma'_1$ is a basic unit.
Conjugating a specialized decomposition in $\SSB(\calA)$ by a basic unit of $\calA$ yields a specialized decomposition of the same length,
so we find that the only option is that either $m=1$, or $m=2$ and $\beta=\alpha$,
and in both cases we find $1+z'=\gamma (1+z) \gamma^{-1}=\gamma'_1 (1+z) (\gamma'_1)^{-1}$.
We conclude by Proposition \ref{prop:actionbasiconsemibasic} that
$1+z'=1+\delta z$ for some $\delta \in \F^\times$. Hence $z \sim z'$, and by Lemma \ref{lemma:equalityofsharp} we infer that $\alpha=\alpha'$.
\end{proof}

We can now conclude the study of the conjugacy classes of quadratic elements in $\calW_{p,q}$.

\begin{theo}\label{theo:conjugacyclassesquadratic}
Let $x \in \calW_{p,q} \setminus \F$ be quadratic, with minimal polynomial $r$.
\begin{enumerate}[(a)]
\item If $r$ does not split with a double root, then the conjugacy class of $x$ contains a unique basic vector.
\item If $r$ splits with a double root $\lambda$, then either the normalization of $x-\lambda$
is conjugated to a unique basic vector of a degenerate basic subalgebra or it
is conjugated to a normalized vector of $\alpha^\sharp$ for a unique nontrivial basic idempotent $\alpha$.
In the second case, the conjugacy class of $x$ is uniquely determined by $\alpha$, $\lambda$ and by the
modular norm $s$ of $x-\lambda$.
\end{enumerate}
\end{theo}

In particular, the conjugacy classes of the \emph{normalized} vectors of $\calW_{p,q} \setminus \F$
with square zero are in one-to-one correspondence with the set consisting of the nonzero basic square-zero elements
and of the nontrivial basic idempotents.

\begin{Rem}
We have finished the study of the conjugacy classes of quadratic elements of $\calW_{p,q} \setminus \F$,
and hence of the conjugacy classes of all the algebraic elements (see Proposition \ref{prop:algebraicimpliesquadratic}).

One could also investigate the orbits of the quadratic elements of $\calW_{p,q} \setminus \F$
under the action of the full automorphism group $\Aut(\calW_{p,q})$.
By combining the previous result with the study of basic automorphisms,
one sees that for each monic polynomial $r \in \F[t]$ of degree $2$ that does not split with a double root,
there is at most one orbit of elements that are annihilated by $r$, and there is exactly one if and only if the algebra $\F[t]/(r)$
is isomorphic to one of the basic subalgebras.
If we consider $r=t^2$ and the normalized elements that are annihilated by $r$,
then we can combine the Automorphism Theorem with the previous study to see that, given a nontrivial basic idempotent $\alpha$,
no normalized vector of $\alpha^\sharp$ is the image of a basic vector under an automorphism,
and hence:
\begin{itemize}
\item Either one of $p$ and $q$ splits with simple roots and the other one splits with a double root,
in which case there are exactly two orbits of normalized square-zero elements under the action
of $\Aut(\calW_{p,q})$ (the one of a basic nonzero square-zero element, and the one of
any normalized element in $\alpha^\sharp$ for some nontrivial basic idempotent $\alpha$).
\item Or there is at most one orbit of normalized square-zero elements under the action
of $\Aut(\calW_{p,q})$.
\end{itemize}
As for non-normalized square-zero elements, the study of their orbits under the action of $\Aut(\calW_{p,q})$
is more difficult because of the fact that not all the automorphisms leave the elements of the center invariant, and hence
the modular norm is in general not an invariant under this action.
This requires a case-by-case study that we prefer leaving out, but there is no fundamental difficulty there.
\end{Rem}

\subsection{Conjugacy classes of finite-dimensional subalgebras}\label{section:misc:finitedimsubalg}

Now that we have examined the conjugacy classes of quadratic elements, we turn to conjugacy classes of
finite-dimensional subalgebras.

We start with the two-dimensional subalgebras, for which the result directly ensues from the previous ones.

\begin{prop}
Let $\calC$ be a $2$-dimensional subalgebra of $\calW_{p,q}$.
If $\calC$ is nondegenerate then it is conjugated to a unique basic subalgebra.

If $\calC$ is degenerate then either $\calC$ is conjugated to a unique basic subalgebra
or there is a unique nontrivial basic idempotent $\alpha$ such that $\calC$ is conjugated to a linear subspace of $\calU(\alpha)$,
moreover if in addition all the nonzero elements of $\calC$ are normalized then in the second stated case $\calC$ is actually conjugated to $\calU(\alpha,\F)$.
\end{prop}

The second case will be dealt with in greater depth shortly (by considering the situation where $\calC$ contains nonzero elements that are not normalized).

As for the orbits under the action of the full automorphism group $\Aut(\calW_{p,q})$,
it is easy to obtain the following result:

\begin{prop}
\begin{enumerate}[(a)]
\item Let $\calC$ be a nondegenerate $2$-dimensional algebra over $\F$.

If there is a subalgebra of $\calW_{p,q}$ that is isomorphic to $\calC$,
then these subalgebras form a single orbit under the action of $\Aut(\calW_{p,q})$.

\item There is a degenerate $2$-dimensional subalgebra of $\calW_{p,q}$ if and only if at least one of $p$ and $q$ splits.
In that case, these subalgebras that contain a normalized square-zero element form a single orbit under the action of $\Aut(\calW_{p,q})$ unless
both $p$ and $q$ split and exactly one has a double root, in which case they form two orbits.
\end{enumerate}
\end{prop}

Now, we turn to subalgebras of larger (finite) dimension, which exist only if one of $p$ and $q$ splits.
We start with the subalgebras of degenerate type (see Section \ref{section:finitedimalg}).

\begin{prop}\label{prop:conjugatedegenerate}
Let $\beta$ and $\beta'$ be normalized vectors such that $N(\beta')=N(\beta)=0$,
and let $V$ and $V'$ be nonzero $\F$-linear subspaces of $C$. For
$\calU(\beta,V)$ to be conjugated to $\calU(\beta',V')$, it is necessary and sufficient that
$V'=V$ and $\beta'$ be conjugated to $\lambda \beta$ for some $\lambda \in \F^\times$.

Moreover, if $\beta$ and $\beta'$ are basic then the subalgebras $\calU(\beta,V)$ and $\calU(\beta',V')$ are conjugated if and only if they are equal.
\end{prop}

\begin{proof}
We start with the first statement, in which the sufficiency of the stated condition is obvious.
Assume conversely that there exists $\gamma \in \calW_{p,q}^\times$
such that $\calU(\beta',V')=\gamma \calU(\beta,V) \gamma^{-1}=\calU(\gamma \beta \gamma^{-1},V)$.
Take a zero divisor $x \in \calU(\beta,V)$ with modular norm $s$ of least degree, and denote by $x_n$ its normalization.
Then $\gamma x \gamma^{-1}$ is a zero divisor with modular norm $s$ of least degree in $\calU(\beta',V') \setminus \{0\}$.
Then, for the normalization $x'_n$ of $\gamma x \gamma^{-1}$ we have $\gamma x \gamma^{-1}=s x'_n$ and hence
$\gamma x_n \gamma^{-1}=x'_n$. Since $x_n \sim \beta$ and $x'_n \sim \beta'$ we deduce that
$\gamma \beta \gamma^{-1} \sim \beta'$.
Hence $\calU(\beta',V')=\calU(\gamma \beta \gamma^{-1},V)=\calU(\beta',V)$, and it is clear from here that $V'=V$.

The last statement then ensues from the first one and from Proposition \ref{prop:conjugatetoatmostonebasic}.
\end{proof}

As a consequence, we conclude the study of degenerate finite-dimensional subalgebras of $\calW_{p,q}$:

\begin{theo}
Assume that one of $p$ and $q$ splits.
Let $n \geq 2$, and $\calC$ be a finite-dimensional subalgebra of degenerate type of $\calW_{p,q}$
with dimension $n$.
Then one and only one of the following holds:
\begin{enumerate}[(i)]
\item $\calC$ is conjugated to a unique subalgebra of the form $\calU(\beta,V)$, where $\beta$
is a nonzero basic vector such that $\beta^2=0$.
\item $\calC$ is conjugated to a unique subalgebra of the form $\calU(\beta,V)$, where
$\beta$ is a normalized vector of $\alpha^\sharp$ for a nontrivial basic  idempotent $\alpha$.
\end{enumerate}
The two cases are possible only if both $p$ and $q$ split and exactly one has simple roots.
\end{theo}

The orbits for the action of the automorphism group $\Aut(\calW_{p,q})$ are in general less abundant
than the ones for its subgroup of inner automorphism, but their classification is tedious because it requires
a case-by-case study. We will not undertake this discussion.

We finish by classifying the finite-dimensional subalgebras of idempotent type of $\calW_{p,q}$:

\begin{theo}\label{theo:conjugtypeidempotent}
Assume that one of $p$ and $q$ splits with simple roots.
Let $n \geq 3$ and $\calA$ be an $n$-dimensional finite-dimensional subalgebra of $\calW_{p,q}$
of idempotent type.
Then $\calA$ is conjugated to $\calH(\alpha,V)$
for a unique nontrivial basic idempotent $\alpha$ and a unique $\F$-linear subspace $V$ of $\alpha^\sharp$.
\end{theo}

Before we prove the result, it is useful to note that if $\alpha$ is a nontrivial idempotent in $\calW_{p,q}$ and $\Phi$ is an arbitrary
automorphism of $\calW_{p,q}$, the equality $\Phi(\alpha^\sharp)=\Phi(\alpha)^\sharp$ clearly holds thanks to any one of the
various characterizations from Lemma \ref{lemma:idempotentlemma}.

\begin{proof}
We know that $\calA=\calH(\beta,V)$ for some nontrivial idempotent $\beta$ and some $\F$-linear subspace $V$ of $C$, with $\dim V =n-2$.
By Theorem \ref{theo:conjugateofquadratictobasic}, there exists $\gamma \in \calW_{p,q}^\times$ such that $\alpha:=\gamma \beta \gamma^{-1}$ is basic. Then
$\gamma \calA \gamma^{-1}=\calH(\gamma \beta \gamma^{-1},V')$ for some $\F$-linear subspace $V'$ of $(\gamma \beta \gamma^{-1})^\sharp$.
This proves the existence part of the theorem.

We now turn to uniqueness. Let $\alpha$ and $\alpha'$ be
nontrivial basic idempotents, $V$ and $V'$ be $(n-2)$-dimensional linear subspaces of $\alpha^\sharp$ and $(\alpha')^\sharp$, respectively,
and assume that there exists a unit $\gamma$ such that $\gamma \calH(\alpha,V)\gamma^{-1}=\calH(\alpha',V')$.
In particular $x \mapsto \gamma x \gamma^{-1}$ takes the nilpotent cone of $\calH(\alpha,V)$ to the one of $\calH(\alpha',V')$,
and hence it takes $\calU(\beta_1,V_1)$ to  $\calU(\beta_2,V_2)$, where $\beta_1$ is an arbitrary generator of
$\alpha^\sharp$, $\beta_2$ an arbitrary generator of $(\alpha')^\sharp$, and $V=V_1 \beta_1$ and $V'=V_2 \beta_2$.
Applying Proposition \ref{prop:conjugatedegenerate}, we obtain $\beta_1 \sim \beta_2$ and $V_1=V_2$,
which yields $V=V'$. Then $\alpha^\sharp=(\alpha')^\sharp$ and we deduce
from Lemma \ref{lemma:equalityofsharp} that $\alpha=\alpha'$.
Hence $\calH(\alpha,V)=\calH(\alpha',V')$, proving the uniqueness statements.
\end{proof}

For the action of the full automorphism group, the discussion is here substantially easier than for subalgebras of degenerate type. The result involves the range of the group homomorphism
$$\Xi : \Aut(\calW_{p,q}) \rightarrow \Aut_\F(C)$$
defined by restricting any automorphism to the center: we studied it in Section \ref{section:automorphismsI}.
The range $\im(\Xi)$ of $\Xi$ naturally acts on any $1$-dimensional $C$-module, with generator $e$, by
$\varphi \mapsto [\lambda e \mapsto \varphi(\lambda) e]$,
which is independent on the choice of generator. The exact nature of $\im(\Xi)$, depending on the specific polynomials $p$ and $q$, has been discussed in
Section \ref{section:automorphismsI}. The statement and proof of the next result require no deep understanding of it.

\begin{theo}
Assume that at least one of $p$ and $q$ splits with simple roots.
Let $n \geq 3$ and $\calA$ be an $n$-dimensional subalgebra of $\calW_{p,q}$ of idempotent type.
Let $\alpha$ be a nontrivial basic idempotent.
Then:
\begin{enumerate}[(a)]
\item There exists a linear subspace
$V$ of $\alpha^\sharp$ and an automorphism $\Phi$ of $\calW_{p,q}$ such that $\Phi(\calA)=\calH(\alpha,V)$.
\item For an arbitrary $\F$-linear subspace $V'$ of $\alpha^\sharp$,
the subalgebras $\calH(\alpha,V)$ and $\calH(\alpha,V')$ are conjugated under the action of $\Aut(\calW_{p,q})$
if and only if $V$ and $V'$ are conjugated under the action of the group $\im(\Xi)$.
\end{enumerate}
\end{theo}

\begin{proof}
We know that $\calA$ is conjugated to $\calH(\alpha',V')$ for some nontrivial basic idempotent $\alpha'$ and some linear subspace $V'$ of
$\alpha^\sharp$. Then there exists a basic automorphism $\Psi$ that takes $\alpha'$ to $\alpha$.
Hence $\Psi(\calH(\alpha',V'))=\calH(\alpha,V)$ for some $\F$-linear subspace $V$ of $\alpha^\sharp$.

Next, we take two $\F$-linear subspaces $V$ and $V'$ of $\alpha^\sharp$.
We shall prove that there exists $\Phi \in \Aut(\calW_{p,q})$ such that $\Phi(\calH(\alpha,V))=\calH(\alpha,V')$
if and only if there exists $\Phi \in \Aut(\calW_{p,q})$ such that $\Phi(\alpha)=\alpha$ and $\Phi(V)=V'$.
The converse implication is obvious. For the direct one, assume that there exists $\Phi \in \Aut(\calW_{p,q})$ such that
$\Phi(\calH(\alpha,V))=\calH(\alpha,V')$. By the Automorphisms Theorem, we can split $\Phi=i_\gamma \circ \Psi$
for some basic automorphism $\Psi$ and some unit $\gamma$. Then $\calH(\alpha,V')=i_\gamma(\calH(\Psi(\alpha),\Psi(V)))$.
Since $\Psi(\alpha)$ is a nontrivial basic idempotent we deduce from Theorem \ref{theo:conjugtypeidempotent} that $\Psi(\alpha)=\alpha$
and $\Psi(V)=V'$. The claimed statement is proved.

We deduce that $\calH(\alpha,V)$ and $\calH(\alpha,V')$ are conjugated under the action of the automorphism group
$\Aut(\calW_{p,q})$ if and only if $V$ and $V'$ are conjugated under the action of the subgroup of $\im \Xi$
consisting of the $\Xi(\Phi)$ where $\Phi \in \Aut(\calW_{p,q})$ leaves $\alpha$ invariant.
In order to conclude, it suffices to prove that this subgroup equals $\im \Xi$.

Let $\Phi \in \Aut(\calW_{p,q})$. We must prove that there exists $\Psi \in \Aut(\calW_{p,q})$ such that $\Xi(\Phi)=\Xi(\Psi)$
and $\Psi(\alpha)=\alpha$. By the Automorphisms Theorem, and since the interior automorphisms act trivially on the center,
we need only consider the case where $\Phi$ is basic. Then, the statement is obvious if
$\Phi(\alpha)=\alpha$, so now we assume that $\Phi(\alpha) \neq \alpha$.
\begin{itemize}
\item If $\Phi$ is positive then $\Phi(\alpha)=\alpha^\star$ because $\F[\alpha]$ is non-degenerate, and it suffices to take $\Psi:=\Phi_\star \circ \Phi$,
where $\Phi_\star$ denotes the pseudo-adjunction (see Section \ref{section:basicCauto}), and the conclusion is obtained by noting that $\Phi_\star$ is a $C$-automorphism.
\item If $\Phi$ is negative then there is a swap $S$ (see Section \ref{section:basicCauto}) that coincides with $\Phi$ on $\F[\alpha]$,
and since every swap is a $C$-automorphism we find that $\Psi:=S^{-1} \circ \Phi$ satisfies our requirements.
\end{itemize}
\end{proof}

\subsection{The center of $\Aut(\calW_{p,q})$}\label{section:centerofautomorphismgroup}

Here, we prove that the only automorphism of $\calW_{p,q}$ that commutes with all the other ones is the identity:

\begin{theo}\label{theo:centerautomorphisms}
The only central element in $\Aut(\calW_{p,q})$ is $\id$.
\end{theo}

As a consequence, the only antiautomorphism of $\calW_{p,q}$ that commutes with all the automorphisms is the adjunction,
and it is also the only antiautomorphism that commutes with all the automorphisms as well as all the antiautomorphisms.

We will need the obvious identity
$$\forall \Phi \in \Aut(\calW_{p,q}), \; \forall \gamma \in \calW_{p,q}^\times, \; \Phi \circ i_\gamma \circ \Phi^{-1}=i_{\Phi(\gamma)}$$
as well as simple observations on the action of basic automorphisms on subgroups of basic units, which are essentially obvious:
\begin{itemize}
\item For every basic automorphism $\Phi$ and every basic zero divisor $\alpha$, one bas $\Phi(\SB(\alpha))=\SB(\Phi(\alpha))$;
\item As a consequence, for every basic automorphism $\Phi$ and every basic subalgebra $\calA$, one has $\Phi(\SB(\calA))=\SB(\Phi(\calA))$.
\end{itemize}

\begin{proof}[Proof of Theorem \ref{theo:centerautomorphisms}]
Let $\Phi \in \Aut(\calW_{p,q})$ be central. By the Automorphisms Theorem, we have a unique decomposition
$\Phi=\Phi_b \circ i_\gamma$ with $\Phi_b \in \BAut(\calW_{p,q})$ and $\gamma \in \calW_{p,q}^\times$.

By the Strong Units Theorem, if $\gamma \not\in \F^\times$ we also consider a reduced decomposition $\gamma=\gamma_1 \cdots \gamma_n$
(see Definition \ref{def:reduceddecomposition}), and denote by $\calA_i$ the basic subalgebra such that $\gamma_i \in \SB(\calA_i)$.

We split the discussion into two cases, whether $\Phi_b$ exchanges the basic subalgebras or not,
i.e., whether it is negative or not (see the terminology in Section \ref{section:automorphismsI}).

\vskip 3mm
\noindent \textbf{Case 1: $\Phi_b$ is negative.} \\
Then $\Phi=\Phi_b \circ \Phi \circ \Phi_b^{-1}= \Phi_b \circ i_{\Phi_b(\gamma)}$
and hence $\gamma \sim \Phi_b(\gamma)$.
If $\gamma$ is nonscalar we note that $\Phi_b(\gamma)=\Phi_b(\gamma_1) \cdots \Phi_b(\gamma_n)$, and it is clear that this is a reduced decomposition, but now
$\Phi_b(\gamma_1) \not\in \SB(\calA_1)$. By the Strong Units Theorem
this is not possible (due to the uniqueness of reduced decompositions up to scalar multiplication).
Hence $\gamma \in \F^\times$ and $\Phi=\Phi_b$.

Then we choose an arbitrary nonscalar element $\alpha \in \SB(\F[a]) \setminus \F$ and we note that the same argument shows that $i_{\alpha}$ does not commute with $\Phi_b$.
This is a contradiction.

\vskip 3mm
\noindent \textbf{Case 2: $\Phi_b$ is positive.} \\
Assume once more that $\gamma$ is nonscalar, and choose an arbitrary nonscalar unit $\gamma_0 \in \SB(\calB) \setminus \F$
where $\calB$ is the basic subalgebra opposite to $\calA_1$.
Then
$$\Phi=i_{\gamma_0}^{-1}\, \Phi\, i_{\gamma_0}=\Phi_b i_{\Phi_b^{-1}(\gamma_0^{-1})} i_{\gamma} i_{\gamma_0}=\Phi_b i_{\gamma'}$$
for $\gamma':=\Phi_b^{-1}(\gamma_0^{-1}) \gamma \gamma_0$. Then $\gamma' \sim \gamma$, but we prove that this is not true. Indeed, if $n$ is odd it is obvious that
$\gamma'=\Phi_b^{-1}(\gamma_0^{-1}) \gamma_1\cdots \gamma_n \gamma_0$ is a reduced decomposition of length $n+2$,
otherwise $n \geq 2$ and either $\gamma_n \gamma_0$ is nonscalar and
$\gamma'=\Phi_b^{-1}(\gamma_0^{-1}) \gamma_1\cdots \gamma_{n-1}(\gamma_n \gamma_0)$
is a reduced decomposition of length $n+1$, or it is a scalar and
$\gamma'=(\gamma_n \gamma_0\Phi_b^{-1}(\gamma_0^{-1})) \gamma_1\cdots \gamma_{n-1}$ is a reduced decomposition.
In any case, the Strong Units Theorem yields a contradiction.

Hence $\gamma \in \F^\times$ and we simply have $\Phi=\Phi_b$, whence $\Phi$ is a positive basic automorphism.

Assume now that $\Phi \neq \id$. Then we find a basic subalgebra $\calC$ on which $\Phi$ is not the identity.
Since $\Phi$ commutes with every inner automorphism, we find that $\Phi(\gamma) \in \F^\times \gamma$
for every unit $\gamma$. We shall find a contradiction by considering $\gamma$ in $\SB(\calC)$.
This requires a case-by-case discussion.

\begin{itemize}
\item If $\calC$ is a field then we find $\forall x\in \calC \setminus \{0\}, \; \Phi(x) \in \F x$,
and since $\Phi$ is linear it is classical that there exists $\lambda \in \F$ such that $\forall x \in \calC, \; \Phi(x)=\lambda x$;
the case $x=1$ then yields $\lambda=1$, thereby contradicting the assumption that $\Phi$ is not the identity on $\calC$.

\item If $\calC$ is degenerate then the zero element and the zero divisors in $\calC$ form a $1$-dimensional $\F$-linear subspace, and $\Phi$ must leave it invariant.
We are then in the same position as in the previous case, and a contradiction is found likewise.

\item Assume finally that $\calC$ splits. Then $\Phi$ is its nonidentity automorphism.
In that case we pick a nontrivial idempotent $\alpha \in \calC$ and an arbitrary $z \in \alpha^\sharp \setminus \{0\}$.
Noting that $\Phi(\alpha)=\alpha^\star$, we see that $\Phi(z) \in (\alpha^\star)^\sharp$.
Now, we remark that the linear subspaces $\F$, $\alpha^\sharp$ and $(\alpha^\star)^\sharp$ are linearly independent.
This is easily deduced from the earlier result that the nilpotent cone of $\F \oplus \alpha^\sharp$
equals $\alpha^\sharp$ (Lemma \ref{lemma:nilpotentcone}), so the intersection with $(\alpha^\star)^\sharp$ must be included in $\alpha^\sharp$.
Yet, Lemma \ref{lemma:equalityofsharp} shows that $\alpha^\sharp \cap (\alpha^\star)^\sharp=\{0\}$.

Finally, there exists $\lambda \in \F^\times$ such that $\Phi(1+z)=\lambda (1+z)$,
yielding $(\lambda-1)+\lambda z-\Phi(z)=0$. By the previous linear independence we obtain a contradiction because $z \neq 0$.
\end{itemize}
Therefore the only option left is that $\Phi=\id$, which completes the proof.
\end{proof}

\begin{Rem}
If $|\F|>3$ then the last case in the previous proof can be radically shortened: in that case indeed
$\Phi$ induces a projective automorphism of the projective line $\calC/\F^\times$ with at least three fixed points, so the latter is the identity
and one concludes that $\Phi$ is a scalar multiple of the identity, just like in the previous cases.
However, this argument fails if $|\F|\leq 3$: in that case it can be checked that the non-identity automorphism of the $\F$-algebra $\F^2$
maps each unit to a scalar multiple of itself. Hence the use of semi-basic units to circumvent this difficulty.
\end{Rem}

\subsection{Conjugacy classes of involutions in $\Aut(\calW_{p,q})$}\label{section:involutions}

One of our initial motivations for the study of $\calW_{p,q}$ was to understand its involutions, i.e., 
its antiautomorphisms $\Phi$ of $\F$-algebra that satisfy $\Phi^2=\id$. 
Every antiautomorphism $\Psi$ of $\calW_{p,q}$ can be written as the composite $(-)^\star \circ \Phi$ for some automorphism $\Phi$,
and since the adjunction commutes with all the automorphisms we observe that $\Psi^2=\id_{\calW_{p,q}}$ if and only if $\Phi^2=\id_{\calW_{p,q}}$.
The same argument shows that two antiautomorphisms $\Psi_1$ and $\Psi_2$ of $\calW_{p,q}$ are conjugated through an element of $\Aut(\calW_{p,q})$ if and only if 
the automorphisms $(-)^\star \circ \Psi_1$ and $(-)^\star \circ \Psi_2$ are conjugated through an element of $\Aut(\calW_{p,q})$.
Therefore, our problem amounts to determining the conjugacy classes, in the group $\Aut(\calW_{p,q})$,
of the involutory elements, i.e., of the elements of order $1$ or $2$.

Here, we will systematically assume $\car(\F) \neq 2$: The case where $\car(\F)=2$ can be obtained by using the same methods,
but we feel that treating it adds unnecessary complexity.

Like the previous section, this one will use the full power of our two main theorems, the Strong Units Theorem and the Automorphisms Theorem.
We consider the projective groups $\mathrm{P}\SB(\F[a])$,
$\mathrm{P}\SB(\F[b])$, $\mathrm{P}\F[a]^\times$ and $\mathrm{P}\F[b]^\times$, and will naturally identify them with subgroups of
$\Inn(\calW_{p,q})$. The Strong Units Theorem essentially states that $\Inn(\calW_{p,q})$ is a free product of
$\mathrm{P}\SB(\F[a])$ and $\mathrm{P}\SB(\F[b])$.
Likewise, for a split basic subalgebra $\calA$ with nontrivial idempotents $\alpha$ and $\alpha^\star$,
the subgroups $\SB(\alpha)$ and $\SB(\alpha^\star)$ are naturally identified with subgroups of
$\mathrm{P}\SB(\calA)$, and we know that the subgroup generated by their union is an internal free product.
Moreover, we recall that $\mathrm{P}\SB(\calA)=(\SB(\alpha) * \SB(\alpha^\star)) \rtimes \mathrm{P}\F[a]^\times$ internally.
Finally, if $\calA$ is degenerate then $\mathrm{P}\SB(\calA)$ is naturally identified with $\SB(\alpha)$, where $\alpha$ is an arbitrary
zero divisor in $\calA$.

Given an (internal) free product of two subgroups $G_1$ and $G_2$, a reduced decomposition of a non trivial element $\gamma$
of their free product consists in writing $\gamma=\gamma_1\cdots \gamma_n$,
where all factors belong to $G_1$ or $G_2$, and no two consecutive factors belong to the same subgroup.

We can now give first insights into the problem. We let $\Phi \in \Aut(\calW_{p,q})$
be an involution, and consider its unique decomposition $\Phi=\Phi_b \circ i$
where $\Phi_b \in \BAut(\calW_{p,q})$ is the \textbf{basic part} of $\Phi$, and $i \in \Inn(\calW_{p,q})$ is the \textbf{inner part.}
Since $\Inn(\calW_{p,q})$ is a normal subgroup, $\Phi_b$ must be an involution.
This leaves only finitely many possibilities, and we will classify them in Section \ref{section:invol:basic}.
Next, for each involutory basic automorphism $B$, we must classify up to conjugation in $\Aut(\calW_{p,q})$
the involutions in $\Aut(\calW_{p,q})$ with basic part $B$. The case $B=\id$ is very easily deduced
from the classification of quadratic elements up to conjugation, but the other cases are more difficult.
This will require that we prove some general results on semi-direct products of free products with $\Z/2$
(see Lemma \ref{lemma:extensionfreeproductinvol}).

Finally, let us consider an involutory positive basic automorphism $\Phi \in \BAut_+(\calW_{p,q})$,
and let $\calC$ be a basic subalgebra of $\Aut(\calW_{p,q})$. Then
$\Phi$ is the identity on $\calC$ or coincides with the standard involution on it.
As for all $\alpha \in \calC^\times$ we have $\alpha^\star \sim \alpha^{-1}$,
this has the consequence that $\Phi$ induces the identity on the projective group $\mathrm{P} \calC^\times$
if it is the identity on $\calC$, otherwise it induces the inversion $x \mapsto x^{-1}$ of the projective group $\mathrm{P} \calC^\times$.

Finally, like in the preceding section we will make systematic use of the identity
$$\forall \Phi \in \Aut(\calW_{p,q}), \; \forall \gamma \in \calW_{p,q}^\times, \; \Phi \circ i_\gamma \circ \Phi^{-1}=i_{\Phi(\gamma).}$$

\subsubsection{Conjugacy classes of involutions in $\BAut(\calW_{p,q})$}\label{section:invol:basic}

Given a $2$-dimensional $\F$-algebra $\calA$, it is known that there are at most $2$ involutions
of $\calA$.
Hence, there are essentially three types of involutions in $\BAut(\calW_{p,q})$:
\begin{itemize}
\item The identity.
\item The pseudo-adjunction (see Section \ref{section:basicCauto}), denoted by $\Phi_\star$.
\item The swaps (see Section \ref{section:basicCauto}), if (and only if) $\F[a]$ and $\F[b]$ are isomorphic.
\item The \textbf{unbalanced involutions}, i.e., the positive basic automorphisms that induce the identity on
one basic subalgebra but not on both. Note that in some cases the pseudo-adjunction is unbalanced.
\index{Unbalanced involutions}
\end{itemize}
It is clear that the pseudo-adjunction is central in $\BAut(\calW_{p,q})$
because every basic automorphism commutes with the adjunction, and the adjunction and the pseudo-adjunction
coinciding on the basic subalgebras.
As a consequence, the pseudo-adjunction forms a single conjugacy class in $\BAut(\calW_{p,q})$, and of course so does the identity.
Next, it is clear that any conjugation in $\BAut(\calW_{p,q})$ takes a swap to another swap and an unbalanced involution to an unbalanced involution.
The number of unbalanced involutions equals the number algebras among $\F[a]$ and $\F[b]$ that are either separable field extensions of $\F$ or split algebras, and if there are exactly two they are conjugated in $\BAut(\calW_{p,q})$ if and only if there exists a swap.

Assume finally that there exists a swap: then any two swaps are conjugated. Indeed, take two swaps $S_1$ and $S_2$,
and consider the basic automorphism $\Phi$ that is the identity on $\F[a]$ and coincides with $S_2 \circ S_1^{-1}$ on $\F[b]$.
Then $\Phi \circ S_1 \circ \Phi^{-1}$ is a swap that coincides with $S_2$ on $\F[a]$, and hence it equals $S_2$.

\subsubsection{A lemma on extensions of free product}

\begin{lemma}\label{lemma:extensionfreeproductinvol}
Let $G$ be a group, $H$ be a subgroup of $G$ of index $2$, and
$\varepsilon \in G \setminus H$ be an element of order $2$. Assume that $H$ splits internally as the free product of two nontrivial subgroups
$H_1$ and $H_2$, and consider the conjugation $\sigma : x \in H \mapsto \varepsilon x \varepsilon^{-1} \in H$. Then:
\begin{enumerate}[(a)]
\item Every involution of $H$ is conjugated in $H$ to an involution of $H_1$ or $H_2$.
\item If $\sigma$ leaves each one of $H_1$ and $H_2$ invariant, then every involution in $G \setminus H$
is conjugated to an element of the form $\varepsilon y$ for some $y \in H_1 \cup H_2$.
\item If $\sigma$ maps $H_1$ to $H_2$ and $H_2$ to $H_1$, then every involution in $G \setminus H$
is conjugated to $\varepsilon$.
\end{enumerate}
\end{lemma}

We will actually not need the first result, but it is a very good warm up for the next two.

\begin{proof}
Let $h=h_1\cdots h_n$ be a nontrivial involution of $H_1 * H_2$, given in reduced form.
Then $h^{-1}=h_n^{-1} \cdots h_1^{-1}$ is also a reduced decomposition of $h$, and hence $h_i=h_{n+1-i}^{-1}$ for all $i \in \lcro 1,n\rcro$,
which forces $n$ to be odd. Then we write $n=2p+1$ for some integer $p \geq n$ and we note that
$h$ is conjugated to $h_{p+1}$ in $H$.

Consider an involution in $G \setminus H$, which we can then write $\varepsilon x$ for some $x \in H$.
Consider a reduced decomposition $x=h_1 \cdots h_n$. If $x=1$ then nothing needs to be proved. So, we assume $x \neq 1$ from now on.
Then
$$(\varepsilon x)^{-1}=x^{-1} \varepsilon^{-1}=\varepsilon^{-1} \prod_{k=1}^n \sigma(h_{n+1-k}^{-1}).$$
Assume now that $\sigma$ either leaves each $H_i$ invariant, or maps each one into the other.
Then by the definition of an internal free product
$\sigma(h_{n+1-k}^{-1})=h_k$ for all $k \in \lcro 1,n\rcro$, and we distinguish between two cases from here:
\begin{itemize}
\item Assume first that $\sigma$ leaves each one of $H_1$ and $H_2$ invariant. Then $n$ must be odd.
Setting $p:=\frac{n-1}{2}$ and $y:=\prod_{k=p+2}^{2p+1} h_k$, we check that
$$y (\varepsilon x) y^{-1}=\varepsilon \left(\prod_{k=p+2}^{2p+1} \sigma^{-1}(h_k)\right) \prod_{k=1}^{p+1} h_k=\varepsilon h_{p+1}.$$
\item Assume finally that $\sigma$ maps $H_1$ to $H_2$ and $H_2$ to $H_1$.
Then $n$ is even and we check for $p:=\frac{n}{2}$ and $y:=\prod_{k=p+1}^{2p} h_k$ that $y (\varepsilon x) y^{-1}=\varepsilon$.
\end{itemize}
\end{proof}

\subsubsection{Conjugacy classes of involutory inner automorphisms}\label{section:invol:inner}

Here we determine the conjugacy classes of involutory automorphisms of $\calW_{p,q}$ among the inner automorphisms.
This is intimately connected with the study of conjugacy classes of nonscalar quadratic elements.

Let $x \in \calW_{p,q}^\times$. Then $(i_x)^2=\id$ if and only if $x^2 \in \F^\times$, and
for every $y$ in $\calW_{p,q}$ the fact that $i_x$ is conjugated to $i_y$ in $\Inn(\calW_{p,q})$
is equivalent to the existence of a unit $z$ such that $x \sim z y z^{-1}$.
Now, if $x \not\in \F$ then $x^2 \in \F^\times$ holds if and only if $x$ is quadratic with trace zero
(because here we have $\car(\F) \neq 2$).
Moreover, in a $2$-dimensional $\F$-algebra $\calA$, the nonscalar elements that satisfy this property
form a $1$-dimensional subspace without the zero vector unless $\calA$ is degenerate (in which case
no unit in $\calA$ has trace zero).
Hence, there is exactly one involution in $\mathrm{P}\calA^\times$ if $\calA$ is degenerate,
and exactly two otherwise.

Now let $x \in \calW_{p,q}^\times \setminus \F$ be quadratic and with trace zero. Then $\F[x]$
is nondegenerate, and hence we know that $x$ is conjugated to a basic vector (see point (a) of Theorem \ref{theo:conjugacyclassesquadratic}), which is then of course
quadratic with trace zero. Hence, every involutory inner automorphism that is not the identity is conjugated to
a nontrivial involutory element of $\mathrm{P}\F[a]^\times$ or of $\mathrm{P}\F[b]^\times$.
Let us conclude:

\begin{prop}
There are at most two conjugacy classes of nontrivial involutions of $\Inn(\calW_{p,q})$ in $\Aut(\calW_{p,q})$.
More precisely:
\begin{itemize}
\item There is no such conjugacy class if and only if both $\F[a]$ and $\F[b]$ are degenerate.
\item There is exactly one conjugacy class if and only if exactly one of $\F[a]$ and $\F[b]$ is degenerate
or none is degenerate and $\F[a]$ and $\F[b]$ are isomorphic.
\end{itemize}
\end{prop}

\subsubsection{Automorphisms with nontrivial basic part}

We turn to the classification of the involutions in $\Aut(\calW_{p,q})$ with nontrivial basic part.

The easiest situation occurs when the basic part is a swap.

\begin{prop}
Let $\Phi \in \Aut(\calW_{p,q})$ be an involution whose basic part $S$ is a swap. Then
$\Phi$ is conjugated to $S$.
\end{prop}

\begin{proof}
We note that Lemma \ref{lemma:extensionfreeproductinvol} can be applied in the subgroup $G:=\{\id,S\}\cdot \Inn(\calW_{p,q})$,
with $H_1:=\mathrm{P}\SB(\F[a])$ and $H_2:=\mathrm{P}\SB(\F[b])$. Indeed, since $S$
is a swap the conjugation $Z \mapsto S Z S^{-1}$ clearly exchanges $H_1$ and $H_2$.

As $\Phi=S i$ for some $i \in H_1 * H_2$, we deduce
from point (c) of Lemma \ref{lemma:extensionfreeproductinvol} that $\Phi$ is conjugated to $S$ in $G$ (and hence in $\Aut(\calW_{p,q})$).
\end{proof}

We continue with the more difficult case of a positive basic automorphism $\Phi$ that is not the identity.

\begin{lemma}\label{lemma:invol:reductionpositive}
Let $\Phi \in \Aut(\calW_{p,q})$ be an involution with basic part $B$ in $\BAut_+(\calW_{p,q}) \setminus \{\id\}$.
Then there is a basic subalgebra $\calA$ and an element
$i \in \mathrm{P}\calA^\times$ such that $\Phi$ is conjugated to $B\, i$ in $\Aut(\calW_{p,q})$.
Moreover, if $B$ is the identity on $\calA$ then $i$ is an involution.
\end{lemma}

\begin{proof}
Once more, we can apply Lemma \ref{lemma:extensionfreeproductinvol}, but this time around we are in the situation of statement (b).
This yields a basic subalgebra $\calA$ and an element $i \in \mathrm{P}\SB(\calA)$ such that
$\Phi$ is conjugated to $B i$.
The proof of the first point is complete if $\calA$ is a field, but we need to go further if $\calA$ splits or is degenerate.

In any case, we note that $B i$ is an involution and we will analyze this.

Assume first that $\calA$ is degenerate and denote by $\beta$ a zero divisor in it.
Note then that $B$ is the identity on $\calA$, and hence $B=\Phi_\star$ is a $C$-automorphism.
Hence $B$ is the identity on $C+C\calA$, to the effect that it commutes with $i$.
Now, from $(Bi)^2=\id$ we deduce that $i$ is an involution. Finally, since
$\mathrm{P}\SB(\calA) \simeq \SB(\beta)$ is isomorphic to $(\F[t],+)$, with $\car(\F) \neq 2$, it contains no element of order $2$, and hence $i=\id$.
We conclude that $Bi=B$.

Assume finally that $\calA$ splits, and choose $\alpha$ as one of its nontrivial idempotents.
Assume that $i \not\in \mathrm{P}\calA^\times$.
We consider a decomposition $i=u_0 u_1 \cdots u_n$, where $u_1 \cdots u_n$ is a reduced decomposition with respect to the subgroups $\SB(\alpha)$ and $\SB(\alpha^\star)$, and $u_0 \in \mathrm{P}\calA^\times$ (possibly $u_0=\id$).
Then we write
$$(Bi)^{-1}=u_n^{-1} \cdots u_1^{-1} (Bu_0)^{-1} =(Bu_0)^{-1} \sigma(u_n^{-1}) \cdots \sigma(u_1^{-1})$$
for $\sigma : x \mapsto (B u_0) x (B u_0)^{-1}$, and finally $(Bu_0)^{-1}=u_0^{-1}B^{-1}=B^{-1} j_0$ for some $j_0 \in \mathrm{P}\calA^\times$.
We note that $\sigma$ either leaves each one of $\SB(\alpha)$ and $\SB(\alpha^\star)$ invariant, or it exchanges them.
It follows from the Semi-Basic Units theorem that $(Bu_0)^{-1}=(Bu_0)$, and hence $\varepsilon:=B u_0$ is an involution.

Now, we must discuss whether $B$ is the identity on $\calA$ or not.
If $B (\alpha)=\alpha^\star$ then $\sigma$ exchanges $\SB(\alpha)$ and $\SB(\alpha^\star)$,
and hence point (c) of Lemma \ref{lemma:extensionfreeproductinvol} yields that $Bi$ is conjugated to $B u_0$, and we are done.

Assume now that $B(\alpha)=\alpha$.
Then point (b) of Lemma \ref{lemma:extensionfreeproductinvol} shows that
$B i$ is conjugated to $(B u_0) i'$ for some $i'$ in $\SB(\beta)$, with $\beta \in \{\alpha,\alpha^\star\}$.
We also remember that $B':=(B u_0)$ is an involution. We shall prove that $(Bu_0)i'$ is conjugated to $Bu_0$.
We know that $u_0$ and $B$ leave $\SB(\beta)$ invariant because they fix $\beta$,
hence $B'$ leaves $\SB(\beta)$ invariant, and the resulting automorphism of $\SB(\beta)$ is then an involution
which we denote by $j \mapsto \overline{j}$.
Now since $B' i'$ is an involution we get $\overline{i'} i'=1$.
Moreover, for all $j \in \SB(\beta)$, we find $j (B'i') j^{-1}=B' \overline{j} i' j^{-1}$.
Hence it suffices to prove that $\overline{j} i' j^{-1}=\id$ for some $j \in \SB(\beta)$.
Yet here we observe that $\SB(\beta)$ is isomorphic to $\F[t]$, and since $\car(\F) \neq 2$
it follows that the mapping $x \mapsto x^2$ is an automorphism of $\SB(\beta)$ (and of course $\SB(\beta)$ is commutative).
Then we take the square root $j$ of $i'$ in $\SB(\beta)$, observe that $(\overline{j} j)^2=\overline{i'} i'=1$,
to the effect that $\overline{j} j=1$, and we deduce that $\overline{j} i' j^{-1}=j^{-2} i'=1$.
Hence $Bi$ is conjugated to $Bu_0$.

Finally, since $B$ is the identity on $\calA$, we note that
$B$ and $u_0$ commute, and as $Bu_0$ and $B$ are involutions we infer that $u_0$ is an involution.
\end{proof}

Conversely, if $B \in \Aut_+(\calW_{p,q}) \setminus \{\id\}$ is an involution that is the identity on
some basic subalgebra $\calA$, then $Bi$ is an involution for every involution $i \in \mathrm{P}\calA^\times$.

Finally, let $B \in \Aut_+(\calW_{p,q}) \setminus \{\id\}$ be an involution, and
$i \in \mathrm{P}\calA^\times$ be an arbitrary element, where $\calA$ is a basic subalgebra on which
$B$ is \emph{not} the identity. Then $B i B^{-1}=i^{-1}$, and hence $Bi$ is involutory!
Hence the lack of constraint on $i$ in Lemma \ref{lemma:invol:reductionpositive} in that case.

We conclude that every involution in $\Aut(\calW_{p,q})$ is conjugated to an element of
the form $B i$ where $B$ is an involutary basic automorphism, $i$ is an element of $\mathrm{P}\calA^\times$
for some basic subalgebra $\calA$, and we can (and must!) take $i=\id$ in case $B$ is a swap,
and $i$ involutory if $B$ is the identity on $\calA$.
We have also determined the conjugacy class of $Bi$ when $B=\id$ (see Section \ref{section:invol:inner}).

It remains to understand the conjugacy class of $Bi$ when $B \in \BAut_+(\calW_{p,q}) \setminus \{\id\}$ is involutory and $i \in \mathrm{P}\calA^\times$
(not necessarily involutory!)
for some basic subalgebra $\calA$. First of all, if there exists a swap $S$, then
$S (Bi) S^{-1}=(S B S^{-1}) (SiS^{-1})$, and $Si S^{-1} \in \mathrm{P}\calB^\times$ where $\calB$ stands for the basic subalgebra opposite to $\calA$.
In that case, we deduce that every involution in $\Aut(\calW_{p,q})$ whose basic component is positive
is conjugated to $B i$ for some $B \in \BAut_+(\calW_{p,q}) \setminus \{\id\}$ and some $i \in \F[a]^\times$.
Finally, the conjugacy classes of the involutory elements of $\BAut_+(\calW_{p,q}) \setminus \{\id\}$
are easily found, and we will abstain from giving any details.

Hence, essentially one question remains: given an involution $B \in \BAut_+(\calW_{p,q}) \setminus \{\id\}$,
and given inner automorphisms $i_1$ and $i_2$ associated with basic units so that $B i_1$ and $B i_2$ are involutions,
when are the elements $Bi_1$ and $Bi_2$ conjugated in $\Aut(\calW_{p,q})$?
The answer depends on the precise nature of $B$, and it is quite natural to conjecture the answer.
To simplify the discourse, say that $i_1$ and $i_2$ have conjugators in the same basic subalgebra $\calA$.
If $B$ is the identity on $\calA$ then as $\mathrm{P}\calA^\times$ contains at most two involutions it is tempting to think that
$i_1=i_2$ whenever $Bi_1$ and $Bi_2$ are conjugated. This is indeed the case (see Proposition \ref{prop:invol:conjugcasident} below).
The least obvious case is the one where $B$ is not the identity on $\calA$, as then
the assumption that $B i_1$ is involutory does not restrict the choice of $i_1$.
However, by taking $j \in \mathrm{P}\calA^\times$ we note that $j^{-1} (B i_1) j=B j i_1 j=B (i_1 j^2)$ because $\mathrm{P}\calA^\times$
is commutative. Hence the conjugacy class of $B i_1$ depends only on the coset of $i_1$ modulo the
subgroup of squares $(\mathrm{P}\calA^\times)^{[2]}$, and the conjecture is that this coset
encodes the said conjugacy class (of course the simple computation we have just performed is not sufficient to prove this conjecture, as
we should consider conjugating with more general elements of $\Aut(\calW_{p,q})$).
The following three propositions will confirm these conjectures:

\begin{prop}\label{prop:invol:conjugcasident}
Let $B \in \BAut_+(\calW_{p,q}) \setminus \{\id\}$, and let $\calA$ be a basic subalgebra on which $B$ is the identity.
Let $i_1$ be an involution in $\mathrm{P} \calA^\times$ and let $i_2$ be another involutory inner automorphism associated with a basic unit.
Then $B i_1$ and $B i_2$ are conjugated in $\Aut(\calW_{p,q})$ if and only if $i_1=i_2$.
\end{prop}

\begin{proof}
Assume that $B i_2=\Phi B i_1 \Phi^{-1}$ for some automorphism $\Phi$.
We split $\Phi= j \Phi_b$ for some basic automorphism $\Phi_b$ and some
inner automorphism $j$. Then $\Phi_b$ commutes with $B$. It follows that
$\Phi_b$ cannot be a negative basic automorphism, otherwise $B$ would be the identity on the
basic subalgebra opposite to $\calA$, and hence would be the identity of $\calW_{p,q}$.
It follows that $\Phi_b$ is positive. Remembering that $\Phi_b$ induces the identity or the inversion in the group $\mathrm{P} \calA^\times$,
we find that it commutes with the involution $i_1$, and hence
$B i_2=j (Bi_1) j^{-1}$ directly. It follows that $i_2=(B^{-1} j B) i_1 j^{-1}$.
If $j=\id$ then $i_2=i_1$ and we are done. Assume now that $j \neq \id$.
Now, let us consider a reduced decomposition $j=j_1 \cdots j_p$ with respect to the decomposition of $\Inn(\calW_{p,q})$
as a free product of $\mathrm{P}\SB(\calA)$ and $\mathrm{P}\SB(\calB)$, where $\calB$ stands for the basic subalgebra opposite to $\calA$.
Then
$$i_2=(B^{-1} j_1 B) \cdots (B^{-1} j_p B) i_1 j_p^{-1} \cdots j_1^{-1}.$$
To start with, we assume that $i_1 \neq \id$, which discards the possibility that $\calA$ be degenerate.

Note that if $j_p \not\in \mathrm{P}\SB(\calA)$ then the above is a reduced decomposition of $i_2$ with
size greater than $1$, which is impossible.
Moreover, if $j_p \in \mathrm{P}\SB(\calA)$, $(B^{-1} j_p B) i_1 j_p^{-1} \neq \id$
and $p>1$, then we observe a reduced decomposition of $i_2$ with size at least $3$, and again this is not possible.

Now, we will observe that $(B^{-1} j_p B) i_1 j_p^{-1} \neq \id$ in any case.
\begin{itemize}
\item If $\calA$ is a field, then $\mathrm{P}\SB(\calA)=\mathrm{P} \calA^\times$ is commutative and $B$ commutes with every element of it, so it is clear that
$(B^{-1} j_p B) i_1 j_p^{-1}=i_1 \neq \id$.
\item Assume now that $\calA$ splits.
Then, by moding out $\SSB(\calA)$, we obtain a quotient that is isomorphic to $\mathrm{P}\calA^\times$,
and in this quotient the cosets of $B^{-1} j_p B$ and $j_p$ are equal because $B$ is the identity on $\mathrm{P}\calA^\times$.
Hence $(B^{-1} j_p B) i_1 j_p^{-1}\neq \id$ in that case also.
\item Finally, if $\calA$ is degenerate then we immediately have that $i_1=\id$.
\end{itemize}

We conclude that $i_1=\id$ or $p=1$.
Note then that $\Phi^{-1}=\Phi_b^{-1} j^{-1}=j' \Phi_b^{-1}$
where $j':=\Phi_b^{-1} j^{-1} \Phi_b$, and observe that $j'$ has a reduced decomposition with length $p$
(we obtain such a decomposition by inverting the above reduced decomposition of $j$, and then by conjugating each factor with $\Phi_b^{-1}$).
Hence, by applying the previous proof in this new situation we find $i_2=\id$ or $p=1$.
Therefore $i_1=\id=i_2$ or $p=1$.

Assume finally that $p=1$. Then $i_2=(B^{-1} j_1 B) i_1 j_1^{-1}$ and $j_1 \in \mathrm{P}\SB(\calA)$, to the effect that 
$i_2 \in \mathrm{P}\SB(\calA)$. If $\calA$ is a field
this directly yields $i_2=i_1$, just like in the above.
If $\calA$ splits then we see that the coset of $(B^{-1} j_1 B) i_1 j_1^{-1}$ mod $\SSB(\calA)$
equals the one of $i_1$, which yields $i_2=i_1$. If $\calA$ is degenerate then $i_2 \in \mathrm{P}\calA^\times$, so $i_2=\id=i_1$
since $i_1$ and $i_2$ are involutory elements of $\mathrm{P}\calA^\times$.
Therefore, $i_1=i_2$ in any case.
\end{proof}

\begin{prop}\label{prop:invol:conjugcasnonidentnoswap}
Assume that $\F[a]$ and $\F[b]$ are nonisomorphic.
Let $B \in \BAut_+(\calW_{p,q}) \setminus \{\id\}$ be an involution, and $\calA$ be a basic subalgebra on which $B$ is not the identity, and denote by
$\calB$ the opposite basic subalgebra.
Let $i_1$ belong to $\mathrm{P} \calA^\times$, and $i_2$ be an inner automorphism attached to a basic unit.
Then $B i_1$ and $B i_2$ are conjugated in $\Aut(\calW_{p,q})$ if and only if
either one of the following situations holds:
\begin{enumerate}[(i)]
\item $i_1=i_2\, j^2$ for some $j \in \mathrm{P}\calA^\times$;
\item $i_1$ is a square in $\mathrm{P}\calA^\times$, and $i_2$ is a square in $\mathrm{P} \calB^\times$.
\end{enumerate}
Moreover, in case (ii) both $Bi_1$ and $Bi_2$ are conjugated to $B$.
\end{prop}

\begin{proof}
We have already proved the converse implication, as well as the fact that
$B i_1$ is conjugated to $B$ whenever $i_1$ is a square in $\mathrm{P}\calA$, which yields the last point.

Now we prove the direct implication. So, we assume that there exists $\Phi \in \Aut(\calW_{p,q})$ such that
$B i_2=\Phi (B i_1) \Phi^{-1}$.
Again, we split $\Phi= j \Phi_b$ for some basic automorphism $\Phi_b$ and some
inner automorphism $j$. The basic automorphism $\Phi_b$ is necessarily positive because $\F[a]$ and $\F[b]$ are nonisomorphic, and in particular it commutes with $B$.
In any case $\Phi_b i_1=i_1^{\varepsilon} \Phi_b$ for some $\varepsilon \in \{1,-1\}$.

Hence $B i_2=B (B^{-1} j B) (i_1^{\varepsilon} j^{-1})$
and it follows that $i_2=(B^{-1} j B) i_1^{\varepsilon} j^{-1}$.
If $j=\id$ then $i_2=i_1^{\varepsilon}$, and outcome (i) holds (if $\varepsilon=-1$, note that $i_1=(i_1)^2 i_1^{-1}$).
Hence, from now on we assume that $j \neq \id$.

We will assume first that $i_1 \neq \id$.
Let us consider a reduced decomposition $j=j_1 \cdots j_p$ with respect to the decomposition of $\Inn(\calW_{p,q})$
as the internal free product of the subgroups $\mathrm{P}\SB(\calA)$ and $\mathrm{P}\SB(\calB)$.
Then $i_2=(B^{-1} j_1 B) \cdots (B^{-1} j_p B) i_1^{\varepsilon} j_p^{-1} \cdots j_1^{-1}$
and we note that if $j_p \not\in \mathrm{P}\SB(\calA)$ then we have just found a reduced decomposition of $i_2$ with
size greater than $1$, which is impossible. Hence $j_p \in \mathrm{P}\SB(\calA)$.

\begin{itemize}
\item \textbf{The case where $p=1$.} Then $i_2=(B^{-1} j_p B) i_1^{\varepsilon} j_p^{-1}$,
and we remember that conjugating by $B^{-1}$ induces the inversion of the group $\mathrm{P}\calA^\times$
and that $\calA$ is not degenerate because $B$ does not induce the identity on it. Hence, by moding out the subgroup $\SSB(\calA)$ if $\calA$ splits, or otherwise directly, we deduce that
$i_2=k^{-1} i_1^{\varepsilon} k^{-1}$ for some $k \in \mathrm{P}\calA^\times$.
If $\varepsilon=1$ we deduce that $i_2=(k^{-1})^2 i_1$, otherwise
$i_2=((i_1 k)^{-1})^2 i_1$. In any case we find that (i) holds.

\item \textbf{The case where $p>1$.}
Then we must have $(B^{-1} j_p B) i_1^{\varepsilon} j_p^{-1}=\id$,
otherwise in the above we find a reduced decomposition of $i_2$ with length $2p-1>1$.
Note already that this requires, by the same technique as in the previous case, that $i_1 \in (\mathrm{P} \calA^\times)^{[2]}$.
Then, thanks to $(B^{-1} j_p B) i_1^{\varepsilon} j_p^{-1}=\id$,
we continue by downward induction and find $i_2=(B^{-1} j_1 B)  j_1^{-1}$, which yields that $i_2 \in (\mathrm{P}\calC^\times)^{[2]}$
for some basic subalgebra $\calC$. Hence, outcome (i) occurs if $\calC = \calA$, otherwise outcome (ii) occurs.
\end{itemize}

Assume finally that $i_1=\id$. Denote by $\calC$ the basic subalgebra such that $i_2 \in \mathrm{P} \calC$.
\begin{itemize}
\item If $B$ is not the identity on $\calC$, reversing the roles of $i_1$ and $i_2$ shows that $i_2 \neq \id$ would lead to one of conclusions (i) and (ii),
whereas outcome (i) holds if $i_2=\id$.
\item If $B$ is the identity on $\calC$, then we directly know by Proposition \ref{prop:invol:conjugcasident}
that $i_2=i_1$, and hence outcome (i) holds.
\end{itemize}
\end{proof}

\begin{prop}\label{prop:invol:conjugcasnonidentswap}
Assume that $\F[a]$ and $\F[b]$ are isomorphic.
Let $B \in \BAut_+(\calW_{p,q}) \setminus \{\id\}$ be an involution, and $\calA$ be a basic subalgebra on which $B$ is not the identity.
Let $i_1,i_2$ belong to $\mathrm{P} \calA^\times$.
Then $B i_1$ and $B i_2$ are conjugated in $\Aut(\calW_{p,q})$ if and only if $i_1=i_2 j^2$ for some $j \in \mathrm{P}\calA^\times$.
\end{prop}

\begin{proof}
We already know that the converse implication holds. Assume now that $B i_1$ and $B i_2$ are conjugated in $\Aut(\calW_{p,q})$.
Assume that there exists $\Phi \in \Aut(\calW_{p,q})$ such that
$B i_2=\Phi (B i_1) \Phi^{-1}$.
Again, we split $\Phi= j \Phi_b$ for some involutory basic automorphism $\Phi_b$ and some
inner automorphism $j$. If $\Phi_b$ is not a swap, we can follow the chain of arguments from the proof of Proposition
\ref{prop:invol:conjugcasnonidentnoswap}, and end up
with the stated outcome because we have assumed that $i_2 \in \mathrm{P} \calA^\times$.

It remains to consider the case where $\Phi_b$ is a swap.
Then we note that $B$ commutes with $\Phi_b$, and hence $B$ is not the identity on $\calB$ (since it is not the identity on $\calA$).
Set $i'_1:=\Phi_b i_1 \Phi_b^{-1}$, which belongs to $\mathrm{P}\calB^\times$, where $\calB$ stands for the basic subalgebra opposite to $\calA$.
Now $i_2=(B^{-1} j B) i'_1 j^{-1}$.
Then we follow the same method as in the proof of Proposition \ref{prop:invol:conjugcasnonidentnoswap}, but because
$i_2$ and $i'_1$ are associated with distinct basic subalgebras the case $p=1$ now
requires that $i_2=1$ and that $i'_1$ is a square in $\mathrm{P}\calB^\times$, to the effect that $i_1$ is a square in $\mathrm{P}\calA^\times$.
The case $p>1$ yields exactly the same results, and the conclusion follows.
\end{proof}

\subsubsection{Conclusion}

Now, we can conclude our study of involutions.

We roughly put the orbits of non-identity involutions into three categories:
\begin{itemize}
\item The orbits of the basic automorphisms that are involutory (but not the identity);
\item The orbits of the inner automorphisms attached to basic units (and that are involutory but different from the identity);
\item The orbits of \textbf{mixed type}, that contain an element of the form $B i$
in which $B$ is a positive involutory basic automorphism that is not the identity, and $i$ is an inner automorphism
attached to a basic unit, but that contain no basic automorphism.
\end{itemize}
The classification of the two former types has already been explained.
As for the classification of the latter, it is given in Propositions \ref{prop:invol:conjugcasident},
\ref{prop:invol:conjugcasnonidentnoswap} and \ref{prop:invol:conjugcasnonidentswap}.
We note finally that if a basic subalgebra $\calA$ is degenerate, then $\mathrm{P}\calA^\times$ is isomorphic to $(\F,+)$,
so every element of $\mathrm{P}\calA^\times$ is a square (and we already knew that no element of $\mathrm{P}\calA^\times$ has order $2$).
This largely simplifies the classification of orbits of mixed type (there is no such orbit if both $\F[a]$ and $\F[b]$ are degenerate).

We split the discussion into two cases, whether there is a swap or not.
If there is a swap and the basic subalgebras are degenerate, then we note that
every element of order $2$ in $\Aut(\calW_{p,q})$ is a swap, and we have already shown that all swaps are conjugated in $\Aut(\calW_{p,q})$.

\begin{table}[H]\label{table2}
\begin{center}
\caption{Conjugacy classes of elements of order $2$ in $\Aut(\calW_{p,q})$ when $\car(\F) \neq 2$
and the basic subalgebras are nonisomorphic}
\begin{tabular}{| c || c | c | c | c |}
\hline
Represen- & Nature  & Nature & Nature & Classifying \\
tative  & of $B$ & of $\calA$ & of $i$ & data \\
\hline
\hline
 &  & Basic & Element of &  \\
$i$ & - & subalgebra & order $2$ & $i$ \\
 &  & & in $\mathrm{P}\calA^\times$ & \\
\hline
 & Involutory &  &  &  \\
$B$ &  basic  & - & - & $B$ \\
&  automorphism &  &  & \\
& $\neq \id$ & & & \\
\hline
 & Involutory  &  &   & \\
&  basic  & Basic & Element of  & \\
$B i$ &  automorphism & subalgebra & order $2$ & $B,i$ \\
& $\neq \id$ & & in $\mathrm{P}\calA^\times$ & \\
& $B_{|\calA} =\id_\calA$ & &  & \\
\hline
 & Involutory  & Basic  & Non-square  & $B,\calA,$ \\
$B i$ &  basic   & subalgebra &  element  & coset of $i$ \\
&  automorphism  & of $\mathrm{P}\calA^\times$ &  of $\mathrm{P}\calA^\times$  & in $\mathrm{P} \calA^\times/(\mathrm{P} \calA^\times)^{[2]}$ \\
& $B_{|\calA} \neq \id_\calA$ &  &  &  \\
\hline
\end{tabular}
\end{center}
\end{table}

\begin{table}[H]\label{table3}
\begin{center}
\caption{Conjugacy classes of elements of order $2$ in $\Aut(\calW_{p,q})$ when $\car(\F) \neq 2$
and the basic subalgebras are isomorphic and nondegenerate}
\begin{tabular}{| c || c | c | c |}
\hline
Representative & Nature  & Nature & Classifying \\
  & of $B$ & of $i$ & data \\
\hline
\hline
 &  & Element of &  \\
$i$ & - & order $2$ & None \\
 &  & in $\mathrm{P}\F[a]^\times$ &  \\
\hline
$B$ & Swap  & - & None \\
\hline
$B$ & Pseudo-adjunction & - & None \\
\hline
 & Positive involutory   &  &  \\
$B$ &  basic automorphism & - & None \\
& $\neq \id$, $\neq \Phi_\star$ & & \\
\hline
 & Positive involutory  & Element of & \\
$B i$ & basic automorphism & order $2$  & $B,i$ \\
& $\neq \id$   & in $\mathrm{P}\F[a]^\times$ &  \\
& $B(a)=a$  &  & \\
\hline
 & Positive involutory  & Non-square  & $B,$ \\
$B i$ & basic automorphism  &   element & coset of $i$ \\
& $B(a)=a^\star$   & of $\mathrm{P}\F[a]^\times$ & in $\mathrm{P} \F[a]^\times/(\mathrm{P} \F[a]^\times)^{[2]}$ \\
\hline
\end{tabular}
\end{center}
\end{table}

\printindex

\end{document}